\documentclass{amsart}
\usepackage{amsthm}
\usepackage{fancyhdr}
\usepackage{amssymb}
\usepackage{mathtools}
\usepackage{indentfirst}
\usepackage{amsmath}
\usepackage{color}
\usepackage{bm}
\usepackage{extarrows}
\usepackage[pagewise]{lineno}%\linenumbers
\usepackage{footnote}

\newcommand{\RN}[1]{\mbox{\uppercase\expandafter{\romannumeral #1}}}

\begin{document}
\def\dis{\displaystyle}
\renewcommand{\theequation}{\thesection.\arabic{equation}}
\newtheorem{Theorem}{Theorem}[section]
\newtheorem{Lemma}{Lemma}[section]
\newtheorem{Proposition}{Proposition}[section]
\newtheorem{Corollary}{Corollary}[section]
\newtheorem{Remark}{Remark}[section]
\renewcommand{\thefootnote}{}

\numberwithin{equation}{section}

\title[]{POINTWISE STRUCTURES AND BACKGROUND PERTURBATION ANALYSES FOR THERMAL NON-EQUILIBRIUM IRROTATIONAL FLOWS IN THREE SPACE DIMENSIONS}

\author{Shijin Deng and Xiaochun Yang}
\address{Shijin Deng, School of Mathematical Sciences, CMA-Shanghai, Shanghai Jiao Tong University}
\address{Xiaochun Yang, School of Mathematical Sciences, Shanghai Jiao Tong University}
\email{}
\date{}
%\footnotetext{*Corresponding author.}
\footnotetext{Email addresses: matdengs@sjtu.edu.cn (Shijin Deng), yangxc0220@sjtu.edu.cn (Xiaochun Yang).}

\maketitle

% 开始设置页眉页脚，在\maketitle之后添加
\pagestyle{fancy} % 启用fancyhdr宏包的自定义样式
\fancyhf{} % 清空页眉页脚内容，但保留默认字号等设置
\renewcommand{\headrulewidth}{0pt} % 设置页眉线宽度，设为0则不显示

% 设置偶数页页码在左侧
\fancyhead[LE]{\fontsize{7}{\baselineskip}\selectfont\thepage} 
\fancyhead[CE]{\fontsize{7}{\baselineskip}\selectfont SHIJIN DENG AND XIAOCHUN YANG} 
% 设置奇数页页码在右侧
\fancyhead[RO]{\fontsize{7}{\baselineskip}\selectfont\thepage}
\noindent{\bf Abstract.} In this paper, both the pointwise space-time structures of the solution for a three-dimensional model describing gas dynamics in thermal non-equilibrium and the corresponding background stabilities within the $L^p$ framework are investigated under the small initial perturbation assumption and the irrotationality condition. Our main tools are the modified Green's function according to the irrotationality condition, a priori energy method and a refined decay estimates of products composed of the entropy function and high-order derivatives. 

\vskip .05in 


%%%%%%%%%%%%%%%%%%%%%%%%%%%%%%%%%%%%%%%%%%%%%%%%%%%%%%%%%%%%%%%%%%%%%%%%%%%%%%%%%5
%
%
%                        Section 1: Introduction 
%
%
%%%%%%%%%%%%%%%%%%%%%%%%%%%%%%%%%%%%%%%%%%%%%%%%%%%%%%%%%%%%%%%%%%%%%%%%%%%%%%%%%

\section{Introduction}\label{sec:1}
In this paper, we study the motion of a three-dimensional vibrational non-equilibrium flow. This physical process can be described by the following system \cite{vk,z1}:
\begin{equation}  \label{vne}
	\left \{
	\begin{aligned}
		&\partial_t \rho + \nabla \cdot (\rho \vec{u} ) = 0, \quad (\vec{x} , t) \in \mathbb{R}^3 \times \mathbb{R}_{+},\\
		&\partial_t ( \rho \vec{u} ) + \nabla \cdot ( \rho \vec{u} \otimes \vec{u} ) + \nabla p = 0, \\
		&\partial_t ( \rho \mathcal{E} ) + \nabla \cdot ( \rho \mathcal{E} \vec{u} + p\vec{u}) = 0,\\
		&\partial_t ( \rho q ) + \nabla \cdot ( \rho q \vec{u}) = \rho \frac{Q - q}{\tau}, 
	\end{aligned}
	\right.
\end{equation}
where $\rho$, $\vec{u}=(u_1,u_2,u_3)^T$, $p$, $\mathcal{E}$, $q$ and $Q$ denote the gas density, velocity, pressure, specific total energy, specific vibrational energy and local equilibrium value of the specific vibrational energy, respectively. The parameter $\tau$ represents the local relaxation time, which is taken as a fixed constant in the rest of the paper. For the sake of simplicity, we set $\tau \equiv 1$ and assume that the gas is in translational and rotational equilibrium. The total energy $\mathcal{E}$ is composed of the internal energy and the kinetic energy, satisfying the relation:
$$ \mathcal{E} = e + \frac{1}{2} \left| \vec{u} \right|^2, \quad \left| \vec{u} \right|^2 = \sum_{j=1}^3 u_j^2. $$
In addition, the internal energy $e$ can be further decomposed into the equilibrium component $e_1$ and the non-equilibrium component $q$:
$$ e = e_1 + q.$$
The corresponding thermodynamic equations are given by:
\begin{equation} \nonumber 
	de_1 = T_1 ds_1 - pdv, \quad dq = T_2 ds_2,  
\end{equation}
where $v = 1/\rho$ is the specific volume; $T_1$ and $s_1$ are the equilibrium temperature and entropy, whereas $T_2$ and $s_2$ are the non-equilibrium temperature and entropy.

System \eqref{vne} can be rewritten as
\begin{equation} \label{vnes}
	\left \{
	\begin{aligned} 
		& \partial_t \rho + \nabla \cdot (\rho \vec{u} ) = 0, \quad (\vec{x} , t) \in \mathbb{R}^3 \times \mathbb{R}_{+},\\ 
		& \partial_t \vec{u} + \vec{u} \cdot \nabla \vec{u} +  \frac{\nabla p}{\rho} = 0, \\ 
		& \partial_t e_1 + \vec{u} \cdot \nabla e_1 + \frac{p}{\rho} \nabla \cdot \vec{u} = - \chi,\\
		& \partial_t q + \vec{u} \cdot \nabla q = \chi,
	\end{aligned} 
	\right.
\end{equation}
with $\chi := Q - q$.

The previous study of this model shows that there is a non-decaying component in any spatial dimension and a usual treatment is to choose four new unkowns $p$, $\vec{u}$, $\chi$, $s$. In the system for these new unkowns, equations for decaying components and this non-decaying component are decoupled in the linear level:
\begin{equation} \label{vnep}
	\left \{
	\begin{aligned} 
		& \partial_t p + \vec{u} \cdot \nabla p + \rho c_f^2  \nabla \cdot \vec{u} = - p_{e_1} \chi, \quad (\vec{x} , t) \in \mathbb{R}^3 \times \mathbb{R}_{+},\\ 
		& \partial_t \vec{u} + \vec{u} \cdot \nabla \vec{u} + v \nabla p = 0, \\
		& \partial_t \chi  + \vec{u} \cdot \nabla \chi + a \nabla \cdot \vec{u} = - \zeta \chi,\\
		& \partial_t s + \vec{u} \cdot \nabla s =  \left( \frac{1}{T_2} - \frac{1}{T_1} \right) \chi.
	\end{aligned} 
	\right.
\end{equation}
Here, one adopts the state equations in \cite{hl}:
\begin{equation} \label{seq}
	\begin{aligned} 
		& e_1 = \frac{\alpha}{2} pv {{=}} \frac{\alpha R}{2} T_1, \quad Q = \frac{\alpha_f}{2} pv {{=}}  \frac{\alpha_f R}{2} T_1, \quad q = \frac{\alpha_f R}{2} T_2, \\ 
		& s_1 = R \left( \ln v + \frac{\alpha}{2} \ln e_1 \right), \quad s_2 = \frac{\alpha_f R}{2} \ln q, \quad {{s=s_1+s_2}},
	\end{aligned} 
\end{equation}
as well as the following identities (see \cite{hl} for more details): 
\begin{equation} \label{equ}
	\begin{aligned} 
		& \gamma := \frac{2}{\alpha} + 1, \quad c_f^2 := p_{\rho}(\rho, s_1) = v^2 (pp_{e_1} - p_v) = \gamma pv, \quad p_{e_1} = \frac{2}{\alpha v}, \\
		& a := \frac{\alpha_f}{\alpha} pv, \quad \zeta := 1 + Q_{e_1} = 1 + \frac{\alpha_f}{\alpha}. 
	\end{aligned}
\end{equation}
In \eqref{seq}, $\alpha$ and $\alpha_f$ denote the degrees of freedom which adjusts instantaneously and requires a longer relaxation time, respectively, and $R > 0$ is the gas constant.
In \eqref{equ}, $c_f$ denotes the frozen speed of sound, while $c$, the equilibrium speed of sound, is defined via the relation $c^2 = p_{\rho}(\rho, s) = c_f^2/(1+b)$, with $b$ given by
\begin{equation} \label{bb}
	b:= \frac{p_{e_1} v}{(1 + Q_{e_1}) c_f^2 / a - p_{e_1} v } = \frac{\gamma - 1}{c_f^2/a + 1}.
\end{equation}
By utilizing \eqref{equ}, one can rewrite system \eqref{vnep} as a perturbation problem around a given steady state $(\bar{p},\vec{0},0,\bar{s})$:
\begin{equation} \label{vnepre}
	\left \{
	\begin{aligned} 
		& \partial_t p + \gamma \bar{p} \nabla \cdot \vec{u} + \frac{2}{\alpha \bar{v}} \chi = -\vec{u} \cdot \nabla p  - \gamma (p - \bar{p}) \nabla \cdot \vec{u} - \frac{2}{\alpha} \left( \frac{1}{v} - \frac{1}{\bar{v}} \right) \chi:=R_p, \\ 
		& \partial_t \vec{u} + \bar{v} \nabla p = -\vec{u} \cdot \nabla \vec{u} - (v-\bar{v}) \nabla p:=R_{\vec{u}}, \\
		& \partial_t \chi + \frac{\alpha_f}{\alpha} \bar{p} \bar{v} \nabla \cdot \vec{u} +\zeta \chi = - \vec{u} \cdot \nabla \chi - \frac{\alpha_f}{\alpha} (pv - \bar{p}\bar{v})\nabla \cdot \vec{u}:=R_{\chi} , \\
		& \partial_t s = -\vec{u} \cdot \nabla s + \frac{2}{\alpha_f R} \cdot \frac{\chi^2}{T_1 T_2}:=R_{s}, \\
		& (p-\bar{p}, \vec{u}, \chi, s-\bar{s})(\vec x,0)=(p_0-\bar{p},\vec{u}_0, \chi_0, s_0-\bar{s})(\vec{x}).
	\end{aligned} 
	\right.
\end{equation}

In this paper, we are interested in the precise pointwise space-time structure of the solution of \eqref{vnepre}, as well as the stability analysis of the solutions with respect to background perturbations under a comparative framework. The background stability problem was introduced in \cite{lww} for Boltzamnn equation and it could be viewed as a kind of comparison principle: Let $\vec{F}'_{\beta} := (p_{\beta}, \vec{u}_{\beta}, \chi_{\beta}, s_{\beta} )$ for $\beta = a, b$ denote two close constant solutions of \eqref{vnep}, with $\vec{u}_a \equiv 0$ and $\chi_{\beta} \equiv 0$ for $\beta = a, b$. The state equations in \eqref{seq} reveal that the vectors $\vec{F}'_{\beta}$ are equivalent to $\vec{W}'_{\beta}:= (v_{\beta}, \vec{u}_{\beta}, e_{1,\beta}, q_{\beta} )$ for $\beta = a, b$ in the sense that there exists a positive constant $C>1$ such that $\frac{|\vec{F}'_{\beta}|}C \le |\vec{W}'_{\beta}| \le C |\vec{F}'_{\beta}|$. Note also that the condition $\chi_{\beta} \equiv 0$ implies $T_{1,\beta} = T_{2,\beta}$ and $e_{1,\beta} = \alpha q_{\beta}/\alpha_f$. Thus, one could simplify notations by setting $\vec{F}_{\beta} := (p_{\beta}, \vec{u}_{\beta}, s_{\beta} )$, $\vec{W}_{\beta}:= (v_{\beta}, \vec{u}_{\beta}, e_{1,\beta})$, and $\vec{\theta}_{\beta}:= (v_{\beta}, e_{1,\beta})$ for $\beta = a, b$. By substituting $p \to p_{\beta} + p^{\beta}$, $\vec{u} \to \vec{u}_{\beta} + \vec{u}^{\beta}$, $\chi \to \chi^{\beta}$, and $s \to s_{\beta} + s^{\beta}$ into \eqref{vnep}, one could derive two individual systems for comparison:
\begin{equation} \label{inlin}
	\left \{
	\begin{aligned}
		& \partial_t p^{\beta} + \vec{u}_{\beta} \cdot \nabla p^{\beta} + \gamma p_{\beta} \nabla \cdot \vec{u}^{\beta} + \frac{2}{\alpha v_{\beta}} \chi^{\beta}= \frac{2}{\alpha} \cdot \frac{v^{\beta} \chi^{\beta}}{ v_{\beta} (v_{\beta} + v^{\beta} )} - \vec{u}^{\beta} \cdot \nabla p^{\beta} - \gamma p^{\beta} \nabla \cdot \vec{u}^{\beta}, \\   
		& \partial_t \vec{u}^{\beta} + \vec{u}_{\beta} \cdot \nabla \vec{u}^{\beta} + v_{\beta} \nabla p^{\beta} = -\vec{u}^{\beta} \cdot \nabla \vec{u}^{\beta} - v^{\beta} \nabla p^{\beta}, \\
		& \partial_t \chi^{\beta} + \vec{u}_{\beta} \cdot \nabla \chi^{\beta} + \frac{\alpha_f}{\alpha} p_{\beta} v_{\beta} \nabla \cdot \vec{u}^{\beta} + \zeta\chi^{\beta} = - \vec{u}^{\beta} \cdot \nabla \chi^{\beta} - \frac{\alpha_f}{\alpha} \left[ p^{\beta} (v_{\beta} + v^{\beta}) + p_{\beta} v^{\beta} \right] \nabla \cdot \vec{u}^{\beta}, \\
		& \partial_t s^{\beta} + \vec{u}_{\beta} \cdot \nabla s^{\beta} = -\vec{u}^{\beta} \cdot \nabla s^{\beta} + \frac{2}{\alpha_f R} \cdot \frac{(\chi^{\beta})^2}{(T_{2,\beta} + T_2^{\beta} )(T_{1,\beta} + T_1^{\beta} )}, \\
		& \bigl(p^\beta, \vec{u}^{\beta}, \chi^{\beta}, s^{\beta}\bigr)(\vec{x},0)=\bigl(p_0^{\beta},\vec{u}_0^{\beta}, \chi_0^{\beta}, s_0^{\beta}\bigr)(\vec{x}).
	\end{aligned} 
	\right. 
\end{equation}

For thermal non-equilibrium flows, extensive research has been conducted, such as the global existence for solutions of the Cauchy problem for gas flow models in vibrational non-equilibrium \cite{hl,z2,z4} and, for those with finite thermal non-equilibrium states \cite{z3,z5}. Relevant results include $L^p$ decay estimates \cite{z6,z7}, temporal asymptotic behaviors \cite{z8}, nonlinear singular limits \cite{lw2}, and relaxation limits \cite{lw3}. There are also studies on the existence theory for initial-boundary value problems \cite{cl,clz} and the asymptotic stability of traveling waves and rarefaction waves \cite{lw1,lz,z}.

In the investigations on the global existence for solutions of general hyperbolic systems, or hyperbolic-parabolic systems, dissipative assumptions are always required (see \cite{hn,ky,y,z9,z5} and the references therein), among which the Kawashima condition (see \cite{sk} for more details) constitutes one of the most frequently used criterion; roughly speaking, it implies the strict negativity of the real parts of all eigenvalues corresponding to the linearized system after Fourier transform, with the only exception of the origin. Studies on global existence for flows with multiple thermal non-equilibrium states \cite{z3} are conducted under the assumption that this condition is satisfied. 

However, there is at least one zero eigenvalue for flows featuring only vibrational non-equilibrium states \cite{hl,z2,z1}; thereby the corresponding equations do not satisfy the Kawashima condition. For the one-dimensional case, relevant research has been carried out by Zeng in \cite{z2}. A key strategy is to introduce new unknowns $(p,\vec u,\chi,s)$ (as done in \eqref{vnep}) which decouples the equations for decaying components $(p,\vec u,\chi)$ and the non-decaying component $s$ in linear level. In \cite{qw}, this strategy is connected with the linearly degenerate condition in hyperbolic theory and generalized to partially dissipative quasilinear hyperbolic balance law systems; the authors introduced the condition of (weakly) linear degeneracy under which a system could be divided into a subsystem satisfying the Kawashima condition and the remaining subsystem violating Kawashima condition but being (weakly) linearly degenerate. This condition is applicable to the damped full Euler system for adiabatic flows \cite{hs,wm,wth}, but not to the thermal nonequilibrium flow model in three-dimensional case \cite{hl,z1}.

In the three-dimensional case, for the subsystem of $(p,\vec u,\chi)$, zero is still a double eigenvalue, thereby violating the Kawashima condition and bringing troubles to the study of the global existence.  Therefore, further research on three-dimensional thermal nonequilibrium flows is highly significant, and it is essential to exploit their spectral information. By leveraging the detailed estimates of the Green's function constructed in \cite{z1}, one could characterize this violation with precision: beyond the entropy function $s$, the velocity field $\vec u$ contains a non-decaying component exactly given by 
$$\left( \mathbb{I}_{3 \times 3} - \Delta^{-1} \nabla \nabla^t \right) \delta (\vec{x}).$$
An observation is, the identity $\vec{u} = \Delta^{-1} \nabla \nabla \cdot \vec{u} - \Delta^{-1} \nabla \times \nabla \times \vec{u} = \Delta^{-1} \nabla \nabla \cdot \vec{u}$ holds for irrotational flows (i.e. $\nabla \times \vec{u} \equiv 0$) and thus the non-decaying component in the velocity field disappears under the irrotationality assumption. Two typical examples of irrotational flows are given in \cite{hl}. The first is a spherically symmetric flow, which satisfies $(\rho, e_1, q)(\vec{x},t) = (\rho, e_1, q)(|\vec{x}|, t)$ and $\vec{u}(\vec{x}, t) = u(|\vec{x}|,t)\vec{x}/|\vec{x}|$; the second type of irrotational flow obeys $(\rho, e_1, q)(\vec{x},t) = (\rho, e_1, q)(\lambda, t)$ and takes the form $\vec{u}(\vec{x},t) = u(\lambda, t) \nabla \lambda$, where $\lambda:= \arctan(x_2/x_1)+ \mu x_3$ for any constant $\mu \in \mathbb{R}$. Based on this observation and also the Green's function and linear estimates established in \cite{z1}, the global existence and also $L^2$ decay estimates for $(p,\vec u,\chi)$ are obtained by energy method under the assumption that the flow is irrotational in \cite{hl}. In this paper, we aim to improve the  $L^2$ decay estimates of the solution into a more precise space-time pointwise structure. It allows us to make a direct comparison between the one-dimensional case \cite{z2} and the three-dimensional case. The rich information provided by the pointwise estimate also contributes to the background stability problem which we will discuss later on.

The work in this paper is conducted under the irrotational assumption. It is essential for the global existence as shown in \cite{hl}. In this paper, it is also found to be related to the decay structure even in the linear level. It is well-known that even for Navier-Stokes equations, there is a Caler\'{o}n-Zygmund operator in the velocity field for high-dimensional cases which slows the decay rates. There are similar issues for equations considered here. For instance, the component corresponding to the velocity field contains the term $(1+t)^{-3}$ (see \eqref{Gkstar}), which is contributed by this operator. It is already known that by the Hodge's decomposition, the identity $\vec{u}  = \Delta^{-1} \nabla \nabla \cdot \vec{u}$ holds for irrotational flows $\vec u$. It leads to a consideration of retaining the operator $\Delta^{-1} \nabla \nabla^t$ in the Green's function and obtaining the estimates of Green's function acting on irrotational flows, which makes the long wave part of the Green's function more clear and sharper. The high frequency part of Green's function obtained in \cite{z1} is also modified by following a global singularity removal. Those modifications together make decay rates faster and provide the specific information about the Green’s function required in the process of obtaining a space-time pointwise structure of the solutions. We should point out that the main reason on the very careful estimate goes back to the entropy function $s$. Even under the irrotational assumption, it does not decay in time; moreover, according to its equation in \eqref{vnep}, it lacks the exponential decay rates in spatial variables like the entropy functions for adiabatic flows. Here, we generalize an energy functional introduced in \cite{z2} for one-dimensional case to the three-dimensional case to improve decay rates of nonlinear terms involving the entropy function $s$, which is also quite crucial for the closure of nonlinearity.

The following is the first main result of this paper:
\begin{Theorem}[\textbf{Pointwise Space-Time Structures}] \label{main1}
	Assume that \eqref{seq} holds, and let $\bar{v}$, $\bar{e}_1$, and $\bar{q}$ denote positive constants satisfying either $\bar{q} = \bar{Q}$ or $\bar{T}_2 = \bar{T}_1$. Suppose that irrotational initial data $(v_0, \vec{u}_0, e_{1,0}, q_0)$ around the constant state $(\bar{v}, \vec{0}, \bar{e}_1, \bar{q})$ satisfies
	\begin{equation} \label{inida}
		\begin{aligned} 
			& \mathcal{I}_0 := (v_0 - \bar{v}, \vec{u}_0, e_{1,0} -\bar{e}_1, q_0 - \bar{q}) \in H^8_{\vec{x}}(\mathbb{R}^3), \\ 
			& \left\| \mathcal{I}_0 \right\|_{H^8_{\vec{x}}} + \sup_{\vec{x} \in \mathbb{R}^3} \Bigl\{ (|\vec{x} |^2 + 1)^{\frac{5}{2}} \bigl(|\mathcal{I}_0| + |\nabla \mathcal{I}_0| + |\nabla^2 (p_0, \vec{u}_0)| \bigr)(\vec{x}) \\ 
			& \quad + (|\vec{x}|^2 + 1)^{\frac{3}{2}} |\nabla^2 s_0|(\vec{x}) + (|\vec{x}|^2 + 1)^{\frac{3}{4}} |\nabla^3 s_0|(\vec{x}) \Bigr\} \equiv \epsilon_0 \ll 1,
		\end{aligned}
	\end{equation}
	where $s_0$ is the initial entropy. Then there exists a small positive constant $\varepsilon_*$ such that, if $\epsilon_0<\varepsilon_*$,  the following estimates hold for the irrotational solution $(p,\vec{u},\chi,s)$ of \eqref{vnepre}:
	\begin{equation} \label{pointwise}
		\begin{aligned}
			& (p - \bar{p}, \vec{u})(\vec{x}, t) = \mathcal{O}(1) \epsilon_0 \left\lbrace \left( (|\vec{x}| - \bar{c} t)^2 + t + 1 \right)^{-2} + \frac{H(\bar{c}t - |\vec{x}|)}{1 + (|\vec{x}| + \sqrt{t})(\bar{c}t - |\vec{x}| + \sqrt{t})^2} \right\rbrace, \\
			& (\nabla p, \nabla \vec{u}, \chi)(\vec{x}, t) = \mathcal{O}(1) \epsilon_0 (1 + t)^{-\frac{1}{2}} \left\lbrace \left( (|\vec{x}| - \bar{c} t)^2 + t + 1 \right)^{-2} + \frac{H(\bar{c}t - |\vec{x}|)}{1 + (|\vec{x}| + \sqrt{t})(\bar{c}t - |\vec{x}| + \sqrt{t})^2} \right\rbrace, \\
			& \nabla \chi (\vec{x}, t) = \mathcal{O}(1) \epsilon_0 \bigl(1 + |\vec{x}| + \sqrt{t} \, \bigr)^{-2} \left( (|\vec{x}| - \bar{c}t)^2 + t + 1 \right)^{-1}, \\
			& \left\| \nabla^2(p, \vec{u}, \chi)(\cdot, t) \right\|_{L^{\infty}_{\vec{x}}} = \mathcal{O}(1)\epsilon_0 (1 + t)^{-3}, \\
			& (s - \bar{s}, \nabla s)(\vec{x}, t) = \mathcal{O}(1) \epsilon_0 (|\vec{x}|^2 + 1)^{-\frac{9}{4}}, \enspace \nabla^2 s(\vec{x}, t) = \mathcal{O}(1) \epsilon_0 (|\vec{x}|^2 + 1)^{-\frac{3}{2}}, \enspace \nabla^3 s(\vec{x}, t) = \mathcal{O}(1) \epsilon_0 (|\vec{x}|^2 + 1)^{-\frac{3}{4}},
		\end{aligned}
	\end{equation}
	where $\bar{c}^2=\gamma \bar{p} \bar{v}/(1+b)$ (with constants $\gamma$ and $b$ given by \eqref{equ} and \eqref{bb} respectively), and the Heaviside step function $H(r)$ is defined as
	\begin{equation} \label{stepfun}
		H(r):= 
		\begin{cases} 
			0 & \text{for } r < 0, \\
			1 & \text{for } r \ge 0.
		\end{cases} 
	\end{equation}
\end{Theorem}

\begin{Remark}\label{rebetalq}
	The system \eqref{inlin} turns to \eqref{vnepre} after a coordinate transformation $\vec{x}_{\beta}:=\vec{x}-\vec{u}_{\beta} t$ and thus \eqref{pointwise} still holds for the solution of \eqref{inlin}, with only the replacement of $(\vec{x},\bar{c}^2,\epsilon_0)$ by $(\vec{x}_{\beta},c_{\beta}^2:= \gamma p_{\beta} v_{\beta}/(1+b),\epsilon_{\beta})$. Furthermore, this coordinate transformation preserves the $L^q$-estimates with respect to the spatial variables; thus accordingly, the following estimates hold for the solution $(p^\beta,\vec{u}^\beta,\chi^\beta,s^\beta)$ of \eqref{inlin} with $\beta \in \{a,b\}$, $|k|=0,1$ and $q \in [2, +\infty]$\textup{:}
	\begin{equation} \label{betalq}
		\begin{aligned}
			& \left\| \partial_{\vec{x}}^k (p^{\beta}, \vec{u}^{\beta})(\cdot,t) \right\|_{L_{\vec{x}}^q(\mathbb{R}^3)} = \mathcal{O}(1) \epsilon_{\beta} (1+t)^{-2-\frac{|k|}{2}+\frac{5}{2q}}, \quad \left\| \partial_{\vec{x}}^k \chi^{\beta} (\cdot,t) \right\|_{L_{\vec{x}}^q(\mathbb{R}^3)} = \mathcal{O}(1) \epsilon_{\beta} (1+t)^{-\frac{5+|k|}{2} +\frac{5}{2q}},\\
			& \left\| \nabla^2 (p^{\beta}, \vec{u}^{\beta}, \chi^{\beta})(\cdot,t) \right\|_{L_{\vec{x}}^{\infty}(\mathbb{R}^3)} = \mathcal{O}(1) \epsilon_{\beta} (1+t)^{-3}.
		\end{aligned}
	\end{equation}
\end{Remark}

The second issue in this paper is about background stability problems. In practical scenarios, background states are usually obtained by taking measurements, and thus inherent errors are unavoidable. In this context, the background stability problem emerges as a natural and critical research extension. In \cite{lww}, the background stability problem was investigated for the Boltzmann equation, where the authors proposed a refined Green's function method tailored to such problems. In \cite{dy1}, we employed both the refined Green's function method and a refined energy method to investigate the background perturbation problem for compressible Navier-Stokes equations. Analogous problems are also studied in related research fields, including potential stability problems \cite{ltw} and force stability problems \cite{lw}.

We now formulate the background perturbation problem as follows: denote 
\begin{equation} \label{yy1}
	\vec{Y}(\vec{x},t) := \bigl(\vec{g}_h, f, z\bigr)(\vec{x}, t):= \bigl( \vec{Y}_1, z\bigr)(\vec{x}, t),
\end{equation}
with
\begin{equation}\label{diff}
	\begin{aligned}
		& \vec{g}_h(\vec{x} , t ) := \bigl(g^0, \vec{h}\bigr)(\vec{x} , t) = \bigl(g^0, h^1, h^2, h^3\bigr)(\vec{x} , t) = \bigl(p^b - p^a ,  \vec{u}^b -\vec{u}^a\bigr)(\vec{x} , t),\\
		& f(\vec{x} , t ) := \bigl(\chi^b -\chi^a\bigr)(\vec{x} , t), \quad z(\vec{x} , t ) := \bigl(s^b - s^a\bigr)(\vec{x} , t), 
	\end{aligned}
\end{equation}
and one could obtain the equations for $\vec{Y}(\vec x,t)$ from \eqref{inlin}:
\begin{equation} \label{gbp}
	\left \{ 
	\begin{aligned}
		& \partial_t g^0 + \gamma p_a \nabla \cdot \vec{h} + \frac{2}{\alpha v_a} f 
		=R_{g^0}^ {\RN{1}}+R_{g^0}^ {\RN{2}}:= R_{g^0}, \\ 
		& \partial_t \vec{h} + v_a \nabla g^0 =R_{\vec h}^ {\RN{1}}+R_{\vec h}^ {\RN{2}}:= R_{\vec{h}}, \\ 
		& \partial_t f + \frac{\alpha_f}{\alpha} p_a v_a \nabla \cdot \vec{h} + \zeta f =R_{f}^ {\RN{1}}+R_{f}^ {\RN{2}}:= R_f, \\
		& \partial_t z = -\vec{u}_b \cdot \nabla s^b 
		+R_{z}^ {\RN{2}}:= R_{z},\\
		& \vec{Y}(\vec{x},0)=(p_0^b - p_0^a, \vec{u}_0^b - \vec{u}_0^a, \chi_0^b - \chi_0^a, s_0^b - s_0^a )(\vec{x}).
	\end{aligned} 
	\right.
\end{equation}
Here, the source terms and coupling terms are defined respectively as follows:
\begin{equation} \label{lin}
	\begin{cases}
		\begin{aligned}
			&R_{g^0}^ {\RN{1}} 
			:= - \vec{u}_b \cdot \nabla p^b + \gamma (p_a - p_b) \nabla \cdot \vec{u}^b + \frac{2}{\alpha} \left( \frac{1}{v_a} - \frac{1}{v_b} \right) \chi^b, \\
			& R_{\vec h}^ {\RN{1}}
			:= - \vec{u}_b \cdot \nabla \vec{u}^b + (v_a - v_b) \nabla p^b , \\
			&R_{f}^ {\RN{1}}
			:= -\vec{u}_b \cdot \nabla \chi^b + \frac{\alpha_f}{\alpha} ( p_a v_a - p_b v_b) \nabla \cdot \vec{u}^b, 
		\end{aligned}
	\end{cases}
\end{equation}
and
\begin{equation}\label{nlin}
	\begin{cases}
		\begin{aligned}
			& R_{g^0}^ {\RN{2}} 
			:= \frac{2}{\alpha} \left[ \frac{v^a f}{v_b \tilde{v}^b }  + \frac{v^b - v^a}{v_b \tilde{v}^b } \chi^b + \left( \frac{v_a - v_b}{v_a v_b \tilde{v}^b} + \frac{\tilde{v}^a - \tilde{v}^b}{v_a \tilde{v}^a \tilde{v}^b} \right) v^a \chi^a \right] - \vec{u}^b \cdot \nabla g^0 - \vec{h} \cdot \nabla p^a \\
			&\quad \quad \quad - \gamma \left( g^0 \nabla \cdot \vec{u}^b + p^a \nabla \cdot \vec{h} \right) , \\
			& R_{\vec h}^ {\RN{2}} 
			:= -( v^b - v^a ) \nabla p^b - \vec{h} \cdot \nabla \vec{u}^b - \vec{u}^a \cdot \nabla \vec{h} - v^a \nabla g^0, \\
			& R_{f}^ {\RN{2}}
			:= - \vec{h} \cdot \nabla \chi^b - \vec{u}^a \cdot \nabla f - \frac{\alpha_f}{\alpha} \left[ \bigl( p^b \tilde{v}^b + p_a v^a \bigr) \nabla \cdot \vec{h} + \left( g^0 \tilde{v}^b + p^a ( \tilde{v}^b - \tilde{v}^a ) \right) \nabla \cdot \vec{u}^a \right] ,\\
			&\quad \quad \quad - \frac{\alpha_f}{\alpha} \left[(p_b - p_a )v^a + p_b (v^b - v^a ) \right] \nabla \cdot \vec{u}^b , \\
			& R_{z}^ {\RN{2}} 
			:= - \vec{u}^b \cdot \nabla z - \vec{h} \cdot \nabla s^a + \frac{2}{\alpha_f R} \left[\frac{(\chi^b + \chi^a ) f}{\tilde{T}_1^b \tilde{T}_2^b } + \left( \frac{\tilde{T}_1^a - \tilde{T}_1^b}{\tilde{T}_2^b \tilde{T}_1^a \tilde{T}_1^b } + \frac{ \tilde{T}_2^a - \tilde{T}_2^b}{\tilde{T}_1^a \tilde{T}_2^a \tilde{T}_2^b} \right) \! (\chi^a)^2 \right]. 
		\end{aligned}
	\end{cases}
\end{equation}
Throughout the paper, quantities marked with a tilde ($\sim$) denote the sum of the constant-state solution and the perturbation solution, e.g. $\tilde{v}^b: = v_b + v^b$.

In our previous work \cite{dy1}, the background stability problem for the compressible Navier-Stokes equations is considered. The study in this paper differs in that the entropy function does not decay in time, which causes greater troubles dealing with terms involving it. Specifically, if one only relies on energy estimates to derive the dependence of high-order derivatives on background perturbation quantities, the nonlinear terms consisting of the entropy function and high-order derivatives do not decay fast enough in time for the closure of nonlinearity. To address this problem, the pointwise estimates of the entropy function are made full use of; in the region $|\vec{x}| > \epsilon t$, the entropy function naturally supplements the decay rate in time by the one in space; for $|\vec{x}| \le \epsilon t$, we extend the enhanced estimates for one-dimensional problems in \cite{z2} to fit the three-dimensional setting and also the comparison setting to compensate the decay rates. Further difficulties stem from the entropy difference function $z= s^b - s^a$, which even grows in time with 
a growth factor of $(1+t)$. It affects terms containing $v^b - v^a$, such as $(v^b - v^a)\chi^b/(v_b \tilde{v}^b)$ in the coupling term $R_{g^0}^{\RN{2}}$, since $v^b - v^a = \mathcal{O}(1)( |v_b - v_a| + |e_{1,b} - e_{1,a}| + |g^0| + |f| + |z|)$. It is resolved by a very precise estimate and a careful partition of the time interval during time integrals. More details could be found in Section 4.2.

The following is the main result about the background stability:
\begin{Theorem}[\textbf{Background Stabilities}] \label{main2}
	For $\beta=a,b$, assume that irrotational initial data for the system \eqref{inlin} satisfy
	$$\begin{aligned} 
		& \mathcal{I}_{\beta} := \bigl(v_0^{\beta}, \vec{u}_0^{\beta}, e_{1,0}^{\beta}, q_0^{\beta}\bigr) \in H^8_{\vec{x}}(\mathbb{R}^3), \\ 
		& \left\| \mathcal{I}_{\beta} \right\|_{H^8_{\vec{x}}} + \sup_{\vec{x} \in \mathbb{R}^3} \Bigl\{ (|\vec{x}|^2 + 1)^{\frac{5}{2}} \bigl( \bigl|\mathcal{I}_{\beta}\bigr| + \bigl|\nabla \mathcal{I}_{\beta}\bigr| + \bigl|\nabla^2 (p_0^{\beta}, \vec{u}_0^{\beta})\bigr| \bigr)(\vec{x}) \\ 
		& \quad + (|\vec{x}|^2 + 1)^{\frac{3}{2}} \bigl|\nabla^2 s_0^{\beta}\bigr|(\vec{x}) + (|\vec{x}|^2 + 1)^{\frac{3}{4}} \bigl|\nabla^3 s_0^{\beta}\bigr|(\vec{x}) \Bigr\} \equiv \epsilon_{\beta} \ll 1.
	\end{aligned}$$
	Furthermore, assume the following comparative initial condition holds:
	$$\left\|(p_0^b - p_0^a, \vec{u}_0^b - \vec{u}_0^a, \chi_0^b - \chi_0^a, s_0^b - s_0^a )\right\|_{L^1_{\vec{x}}(\mathbb{R}^3) \cap H^8_{\vec{x}}(\mathbb{R}^3)} = \mathcal{O}(1) \bigl| \vec{F}_b - \vec{F}_a \bigr| = \mathcal{O}(1) \bigl| \vec{W}_b - \vec{W}_a \bigr|,$$
	with two constant background states $\vec{F}_a \equiv (p_a, \vec{0}, 0, s_a)$ and $\vec{F}_b \equiv (p_b, \vec{u}_b, 0, s_b)$\textup{(}or equivalently, $\vec{W}_a \equiv (v_a, \vec{0}, e_{1,a}, q_a)$ and $\vec{W}_b \equiv (v_b, \vec{u}_b, e_{1,b}, q_b)$\textup{)}. Then there exists a small positive constant $\epsilon_*$ such that, if $\epsilon_{\beta}<\epsilon_*$ and $|\vec{W}_b-\vec{W}_a|<\epsilon_*$, the irrotational solution 
	$$\bigl(\vec{g}_h, f, z\bigr)(\vec{x}, t) = \bigl(g^0, h^1, h^2, h^3, f, z \bigr)(\vec{x}, t )=\bigl(p^b - p^a , \vec{u}^b -\vec{u}^a, \chi^b - \chi^a, s^b - s^a \bigr)(\vec{x}, t)$$
	of \eqref{gbp} satisfies 
	$$\begin{aligned}
		& \left\| \vec{g}_h (\cdot , t) \right\|_{L_{\vec{x}}^2(\mathbb{R}^3)} = \mathcal{O}(1)\! \left( (1 + t)^{-\frac{1}{4}} \bigl| \vec{u}_b - \vec{u}_a \bigr| + (1 + t)^{-\frac{3}{20}} \bigl| \vec{\theta}_b - \vec{\theta}_a \bigr| \right), \\ 
		& \left\| f (\cdot , t) \right\|_{L_{\vec{x}}^2(\mathbb{R}^3)} = \mathcal{O}(1) \! \left( (1 + t)^{-\frac{3}{4}} \bigl| \vec{u}_b - \vec{u}_a \bigr|+ (1 + t)^{-\frac{13}{20}} \bigl| \vec{\theta}_b - \vec{\theta}_a \bigr| \right), \\
		& \left\| \vec{g}_h (\cdot , t) \right\|_{L_{\vec{x}}^{\infty}(\mathbb{R}^3)} = \mathcal{O}(1) (1 + t)^{-1} \bigl| \vec{W}_b - \vec{W}_a \bigr|, \\
		& \left\| f (\cdot , t) \right\|_{L_{\vec{x}}^{\infty}(\mathbb{R}^3)} = \mathcal{O}(1) (1 + t)^{-\frac{3}{2}} \bigl| \vec{W}_b - \vec{W}_a \bigr|,\\
		& \left\| z (\cdot , t) \right\|_{L_{\vec{x}}^p (\mathbb{R}^3)} = \mathcal{O}(1) \! \left( (1 + t) \bigl| \vec{u}_b - \vec{u}_a \bigr|+ \log(1 + t) \bigl| \vec{\theta}_b - \vec{\theta}_a \bigr| \right),
	\end{aligned}$$
	with $p \in [1, +\infty]$ and $\vec{\theta}_b - \vec{\theta}_a = (v_b - v_a, e_{1,b} - e_{1,a})$.
\end{Theorem}
\begin{Remark}
	Establishing dependence on background perturbation quantities always necessitates sacrificing a decaying factor $(1+t)^{-1/2}$ compared with the solutions of the individual perturbation problem \textup{(}as shown in \cite{lww}\textup{)}. The slightly slower decay rate preceding $|\vec{\theta}_b - \vec{\theta}_a|$ may be caused by technical reasons in the sense that although $\chi^b$ shares the same decay rate as $\nabla p^b$ and $\nabla \cdot \vec{u}^b$, the latter two permit a transfer of the differential operator during convolutions with Green's function $\mathbb{G}_L$ during the time interval $[0,t/2]$ \textup{(}see \eqref{y2l0}--\eqref{y2l1} for more details\textup{)}.
\end{Remark}

The rest of this paper is structured as follows: In Section 2, one presents a modified estimate of the Green's function for three-dimensional thermally non-equilibrium flows derived under the irrotationality condition and also modifies the energy estimate in \cite{hl} to general energy estimates especially for higher derivatives, to prepare for the closure of nonlinearity. The details of the modifications could be found in Appendix $\RN{1}$ and $\RN{2}$ respectively. One also lists computational lemmas for wave interactions in Section 2 and their proofs are left in Appendix $\RN{3}$. In Section 3, pointwise estimates for solutions of the individual perturbation problem are obtained by the Green's function method and energy estimates. In Section 4, the $L^p_{\vec{x}}$-estimates for the difference between two perturbation solutions are established for $p \in [2, +\infty]$. Notably, the term $z$ (i.e. the difference between two entropy perturbations) introduces a growth factor $(1+t)$. It is a main difference and also a new trouble compared with previous studies and necessitates the combined use of the Green's function method and weighted energy estimates to resolve such challenges.

{\bf Notations.} For a multi-index $k := (k_1, k_2, \dots, k_n)$, one denotes
$$\partial_{\vec{x}}^k := \partial_{x_1}^{k_1}\partial_{x_2}^{k_2}\cdots\partial_{x_n}^{k_n},\quad \vec{x}^k := x_1^{k_1}x_2^{k_2}\cdots x_n^{k_n},\quad |k| := k_1+k_2+\cdots+k_n.$$
The symbol $\nabla_{\vec{x}}^m$ ($m\in\mathbb{N}^+$), abbreviated as $\nabla^m$, represents all $m$th-order derivatives with respect to the spatial variable $\vec{x}$. For a multi-index $k$,  the notations $k\pm 1$ are referred to as $k \pm 1:= (k_1,\dots,k_{j-1},k_j \pm 1,k_{j+1},\dots,k_n)$ for arbitrary $1\le j\le n$.

For a vector $\vec{x} =(x_1, \dots, x_n)^t \in \mathbb{R}^n$, its Euclidean norm is denoted as $|\vec{x}| \equiv \sqrt{\sum_{i=1}^{n} (x_i)^2}$. The notations $\left\| \cdot \right\|_{L^p}$ and $\left\| \cdot \right\|_{H^m}$ denote the classical norms on $L^p_{\vec{x}} (\mathbb{R}^n)$ and Sobolev space $H^m_{\vec{x}} (\mathbb{R}^n)$, respectively. The Fourier transform of a function $f$ is denoted by $\mathbb{F}[f](\vec{\eta})$, defined as
$$\mathbb{F}[f](\vec{\eta}) := \int_{\mathbb{R}^n} e^{-i\vec{\eta} \cdot \vec{x}} f(\vec{x}) \, d\vec{x},$$
with the Fourier variable $\vec{\eta} = (\eta_1, \dots, \eta_n)^t \in \mathbb{R}^n$ and $\mathbb{F}^{-1}$ denotes the inverse Fourier transform; its long-wave and short-wave parts, denoted by $f_L(\vec{x})$ and $f_S(\vec{x})$ respectively, are defined as follows:
$$f_L(\vec{x})\equiv \mathbb{F}^{-1}\Big[\chi(|\vec\eta|)\mathbb{F}[f](\vec{\eta})\Big](\vec x),\ \ \ \ f_S(\vec{x})\equiv \mathbb{F}^{-1}\Big[ \left(1-\chi(|\vec\eta|)\right)\mathbb{F}[f](\vec{\eta})\Big](\vec x),$$ by a cut-off function $\chi$.

The notation $f = \mathcal{O}(1) g$ indicates that there exists a positive constant $C $ such that $|f| \le C |g|$. Throughout this paper, $C > 0$ denotes a generic constant and its value may vary from line to line.

%%%%%%%%%%%%%%%%%%%%%%%%%%%%%%%%%%%%%%%%%%%%%%%%%%%%%%%%%%%%%%%%%%%%%%%%%%%%%%%%%5
%
%
%                        Section 2: Preliminaries 
%
%
%%%%%%%%%%%%%%%%%%%%%%%%%%%%%%%%%%%%%%%%%%%%%%%%%%%%%%%%%%%%%%%%%%%%%%%%%%%%%%%%%

%%%%%%%%%%%%%%%%%%%%%%%%%%%%%%%%%%%%%%%%%%%%%%%%%%%%%%%%%%%%%%%%%%%%%%%%%%%%%%%%%5
%
%
%                        Section 2: Preliminaries 
%
%
%%%%%%%%%%%%%%%%%%%%%%%%%%%%%%%%%%%%%%%%%%%%%%%%%%%%%%%%%%%%%%%%%%%%%%%%%%%%%%%%%

\section{Preliminaries: Green's Function and Energy Estimates for Individual Perturbations}\label{sec:2}

In this section, we present several key results on the Green's function, along with the global existence and decay properties of each individual solution under small initial perturbations. Meanwhile, in this section and subsequent sections, we introduce the parameter $d_i$ to distinguish the distinct decay rates, where $d_1 = d_2 = 0$ and $d_3 = 1/2$.

\subsection{Green's Function} An obvious observation from \eqref{inlin} is that the linear part for $(p^\beta,\vec u^\beta,\chi^\beta)$ and the one for the non-decaying component $s^\beta$ are decoupled. It allows us to consider the Green's function for $(p^\beta,\vec u^\beta,\chi^\beta)$ separately. 

Define a Green's function $\mathbb{G}_{\beta}(\vec{y},t)$ by
\begin{equation}\label{greeng}
	\begin{cases}
		\left(\partial_t +
		\begin{pmatrix}
			0&\gamma p_{\beta} \nabla^t_{\vec{y}} & \frac{2}{\alpha v_{\beta}}\\
			v_{\beta} \nabla_{\vec{y}}&0&0\\
			0& \frac{\alpha_f}{\alpha} p_{\beta}v_{\beta} \nabla^t_{\vec{y}} & \zeta
		\end{pmatrix} \right)
		\mathbb{G}_{\beta} (\vec{y},t)=0,\\
		\mathbb{G}_{\beta} (\vec{y},0)=\delta(\vec{y})\mathbb{I}_{5},
	\end{cases}
\end{equation}
where $\mathbb{I}_{5}$ and $\delta(\vec{x})$ denote the $5\times5$ identity matrix and the $3$-dimensional Dirac delta function respectively. Then the Green's function $\mathbb{F}_{\beta}(\vec{x}, t)$ for $(p^\beta,\vec u^\beta,\chi^\beta)$ could be obtained by a coordinate transformation:
$$\mathbb{F}_{\beta}(\vec{x}, t)\equiv \mathbb{G}_{\beta}(\vec{x}- \vec{u}_{\beta} t,t).$$

The pointwise spacetime structure of $\mathbb{G}_{\beta}(\vec{y}, t)$ has been established in \cite{z1}. Here, we introduce a further estimate suitable for the irrotationality assumption. The main difference lies in the treatment of the case $|\vec \eta|\ll 1$ ($\vec\eta$ is the Fourier variable) where one retains the operator $\Delta^{-1}\nabla\nabla^t$ instead of estimating the inverse Fourier transform of the symbol $\vec{\eta}\vec{\eta}^t/|\vec{\eta}|^2$. This modification shows a cancelation of the algebraically decaying term $H(c_{\beta}t-|\vec{y}|)(1+t)^{-3-|k|}$ in the velocity field with the irrotationality condition. In the subsequent theorem, for simplicity, the subscript $\beta$ is replaced by the overbar symbol $\bar{\cdot}$ as a uniform notation of the function values at a constant equilibrium state.

\begin{Theorem}[\textbf{Green's function}: a modification of Theorem~1.3 in \cite{z1}] \label{green}
Let $k$ be a multi-index and $0 < \tilde{c}_2 < \bar{c}$ be a constant, where $\bar{c}$ represents the equilibrium sound speed, taken at the equilibrium constant state. After eigendecomposition, the Green's function $\mathbb{G}_{\beta}(\vec{y}, t)$, defined by \eqref{greeng} and abbreviated hereafter as $\mathbb{G}(\vec{y}, t)$, satisfies
	\begin{equation} \label{g1}
		\partial_{\vec{y}}^{k} \Big(\mathbb{G}(\vec{y},t)-\operatorname{diag}\! \left( 0, (\mathbb{I}_{3 \times 3} - \Delta^{-1} \nabla \nabla^t ) \delta(\vec{y}), 0 \right)\Big) :=\partial_{\vec{y}}^{k} \mathbb{G}^D(\vec{y},t) = \partial_{\vec{y}}^{k} \mathbb{G}^D_L (\vec{y},t) + \partial_{\vec{y}}^{k} \mathbb{G}^D_S (\vec{y},t).
	\end{equation}
{{Here, $\mathbb{G}^D_L$ and $\mathbb{G}^D_S$ denote the long-wave and short-wave components of the decaying components $\mathbb{G}^D$ of the Green's function, respectively,}} and are abbreviated hereafter as $\mathbb{G}_L$ and $\mathbb{G}_S$.

There exists a positive constant $D_0$ such that for all $t > 0$, the long-wave component $\mathbb{G}_L$ satisfies
	\begin{equation} \label{gL}
		\begin{aligned}
			\partial_{\vec{y}}^{k} \Big(\mathbb{G}_L(\vec{y},t) -\operatorname{diag}\! \left(0, \Delta^{-1} \nabla \nabla^t \left(\mathbb{G}_L\right)_{11}(\vec{y},t), 0 \right)\Big) &= H(t-1) \Bigl[ \tilde{G}_{k}^{*} (\vec{y},t) + \mathcal{O}(1) (1 + t)^{-\frac{5 + |k|}{2}} e^{-\frac{(|\vec{y}|-\bar{c} t)^2}{Kt}} \Bigr]  \\
			& \quad + \mathcal{O}(1) e^{-(|\vec{y}|+t)/D_0}, 
		\end{aligned}
	\end{equation}
	with $\tilde{G}_{k}^{*}$ taking the form:
	\begin{equation} \label{gkstar0}
		\begin{aligned}
			\tilde{G}_{k}^{*}(\vec{y},t) &:= H(|\vec{y}|-\bar{c}_2 (1+t)) \frac{1}{2(2 \pi \bar{\alpha}_1 t)^{3/2}} \partial_{\vec{y}}^k \left[ \frac{|\vec{y}| - \bar{c} t}{t} e^{-\frac{(|\vec{y}|-\bar{c}t)^2}{2 \bar{\alpha}_1 t}} \right]\\
			& \quad \times \begin{pmatrix}
				\frac{1}{\bar{c}} & \frac{\bar{c}\bar{\rho}t}{|\vec{y}|^2} \vec{y}^t & -\frac{\bar{p}_{e_1}}{\bar{c} \zeta}\\
				\frac{\bar{v} t}{\bar{c}|\vec{y}|^2} \vec{y} & 0_{3 \times 3}& -\frac{\bar{v} \bar{p}_{e_1} t}{\bar{c} \zeta |\vec{y}|^2} \vec{y}\\
				0 & 0_{1 \times 3} & 0
			\end{pmatrix}\\
			&= \mathcal{O}(1) t^{-2-\frac{|k|}{2}} e^{-\frac{(|\vec{y}|-\bar{c}t)^2}{Kt}}.
		\end{aligned}
	\end{equation}
Here, the equilibrium speed of sound $\bar{c}=\sqrt{\gamma \bar{p} \bar{v}/(1+b)}$ with $\gamma$ and $b$ given in \eqref{equ} and \eqref{bb} respectively, the Heaviside step function $H(r)$ is defined by \eqref{stepfun}, $(\cdot)_{lm}$ represents the $(l,m)$-entry of the block third-order matrix and $K$ is any constant greater than $ 2\bar{\alpha}_1$, with 
\begin{equation}\nonumber
 \bar{\alpha}_1 := \frac{\bar{a} \bar{p}_{e_1} \bar{v}}{\zeta^2} = \frac{2\alpha_f \bar{p}\bar{v}}{\alpha^2 \zeta^2}.
\end{equation}
	
	Moreover, there exists a universal positive constant $D_1:= 4 \max \{ 4 \bar{c}_f/\mu_1, 1 \}$ such that one could decompose the short-wave component $\mathbb{G}_S$ as follows:
	\begin{equation} \label{gS}
		\begin{aligned}
			\mathbb{G}_S(\vec{y},t) &= G_{0,1}^{s,*}(\vec{y},t) + G_{0,2}^{s}(\vec{y},t) + \operatorname{diag} \! \left(0, \Delta^{-1} \nabla \nabla^t \theta (\vec{y}, t), 0 \right) + \mathcal{O}(1) e^{-(|\vec{y}| + t)/D_0}\\
			& \quad + \mathcal{O}(1) |\vec{y}|^{-2} e^{-|\vec{y}|/D_1} \bigl( e^{-\mu_1 t/4} + e^{-\mu_2 t} \bigr), \\ 
			G_{0,1}^{s,*}(\vec{y},t) &:= e^{- \mu_2 t}L_{1,0} \delta (\vec{y}) + e^{- \mu_1 t} \bigl[ L_{2,0} \omega_t(\vec{y}, t; \bar{c}_f) + (L_{3, -1}(t, \nabla) + L_{3,0}) \omega(\vec{y}, t; \bar{c}_f)\bigr] \\
			& \quad + t e^{-\mu_1 t} L_{5,0} \omega(\vec{y},t; \bar{c}_f),\\
			G_{0,2}^{s}(\vec{y},t) &:= e^{-\mu_1 t} \operatorname{diag} \! \left(0, \Delta^{-1} \nabla \nabla^t, 0 \right)\!\bigl[ \omega_t(\vec{y}, t; \bar{c}_f) + \left( \mu_1 + \mu_1 (\mu_1/2 + \mu_2) t\right) \! \omega(\vec{y}, t; \bar{c}_f )\bigr],
		\end{aligned}
	\end{equation}
where $\omega$ and $\omega_t$ are defined by
\begin{equation}\label{wwt}
\mathbb{F}[\omega](\vec{\eta}, t; \bar{c}_f) = \sin (\bar{c}_f |\vec{\eta}| t)/(\bar{c}_f |\vec{\eta}|),\ \ \ \ 
\mathbb{F}[\omega_t](\vec{\eta}, t; \bar{c}_f) = \cos(\bar{c}_f |\vec{\eta}| t),
\end{equation}
all quantities $\bar{c}_f$, $\bar{a}$ and $\bar{p}_{e_1}$ are defined in \eqref{equ},
and $\mu_1$, $\mu_2$ and $L_{k,j}$ are given by
\begin{equation}\label{mumu}
 \mu_1:= \frac{a p_{e_1} v}{2 c_f^2} = \frac{\alpha_f}{\alpha ( \alpha + 2 )},\quad \mu_2 := \frac{c^2 (1 + Q_{e_1})}{c_f^2} = \frac{\zeta}{1 + b},
\end{equation}
	\begin{equation} \label{leading}
		\begin{aligned}
			& L_{1,0}:= \begin{pmatrix}
				0&0&0\\
				0&0&0\\
				-\frac{\bar{a} \bar{v}}{\bar{c}^2_f} &0&1
			\end{pmatrix}, \quad
			L_{2,0}:= \begin{pmatrix}
				1&0&0\\
				0&0&0\\
				\frac{\bar{a} \bar{v}}{\bar{c}^2_f}&0&0
			\end{pmatrix}, \quad
			L_{3,-1}(t, \nabla) := \begin{pmatrix}
				0&- \bar{\rho} \bar{c}_f^2 \nabla^t &0\\
				- \bar{v} \nabla &0&0\\
				0&- \bar{a} \nabla^t &0
			\end{pmatrix}, \\
			& L_{3,0} := \begin{pmatrix}
				\mu_1&0&-\bar{p}_{e_1}\\
				0&0&0\\
				\frac{\bar{a} \bar{v} (\mu_1 - \mu_2)}{\bar{c}_f^2}&0&-2 \mu_1
			\end{pmatrix}, \quad
			L_{5,0} := \mu_1 \left( \frac{\mu_1}{2} + \mu_2 \right) \begin{pmatrix}
				1&0&0\\
				0&0&0\\
				\frac{\bar{a} \bar{v}}{\bar{c}^2_f}&0&0
			\end{pmatrix},
		\end{aligned}
	\end{equation}
	while $\theta$ satisfies
	$$\left| \theta (\vec{y}, t) \right| = \mathcal{O}(1) |\vec{y}|^{-1} e^{-|\vec{y}|/D_1} \bigl( e^{-\mu_1 t/4} + e^{-\mu_2 t} \bigr).$$
\end{Theorem}

\begin{Remark} Theorem~\ref{green} is a refined version to fit the irrotationality condition and the details of this refinement could be found in Appendix $\RN{1}$.
\end{Remark}

Theorem~\ref{green} is well-prepared for irrotational data in the sense that it combined with Kirchhoff's formula results in the following conclusion for irrotational initial data:

\begin{Theorem}\label{lpg} 
	For irrotational data $\vec{g}(\vec{x}) := (g_0, \vec{h}_0, f_0)(\vec{x})=(g_0, h_0^1, h_0^2, h_0^3, f_0)(\vec{x})$ satisfying $\nabla \times \vec{h}_0 \equiv 0$,
	there exist positive constants $C_p$ and $C_0$ such that for $|k| \ge 0$, $p \ge 2$, and $q \ge p/(p+1)$,
	\begin{equation}\label{glpl}
		\Big\|\partial_{\vec{x}}^k \mathbb{G}_{L,i}^t \vec{g}  \Big\|_{L^p} \le C_p (1 + t)^{\frac{5}{2}\left(\frac{1}{p} - \frac{1}{q}\right) + \frac{1}{2} - \frac{|k|}{2}} \left\|\vec{g} \right\|_{L^q}
		\begin{cases}
			1 & \text{for } i = 1, 2, \\
			(1 + t)^{-\frac{1}{2}} & \text{for } i = 3,
		\end{cases}
	\end{equation}
	\begin{equation}\label{glps2}
		\Big\| \partial_{\vec{x}}^k \mathbb{G}_S^t \vec{g} \Big\|_{L^2} \le C_2 e^{-t/C_2} \left\| \partial_{\vec{x}}^k \vec{g} \right\|_{L^2},
	\end{equation}
	and
	\begin{equation}\label{glps0}
		\Big\| \partial_{\vec{x}}^k \mathbb{G}_S^t \vec{g}       \Big\|_{L^\infty} \le C_0 e^{-t/C_0} \left( \left\| \partial_{\vec{x}}^k \vec{g} \right\|_{L^\infty} + \left\| \partial_{\vec{x}}^k \nabla g_0 \right\|_{L^\infty} + \left\| \partial_{\vec{x}}^k \nabla \vec{h}_0 \right\|_{L^\infty} \right).
	\end{equation}
	Here, $\mathbb{G}_{L,i}$ denotes the $i$-th row of the block third-order matrix $\mathbb{G}_L$.
\end{Theorem}

The detail about the proof of Theorem~\ref{lpg} is left in Appendix $\RN{1}$ and the following corollary is about the Kirchhoff's formula:

\begin{Corollary}[\textbf{Kirchhoff's Formula}, Theorem~4.1, \cite{z1}] \label{kirch}
	Let $\omega = \omega(\vec{x}, t; c)$ and $\omega_t = \omega_t(\vec{x}, t; c)$ be defined in \eqref{wwt}. For a function $g(\vec{x})$, the inverse Fourier transforms of $\mathbb{F}[\omega] \cdot \mathbb{F}[g]$ and $\mathbb{F}[\omega_t] \cdot \mathbb{F}[g]$ are, respectively,
	\begin{equation} \label{wg}
		(\omega \underset{\vec{x}}{\ast} g)(\vec{x}, t) = \frac{t}{4 \pi} \int_{|\vec{y}|=1} g(\vec{x} + ct\vec{y}) \, dS(\vec{y})
	\end{equation}
	and
	\begin{equation} \label{wtg}
		(\omega_t \underset{\vec{x}}{\ast} g)(\vec{x}, t) = \frac{1}{4 \pi} \int_{|\vec{y}|=1} g(\vec{x} + ct\vec{y}) \, dS(\vec{y}) + \frac{ct}{4 \pi} \int_{|\vec{y}|=1} \nabla g(\vec{x} + ct\vec{y}) \cdot \vec{y} \, dS(\vec{y}).
	\end{equation}
\end{Corollary}

The estimate \eqref{glpl} in Theorem~\ref{lpg} allows $(p,q)=(2,2)$ with a growing-in-time factor $t^{1/2}$ and the following lemma is an improvement of this case. The detailed proof is again left in the Appendix $\RN{1}$. 

\begin{Lemma}\label{impr}
	Under the assumption in Theorem~\ref{lpg}, there exists a positive constant $C_I$ such that
	\begin{equation}\label{impr0}
		\Big\|\partial_{\vec{x}}^k \mathbb{G}_{L,i}^t \vec{g} \Big\|_{L^2} \le C_I (1+t)^{-\frac{|k|}{2}} \left\|\vec{g}\right\|_{L^2} 
		\begin{cases}
			1 & \text{for } i=1,2,\\
			(1+t)^{-\frac{1}{2}} & \text{for } i=3.
		\end{cases}
	\end{equation}
\end{Lemma} 

\subsection{Energy Estimates}

Since \eqref{inlin} is quasi-linear, a priori estimates are needed for the closure of the nonlinearity. We quote the following theorem for a priori bounds for the high order derivatives of the solution:

\begin{Theorem}[\textbf{Energy estimates}: a modification of Theorem~1.1, \cite{hl}] 
\label{genes} 
	Let $\bar{\rho}$, $\bar{e}_1$, $\bar{q}$ be positive constants such that either $\bar{q} = \bar{Q}$ or $\bar{T}_1 = \bar{T}_2$ holds, the irrotational initial data $(\rho_0, \vec{u}_0, e_{1,0}, q_0)$ be such that $\|(p_0 - \bar{p}, \vec{u}_0, \chi_0)\|_{L^1}$ is bounded and $\|(\rho_0 - \bar{\rho}, \vec{u}_0, e_{1,0} -\bar{e}_1, q_0 - \bar{q})\|_{H^l}$ is sufficiently small for some integer $l \ge 3$. Then a unique global solution $(\rho, \vec{u}, e_1, q)$ satisfying the irrotational condition $\nabla \times \vec{u} \equiv 0$ with $\rho > 0, e_1 > 0, q > 0$ to the Cauchy problem \eqref{vnes} exists and satisfies
	\begin{equation} \nonumber
		(\rho - \bar{\rho}, \vec{u}, e_1 - \bar{e}_1, q - \bar{q}) \in C^0([ 0, \infty); H^l(\mathbb{R}^3)) \cap C^1( [0, \infty); H^{l-1}(\mathbb{R}^3)).
	\end{equation}
	Moreover, there exists a constant $D_1>0$ such that
	\begin{equation} \label{genenes}
		\begin{aligned}
			& \left\| (\rho - \bar{\rho}, \vec{u}, e_1 -\bar{e}_1, q - \bar{q})(\cdot,t) \right\|_{L^2}^2 + \left\| \nabla (p, \vec{u}, \chi, s)(\cdot,t) \right\|_{H^{l-1}}^2 \! + \int_0^t \! \left( \left\| \chi(\cdot,\tau) \right\|^2_{H^l} + \left\| (\nabla p, \nabla \cdot \vec{u})(\cdot,\tau) \right\|^2_{H^{l-1}} \right)\! d\tau \\
			& = \mathcal{O}(1) \Bigl( \left\| (\rho_0 - \bar{\rho}, \vec{u}_0, e_{1,0} -\bar{e}_1, q_0 - \bar{q}) \right\|_{L^2}^2 + \left\| \nabla (p_0, \vec{u}_0, \chi_0) \right\|_{H^{l-1}}^2\\
			& \quad \quad \quad \quad + \exp\left\lbrace D_1 \left\| (p_0 - \bar{p}, \vec{u}_0, \chi_0 ) \right\|_{L^1 \cap H^l} \right\rbrace \left\| \nabla s_0 \right\|_{H^{l-1}}^2 \Bigr).
		\end{aligned}
	\end{equation}
And the solution $(p, \vec{u}, \chi)$ has the following decay property:
	\begin{equation} \label{pu1xde}
		\left\| \chi(\cdot, t) \right\|_{H^l} + \left\| \nabla (p, \vec{u})(\cdot, t) \right\|_{H^{l-1}} = \mathcal{O}(1) (1 + t)^{-\frac{5}{4}} \left\| (p_0 - \bar{p}, \vec{u}_0, \chi_0 ) \right\|_{L^1 \cap H^l}, \quad \forall t \ge 0.
	\end{equation}
\end{Theorem}

\begin{Remark} Take $l=3$ in Theorem~\ref{genes} and it goes back to Theorem~1.1 in \cite{hl}.
\end{Remark}

A priori estimates for higher order derivatives of the solution are required for a pointwise space-time structure of the solution in Theorem \ref{main1}. Thus one applies energy method for higher order derivatives to obtain Theorem~\ref{genes} and the proof is similar to the one for Theorem~1.1 in \cite{hl}. For the convenience of the readers, we provide the details of the proof in Appendix $\RN{2}$. 

The followings are some Sobolev inequalities for the nonlinear terms:
\begin{Lemma}[Lemma~2.1, \cite{hl}] For $\omega_1\in H^l(\mathbb{R}^3), \nabla\omega_1\in L^\infty(\mathbb{R}^3)$ and $\omega_2\in H^{l-1}(\mathbb{R}^3)\cap L^\infty(\mathbb{R}^3)$ and a multi-index $k$,
\begin{equation} \label{sob1}
		\sum_{1 \le  |k|\le l} \left\|\bigl[\partial_{\vec{x}}^k, \omega_1\bigr]\omega_2 \right\|_{L^2} \le C_{l} \left( \left\| \nabla \omega_1\right\|_{L^\infty} \cdot\left\|\nabla^{l-1} \omega_2 \right\|_{L^2} + \left\|\omega_2 \right\|_{L^\infty}\cdot \left\|\nabla^{l} \omega_1 \right\|_{L^2} \right),
	\end{equation}
	where the commutator $[\cdot,\cdot]$ is defined by $$\bigl[\partial_{\vec{x}}^k, \omega_1\bigr]\omega_2:= \partial_{\vec{x}}^k (\omega_1 \omega_2) - \omega_1 \partial_{\vec{x}}^k \omega_2.$$ 
	
	Particularly, for $\omega_1, \omega_2\in H^{l}(\mathbb{R}^3)\cap L^\infty(\mathbb{R}^3)$,
	\begin{equation} \label{sob2}
		\sum_{|k| \le l} \left\|\partial_{\vec{x}}^k (\omega_1 \omega_2)\right\|_{L^2} \le C_{l} \left(\left\|\omega_1\right\|_{L^\infty} \cdot\left\|\nabla^{l} \omega_2 \right\|_{L^2} + \left\|\omega_2 \right\|_{L^\infty}\cdot \left\|\nabla^{l} \omega_1 \right\|_{L^2} \right),
	\end{equation}
	where 
	\begin{equation}\label{wemb}
		\left\|\omega_1\right\|_{L^\infty}\le C\left\|\nabla\omega_1\right\|^{1/2}_{L^2}\cdot \left\|\nabla^2\omega_1\right\|_{L^2}^{1/2}\le C\left\|\nabla\omega_1\right\|_{H^1},\ \ \forall \omega_1\in H^2(\mathbb{R}^3).
	\end{equation}
\end{Lemma}

\subsection{Computational Lemmas for Wave Interactions} Later, in the study of pointwise estimates of the solutions, wave interactions could be quite complex. Here, one lists all the computational lemmas used there and their proofs are left in Appendix $\RN{3}$.

The first lemma below is for the estimates of the entropy function $s-\bar{s}$ and its derivatives:
\begin{Lemma}\label{lemmains}
	Let $\bar{c}$ be a positive constant. For $d \ge 0$, $i \ge 0$, $m > 1$, and $j \ge 2$, it holds that
	\begin{equation} \label{tconv1}
		\begin{aligned}
			&\int_0^t (1+\tau)^{-d} (1 + |\vec{x}| + \sqrt{\tau})^{-i} \left( (|\vec{x}|-\bar{c}\tau)^2 + \tau + 1 \right)^{-m} d\tau\\
			&\le C (|\vec{x}|^2 + 1)^{-\frac{i}{2}-m} 
			\begin{cases}
				(1+|\vec{x}|/(2\bar{c}))^{-d+1} & \text{for } 0 \le d < 1 \\
				\ln(1+|\vec{x}|/(2\bar{c})) & \text{for } d=1 \\
				1 & \text{for } d>1
			\end{cases}\\
			& \quad + C(|\vec{x}|^2 + 1)^{-\frac{i+d+m}{2}+\frac{1}{4}} + C(|\vec{x}|^2 + 1)^{-\frac{i+d}{2}-m+1},
		\end{aligned}
	\end{equation}
	and
		\begin{equation} \label{tconv2}
		\begin{aligned}
			&\int_0^t (1+\tau)^{-d} \frac{H(\bar{c}\tau - |\vec{x}|)}{1 + (|\vec{x}| + \sqrt{\tau})^i(\bar{c}\tau - |\vec{x}| + \sqrt{\tau})^j} d\tau \\
			&\le C H(\bar{c}t - |\vec{x}|) (|\vec{x}| + 1)^{-d-i}	\begin{cases}
				\ln(1+2|\vec{x}|/\bar{c}) + (|\vec{x}|+1)^{-1} & \text{for } j=2, \\
				(|\vec{x}|+1)^{-\frac{j}{2}+1} & \text{for } j>2.
			\end{cases}
		\end{aligned}
	\end{equation}
	
\end{Lemma}

The next lemma is prepared for the estimates of initial propagations:
\begin{Lemma}\label{lemmain}
Let $K$, $\mu$, $\bar{c}$ and $\bar{c}_f$ be positive constants. For $j=0,1,2$ and $d>3/2$, it holds that
\begin{align} \label{iniconv1}
		&e^{-\frac{(|\vec{x}|-\bar{c}t)^2}{K(1+t)}} \underset{\vec{x}}{\ast} (|\vec{x}|^2+1)^{-d} 
		 \le C  e^{-\frac{(|\vec{x}|-\bar{c}t)^2}{4K(1+t)}} + C(1+t)^{\frac{5}{2}} \left( (|\vec{x} | - \bar{c} t)^2 + t+1\right)^{-d},\\
 \label{iniconv2}
		& |\vec{x}|^{-j} e^{-\mu(|\vec{x}|+t)} \underset{\vec{x}}{\ast} (|\vec{x}|^2+1)^{-d} \le C e^{-\mu t}(|\vec{x}|^2+1)^{-d} ,\\
\label{iniconv3}
		 & e^{-\mu t} \omega(\vec{x},t;\bar{c}_f) \underset{\vec{x}}{\ast} (|\vec{x}|^2+1)^{-d} \le C e^{-\mu t/4} (|\vec{x}|^2+1)^{-d}.
\end{align}
Here, $\omega(\vec{x},t;\bar{c}_f)$ is defined in \eqref{wwt}.
\end{Lemma}

Finally, all the remaining lemmas are for the nonlinear wave interactions.  One needs to consider the spatial and temporal convolution of the terms from the Green’s function in \eqref{g1}--\eqref{leading} with the nonlinear ansatz in \eqref{f0}--\eqref{f01im}:
\begin{Lemma}\label{lemmainker}
	Let $\bar{c}, m$ and $K$ be positive constants. It holds that
	\begin{equation} \label{heke2}
		\begin{aligned}
			&(1+t)^{-2-\frac{m}{2}} e^{-\frac{(|\vec{x}| - \bar{c} t)^2}{K(1 + t)}} \underset{(\vec{x},t)}{\ast} (1+t)^{-\frac{1}{2}} e_1^2(\vec{x},t)\\
			&\le C (1+t)^{-\frac{\min\{m,5\}}{2}} \left( (|\vec{x}| - \bar{c} t)^2 + t+1\right)^{-2} \\
			&\quad + C H(\bar{c}t - |\vec{x}|)\left[(1+t)^{-2-\frac{m}{2}} \left( (|\vec{x}| - \bar{c} t)^2 + t+1\right)^{-\frac{1}{2}} + (1+t)^{-\frac{9}{2}} 
			\begin{cases}
				(1+t)^{\frac{3-m}{2}} & \text{for } 0 \le m < 3 \\
				\ln(1+t) & \text{for } m=3 \\
				1 & \text{for } m > 3
			\end{cases} \right], 
			\end{aligned}
			\end{equation}
		\begin{equation} \label{heke20}
		\begin{aligned}	
			& (1+t)^{-2-\frac{m}{2}} e^{-\frac{(|\vec{x}| - \bar{c} t)^2}{K(1 + t)}} \underset{(\vec{x},t)}{\ast} (1+t)^{-\frac{1}{2}} (|\vec{x}|^2 + 1)^{-\frac{9}{4}} \left( (|\vec{x}| - \bar{c} t)^2 + t+1\right)^{-2} \\
			& \le C(1+t)^{-\frac{\min\{m,5\}}{2}} \left( (|\vec{x}| - \bar{c} t)^2 + t+1\right)^{-2}  \\
			&\quad + C H(\bar{c}t - |\vec{x}|)\left[(1+t)^{-\frac{9}{2}} + (1+t)^{-1-\frac{m}{2}} \left( (|\vec{x}| - \bar{c} t)^2 + t+1\right)^{-\frac{9}{4}} \right],\\[1ex] 
			& (1+t)^{-2-\frac{m}{2}} e^{-\frac{(|\vec{x}| - \bar{c} t)^2}{K(1 + t)}} \underset{(\vec{x},t)}{\ast} (1+t)^{-\frac{1}{2}} (|\vec{x}|^2 + 1)^{-\frac{9}{4}} \frac{H(\bar{c}t - |\vec{x}|)}{1 + (|\vec{x}| + \sqrt{t})(\bar{c}t - |\vec{x}| + \sqrt{t} )^2 } \\
			&\le C (1+t)^{-2-\frac{\min\{m,2\}}{2}}e^{-\frac{(|\vec{x}| - \bar{c} t)^2}{C_{K}(1 + t)}} + C H(4\bar{c}\sqrt{t}-|\vec{x}|)(1+t)^{-3} \\
			&\quad + C H(\bar{c}t - |\vec{x}|) \left[(1+t)^{-2-\frac{m}{2}} \left( (|\vec{x}| - \bar{c} t)^2 + t+1\right)^{-1} + 
			(1+t)^{-\frac 72-\frac{\min\{m,7\}}{4}}\right],  \\[1ex]
			\end{aligned}
			\end{equation}
		\begin{equation} \label{heke21}
		\begin{aligned}	
			& (1+t)^{-2 -\frac{m}{2}} e^{-\frac{(|\vec{x}| - \bar{c} t)^2}{K(1 + t)}} \underset{(\vec{x},t)}{\ast} \left( (|\vec{x}| - \bar{c} t)^2 + t+1\right)^{-2} e_2(\vec{x},t) \\
			&\le C(1+t)^{-\frac{\min\{m,6\}}{2}}  \left( (|\vec{x}| - \bar{c} t)^2 + t+1\right)^{-2}  \\
			&\quad + C H(\bar{c}t - |\vec{x}|) \left[(1+t)^{-2-\frac{m}{2}} \left( (|\vec{x}| - \bar{c} t)^2 + t+1\right)^{-\frac{3}{4}} + (1+t)^{-5} 
			\begin{cases}
				(1+t)^{\frac{3-m}{2}} & \text{for } 0 \le m < 3 \\
				\ln(1+t) & \text{for } m=3 \\
				1 & \text{for } m > 3
			\end{cases} \right].
		\end{aligned}
	\end{equation}
	Here, $C_{K}$ denotes a positive constant depending on $K$ and
	\begin{equation} \label{e1e2}
	\begin{aligned}
		e_1(\vec{x}, t) &:= \left((|\vec{x} | - \bar{c} t)^2 + t+1 \right)^{-2} +  \frac{H(\bar{c}t - |\vec{x}|)}{1 + (|\vec{x}| + \sqrt{t})(\bar{c}t - |\vec{x}| + \sqrt{t} )^2 }, \\
		e_2(\vec{x}, t) &:= \bigl(1 + |\vec{x}| + \sqrt{t} \, \bigr)^{-2} \left( (|\vec{x}| -\bar{c}t )^2 + t + 1 \right)^{-1}.
	\end{aligned}
\end{equation}	
\end{Lemma}
By virtue of the inequality
$$e_2(\vec{x},t) \cdot \frac{H(\bar{c}t - |\vec{x}|)}{1 + (|\vec{x}| + \sqrt{t})(\bar{c}t - |\vec{x}| + \sqrt{t})^2 } \le C(1+t)^{-\frac{1}{2}} e_1^2(\vec{x},t),$$
we omit the extra convolution computation of this term: $H(\bar{c}t - |\vec{x}|) e_2(\vec{x},t)\bigl[1 + (|\vec{x}| + \sqrt{t})(\bar{c}t - |\vec{x}| + \sqrt{t} )^2 \bigr]^{-1}$ in Lemma~\ref{lemmainker}.

\begin{Lemma}\label{lemmainexp}
Let $D_0,\bar{c}$ be positive constants and $e_1,e_2$ be defined in \eqref{e1e2}. For $i \ge 0$, $m \ge 0$, and $j=0,1,2$, it holds that
\begin{equation}\label{exponxt}
	\begin{aligned}
		& |\vec{x}|^{-j} e^{-(|\vec{x}|+t)/D_0} \underset{(\vec{x},t)}{\ast} (1+t)^{-i} e_1^m(\vec{x},t)  \\
		&\le C \left[ (1+t)^{-i} \left( (|\vec{x}| - \bar{c} t)^2 + t + 1 \right)^{-2m} + H(\bar{c}t - |\vec{x}|)(1+t)^{-i-2m} \right], 
	\end{aligned}
\end{equation}
\begin{equation}\label{exponxt0}
	\begin{aligned}
		& |\vec{x}|^{-j} e^{-(|\vec{x}|+t)/D_0} \underset{(\vec{x},t)}{\ast} (1+t)^{-\frac{1}{2}} (|\vec{x}|^2 + 1)^{-\frac{9}{4}} e_1(\vec{x},t) \\ 
		&\le C \Bigl[ (1+t)^{-\frac{1}{2}} (|\vec{x}|^2+1)^{-\frac{9}{4}} \left( (|\vec{x}| - \bar{c} t)^2 + t + 1\right)^{-2} \\
		& \quad \quad \quad + H(\bar{c}t - |\vec{x}|) \left( (|\vec{x}|^2+1)^{-\frac{9}{4}} (1+t)^{-3} + (1+t)^{-\frac{9}{2}} \right) \Bigr], 
	\end{aligned}
\end{equation}
\begin{equation}\label{exponxt1}
	\begin{aligned}
		& |\vec{x}|^{-j} e^{-(|\vec{x}|+t)/D_0} \underset{(\vec{x},t)}{\ast} \left( (|\vec{x}| - \bar{c} t)^2 + t + 1\right)^{-2} e_2(\vec{x},t) \\ 
		&\le C \left[ (|\vec{x}|+\bar{c}t+1)^{-2} \left( (|\vec{x}| - \bar{c} t)^2 + t + 1\right)^{-3} + H(\bar{c}t - |\vec{x}|)(1+t)^{-5} \right].
	\end{aligned}
\end{equation}
\end{Lemma}

When computing the convolutions of $G_{0,1}^{s,*}(\vec{x},t)$ and $G_{0,2}^s(\vec{x},t)$ with the source terms, the following results are employed:
\iffalse
\begin{Lemma}\label{lemmainexp2}
	Let $\mu_2$ and $\bar{c}$ denote positive constants, and let $e_1,e_2$ be defined in \eqref{e1e2}. Then it holds for all $i \ge 0$ and $m \ge 0$ that
	\begin{equation} \label{expont}
		\begin{aligned}
			& e^{-\mu_2 t} \underset{t}{\ast} (1+t)^{-i} e_1^m(\vec{x},t) \le C\left[ (1+t)^{-i}\left((|\vec{x}|-\bar{c}t)^2+t+1\right)^{-2m} + H(\bar{c}t-|\vec{x}|)(1+t)^{-i-2m}\right], \\[0.5ex]
			& e^{-\mu_2 t} \underset{t}{\ast} (1+t)^{-\frac{1}{2}} (|\vec{x}|^2 + 1)^{-\frac{9}{4}} e_1(\vec{x},t) \\
			&\le C \Bigl[ (1+t)^{-\frac{1}{2}} (|\vec{x}|^2+1)^{-\frac{9}{4}} \left( (|\vec{x}| - \bar{c} t)^2 + t + 1\right)^{-2}\\
			&\quad \quad \quad + H(\bar{c}t - |\vec{x}|)\left( (|\vec{x}|^2+1)^{-\frac{9}{4}} (1+t)^{-3} + (1+t)^{-7} \right) \Bigr], \\[0.5ex]
			& e^{-\mu_2 t} \underset{t}{\ast} \left( (|\vec{x}| - \bar{c} t)^2 + t + 1\right)^{-2} e_2(\vec{x},t) \le C \left( (|\vec{x}|-\bar{c} t)^2 + t + 1\right)^{-2} e_2(\vec{x},t).
		\end{aligned}
	\end{equation}
\end{Lemma}
\fi

\begin{Lemma}\label{lemmainexpw}
Let $\mu,\bar{c}_f,\bar{c}$ be positive constants with $\bar{c}<\bar{c}_f$, $l,d\ge 0$, $e_1$ be defined in \eqref{e1e2}, and $\omega$ be given by \eqref{wwt} and \eqref{wg}, one has that
	\begin{equation} \label{exponwwt}
		\begin{aligned}
			& e^{-\mu t} \omega(\vec{x},t;\bar{c}_f) \underset{(\vec{x},t)}{\ast} (1+t)^{-l} (|\vec{x}|^2+1)^{-d} e_1(\vec{x},t) \\
			&\le C H(|\vec{x}|-\bar{c}t)(1+t)^{-l}(|\vec{x}|^2 + 1)^{-d}  \left((|\vec{x}|-\bar{c}t)^2+t+1\right)^{-2}\\
			& \quad +CH(\bar{c}t-|\vec{x}|)\left[ (|\vec{x}|^2+1)^{-d}(1+t)^{-l-\frac{5}{2}} + (1+t)^{-l-2-2d}\right],
			\end{aligned}
	\end{equation}
	\begin{equation} \label{exponwwt0}
		\begin{aligned}
			& e^{-\mu t} \omega(\vec{x},t;\bar{c}_f) \underset{(\vec{x},t)}{\ast} H(|\vec{x}| - \epsilon \sqrt{t} ) (1+t)^{-\frac{5}{2}} (|\vec{x}|^2 + 1)^{- \frac{9}{4}} \\
			&\le C(1+t)^{-\frac{5}{2}} \left[ H(|\vec{x}| - \bar{c} t) (|\vec{x}|^2+1)^{-\frac{9}{4}} + H(\bar{c}t - |\vec{x}|) (1+t)^{-\frac{9}{4}} \right], 
		\end{aligned}
	\end{equation}
	\begin{equation} \label{exponwwt1}
		\begin{aligned}
			& e^{-\mu t} \omega(\vec{x},t;\bar{c}_f) \underset{(\vec{x},t)}{\ast} H(\epsilon \sqrt{t} - |\vec{x}|) (1 + t)^{-3} \\
			&\le C (1+t)^{-3} \left[ H(|\vec{x}|-2\epsilon\sqrt{t}) e^{-\frac{\mu(|\vec{x}|+2\epsilon \sqrt{t})}{16\bar{c}_f}} + H(2\epsilon\sqrt{t}-|\vec{x}|) \right].
		\end{aligned}
	\end{equation}
\end{Lemma}

Convolutions involving $\nabla \omega$ and $\omega_t$ are treated identically to $\omega$ as shown above, and we omit the details.

\begin{Lemma} \label{lemmain1}
Let $D_0$ be a positive constant. For $j=0,1,2$, it holds that:
	\begin{equation} \label{exponxt1t3}
		\begin{aligned}
			&|\vec{x}|^{-j} e^{-(|\vec{x}|+t)/D_0} \underset{(\vec{x},t)}{\ast} H(|\vec{x}|-\epsilon \sqrt{t})(1+t)^{-\frac{5}{2}} (|\vec{x}|^2 + 1)^{-\frac{9}{4}}\\
			&\le C \left[ H(|\vec{x}|-\bar{c}t) (1+t)^{-\frac{5}{2}} (|\vec{x}|^2+1)^{-\frac{9}{4}} + H(\bar{c}t - |\vec{x}|) (1+t)^{-\frac{19}{4}} \right], \\[0.3ex]
			& |\vec{x}|^{-j} e^{-(|\vec{x}|+t)/D_0} \underset{(\vec{x},t)}{\ast} H(\epsilon \sqrt{t}-|\vec{x}|)(1+t)^{-3}\\
			& \le C (1+t)^{-3} \left( H(|\vec{x}|-2\epsilon \sqrt{t}) e^{-\frac{|\vec{x}|+2\epsilon \sqrt{t}}{8D_0}} + H(2\epsilon \sqrt{t}-|\vec{x}|) \right).
		\end{aligned}
	\end{equation}
\end{Lemma}

%%%%%%%%%%%%%%%%%%%%%%%%%%%%%%%%%%%%%%%%%%%%%%%%%%%%%%%%%%%%%%%%%%%%%%%%%%%%%%%%%5
%
%
%                       Section 3:
%
%
%%%%%%%%%%%%%%%%%%%%%%%%%%%%%%%%%%%%%%%%%%%%%%%%%%%%%%%%%%%%%%%%%%%%%%%%%%%%%%%%%

\section{Pointwise Space-Time Estimates for Solutions of the Individual Perturbation Problem: Proof of Theorem~\ref{main1}}
In this section, one derives the pointwise space-time structures of $p-\bar{p}$, $\vec{u}$, $\chi$ and $s-\bar{s}$ obtained by the refined estimates of the Green's function given in Theorem~\ref{green} and Duhamel's principle. For this purpose, the following time-dependent functional $\mathcal{N}(t)$ is introduced and the desired pointwise structures are described by the weight functions $e_1, e_2$:
		\begin{equation} \label{anas}
			\begin{aligned}
				\mathcal{N}(t)&:= \underset{0 \le \tau \le t}{\sup} \Big\lbrace \left\|\left(e_1^{-1}(|p-\bar{p}|+|\vec{u}|)\right) \! (\cdot, \tau)\right\|_{L^{\infty}} + (1+\tau)^{\frac{1}{2}} \left\|\left(e_1^{-1}(|\nabla p|+|\nabla \vec{u}|+|\chi|)\right) \! (\cdot, \tau)\right\|_{L^{\infty}}\\
				& \quad \quad \quad \quad \quad + \left\|\left(e_2^{-1}|\nabla \chi|\right) \! (\cdot, \tau)\right\|_{L^{\infty}} + (1+\tau)^3 \left\|\left(|\nabla^2 p|+|\nabla^2 \vec{u}|+|\nabla^2 \chi|\right) \! (\cdot, \tau)\right\|_{L^{\infty}}\\
				& \quad \quad \quad \quad \quad + \left\|\left((|\vec{x}|^2+1)^{\frac{9}{4}}( |s-\bar{s}|+|\nabla s|)\right) \! (\cdot, \tau)\right\|_{L^{\infty}} + \sum_{i=2}^3 \left\|\left((|\vec{x}|^2+1)^{3-\frac{3i}{4}}|\nabla^i s| \right) \! (\cdot, \tau)\right\|_{L^{\infty}} \Big\rbrace
			\end{aligned}
		\end{equation}
with $e_1(\vec{x}, t), e_2(\vec{x}, t)$ defined in \eqref{e1e2}. Here, $\bar{c}^2 = \gamma \bar{p} \bar{v}/(1+b)$, where $\gamma$ and $b$ are specified in \eqref{equ} and \eqref{bb}, respectively, and $H(r)$ denotes the step function defined in \eqref{stepfun}.

\subsection{Temporal Decay Estimates}

One first establishes the following decay estimates for the higher-order derivatives, which rely on the {{temporal}} decay rates of the $L^\infty$-norms of $(p,\vec{u},\chi)(\vec{x},t)$ and their first- and second-order spatial derivatives.

\begin{Lemma}[\textbf{Decay Property}] \label{decay}
	Suppose that the assumptions in Theorem~\ref{main1} are satisfied. Then, for $i = 1,2,3$, the following temporal decay estimates hold for the solution $(p, \vec{u}, \chi)$ and derivatives:
	\begin{equation} \label{decaygen}
		\left\| \nabla^i \chi(\cdot, t) \right\|_{H^{l_i - i}} + \left\| \nabla^{i+1} (p, \vec{u})(\cdot, t) \right\|_{H^{l_i - (i+1)}} \le C(1 + t)^{-m_i} \left( \left\| (p_0 - \bar{p}, \vec{u}_0, \chi_0 ) \right\|_{L^1 \cap H^{l_i}} +{{ \mathcal{N}(t)}}\right), \quad \forall t \ge 0.
	\end{equation}
where $m_1 = 7/4, m_2 = 9/4, m_3 = 5/2$, and $l_i$ are regularity indices for the Sobolev spaces satisfying $l_i-(i + 1) > 3/2$ for $i = 1,2,3$.
\end{Lemma}

\begin{Remark}
	It is obvious that for $l_1,l_2,l_3 \le 8$, the quantity $\left\|(p_0 - \bar{p}, \vec{u}_0, \chi_0)\right\|_{L^1 \cap H^{l_i}}$  may be uniformly replaced by $\epsilon_0$, in agreement with the definition of $\epsilon_0$ from Theorem~\ref{main1}.
\end{Remark}

\begin{Remark} \label{rebetaenes}
	Both Theorem~\ref{genes} and Lemma~\ref{decay} hold for a general steady-state $\bar{\vec{u}}\neq\vec 0$. 
\end{Remark}

\begin{proof}
To verify \eqref{decaygen}, one defines time-dependent functionals
		$$ \mathcal{M}_{i}(t):= \underset{0 \le \tau \le t}{\sup} (1+\tau)^{2m_i} \left( \left\|\nabla^{i} \chi (\cdot, \tau)\right\|_{H^{l-i}}^2 + \left\|\nabla^{i+1} (p, \vec{u})(\cdot, \tau)\right\|_{H^{l-i-1}}^2 \right),\ \ i=1,2,3.$$
Similar to \eqref{puderi1}, from Theorem~\ref{lpg}, Lemma~\ref{impr} and Duhamel’s principle, one has that 
	\begin{equation} \label{puderim}
		\begin{aligned}
			\left\|\nabla^{i+1} (p, \vec{u})(\cdot, t)\right\|_{L^2} &\le C(1+t)^{-\frac{3}{4}-\frac{i+1}{2}} \left\|(p_0-\bar{p},\vec{u}_0,\chi_0)\right\|_{L^1 \cap H^{i+1}} \\
			& \quad + C \int_0^t e^{-(t-\tau)/C_2} \left\|(R_p, R_{\vec{u}}, R_{\chi})(\cdot, \tau)\right\|_{H^{i+1}} d\tau \\
			& \quad + C \int_0^{t/2} (1+t-\tau)^{-\frac{3}{4}-\frac{i+1}{2}} \left\|(R_p, R_{\vec{u}}, R_{\chi})(\cdot, \tau)\right\|_{L^1} d\tau\\
			& \quad +  C \int_{t/2}^t (1+t-\tau)^{-\frac{i+1}{2}} \left\|(R_p, R_{\vec{u}}, R_{\chi})(\cdot, \tau)\right\|_{L^2} d\tau.
		\end{aligned}
	\end{equation}

	For $i=1$, it holds that
	\begin{equation} \label{Rpuxl1}
		\left\|(R_p, R_{\vec{u}}, R_{\chi})(\cdot, \tau)\right\|_{L^1} \le C \epsilon_0 \left(\left\|\chi (\cdot, \tau)\right\|_{H^1} + \left\|\nabla (p, \vec{u})(\cdot, \tau)\right\|_{L^2} \right) \le C \epsilon_0 (1+\tau)^{-\frac{5}{4}} \left\|(p_0-\bar{p},\vec{u}_0,\chi_0)\right\|_{L^1 \cap H^3},
	\end{equation}
	\begin{equation} \label{Rpuxl2}
		\left\|(R_p, R_{\vec{u}}, R_{\chi})(\cdot, \tau)\right\|_{L^2} \le C \epsilon_0 \left( \left\|\nabla (p, \vec{u})(\cdot, \tau)\right\|_{L^\infty} + \left\|\chi (\cdot, \tau) \right\|_{W^{1,\infty}} \right) \le C \epsilon_0 (1+\tau)^{-\frac{5}{2}} {{\mathcal{N}(\tau)}},
	\end{equation}
	and
	$$\left\|(R_p, R_{\vec{u}}, R_{\chi})(\cdot, \tau)\right\|_{H^2} \le C \epsilon_0 \left( \left\|\nabla \chi (\cdot, \tau)\right\|_{H^2} + \left\|\nabla^2 (p, \vec{u})(\cdot, \tau)\right\|_{H^1} \right), $$
where \eqref{Rpuxl1} and \eqref{Rpuxl2} are derived by utilizing the decay estimate \eqref{pu1xde} in Theorem~\ref{genes} and the definition of $\mathcal{N}(\tau)$ in \eqref{anas}, respectively, and $\epsilon_0$ is defined in \eqref{inida}. The above estimates together with \eqref{puderim} yield that
	\begin{equation} \label{puml2}
		\begin{aligned}
			\left\|\nabla^2 (p, \vec{u})(\cdot, t)\right\|_{L^2} &\le  C(1+t)^{-\frac{7}{4}} \left\|(p_0-\bar{p},\vec{u}_0,\chi_0) \right\|_{L^1 \cap H^2} + C \epsilon_0 \int_0^t e^{-(t-\tau)/C_2} (1+\tau)^{-\frac{7}{4}} \sqrt{\mathcal{M}_1(\tau)} \, d\tau \\
			& \quad + C \epsilon_0 \left\|(p_0-\bar{p},\vec{u}_0,\chi_0) \right\|_{L^1 \cap H^3} \int_0^{t/2} (1+t-\tau)^{-\frac{7}{4}} (1+\tau)^{-\frac{5}{4}} \, d\tau \\
			& \quad + C \epsilon_0 \int_{t/2}^t (1+t-\tau)^{-1} (1+\tau)^{-\frac{5}{2}} {{\mathcal{N}(\tau)}} \, d\tau \\
			&\le C(1+t)^{-\frac{7}{4}}\left[ \left\|(p_0-\bar{p},\vec{u}_0,\chi_0)\right\|_{L^1 \cap H^3} + \epsilon_0 \left(\sqrt{\mathcal{M}_1(t)} + (1+t)^{-\frac{3}{4}} \ln (1+t) {{\mathcal{N}(t)}} \right)\right].
		\end{aligned}
	\end{equation}
	Substituting \eqref{puml2} into \eqref{puximpro0} with $m=2$ and then applying Gronwall's inequality to derive that:
	\begin{equation} \nonumber 
		\begin{aligned}
			\left\|\nabla \chi (\cdot, t)\right\|_{H^{l-1}}^2  + \left\|\nabla^2 (p, \vec{u}) (\cdot, t)\right\|_{H^{l-2}}^2 & \le C e^{- t/C} \left(\left\|\nabla \chi_0 \right\|_{H^{l-1}}^2  + \left\|\nabla^2 (p_0, \vec{u}_0) \right\|_{H^{l-2}}^2 \right)\\
			& \quad + C (1+t)^{-\frac{7}{2}} \left(\left\|(p_0-\bar{p},\vec{u}_0,\chi_0) \right\|_{L^1 \cap H^3}^2 + \epsilon_0^2  \mathcal{M}_1(t)+{{\epsilon_0^2 \mathcal{N}^2(t)}}\right),
		\end{aligned}
	\end{equation}
	and thus	$$ \mathcal{M}_1(t) \le C \left( \left\|(p_0-\bar{p},\vec{u}_0,\chi_0)\right\|_{L^1 \cap H^l}^2 + \epsilon_0^2  \mathcal{M}_1(t) + {{\epsilon_0^2 \mathcal{N}^2(t)}} \right),$$
	which implies
	$$\mathcal{M}_1(t) \le C \left( \left\|(p_0-\bar{p},\vec{u}_0,\chi_0) \right\|_{L^1 \cap H^l}^2 + {{\epsilon_0^2 \mathcal{N}^2(t)}} \right),$$
	and verifies \eqref{decaygen} for $i=1$.
	
	For $i=2,3$, it could be verified similarly since the following estimates hold for $(R_p, R_{\vec{u}}, R_{\chi})$:
	\begin{equation} \label{puxh3}
		\begin{aligned}
			\left\|(R_p, R_{\vec{u}}, R_{\chi})(\cdot, \tau)\right\|_{H^3} &\le C \epsilon_0 \left( {{\left\|(\nabla p, \nabla \vec{u}, \chi)(\cdot, \tau)\right\|_{L^\infty}}} + \left\|\nabla^2 \chi (\cdot, \tau)\right\|_{H^2} + \left\|\nabla^3 (p, \vec{u})(\cdot, \tau)\right\|_{H^1} \right) \\
			&\quad +{{ C \left\| (p-\bar{p}, \vec{u})(\cdot, \tau)\right\|_{L^\infty} \cdot \left\|\nabla^2 (p, \vec{u})(\cdot, \tau)\right\|_{L^2}}}\\
			&\le C \epsilon_0 \left( (1+\tau)^{-\frac{5}{2}}{{\mathcal{N}(\tau)}} + (1+\tau)^{-\frac{9}{4}} \sqrt{\mathcal{M}_2(\tau)} \right) \\
			& \quad + C (1+\tau)^{-2-\frac{5}{4}} {{\mathcal{N}(\tau)}} \left\|(p_0-\bar{p},\vec{u}_0,\chi_0)\right\|_{L^1 \cap H^3}
		\end{aligned}
	\end{equation}
	and
	\begin{equation} \label{puxh4}
		\begin{aligned}
			\left\|(R_p, R_{\vec{u}}, R_{\chi})(\cdot, \tau)\right\|_{H^4} &\le C \epsilon_0 \left( {{\left\|\nabla(p,\vec{u})(\cdot, \tau)\right\|_{W^{1,\infty}} + \left\| \chi (\cdot, \tau)\right\|_{W^{2,\infty}} }}+ \left\|\nabla^3 \chi (\cdot, \tau)\right\|_{H^2} + \left\|\nabla^4 (p, \vec{u})(\cdot, \tau)\right\|_{H^1} \right) \\
			&\quad +{{ C \left\|\nabla^4 \left(\frac{1}{\bar{v}}-\frac{1}{v}, \bar{v}-v, \bar{p}\bar{v}-pv\right) \! (\cdot, \tau)\right\|_{L^2} \cdot \left\|(\nabla p, \nabla \cdot \vec{u}, \chi)(\cdot, \tau)\right\|_{L^\infty}}}\\
			&\le C \epsilon_0 (1+\tau)^{-\frac{5}{2}} \left({{\mathcal{N}(\tau)}}  +\sqrt{\mathcal{M}_3(\tau)} \right),
		\end{aligned}
	\end{equation}
	where the final inequalities in \eqref{puxh3} and \eqref{puxh4} are consequences of the global existence result \eqref{genenes} and the decay estimates \eqref{pu1xde} in Theorem~\ref{genes}, as well as the definition of $\mathcal{N}(\tau)$ in \eqref{anas}. Accordingly, invoking \eqref{puderim}--\eqref{Rpuxl2}, \eqref{puxh3}--\eqref{puxh4} and \eqref{puximpro0}, and following the analogous argument for $i=1$, one could conclude that \eqref{decaygen} also holds for $i=2,3$.
\end{proof}

\subsection{Compensation for Unaccounted Temporal Decay in $|\vec{x}| \le \epsilon t$}
 It should be noted that $(s-\bar{s})(\vec{x},t)$ only decays in spatial variables but lacks temporal decay properties. Thus, in finite Mach regions, nonlinear terms involving the entropy function may not decay fast enough. Next, we aim to investigate whether unaccounted temporal decays in the region $|\vec{x}| \le \epsilon t$ exist for certain nonlinear couplings.
 
For this purpose, we develop the approach introduced in \cite{z2} for one-dimensional problems and define a function $F_{m}(t,\epsilon)$ to remedy the deficiency in temporal decay:
\begin{equation} \label{fk}
	\begin{aligned}
		F_{m}(t, \epsilon)&:= \ \int_{C_{\epsilon t}} (\epsilon^2 |\vec{x}|^2 + 1)^{-\frac{9}{2}} \sum_{|k|=m} \left[ \bigl( \partial_{\vec{x}}^k p\bigr)^2 + \bigl| \partial_{\vec{x}}^k \vec{u} \bigr|^2 + \bigl(\partial_{\vec{x}}^k \chi \bigr)^2 \right] \! (\vec{x}, t) \, d\vec{x} \\
		&\quad + \int_{t/2}^t \int_{C_{\epsilon t}} e^{-\epsilon (t - \tau)} (\epsilon^2 |\vec{x}|^2 + 1)^{-\frac{9}{2}} \left\lbrace \sum_{|k|=1}^{m} \left[ \bigl(\partial_{\vec{x}}^k p \bigr)^2 + \bigl| \partial_{\vec{x}}^k \vec{u} \bigr|^2 \right]  + \sum_{|k|=0}^{m} \bigl(\partial_{\vec{x}}^k \chi \bigr)^2 \right\rbrace \! (\vec{x}, \tau) \, d\vec{x} d\tau,
	\end{aligned}
\end{equation}
where $C_{\epsilon t}:= [-\epsilon t, \epsilon t] \times [-\epsilon t, \epsilon t] \times [-\epsilon t, \epsilon t]$ denotes the cubic spatial domain with the original point as the center and scaled by $\epsilon t$. Then one has the following estimates for the function $F_{m}(t, \epsilon)$:

\begin{Lemma} \label{lefk}
	Under the assumptions of Theorem~\ref{main1}, for $t \ge 0$, $2 \le m \le 5$ and sufficiently small $\epsilon > 0$, the following recursive relation holds:
	\begin{equation} \label{recurfk}
		F_{m}(t, \epsilon) \le C_U \left( F_{m-1}(t, \epsilon) + \epsilon_0^2  e^{-\epsilon t/2} + \epsilon \epsilon_0^2  (\epsilon^4 t^2 + 1)^{-\frac{9}{2}} (1+t)^{-\frac{1}{2}}\right),
	\end{equation}
	where $\epsilon_0$ is defined in \eqref{inida} and $C_U$ is a positive constant independent of $t$, $|k|$ and $\epsilon$.
\end{Lemma}
\begin{proof}
Following the approach developed in \cite{z2} for the one-dimensional model, one derives the weighted energy estimate by multiplying \eqref{hl221} by the weight function $e^{-\epsilon (t-\tau)} (\epsilon^2 |\vec{x}|^2 + 1)^{-9/2}$, integrating the product over the spatial domain $C_{\epsilon t}$ (with respect to $\vec{x}$) and the interval $[t/2, t]$ (with respect to $\tau$), and invoking the global existence result \eqref{genenes}, together with the decay property \eqref{pu1xde} and the embedding inequality \eqref{sobemb}:
\begin{equation} \label{puxwei}
	\begin{aligned}
		& \frac{1}{2} \int_{C_{\epsilon t}} (\epsilon^2 |\vec{x}|^2 + 1)^{-\frac{9}{2}}  \left( \left|\partial_{\vec{x}}^k p\right|^2 + \gamma \rho p \left|\partial_{\vec{x}}^k \vec{u}\right|^2 + b \left| \partial_{\vec{x}}^k p - \frac{c_f^2}{a} \rho \partial_{\vec{x}}^k \chi \right|^2 \right) \! (\vec{x},t) \, d\vec{x} \\
		& \quad + \alpha \frac{c_f^2}{a} \int_{t/2}^t \int_{C_{\epsilon t}} e^{-\epsilon (t-\tau)} (\epsilon^2 |\vec{x}|^2 + 1)^{-\frac{9}{2}} \left|\rho \partial_{\vec{x}}^k \chi \right|^2 \! (\vec{x},\tau) \, d\vec{x} d\tau \\
		& \le {{C\epsilon_0 \int_{t/2}^t \int_{C_{\epsilon t}} e^{-\epsilon (t - \tau)} (\epsilon^2 |\vec{x}|^2 + 1)^{-\frac{9}{2}} \left\lbrace \sum_{|j|=1}^{|k|-1} \left( \bigl|\partial_{\vec{x}}^j p \bigr|^2 + \bigl| \partial_{\vec{x}}^j \vec{u} \bigr|^2 \right) + \sum_{|j|=0}^{|k|-1} \bigl|\partial_{\vec{x}}^j \chi \bigr|^2 \right\rbrace (\vec{x}, \tau) \, d\vec{x} d\tau}}\\
		& \quad + C(\epsilon + \epsilon_0) \int_{t/2}^t \int_{C_{\epsilon t}} e^{-\epsilon (t-\tau)} (\epsilon^2 |\vec{x}|^2 + 1)^{-\frac{9}{2}} \Bigl(
		\left| \nabla^{|k|} p\right|^2
		+ \left| \nabla^{|k|} \vec{u}\right|^2+ \left| \nabla^{|k|} \chi \right|^2 \Bigr)(\vec{x},\tau) \, d\vec{x} d\tau  \\
		& \quad 
		+ C\epsilon_0^2 e^{-\epsilon t/2}
		+ C \epsilon \epsilon_0^2(\epsilon^4 t^2 + 1)^{-\frac{9}{2}}(1+t)^{-\frac{1}{2}}.
	\end{aligned}
\end{equation}
Here $|k| \ge 2$, and the decay rate $(\epsilon^4 t^2 + 1)^{-9/2}(1+t)^{-\frac{1}{2}}$ follows from the function and its associated integral inequality given below:
$$\vec{X}_p(\vec{x},t):=(X_{p,1},X_{p,2},X_{p,3})(\vec{x},t)=  \vec{u} \left(  \left|\partial_{\vec{x}}^k p\right|^2 + \gamma \rho p \left|\partial_{\vec{x}}^k \vec{u}\right|^2 + b \left| \partial_{\vec{x}}^k p - \frac{c_f^2}{a} \rho \partial_{\vec{x}}^k \chi \right|^2 \right) +2\gamma p \partial_{\vec{x}}^k \vec{u} \, \partial_{\vec{x}}^k p$$
and
\begin{equation} \nonumber
	\begin{aligned}
		&\int_{t/2}^t \int_{C_{\epsilon t}} e^{-\epsilon (t - \tau)}(\epsilon^2 |\vec{x}|^2 + 1)^{-\frac{9}{2}} \nabla \cdot \vec{X}_p(\vec{x},\tau) \, d\vec{x} d\tau\\
		 & =-\int_{t/2}^t \int_{C_{\epsilon t}} e^{-\epsilon (t - \tau)} \nabla (\epsilon^2 |\vec{x}|^2 + 1)^{-\frac{9}{2}} \cdot \vec{X}_p(\vec{x},\tau) \, d\vec{x} d\tau\\
		 &\quad +\sum_{i=1}^3 \int_{t/2}^t \int_{-\epsilon t}^{\epsilon t} \int_{-\epsilon t}^{\epsilon t} e^{-\epsilon (t - \tau)} (\epsilon^2 |\vec{x}|^2 + 1)^{-\frac{9}{2}} X_{p,i}(\vec{x},\tau) \Big|_{x_i = -\epsilon t}^{\epsilon t} \,  dx_{j_1} dx_{j_2} d\tau\\
		& \le C\epsilon \int_{t/2}^t \int_{C_{\epsilon t}} e^{-\epsilon (t - \tau)} (\epsilon^2 |\vec{x}|^2 + 1)^{-5}\left( \left|\partial_{\vec{x}}^k p\right|^2 + \left|\partial_{\vec{x}}^k \vec{u}\right|^2 + \left| \partial_{\vec{x}}^k \chi \right|^2\right)\!(\vec{x}, \tau) \, d\vec{x} d\tau\\
		&\quad + C(\epsilon^4 t^2 + 1)^{-\frac{9}{2}} \int_{-\epsilon t}^{\epsilon t} \int_{-\epsilon t}^{\epsilon t} dx_{j_1} dx_{j_2} \int_{t/2}^t  e^{-\epsilon (t - \tau)} \left(\left\|\partial_{\vec{x}}^k p(\cdot , \tau) \right\|_{L^{\infty}}^2 + \left\|\partial_{\vec{x}}^k \vec{u}(\cdot , \tau)\right\|_{L^{\infty}}^2 + \left\|\partial_{\vec{x}}^k \chi(\cdot , \tau) \right\|_{L^{\infty}}^2 \right)d\tau\\
		& \le C\epsilon \int_{t/2}^t \int_{C_{\epsilon t}} e^{-\epsilon (t - \tau)} (\epsilon^2 |\vec{x}|^2 + 1)^{-5}\left( \left|\partial_{\vec{x}}^k p\right|^2 + \left|\partial_{\vec{x}}^k \vec{u}\right|^2 + \left| \partial_{\vec{x}}^k \chi \right|^2\right)\!(\vec{x}, \tau) \, d\vec{x} d\tau \\
		& \quad + C \epsilon \epsilon_0^2  (\epsilon^4 t^2 + 1)^{-\frac{9}{2}} (1+t)^{-\frac{1}{2}},
	\end{aligned}
\end{equation}
with $j_1:=\max\{i-1,3-i\}$ and $j_2:=6-i-j_1=\min\{7-2i,3\}$. The last inequality relies on the Sobolev embedding \eqref{wemb} together with the decay property in Theorem \ref{genes} (see \eqref{pu1xde}), namely,
$\left\|\partial_{\vec{x}}^k (p,\vec{u},\chi)(\cdot , \tau)\right\|_{L^{\infty}} \le C\left\|\nabla \partial_{\vec{x}}^k (p,\vec{u},\chi)(\cdot , \tau)\right\|_{H^1} \le C\epsilon_0(1+\tau)^{-5/4}.$

One could also treat \eqref{hl223} and \eqref{hl224} similarly to obtain
\begin{equation} \label{uwei}
	\begin{aligned}
		& \frac{\alpha_f}{4\alpha} \bar{p} \bar{v} \int_{t/2}^t \int_{C_{\epsilon t}} e^{-\epsilon (t-\tau)} (\epsilon^2 |\vec{x}|^2 + 1)^{-\frac{9}{2}} \left|\partial_{\vec{x}}^{k-1} \nabla \vec{u} \right|^2 \! (\vec{x},\tau) \, d\vec{x} d\tau \\
		& \le C \int_{C_{\epsilon t}} (\epsilon^2 |\vec{x}|^2 + 1)^{-\frac{9}{2}} \left( \left|\partial_{\vec{x}}^{k-1} \chi \right|^2 + \left|\partial_{\vec{x}}^{k-1} \nabla \cdot \vec{u}\right|^2 \right) \! (\vec{x},t) \, d\vec{x}\\                            
		& \quad + C (\epsilon + \epsilon_0) \int_{t/2}^t \int_{C_{\epsilon t}} e^{-\epsilon (t-\tau)} (\epsilon^2 |\vec{x}|^2 + 1)^{-\frac{9}{2}} \sum_{|j|=1}^{|k|-1} \left( \bigl|\partial_{\vec{x}}^j p\bigr|^2+ \bigl|\partial_{\vec{x}}^j \vec{u} \bigr|^2 \right) \! (\vec{x},\tau) \, d\vec{x} d\tau\\
		& \quad + C \int_{t/2}^t \int_{C_{\epsilon t}} e^{-\epsilon (t-\tau)} (\epsilon^2 |\vec{x}|^2 + 1)^{-\frac{9}{2}} \sum_{|j|=0}^{|k|-1} \bigl|\partial_{\vec{x}}^j \chi \bigr|^2 (\vec{x},\tau) \,  d\vec{x} d\tau\\
		& \quad + C (\epsilon + \epsilon_0) \int_{t/2}^t \int_{C_{\epsilon t}} e^{-\epsilon (t-\tau)} (\epsilon^2 |\vec{x}|^2 + 1)^{-\frac{9}{2}} \left( \left|\partial_{\vec{x}}^{k-1} \nabla p\right|^2+\left|\partial_{\vec{x}}^{k-1} \nabla \chi \right|^2 \right) \! (\vec{x},\tau) \, d\vec{x} d\tau\\
		& \quad + C \int_{t/2}^t \int_{C_{\epsilon t}} e^{-\epsilon (t-\tau)} (\epsilon^2 |\vec{x}|^2 + 1)^{-\frac{9}{2}} \left|\partial_{\vec{x}}^{k-1} \nabla p\right| \cdot \left|\partial_{\vec{x}}^{k-1} \nabla \chi \right| \!(\vec{x},\tau) \,  d\vec{x} d\tau\\
		& \quad + C \epsilon_0^2  e^{-\epsilon t/2} + C\epsilon \epsilon_0^2  (\epsilon^4 t^2 + 1)^{-\frac{9}{2}} (1+t)^{-\frac{1}{2}},
	\end{aligned}
\end{equation}
and
\begin{equation} \label{pwei}
	\begin{aligned}
		& \frac{\bar{v}}{4} \int_{t/2}^t \int_{C_{\epsilon t}} e^{-\epsilon (t-\tau)} (\epsilon^2 |\vec{x}|^2 + 1)^{-\frac{9}{2}} \left| \partial_{\vec{x}}^{k-1} \nabla p\right|^2 \! (\vec{x},\tau) \, d\vec{x} d\tau \\
		& \le C \int_{C_{\epsilon t}} (\epsilon^2 |\vec{x}|^2 + 1)^{-\frac{9}{2}} \left( \left|\partial_{\vec{x}}^{k-1} \vec{u}\right|^2 + \left|\partial_{\vec{x}}^{k-1} \nabla p\right|^2 \right) \! (\vec{x},t)  \, d\vec{x}\\
		& \quad + C (\epsilon + \epsilon_0) \int_{t/2}^t \int_{C_{\epsilon t}} e^{-\epsilon (t-\tau)} (\epsilon^2 |\vec{x}|^2 + 1)^{-\frac{9}{2}} \sum_{|j|=1}^{|k|-1} \left( \bigl|\partial_{\vec{x}}^j p\bigr|^2+\bigl|\partial_{\vec{x}}^j \vec{u}\bigr|^2 \right) \! (\vec{x},\tau) \, d\vec{x} d\tau\\
		& \quad + C \int_{t/2}^t \int_{C_{\epsilon t}} e^{-\epsilon (t-\tau)} (\epsilon^2 |\vec{x}|^2 + 1)^{-\frac{9}{2}} \sum_{|j|=0}^{|k|-1} \bigl|\partial_{\vec{x}}^j \chi\bigr|^2 (\vec{x},\tau) \, d\vec{x} d\tau\\
		& \quad + C \int_{t/2}^t \int_{C_{\epsilon t}} e^{-\epsilon (t-\tau)} (\epsilon^2 |\vec{x}|^2 + 1)^{-\frac{9}{2}} \left|\partial_{\vec{x}}^{k-1} \nabla \vec{u}\right|^2 \! (\vec{x},\tau) \,  d\vec{x} d\tau\\
		&\quad + C \epsilon_0^2  e^{-\epsilon t/2} + C\epsilon \epsilon_0^2  (\epsilon^4 t^2 + 1)^{-\frac{9}{2}} (1+t)^{-\frac{1}{2}}.
	\end{aligned}
\end{equation}
Here one invokes the lower bounds $pv \ge \bar{p}\bar{v}/2$ and $v \ge \bar{v}/2$ gained via $pv \ge \bar{p}\bar{v} - C\epsilon_0$ and $v \ge \bar{v} - C\epsilon_0$. 

The estimates \eqref{uwei} and \eqref{pwei} result in that
\begin{equation} \nonumber 
	\begin{aligned}
		& \frac{\bar{v}}{8} \int_{t/2}^t \int_{C_{\epsilon t}} e^{-\epsilon (t-\tau)} (\epsilon^2 |\vec{x}|^2 + 1)^{-\frac{9}{2}} \left( \left| \partial_{\vec{x}}^{k-1} \nabla p\right|^2 + \frac{2 \alpha_f}{\alpha} \bar{p} \left| \partial_{\vec{x}}^{k-1} \nabla \vec{u}\right|^2 \right) \! (\vec{x},\tau) \, d\vec{x} d\tau \\
		& \le C \int_{C_{\epsilon t}} (\epsilon^2 |\vec{x}|^2 + 1)^{-\frac{9}{2}} \left( \left|\partial_{\vec{x}}^{k-1} \nabla p\right|^2 + \left|\partial_{\vec{x}}^{k-1} \nabla \cdot \vec{u}\right|^2 \right) \! (\vec{x},t) \, d\vec{x}\\
		& \quad + C \int_{t/2}^t \int_{C_{\epsilon t}} e^{-\epsilon (t-\tau)} (\epsilon^2 |\vec{x}|^2 + 1)^{-\frac{9}{2}} \left|\partial_{\vec{x}}^{k-1} \nabla \chi\right|^2 \! (\vec{x},\tau) \, d\vec{x} d\tau\\
		& \quad + C \left( F_{|k|-1}(t, \epsilon) + \epsilon_0^2 e^{-\epsilon t/2}+\epsilon \epsilon_0^2  (\epsilon^4 t^2 + 1)^{-\frac{9}{2}} (1+t)^{-\frac{1}{2}} \right),
	\end{aligned}
\end{equation}
which finally together with \eqref{puxwei} yields that
\begin{equation} \nonumber 
	\begin{aligned}
		& \frac{1}{4} \sum_{|k|=m} \int_{C_{\epsilon t}} (\epsilon^2 |\vec{x}|^2 + 1)^{-\frac{9}{2}}  \left( \frac{1}{2} \bigl|\partial_{\vec{x}}^k p\bigr|^2 + \gamma \rho p  \bigl|\partial_{\vec{x}}^k \vec{u}\bigr|^2 + \min\left\{ b, \frac{1}{2} \right\} \left| \frac{c_f^2}{a} \rho \partial_{\vec{x}}^k \chi \right|^2 \right) \! (\vec{x},t) \, d\vec{x} \\
		& \quad + \frac{\alpha c_f^2}{2a} \sum_{|k|=m} \int_{t/2}^t \int_{C_{\epsilon t}} e^{-\epsilon (t-\tau)} (\epsilon^2 |\vec{x}|^2 + 1)^{-\frac{9}{2}} \bigl|\rho \partial_{\vec{x}}^k \chi\bigr|^2(\vec{x},\tau) \, d\vec{x} d\tau \\
		& \quad + \sigma \sum_{|k|=m} \int_{t/2}^t \int_{C_{\epsilon t}} e^{-\epsilon (t-\tau)} (\epsilon^2 |\vec{x}|^2 + 1)^{-\frac{9}{2}} \left( \bigl|\partial_{\vec{x}}^k p\bigr|^2 + \bigl|\partial_{\vec{x}}^k \vec{u}\bigr|^2 \right) \! (\vec{x},\tau) \, d\vec{x} d\tau\\
		& \le C \left( F_{m-1}(t,\epsilon) + \epsilon_0^2  e^{-\epsilon t/2} + \epsilon \epsilon_0^2  (\epsilon^4 t^2 + 1)^{-\frac{9}{2}} (1+t)^{-\frac{1}{2}}\right),
	\end{aligned}
\end{equation}
with a positive constant $\sigma$ satisfying $\epsilon, \epsilon_0 \ll \sigma\ll 1$.
This immediately verifies \eqref{recurfk} and completes the proof.
\end{proof}

\begin{Remark}
	From \eqref{anas} and \eqref{recurfk}, it can be deduced that for $t \ge 0$,
	\begin{equation}\label{efm}
		\begin{aligned}
			F_{m}(t, \epsilon) & \le C 
			\begin{cases}
				\epsilon^{-4} \mathcal{N}^2(t) (1 + t)^{-6} & \text{for } m = 1,\\
				F_1(t, \epsilon)+\epsilon^{-17}\epsilon_0^2 (1+t)^{-\frac{19}{2}} & \text{for } m \ge 2 
			\end{cases}\\
			& \le C_{\epsilon^{-1}} \bigl(\mathcal{N}(t)+\epsilon_0\bigr)^2 (1 + t)^{-6},
		\end{aligned}
	\end{equation}
	where $C_{\epsilon^{-1}}$ denotes a positive constant depending on $\epsilon^{-1}$ but whose value may vary from line to line, and one employs the pointwise upper bounds of $\chi$ and $\nabla (p-\bar{p}, \vec{u}, \chi)$ from the ansatz \eqref{anas} to obtain the spacetime structure on $\vec{x} \in C_{\epsilon t}$, which further yields the estimate for $F_1(t, \epsilon)$.
	\end{Remark}
	
	Next, one will show how to develop unaccounted temporal decays from the estimates \eqref{efm} of $F_{m}(t, \epsilon)$.  Define a weight function $\phi (\vec{x}, t)$ as follows:
\begin{equation} \label{phi}
	\phi (\vec{x}, t):= 
	\begin{cases}
		1 & \text{for } |\vec{x}| \le \frac{\epsilon t}{2}, \\
		1 - K \displaystyle\int_{-1}^{\frac{4 |\vec{x}| - 3 \epsilon t}{\epsilon t}} e^{- \frac{1}{1 - |\vec{y}|^2}} d\vec{y} & \text{for } \frac{\epsilon t}{2} \le |\vec{x}| \le \epsilon t, \\
		0 & \text{for } |\vec{x}| \ge \epsilon t,
	\end{cases}
\end{equation}
where $K>0$ is the constant satisfying $\displaystyle K \int_{-1}^1 e^{- \frac{1}{1 - |\vec{y}|^2}} d\vec{y} = 1.$
By the Gagliardo-Nirenberg inequality, one could deduce that  
\begin{equation} \label{puxrecur}
	\begin{aligned}
		(|\vec{x}|^2 + 1)^{-\frac{9}{4}} \bigl[ \left( \left| \partial_{\vec{x}}^k p\right| + \left|\partial_{\vec{x}}^k \vec{u}\right| + \left|\partial_{\vec{x}}^k \chi\right| \right) \! \phi \bigr] (\vec{x}, t) 
		&\le C_{\epsilon^{-1}} F_{|k|}^{\frac{1}{8}} \left( F_{|k|}^{\frac{1}{2}} + F_{|k|+1}^{\frac{1}{2}} + F_{|k|+2}^{\frac{1}{2}} \right)^{\frac{3}{4}} \! (t, \epsilon)\\
		&\le C_{\epsilon^{-1}}\bigl(\mathcal{N}(t)+\epsilon_0\bigr)(1 + t)^{-3},
	\end{aligned} 
\end{equation}
for $|k| \ge 1$.
Herein, a fact that the weight function $\phi$ restricts the region of integration for the spatial variables to the spherical region $B_{\epsilon t}:= \{ \vec{x} \in \mathbb{R}^3 \mid |\vec{x}| \le \epsilon t \} \subset C_{\epsilon t}$ is utilized. Moreover, $\epsilon \in (0,1)$ ensures that $(|\vec{x}|^2 + 1)^{-9/4} \le (\epsilon^2 |\vec{x}|^2 + 1)^{-9/4}$, such that the left-hand side of \eqref{puxrecur} is bounded by $(2F_{|k|})^{1/2}(t,\epsilon)$ after taking the $L^2$-norm.

For $|\vec{x}| \le \epsilon t/2$, $\phi(\vec{x},t) = 1$, which together with \eqref{puxrecur} immediately implies
\begin{equation} \label{slp1}
	| s - \bar{s}| \left( \left| \partial_{\vec{x}}^k p\right| + \left|  \partial_{\vec{x}}^k \vec{u}\right| + \left| \partial_{\vec{x}}^k \chi \right|  \right) \! (\vec{x}, t) \le C_{\epsilon^{-1}} \mathcal{N}(t) \bigl(\mathcal{N}(t)+\epsilon_0\bigr) (1 + t)^{-3}.
\end{equation}
For $|\vec{x}| > \epsilon t/2$, the spatial decay rates of $s-\bar{s}$ could be converted into the time decay rates:
\begin{equation} \label{slp2}
		| s - \bar{s}| \left( \left| \partial_{\vec{x}}^k p\right| + \left| \partial_{\vec{x}}^k \vec{u} \right| + \left| \partial_{\vec{x}}^k \chi \right| \right) \! (\vec{x}, t) \le  C_{\epsilon^{-1}} \epsilon_0 \mathcal{N}(t)(1 + t)^{-\frac{9}{2}}.
\end{equation}

Note that the result in \eqref{slp1} can be extended to the case $|k|=1$. It is also worth noting that replacing $s-\bar{s}$ with $\nabla^i s$ for $i=1,2$ in \eqref{slp1} preserves the validity of the estimate; in contrast, it fails for $\nabla^3 s$. The reason is that the estimates for $|\nabla^3 s|^2$ derived from \eqref{anas} lack $L^1$-integrability, a key property indispensable to the derivation process.

\begin{Remark} \label{transform}
In the derivation of \eqref{uwei}, irrotationality is utilized to convert $\partial_{\vec{x}}^{k-1} \nabla \cdot \vec{u}$ to $\partial_{\vec{x}}^{k-1} \nabla \vec{u}$ on the left-hand side of the inequality. However, the whole process is a bit different from the one in energy estimate in Section 2, due to the boundaries $x_i=\pm\epsilon t$ of the region $C_{\epsilon t}$ of integration. For convenience of the readers, the details are listed as follows\textup{:} One could derive $\sum_{i<j} ( \partial_{\vec{x}}^{k-1} \partial_{x_i} u_j - \partial_{\vec{x}}^{k-1} \partial_{x_j} u_i)^2 = 0$ from the irrotational condition $\nabla \times \vec{u} \equiv 0$, which ensures that
	\begin{equation} \label{nablau1}
		\begin{aligned}
			& \int_{t/2}^t \int_{C_{\epsilon t}} e^{-\epsilon (t-\tau)} (\epsilon^2 |\vec{x}|^2 + 1)^{-\frac{9}{2}} \left| \partial_{\vec{x}}^{k-1} \nabla \cdot \vec{u}\right|^2 \! (\vec{x},\tau) \, d\vec{x} d\tau \\
			& = \int_{t/2}^t \int_{C_{\epsilon t}} e^{-\epsilon (t-\tau)} (\epsilon^2 |\vec{x}|^2 + 1)^{-\frac{9}{2}} \Bigl[ \left| \partial_{\vec{x}}^{k-1} \nabla \vec{u}\right|^2\\
			& \quad \quad \quad \quad \quad + 2 \sum_{i<j} \Bigl( \partial_{x_i} \! \left( \partial_{\vec{x}}^{k-1} u_i \,  \partial_{\vec{x}}^{k-1} \partial_{x_j} u_j \right) - \partial_{x_j} \! \left( \partial_{\vec{x}}^{k-1} u_i \, \partial_{\vec{x}}^{k-1} \partial_{x_i} u_j \right) \Bigr) \Bigr] (\vec{x},\tau) \, d\vec{x} d\tau.
		\end{aligned}
	\end{equation}
It also holds that
	\begin{equation} \nonumber 
		\begin{aligned}
			& \int_{t/2}^t \int_{C_{\epsilon t}} e^{-\epsilon (t-\tau)} (\epsilon^2 |\vec{x}|^2 + 1)^{-\frac{9}{2}} \partial_{x_1} \! \left( \partial_{\vec{x}}^{k-1} u_1 \, \partial_{\vec{x}}^{k-1} \partial_{x_j} u_j \right) \! (\vec{x},\tau) \,  d\vec{x} d\tau\\
			&= - \int_{t/2}^t \int_{C_{\epsilon t}} e^{-\epsilon (t-\tau)} \partial_{x_1} (\epsilon^2 |\vec{x}|^2 + 1)^{-\frac{9}{2}} \left( \partial_{\vec{x}}^{k-1} u_1 \,  \partial_{\vec{x}}^{k-1} \partial_{x_j} u_j \right) \! (\vec{x},\tau) \, d\vec{x} d\tau \\
			& \quad +\int_{t/2}^t \int_{-\epsilon t}^{\epsilon t} \int_{-\epsilon t}^{\epsilon t} e^{-\epsilon (t-\tau)} (\epsilon^2 |\vec{x}|^2 + 1)^{-\frac{9}{2}} \left( \partial_{\vec{x}}^{k-1} u_1 \, \partial_{\vec{x}}^{k-1} \partial_{x_j} u_j \right) \! (\vec{x},\tau) \Big|_{x_1 = -\epsilon t}^{\epsilon t} \,  dx_2 dx_3 d\tau \\
			& \le C\epsilon \int_{t/2}^t \int_{C_{\epsilon t}} e^{-\epsilon (t-\tau)} (\epsilon^2 |\vec{x}|^2 + 1)^{-5} \left( \left| \partial_{\vec{x}}^{k-1} u_1\right|^2 + \left| \partial_{\vec{x}}^{k-1} \partial_{x_j} u_j\right|^2 \right) \! (\vec{x},\tau) \, d\vec{x} d\tau\\
			& \quad + C\epsilon \epsilon_0^2  (\epsilon^4 t^2 + 1)^{-\frac{9}{2}} (1+t)^{-\frac{1}{2}},
		\end{aligned}
	\end{equation}
	which together with \eqref{nablau1} verifies the conversion from $\partial_{\vec{x}}^{k-1} \nabla \cdot \vec{u}$ to $\partial_{\vec{x}}^{k-1} \nabla \vec{u}$ in \eqref{uwei}.
\end{Remark}

\subsection{Pointwise Estimates for the Entropy Function $(s-\bar{s})(\vec{x},t)$ and Its Derivatives}
From the entropy equation in \eqref{vnepre}, one could obtain the equation for $s^*:=s-\bar{s}$:
$$\partial_t s^* = -\vec{u}\cdot\nabla s^* + \frac{2}{\alpha_f R}\cdot\frac{\chi^2}{T_1 T_2}$$
and thus
$$s^*(\vec{x},t) = s^*(\vec{x},0)+ \int_0^t \left( -\vec{u}\cdot\nabla s^* + \frac{2}{\alpha_f R}\cdot\frac{\chi^2}{T_1 T_2} \right) \! (\vec{x},\tau) \, d\tau.$$

From Lemma~\ref{lemmains} and the definition \eqref{anas}, one obtains the following estimates:
\begin{equation} \label{uxxint}
			\begin{aligned}
				&\int_0^t |\vec{u}|(\vec{x}, \tau) \, d\tau \le C \mathcal{N}(t)(|\vec{x}|^2 + 1)^{-\frac{3}{4}}, \quad \int_0^t \left|\nabla \vec{u}\right| \! (\vec{x}, \tau) \, d\tau \le C\mathcal{N}(t)(|\vec{x}|^2 + 1)^{-1},\\
				&\int_0^t \chi^2(\vec{x}, \tau) \, d\tau \le C \mathcal{N}^2(t)(|\vec{x}|^2 + 1)^{-\frac{9}{4}}, \quad \int_0^t \left|\chi \nabla \chi \right| \! (\vec{x}, \tau) \, d\tau \le C \mathcal{N}^2(t) (|\vec{x}|^2 + 1)^{-\frac{5}{2}},\\
				&\int_0^t \left|\nabla \chi\right|^2 \! (\vec{x}, \tau) \, d\tau \le C \mathcal{N}^2(t) (|\vec{x}|^2 + 1)^{-\frac{11}{4}}.
			\end{aligned}
\end{equation}
Furthermore, Lemma~\ref{decay} together with \eqref{anas} results in that
\begin{equation} \label{xxint}
	\int_0^t \Big(\left|\chi \nabla^2 \chi\right| \!  + \left|\chi \nabla^3 \chi\right| \! + \left|\nabla \chi \right|\cdot \left|\nabla^2 \chi \right| \!  \Big)(\vec{x}, \tau)\, d\tau \le C \mathcal{N}(t)\bigl(\mathcal{N}(t)+\epsilon_0\bigr) (|\vec{x}|^2 + 1)^{-2}.
\end{equation}

Combining the definition \eqref{anas} and the above representation of $s^*$, and incorporating the initial condition $|s^*(\vec{x}, 0)| \le C_{in} \epsilon_0 (|\vec{x}|^2 + 1)^{-5/2}$ together with  Lemma~\ref{lemmains}, one has that
\begin{equation} \label{s0}
\begin{aligned}
	\left|s^* (\vec{x}, t)\right| & \le C_{in} \epsilon_0 (|\vec{x}|^2 + 1)^{-\frac{5}{2}} \\
	&\quad + C {{\mathcal{N}^2(t)}} \int_0^t \left(\left[ \left( (|\vec{x}| - \bar{c} \tau)^2 + \tau+1 \right)^{-2} +  \frac{H(\bar{c}\tau - |\vec{x}|)}{1 + (|\vec{x}| + \sqrt{\tau})(\bar{c}\tau - |\vec{x}| + \sqrt{\tau})^2 } \right](|\vec{x}|^2 + 1)^{-\frac{9}{4}}\right.\\
	&\quad\quad\quad\quad\quad\quad\quad\quad\left. +(1 + \tau)^{-1} \left[\left( (|\vec{x}| - \bar{c} \tau)^2 + \tau+1 \right)^{-4} +  \frac{H(\bar{c}\tau - |\vec{x}|)}{1 + (|\vec{x}| + \sqrt{\tau})^2(\bar{c}\tau - |\vec{x}| + \sqrt{\tau})^4 }\right] \right)d\tau\\
	& \le \left(C_{in} \epsilon_0+C  {{\mathcal{N}^2(t)}}\right) {{(|\vec{x}|^2 + 1)^{-\frac{9}{4}}}}. 
\end{aligned}
\end{equation}

By virtue of \eqref{uxxint}, the pointwise estimate for $\nabla s$ could be derived as follows:
\begin{equation} \label{s1}
	\begin{aligned}
		\left|\nabla s (\vec{x}, t)\right|&=\left|\nabla s(\vec{x},0)+\int_0^t \nabla \left( -\vec{u}\cdot\nabla s + \frac{2}{\alpha_f R} \cdot \frac{\chi^2}{T_1 T_2} \right) (\vec x,\tau)d\tau\right| \le C_{in}  \epsilon_0 (|\vec{x}|^2 + 1)^{-\frac{5}{2}}\\
		& \quad + \left|\int_0^t  \left( - \nabla s\cdot\nabla \vec{u} - \vec{u}\cdot\nabla^2 s + \frac{2}{\alpha_f R } \nabla \left( \frac{1}{T_1 T_2} \right) \chi^2 + \frac{4}{\alpha_f R} \cdot \frac{\chi \nabla \chi}{T_1 T_2} \right) (\vec x,\tau)d\tau\right|\\
		& \le C_{in}  \epsilon_0 (|\vec{x}|^2 + 1)^{-\frac{5}{2}} + C \mathcal{N}(t)(|\vec{x}|^2 + 1)^{-1} \cdot \left\| \nabla s (\vec{x} , \cdot) \right\|_{L^{\infty}_t(\mathbb{R}_+)}\\
		&\quad + C\mathcal{N}(t)(|\vec{x}|^2 + 1)^{-\frac{3}{4}} \cdot \left\| \nabla^2 s (\vec{x} , \cdot) \right\| _{L^{\infty}_t (\mathbb{R}_+)} +  C \mathcal{N}^2(t)(|\vec{x}|^2 + 1)^{-\frac{9}{4}} \\
		& \le C_{in} \epsilon_0 (|\vec{x}|^2 + 1)^{-\frac{5}{2}} + C \mathcal{N}^2(t)(|\vec{x}|^2 + 1)^{-\frac{9}{4}}\\
		&\le  \left(C_{in} \epsilon_0+C  {{\mathcal{N}^2(t)}}\right)(|\vec{x}|^2 + 1)^{-\frac{9}{4}}.
	\end{aligned}
\end{equation}
Here, one substitutes the initial condition $|\nabla s (\vec{x}, 0)| \le C_{in}  \epsilon_0 (|\vec{x}|^2 + 1)^{-5/2}$ from \eqref{inida} and the definition \eqref{anas} into the representation of $\nabla s$.

The high order derivatives $\nabla^2 s$ and $\nabla^3 s$ could be estimated similarly by using \eqref{uxxint} and \eqref{xxint}:
\begin{equation} \label{s23}
	\begin{aligned}
		&\left| \nabla^2 s\right| \! (\vec{x}, t) \le {{\left[C_{in} \epsilon_0+C  \mathcal{N}(t)\bigl(\mathcal{N}(t)+\epsilon_0\bigr) \right]}}(|\vec{x}|^2 + 1)^{-\frac{3}{2}}, \\
		&\left|\nabla^3 s\right| \! (\vec{x}, t) \le \left[C_{in} \epsilon_0+C \mathcal{N}(t)\bigl(\mathcal{N}(t)+\epsilon_0\bigr) \right] (|\vec{x}|^2 + 1)^{-\frac{3}{4}},
	\end{aligned}
\end{equation}
where the dominant terms to determine the final estimates are $\displaystyle\int_0^t -\vec{u}\cdot\nabla^3 s(\vec x,\tau)d\tau$ and $\displaystyle\int_0^t -\vec{u}\cdot\nabla^4 s(\vec x,\tau)d\tau$ and one derives $\nabla^4 s(\vec{x}, t) \le C \epsilon_0$ from the general energy estimates in Theorem~\ref{genes} and the embedding theorem $\| \nabla^4 s (\cdot , t) \| _{L^{\infty}} \le C \| \nabla^5 s (\cdot , t)\| _{H^1}$.

\subsection{Pointwise Estimates for $(p - \bar{p}, \vec{u}, \chi)(\vec{x},t)$}
Duhamel's principle together with the Green's function defined by \eqref{greeng} gives the following representation for the solution $(p,\vec{u},\chi)$ of \eqref{vnepre}:
$$\vec{U}(\vec{x},t) = \mathbb{G}(\vec{x},t) \underset{\vec{x}}{\ast} \vec{U}(\vec{x},0) + \int_0^t \mathbb{G}(\vec{x},t-\tau) \underset{\vec{x}}{\ast} 
		\begin{pmatrix}
			R_p\\
			R_{\vec{u}}\\
			R_{\chi}
		\end{pmatrix}(\vec{x},\tau) \, d\tau,$$
and thus
\begin{equation} \label{duhpux}
	\begin{aligned}
		\partial_{\vec{x}}^k \vec{U}(\vec{x},t) &= \partial_{\vec{x}}^k \mathbb{G}(\vec{x},t) \underset{\vec{x}}{\ast} \vec{U}(\vec{x},0) + \int_0^t \partial_{\vec{x}}^k \mathbb{G}(\vec{x},t-\tau) \underset{\vec{x}}{\ast} 
		\begin{pmatrix}
			R_p\\
			R_{\vec{u}}\\
			R_{\chi}
		\end{pmatrix}(\vec{x},\tau) \, d\tau\\
		&= \partial_{\vec{x}}^k \mathbb{G}_L(\vec{x},t) \underset{\vec{x}}{\ast} \vec{U}(\vec{x},0) + \mathbb{G}_S(\vec{x},t) \underset{\vec{x}}{\ast} \partial_{\vec{x}}^k \vec{U}(\vec{x},0) + \partial_{\vec{x}}^k \vec{S}_R (\vec{x},t).
	\end{aligned}
\end{equation}
Here,  one denotes
\begin{equation} \label{ui}
	\vec{U}:=\bigl(p-\bar{p},\vec{u},\chi\bigr):=\left(\vec{U}_1,\vec{U}_2,\vec{U}_3\right)^t,\ \ \ \  \vec{S}_R:=\left( \vec{S}_{R,1}, \vec{S}_{R,2}, \vec{S}_{R,3}\right)^t:=\int_0^t \partial_{\vec{x}}^k \mathbb{G}(\vec{x},t-\tau) \underset{\vec{x}}{\ast} 
		\begin{pmatrix}
			R_p\\
			R_{\vec{u}}\\
			R_{\chi}
		\end{pmatrix}(\vec{x},\tau) \, d\tau,
\end{equation}
and the source terms $R_p$, $R_{\vec{u}}$, and $R_{\chi}$ are defined in \eqref{vnepre}.

\textbf{Initial Propagations.}
Assume that the initial perturbation $(p_0 - \bar{p}, \vec{u}_0, \chi_0)(\vec{x})$ satisfies that 
\begin{equation}\label{inipuc}
	\bigl|\vec{U}\bigr|(\vec{x},0), \, \bigl|\nabla \vec{U}\bigr|(\vec{x},0), \, \bigl|\nabla^2 (p_0, \vec{u}_0)\bigr|(\vec{x})  \le C_{in} \epsilon_0 (|\vec{x}|^2+1)^{-\frac{5}{2}}.
\end{equation}
Lemma~\ref{lemmain} together with \eqref{inipuc} and the estimates of the Green's function $\mathbb{G}(\vec{x},t)$ in Theorem~\ref{kirch} yields that
\begin{equation} \label{inipux}
\begin{aligned}
		&\left|\mathbb{G}_i(\vec{x},t) \underset{\vec{x}}{\ast} \vec{U}(\vec{x},0)\right| \le C_{in} \epsilon_0 \left((1+t)^{-2-d_i}e^{-\frac{(|\vec{x}|-\bar{c}t)^2}{K(1+t)}}+\left(1+ |\vec{x}|^{-2} \right)e^{-(|\vec{x}|+t)/D_0}\right) \underset{\vec{x}}{\ast} (|\vec{x}|^2+1)^{-5/2}\\
		&\quad\quad+C_{in} \epsilon_0e^{-\mu_2 t} \delta (\vec{x}) \underset{\vec{x}}{\ast} (|\vec{x}|^2+1)^{-5/2} +C_{in} \epsilon_0(1+t)e^{-\mu_1 t} \omega(\vec{x},t;\bar{c}_f) \underset{\vec{x}}{\ast} (|\vec{x}|^2+1)^{-5/2}\\
		&\quad\quad+C_{in} \epsilon_0e^{-\mu_1 t} \left[ \omega(\vec{x},t;\bar{c}_f) \underset{\vec{x}}{\ast} 
		\left|\nabla (p_0, \vec{u}_0)(\vec{x})\right|+\left| \omega_t(\vec{x},t;\bar{c}_f) \underset{\vec{x}}{\ast} 
		p_0(\vec{x})\right| + \left| \omega_t(\vec{x},t;\bar{c}_f) \underset{\vec{x}}{\ast} 
		\vec{u}_0(\vec{x})\right| \, \right]\\
		&\le C_{in} \epsilon_0 (1+t)^{-d_i} \left( (|\vec{x} |-\bar{c} t)^2 + t+1\right)^{-2} \le C_{in} \epsilon_0 (1+t)^{-d_i} e_1(\vec{x},t), \ \ i=1,2,3.
\end{aligned}
\end{equation}
Here, $\mathbb{G}_i$ denotes the $i$-th row of the Green’s function $\mathbb{G}$ and the moving heat kernel in $\mathbb{G}_3$ has an extra $(1+t)^{-1/2}$ compared with $\mathbb{G}_1$ and $\mathbb{G}_2$ according to Theorem~\ref{kirch}. Thus, it yields a discrepancy and the constants $d_i$ are as follows
\begin{equation} \label{di}
	d_1=d_2=0, \quad d_3=\frac{1}{2}.
\end{equation}

The estimates for first-order derivatives $\nabla \mathbb{G}_i(\vec{x},t) \underset{\vec{x}}{\ast} \vec{U}(\vec{x},0)$ are similar to \eqref{inipux} and it holds that
\begin{equation} \label{inipux1}
	\begin{aligned}
		\left| \nabla \mathbb{G}_i(\vec{x},t) \underset{\vec{x}}{\ast} \vec{U}(\vec{x},0)\right|  & \le C_{in} \epsilon_0 (1+t)^{-d_i} \left( (|\vec{x} | - \bar{c} t)^2 + t+1\right)^{-\frac{5}{2}} \\
		& \le C_{in} \epsilon_0 
		\begin{cases}
			(1+t)^{-\frac{1}{2}} e_1(\vec{x},t)  & \text{for } i=1,2, \\
			e_2(\vec{x},t)  & \text{for } i=3.
		\end{cases}
	\end{aligned}
\end{equation}

In the estimates for $\nabla^2 \mathbb{G}_i(\vec{x},t) \underset{\vec{x}}{\ast} \vec{U}(\vec{x},0)$, if one takes a similar way to \eqref{inipux}, the pointwise structures of $\nabla^3(p_0,\vec u_0)(\vec x)$ are required and have not been posed in the assumptions of Theorem~\ref{main1}. Thus, instead, we invoke the $L^{\infty}$-boundedness assumptions of high-order derivatives, namely
$$\left\|\nabla^i \vec{U}(\cdot,0) \right\|_{L^{\infty}} \le C \left\|\nabla^{i+1} \vec{U}(\cdot,0) \right\|_{H^1}  \le C_{in} \epsilon_0 , \quad \text{for } i=2,3,$$
which combined with \eqref{glps0} in Theorem~\ref{lpg} and \eqref{iniconv1} results in that
\begin{equation} \label{inipux2}
\begin{aligned}
	&\left| \nabla^2 \mathbb{G}(\vec{x},t) \underset{\vec{x}}{\ast} \vec{U}(\vec{x},0)\right|\le \left| \nabla^2 \mathbb{G}_L(\vec{x},t) \underset{\vec{x}}{\ast} \vec{U}(\vec{x},0)\right|+\left|\nabla^2 \mathbb{G}_S(\vec{x},t) \underset{\vec{x}}{\ast} \vec{U}(\vec{x},0)\right|\\
	&\le C_{in} \epsilon_0 (1+t)^{-2-1}e^{-\frac{(|\vec{x}|-\bar{c}t)^2}{K(1+t)}}\underset{\vec{x}}{\ast} (|\vec{x}|^2+1)^{-5/2}+C_{in} \epsilon_0 e^{-t/C_0}\\
	&\le C_{in} \epsilon_0 (1+t)^{-3}.
\end{aligned}	 
\end{equation}

\textbf{Nonlinear Wave Propagations.}
 By invoking the definition \eqref{anas} of $\mathcal{N}(t)$, one deduces the following estimates for nonlinear terms $(R_p, R_{\vec{u}}, R_{\chi})$:
\begin{align}
	&\left| (R_p, R_{\vec{u}}, R_{\chi})\right| \! (\vec{x},t) \le C {{\mathcal{N}^2(t)}} e_1(\vec{x},t) \left[ (1+t)^{-\frac{1}{2}} \left( e_1(\vec{x},t) + (|\vec{x}|^2 + 1)^{-\frac{9}{4}} \right) + e_2(\vec{x},t) \right], \label{f0} \\
	& \left|\nabla (R_p, R_{\vec{u}}, R_{\chi})\right| \! (\vec{x}, t) \le C{{\mathcal{N}^2(t)}}(1 + t)^{-\frac{5}{2}} \left( e_1(\vec{x}, t) + (|\vec{x}|^2 + 1)^{- \frac{9}{4}} \right). \nonumber 
\end{align}
Here, the factor ${{(|\vec{x}|^2 + 1)^{-9/4}}}$ arises from the pointwise structure of the entropy function $(s - \bar{s})(\vec{x},t)$ and its first-order derivative $\nabla s(\vec{x}, t)$. The time decay rate of this estimate is not fast enough for closure of nonlinearity for certain parts; and it can be refined by \eqref{slp1} and \eqref{slp2} as follows:
\begin{equation} \label{f01im}
	\begin{aligned}
		\left|(R_p, R_{\vec{u}}, R_{\chi})\right| \! (\vec{x},t), \, \left|\nabla (R_p, R_{\vec{u}}, R_{\chi})\right| \! (\vec{x}, t) & \le C{{\mathcal{N}^2(t)}}(1 + t)^{-\frac{5}{2}} \left( e_1(\vec{x}, t) + H(|\vec{x}| - \epsilon \sqrt{t} )(|\vec{x}|^2 + 1)^{- \frac{9}{4}} \right) \\
		& \quad + C_{\epsilon^{-1}} {{\mathcal{N}(t)\bigl(\mathcal{N}(t)+\epsilon_0\bigr)}}  H(\epsilon \sqrt{t} - |\vec{x}|) (1 + t)^{-3}.
	\end{aligned}
\end{equation}
From the solution representation \eqref{duhpux}, to derive the pointwise estimates for the first-order derivatives, pointwise information about the second-order derivatives of $R_p(\vec{x},t)$ and $R_{\vec{u}}(\vec{x},t)$ are required:
\begin{equation} \label{f2}
	\begin{aligned}
		\left|\nabla^2 (R_p, R_{\vec{u}})\right| \! (\vec{x},t) & \le C {{\mathcal{N}(t)\bigl(\mathcal{N}(t)+\epsilon_0\bigr)}} \left[(1 + t)^{-\frac{5}{2}} + (1+t)^{-\frac{1}{2}} (|\vec{x}|^2+1)^{-\frac{3}{2}} \right] e_1(\vec{x}, t) \\
		& \quad +  { {C_{\epsilon^{-1}} \mathcal{N}(t)\bigl(\mathcal{N}(t)+\epsilon_0\bigr)}}\left[ H(|\vec{x}|-\epsilon \sqrt{t}) (1+t)^{-\frac{5}{2}} (|\vec{x}|^2+1)^{-\frac{9}{4}}+H(\epsilon \sqrt{t} - |\vec{x}|) (1 + t)^{-3}\right].
	\end{aligned}
\end{equation}

Applying Lemma~\ref{lemmainker} for $m=2d_i$ and Lemma~\ref{lemmainexp} for $i=1/2, j=0,1,2, m=2$, one obtains that
\begin{equation} \label{soupux00}
\begin{aligned}
	& \left((1+t)^{-2-d_i}e^{-\frac{(|\vec x|-\bar{c}t)^2}{K(1+t)}}+e^{-(|\vec x|+t)/D_0}+\frac{e^{-|\vec x|/D_1}}{|\vec x|^{2}}\left(e^{-\mu_1t/4}+e^{-\mu_2t}\right)\right)\\
	&\quad \quad \underset{(\vec{x},t)}{\ast} {{\mathcal{N}^2(t)}} e_1(\vec{x},t) \left[ (1+t)^{-\frac{1}{2}} \left( e_1(\vec{x},t) + (|\vec{x}|^2 + 1)^{-\frac{9}{4}} \right) + e_2(\vec{x},t) \right] \\
	&\le C_s {\mathcal{N}^2(t)}(1+t)^{-d_i} e_1(\vec{x},t).
\end{aligned}
\end{equation}
Similarly, Lemma~\ref{lemmainexpw} for $l=5/2, d=0$ results in that
\begin{equation} \label{soupux01}
\begin{aligned}
	&(1+t)e^{-\mu_1t}\omega(\vec x,t;\bar{c}_f)
\underset{(\vec{x},t)}{\ast} \left[{{\mathcal{N}^2(t)}}(1 + t)^{-\frac{5}{2}} \left( e_1(\vec{x}, t) + H(|\vec{x}| - \epsilon \sqrt{t} )(|\vec{x}|^2 + 1)^{- \frac{9}{4}} \right)\right. \\
	& \left.\quad + C_{\epsilon^{-1}} {{\mathcal{N}(t)\bigl(\mathcal{N}(t)+\epsilon_0\bigr)}} H(\epsilon \sqrt{t} - |\vec{x}|) (1 + t)^{-3}\right]\\
	&\quad +e^{-\mu_1t} \left| \omega_t(\vec x,t;\bar{c}_f)\underset{(\vec{x},t)}{\ast} R_p(\vec{x},t)\right| + \left| \omega_t(\vec x,t;\bar{c}_f)\underset{(\vec{x},t)}{\ast} R_{\vec{u}}(\vec{x},t)\right|\\
	&\le C_s {{\mathcal{N}(t)\bigl(\mathcal{N}(t)+\epsilon_0\bigr)}}(1+t)^{-d_i} e_1(\vec{x},t),
\end{aligned}
\end{equation}
and direct computations yields that
\begin{equation} \label{soupux02}
\begin{aligned}
	&e^{-\mu_2t} \underset{t}{\ast} {{\mathcal{N}^2(t)}} e_1(\vec{x},t) \left[ (1+t)^{-\frac{1}{2}} \left( e_1(\vec{x},t) + (|\vec{x}|^2 + 1)^{-\frac{9}{4}} \right) + e_2(\vec{x},t) \right]\\
	&\le C_s {{\mathcal{N}^2(t)}}(1+t)^{-d_i} e_1(\vec{x},t).
\end{aligned}
\end{equation}
Combining the pointwise estimates of $\mathbb{G}(\vec{x},t)$ from Theorem~\ref{green} with \eqref{duhpux}, \eqref{f0}, \eqref{f01im} and \eqref{soupux00}--\eqref{soupux02}, it holds that 
\begin{equation} \label{soupux}
	\bigl|\vec{S}_{R,i}\bigr|(\vec{x},t) \le C_s {{\mathcal{N}(t)\bigl(\mathcal{N}(t)+\epsilon_0\bigr)}}(1+t)^{-d_i} e_1(\vec{x},t),
\end{equation}
where $\vec{S}_{R,i}$ and $d_{i}$ are defined in \eqref{ui} and \eqref{di} respectively for $i=1,2,3$. From \eqref{duhpux}, \eqref{inipux}, and \eqref{soupux}, the pointwise estimates for $(p-\bar{p})(\vec{x},t)$, $\vec{u}(\vec{x},t)$, and $\chi(\vec{x},t)$,
which correspond to the components $\vec{U}_1$, $\vec{U}_2$, $\vec{U}_3$ of $\vec{U}$ defined in \eqref{ui}, follow directly:
\begin{equation} \label{puxpoin}
	\bigl|\vec{U}_i\bigr|(\vec{x},t) \le \left( C_{in} \epsilon_0 + C_s {{\mathcal{N}(t)\bigl(\mathcal{N}(t)+\epsilon_0\bigr)}} \right)(1+t)^{-d_i} e_1(\vec{x},t), \quad \text{for } i=1,2,3.
\end{equation}

The first order derivative of the solution could be estimated similarly. After comparing the refined estimates for $|\nabla^2 (R_p, R_{\vec{u}})|(\vec{x},t)$ given in \eqref{f2} and $|\nabla (R_p, R_{\vec{u}}, R_{\chi})|(\vec{x},t)$ as shown in \eqref{f01im}, the only additional term is $(1+t)^{-1/2} (|\vec{x}|^2+1)^{-3/2} e_1(\vec{x},t)$ which could be estimated as follows:
\begin{equation} \nonumber
\begin{aligned}
	&e^{-\mu_1t}\omega(\vec x,t;\bar{c}_f) \underset{(\vec{x},t)}{\ast} (1+t)^{-\frac{1}{2}} (|\vec{x}|^2+1)^{-\frac{3}{2}} e_1(\vec{x},t)\\
	&\le C\left[H(|\vec x|-\bar{c}t)(1+t)^{-\frac 12}(|\vec x|^2+1)^{-\frac 32}\left((|\vec x|-\bar{c}t)^2+t+1\right)^{-2}\right.\\
	&\left.\quad \quad \quad +H(\bar{c}t-|\vec x|)\left((|\vec x|^2+1)^{-\frac 32}(1+t)^{-3}+(1+t)^{-\frac{11}2}\right)\right]
\end{aligned}
\end{equation}
by Lemma~\ref{lemmainexpw} with $l=1/2, d=3/2$. Here, the second-order derivative of the nonlinear term is involved due to the convolutions of $L_{2,0}\omega_t$, $L_{3,-1}(t,\nabla)\omega$, and $\operatorname{diag} \bigl( 0, \Delta^{-1} \nabla \nabla^t, 0 \bigr) \omega_t$ in $\mathbb{G}_S$ with $\nabla (R_p, R_{\vec{u}}, R_{\chi})(\vec{x},t)$.

Using the pointwise estimates for $\mathbb{G}(\vec{x},t)$ from Theorem~\ref{green}, together with \eqref{f0}--\eqref{f2} and Lemmas~\ref{lemmainker}--\ref{lemmainexpw}, one deduces that
\begin{equation} \nonumber 
	\begin{aligned}
		\bigl| \nabla \vec{S}_{R,i} \bigr|(\vec{x},t) &= \left|\nabla \mathbb{G}_L(\vec{x},t) \underset{(\vec{x},t)}{\ast} (R_p, R_{\vec{u}},R_{\chi})(\vec{x},t) + \mathbb{G}_S(\vec{x},t) \underset{(\vec{x},t)}{\ast} \nabla (R_p, R_{\vec{u}},R_{\chi})(\vec{x},t) \right| \\
		& \le C_s {{\mathcal{N}(t)\bigl(\mathcal{N}(t)+\epsilon_0\bigr)}}
		\begin{cases}
			(1+t)^{-\frac{1}{2}} e_1(\vec{x},t)  & \text{for } i=1,2, \\
			e_2(\vec{x},t)  & \text{for } i=3,
		\end{cases}
	\end{aligned}
\end{equation}
which together with \eqref{duhpux} and \eqref{inipux1} yields the pointwise estimates for $\nabla \vec{U}_i(\vec{x},t)$:
\begin{equation} \label{puxpoin1}
	\bigl|\nabla \vec{U}_i \bigr|(\vec{x},t) \le \left( C_{in} \epsilon_0 + C_s {{\mathcal{N}(t)\bigl(\mathcal{N}(t)+\epsilon_0\bigr)}} \right)
	\begin{cases}
		(1+t)^{-\frac{1}{2}} e_1(\vec{x},t) & \text{for } i=1,2, \\
		e_2(\vec{x},t) & \text{for } i=3.
	\end{cases}
\end{equation}

Finally for the second order derivative of the solution, the time decay estimates for the second and third order derivatives of the nonlinear terms are required:
\begin{equation} \label{f23}
	\left|\nabla^i (R_p, R_{\vec{u}}, R_{\chi})\right| \! (\vec{x},t) \le C_{\epsilon^{-1}} {{\mathcal{N}(t)\bigl(\mathcal{N}(t)+\epsilon_0\bigr)}}(1 + t)^{-3}, \quad \text{for } i=2,3,
\end{equation}
which could be obtained by invoking the definitions \eqref{vnepre} and \eqref{anas}, the decay property \eqref{decaygen}, together with the results established in \eqref{slp1} and \eqref{slp2}. Based on \eqref{f23}, one establishes the following space-time convolution estimates:
\begin{equation} \label{expont2}
	\begin{aligned}
		& e^{-\mu_2 t} \underset{t}{\ast} (1+t)^{-3} \le C (1+t)^{-3},\\[0.2ex]
		& |\vec{x}|^{-j} e^{-(|\vec{x}|+t)/D_0} \underset{(\vec{x},t)}{\ast} (1+t)^{-3} \\
		& \le \bigl\| |\vec{x}|^{-j} e^{-|\vec{x}|/D_0}\bigr\|_{L^1} e^{-t/D_0} \underset{t}{\ast} (1+t)^{-3} \le C(1+t)^{-3}, \quad \text{for } j=0,1,2,\\[0.2ex]
		& e^{-\mu_i t} \omega(\vec{x},t;\bar{c}_f) \underset{(\vec{x},t)}{\ast} (1+t)^{-3}=4\pi te^{-\mu_i t} \underset{t}{\ast} (1+t)^{-3} \le C (1+t)^{-3}, \quad \text{for } i=1,2.
	\end{aligned}
\end{equation}
In addition, by setting $m=|k|$ in Lemma~\ref{lemmainker} and for $|k|=2,3$, one obtains that
\begin{equation} \label{heke3}
	(1+t)^{-2-\frac{|k|}{2}} e^{-\frac{(|\vec{x}| - \bar{c} t)^2}{K(1 + t)}} \underset{(\vec{x},t)}{\ast} \bigl|(R_p, R_{\vec{u}}, R_{\chi})\bigr|(\vec{x},t) \le C{{\mathcal{N}^2(t)}}(1+t)^{-3},
\end{equation}
which combined with the pointwise estimates for $\mathbb{G}(\vec{x},t)$ from Theorem~\ref{green}, \eqref{duhpux}, \eqref{inipux2} and \eqref{f23}--\eqref{heke3} results in that
\begin{equation} \label{puxpoin2}
	\begin{aligned}
		\bigl|\nabla^2 \vec{U}\bigr| (\vec{x},t) & \le \left| \nabla^2 \mathbb{G} (\vec{x},t) \underset{\vec{x}}{\ast} \vec{U}(\vec{x},0)\right|+ \left|\nabla^2 \mathbb{G}_L(\vec{x},t) \underset{(\vec{x},t)}{\ast} (R_p, R_{\vec{u}},R_{\chi})(\vec{x},t) \right| \\
		& \quad + \left| \mathbb{G}_S(\vec{x},t) \underset{(\vec{x},t)}{\ast} \nabla^2 (R_p, R_{\vec{u}},R_{\chi})(\vec{x},t) \right| \\
		&\le \left( C_{in}\epsilon_0 + C_s {{\mathcal{N}(t)\bigl(\mathcal{N}(t)+\epsilon_0\bigr)}}\right)(1+t)^{-3}.
	\end{aligned}
\end{equation}

Finally, one combines \eqref{anas}, \eqref{s0}--\eqref{duhpux}, \eqref{puxpoin}, \eqref{puxpoin1} and \eqref{puxpoin2} to obtain that		\begin{equation} \nonumber
			\mathcal{N}(t) \le  C_{in}\epsilon_0 + C_{s} \mathcal{N}(t)\bigl(\mathcal{N}(t)+\epsilon_0\bigr).
		\end{equation}
For $\epsilon_0$ sufficiently small, one can further deduce that
		\begin{equation} \nonumber
			\mathcal{N}(t) \le C\epsilon_0,
		\end{equation}
and completes the proof of Theorem~\ref{main1}.

%%%%%%%%%%%%%%%%%%%%%%%%%%%%%%%%%%%%%%%%%%%%%%%%%%%%%%%%%%%%%%%%%%%%%%%%%%%%%%%%%5
%
%
%                       Section 4:
%
%
%%%%%%%%%%%%%%%%%%%%%%%%%%%%%%%%%%%%%%%%%%%%%%%%%%%%%%%%%%%%%%%%%%%%%%%%%%%%%%%%%

%%%%%%%%%%%%%%%%%%%%%%%%%%%%%%%%%%%%%%%%%%%%%%%%%%%%%%%%%%%%%%%%%%%%%%%%%%%%%%%%%5
%
%
%                       Section 4:
%
%
%%%%%%%%%%%%%%%%%%%%%%%%%%%%%%%%%%%%%%%%%%%%%%%%%%%%%%%%%%%%%%%%%%%%%%%%%%%%%%%%%

\section{Asymptotic Analysis in the $L^p$ Framework for Background Perturbations: Proof of Theorem~\ref{main2}}
As mentioned in previous studies, a direct comparison between two individual perturbations fails to induce a dependence on the difference of the background states. Instead, one considers the equations \eqref{gbp} of the difference function $\vec{Y}$, which is defined by \eqref{yy1} and \eqref{diff} as the difference of the two perturbations $(p^{\beta},\vec{u}^{\beta}, \chi^{\beta}, s^{\beta})$ ($\beta=a,b$). 

The global existence of nonlinear solutions $(p^{\beta},\vec{u}^{\beta}, \chi^{\beta}, s^{\beta})$ and the decay properties of $(\nabla p^\beta, \nabla \vec{u}^\beta, \chi^\beta)$ for $\beta=a,b$ have been established in \cite{hl} (one could refer to the aforementioned Theorem~\ref{genes}); the pointwise structures of the solutions are also obtained in the previous section in this paper (one could refer to Theorem~\ref{main1} for more details). Consequently, under the assumptions of Theorem~\ref{genes}, their difference $\vec{Y}$ also admits a global solution. Moreover, $(p^{\beta},\vec{u}^{\beta}, \chi^{\beta}, s^{\beta})$ ($\beta=a,b$) satisfies estimates stated in Theorem~\ref{main1}, wherein $\vec{x}$ and $\bar{c}$ are replaced by the shifted variable $\vec{x} - \vec{u}_{\beta}t$ and $\sqrt{\gamma p_\beta v_\beta/(1+b)}$ respectively.

In the rest of this section, one aims to derive the relationship between $\vec{Y}=(\vec{Y}_1,z)^t$ and the difference of the background states. First one could formulate $\vec{Y}_1=(g^0, \vec{h}, f)^t$ via the Green's function $\mathbb{G}(\vec x,t)$ and the Duhamel's principle:
\begin{equation}\label{diffrep}
	\begin{aligned}
		\vec{Y}_1(\vec x, t)&=\mathbb{G}(\vec x,t) \underset{\vec{x}}{\ast} \vec{Y}_1(\vec x,0)+\int_0^t  \mathbb{G}(\vec{x},t-\tau) \underset{\vec{x}}{\ast}
		\begin{pmatrix} 
			R_{g^0}^ {\RN{1}}\vspace{1.5mm}
			\\
			R_{\vec h}^ {\RN{1}}\vspace{1.5mm}
                         \\
			R_{f}^ {\RN{1}}
		\end{pmatrix}(\vec x,\tau) \, d\tau 
		+ \int_0^t  \mathbb{G}(\vec x, t-\tau) \underset{\vec{x}}{\ast} \begin{pmatrix}
			R_{g^0}^ {\RN{2}}\vspace{1.5mm}
			\\
			R_{\vec h}^ {\RN{2}}\vspace{1.5mm}
                         \\
			R_{f}^ {\RN{2}}
		\end{pmatrix}(\vec{x},\tau) \, d\tau\\
		& := \mathbb{G}(\vec x,t) \underset{\vec{x}}{\ast} \vec{Y}_1(\vec x,0) + \vec{Y}_A(\vec{x},t) +\vec{Y}_C(\vec{x},t).
	\end{aligned}
\end{equation}
Here, $R_{g^0}^ {\RN{1}}, R_{\vec h}^ {\RN{1}}, R_{f}^ {\RN{1}}$ defined in \eqref{lin} are source terms that contain solely products of the individual perturbation $(p^b,\vec{u}^b, \chi^b)$ and the difference of background states, and thus their upper bounds are available:
\begin{equation} \nonumber 
	\begin{aligned}
		\left| R_{g^0}^ {\RN{1}}\right| & \le \bigl|\vec{u}_b \bigr|\cdot \left|\nabla p^b \right| + \gamma \left| p_b - p_a \right| \cdot \left|\nabla \cdot \vec{u}^b \right| + \frac{2}{\alpha v_a v_b} \left| v_b - v_a \right| \cdot \left| \chi^b \right| \\ 
		& \le \bigl| \vec{u}_b - \vec{u}_a \bigr|\cdot \left|\nabla p^b \right| + C \bigl| \vec{\theta}_b - \vec{\theta}_a \bigr| \! \left( \left|\nabla \cdot \vec{u}^b \right| +  \left| \chi^b \right| \right), \\ 
		\left| R_{\vec h}^ {\RN{1}}\right| & \le  \bigl| \vec{u}_b - \vec{u}_a \bigr| \cdot \left|\nabla \vec{u}^b \right| + \left| v_b - v_a \right| \cdot \left| \nabla p^b \right|,\\
		\left| R_{f}^ {\RN{1}}\right| & \le  \bigl| \vec{u}_b - \vec{u}_a \bigr| \cdot \left|\nabla \chi^b \right| + C\bigl| \vec{\theta}_b - \vec{\theta}_a \bigr| \cdot \left| \nabla \cdot \vec{u}^b \right|.
	\end{aligned}
\end{equation}
In the above inequalities, one uses the relation $|p_b - p_a| \le C |\vec{\theta}_b - \vec{\theta}_a| $, with $\vec{\theta}_b - \vec{\theta}_a := (v_b - v_a, e_{1,b} - e_{1,a})$. The coupling terms $R_{g^0}^ {\RN{2}}, R_{\vec h}^ {\RN{2}}, R_{f}^ {\RN{2}}$ given by \eqref{nlin} consist of the unknown difference function $\vec{Y}$ as well as the individual perturbations $(p^{\beta},\vec{u}^{\beta},\chi^{\beta},s^{\beta})$ ($\beta=a,b$). A significant distinction between source terms $R_{g^0}^ {\RN{1}}, R_{\vec h}^ {\RN{1}}, R_{f}^ {\RN{1}}$ and coupling terms $R_{g^0}^ {\RN{2}}, R_{\vec h}^ {\RN{2}}, R_{f}^ {\RN{2}}$ is that the latter ones do not exhibit any dependence on the background state differences $\vec{F}_b-\vec{F}_a$ or $\vec{W}_b-\vec{W}_a$ from the existing results for $\vec{Y}$. In fact, the recent established estimates for $\vec{Y}$ are analogous to those for the individual perturbations $(p^{\beta},\vec{u}^{\beta}, \chi^{\beta}, s^{\beta})$ ($\beta = a,b$), namely:
$$\left| \vec{Y} (\vec{x}, t) \right| = \left|(p^b - p^a, \vec{u}^b -\vec{u}^a, \chi^b - \chi^a, s^b - s^a)(\vec x, t) \right| \le \left|(p^a, \vec{u}^a, \chi^a, s^a)(\vec x, t) \right| + \left|(p^b, \vec{u}^b, \chi^b, s^b)(\vec x, t) \right|.$$
Accordingly, the coupling terms will be treated as nonlinear terms and closed by combining the ansatz assumptions and the a priori bounds derived via the energy method later.\\

\textbf{Initial Propagations.}
Assume that the two initial perturbations $(p_0^{\beta},\vec{u}_0^{\beta}, \chi_0^{\beta}, s_0^{\beta})(\vec{x})$ ($\beta = a,b$) satisfy the initial assumptions in Theorem~\ref{main2}. Especially, it holds that
\begin{equation}\nonumber
	\left\|\vec{Y}(\cdot,0)\right\|_{L^1\cap H^8} = \left\|(p_0^b - p_0^a,\vec{u}_0^b - \vec{u}_0^a, \chi_0^b - \chi_0^a, s_0^b - s_0^a )\right\|_{L^1\cap H^8} \le C \bigl| \vec{F}_b - \vec{F}_a \bigr| \le C \bigl| \vec{W}_b - \vec{W}_a \bigr| \ll 1,
\end{equation}
which, combined with Theorem~\ref{lpg}, results in that for $i=1,2,3$ and $|k|\le5$,
\begin{equation} \label{ini0}
	\begin{aligned}
		\left|\partial_{\vec{x}}^k \mathbb{G}_{i}(\vec{x}, t) \underset{\vec{x}}{\ast} \vec{Y}_1 (\vec{x}, 0) \right| & \le \left|\mathbb{G}_{S,i}(\vec{x}, t) \underset{\vec{x}}{\ast} \partial_{\vec{x}}^k \vec{Y}_1(\vec{x}, 0 ) \right| + \left|\partial_{\vec{x}}^k \mathbb{G}_{L,i} (\vec{x}, t) \underset{\vec{x}}{\ast} \vec{Y}_1(\vec{x}, 0) \right| \\
		& \le C e^{-t/C_0} \left( \left\|\partial_{\vec{x}}^k \vec{Y}_1 (\cdot,0) \right\|_{L^\infty} + \left\|\partial_{\vec{x}}^k  \nabla (g^0, \vec{h})(\cdot,0) \right\|_{L^\infty} \right) \\
		& \quad + C (1 + t)^{-2-d_i-\frac{|k|}{2}} \left\|\vec{Y}_1(\cdot,0) \right\|_{L^1} \\
		& \le C_{in} \left( 1 + t \right)^{-2-d_i-\frac{|k|}{2}} \bigl| \vec{W}_b - \vec{W}_a \bigr|,
	\end{aligned}
\end{equation}
\begin{equation}\label{ini2}
	\begin{aligned}
		\left\|\partial_{\vec{x}}^k \mathbb{G}_i (\cdot, t) \underset{\vec{x}}{\ast} \vec{Y}_1(\cdot, 0)\right\|_{L^2}
		& \le C e^{-t/C_2} \left\|\partial_{\vec{x}}^k \vec{Y}_1 (\cdot,0) \right\|_{L^2} + C (1 + t)^{-\frac{3}{4}-d_i-\frac{|k|}{2}} \left\| \vec{Y}_1 (\cdot,0) \right\|_{L^1}\\
		& \le C_{in} ( 1 + t)^{-\frac{3}{4} -d_i -\frac{|k|}{2}} \bigl|\vec{W}_b-\vec{W}_a\bigr|,
	\end{aligned}
\end{equation}
and
\begin{equation}\label{iniz}
	 \left\|\partial_{\vec{x}}^k z(\cdot, 0) \right\|_{L^\infty} \le C \left\|\nabla^{|k|+1} z(\cdot, 0) \right\|_{H^1} \le C_{in} \bigl|\vec{W}_b - \vec{W}_a\bigr|, \quad \left\| \partial_{\vec{x}}^k z(\cdot, 0) \right\|_{L^2} \le C_{in} \bigl|\vec{W}_b - \vec{W}_a\bigr|.
\end{equation}
Here, the positive constant $C_{in}$ depends on the initial data and is independent of $|\vec{W}_b-\vec{W}_a|$. The constants $d_i$, which characterize the distinction in decay rates, are defined in \eqref{di}. The constraint $|k| \le 5$ is imposed to ensure the validity of the Sobolev inequality
$$\left\|\partial_{\vec{x}}^k  \nabla ( g^0, \vec{h})  (\cdot,0) \right\|_{L^\infty} \le C\left\|\nabla^{|k|+2} ( g^0, \vec{h})  (\cdot,0) \right\|_{H^1} \le C \left\|( g^0, \vec{h})  (\cdot,0) \right\|_{H^8}.$$

\subsection{Basic and Enhanced Energy Estimates}
It is obvious from \eqref{diffrep} that high order derivatives are involved in the solution representation. In fact, a priori estimates from energy method are always required for the closure of nonlinearity in a quasi-linear system.

\begin{Theorem}\label{enle}
	Let $m \ge 3$. Assume $\|\vec{Y}(\cdot,0)\|_{H^m} = \mathcal{O}(1)|\vec{W}_b - \vec{W}_a|\ll 1$. Then the solution $\vec{Y}(\vec{x}, t)$ of \eqref{gbp} satisfies the following estimates \textup{:}
	\begin{equation} \label{enghf}
		\left\|\vec{Y}_1(\cdot, t)\right\|_{H^m}^2 + \int_0^t \left( \left\| \nabla g^0 (\cdot, \tau)\right\|_{H^{m-1}}^2 + \left\|\nabla \cdot \vec{h} (\cdot, \tau)\right\|_{H^{m-1}}^2 + \left\|f(\cdot, \tau)\right\|_{H^m}^2 \right) d\tau \le C \bigl|\vec{W}_b - \vec{W}_a\bigr|^2
	\end{equation}
	and
	\begin{equation}\label{ensz}
		\bigl\|z(\cdot, t)\bigr\|_{H^m} \le C(1+t)\bigl|\vec{W}_b - \vec{W}_a\bigr|,
	\end{equation}
	where $C > 0$ is a universal constant independent of $t$ and $|\vec{W}_b - \vec{W}_a|$.
\end{Theorem}
\begin{proof}
For a multi-index $k = (k_1, k_2, k_3)$, take the corresponding derivatives $\partial_{\vec{x}}^k = \partial_{x_1}^{k_1}\partial_{x_2}^{k_2}\partial_{x_3}^{k_3}$ on both sides of \eqref{gbp} and obtain that
\begin{align}
	& \partial_t \partial_{\vec x}^k g^0 + \gamma p_a  \partial_{\vec x}^k (\nabla \cdot \vec{h})  + \frac{2}{\alpha v_a} \partial_{\vec x}^k f = \partial_{\vec x}^k R_{g^0}^ {\RN{1}}+ \partial_{\vec x}^k R_{g^0}^ {\RN{2}},  \label{kg0} \\  
	& \partial_t \partial_{\vec x}^k \vec{h} + v_a \partial_{\vec x}^k \nabla g^0 = \partial_{\vec x}^k R_{\vec h}^ {\RN{1}} + \partial_{\vec x}^k R_{\vec h}^ {\RN{2}}, \label{kh} \\ 
	& \partial_t \partial_{\vec x}^k f + \frac{\alpha_f}{\alpha} p_a v_a \partial_{\vec x}^k (\nabla \cdot \vec{h}) + \left( 1 + \frac{\alpha_f}{\alpha} \right) \partial_{\vec x}^k f = \partial_{\vec x}^k R_{f}^ {\RN{1}} + \partial_{\vec x}^k R_{f}^ {\RN{2}}, \label{kf} \\
	& \partial_t \partial_{\vec x}^k z = -\vec{u}_b \cdot \partial_{\vec x}^k \nabla s^b + \partial_{\vec x}^k R_{z}^ {\RN{2}}. \label{kz}
\end{align}

Denote
\begin{equation} \label{b123}
	b_1 := \left( \tilde{c}_f^a \right)^2 \left( \tilde{\rho}^a \right)^2 = \gamma \tilde{p}^a \tilde{\rho}^a,  
	\quad b_2 := \frac{c_f^2}{a} \tilde{\rho}^a = \frac{2+\alpha}{\alpha_f} \tilde{\rho}^a,
\end{equation}
where functions marked with a tilde represent the sum of the constant-state solution and the perturbation solution, as exemplified by $\tilde{p}^a = p_a + p^a$, and the quantities $c_f^2$, $a$, $\gamma$ and $b$ are defined in \eqref{equ} and \eqref{bb}. The equation for $\tilde{\rho}^a$ presented in \eqref{vnes} together with \eqref{kf} and the definition of $b_2$ in \eqref{b123} yields that
\begin{equation}  \label{kfb3}
	\partial_t \! \left( b_2 \partial_{\vec{x}}^k f\right)  + \frac{\alpha_f}{\alpha} p_a v_a b_2 \partial_{\vec{x}}^k (\nabla \cdot \vec{h}) + \left( 1 + \frac{\alpha_f}{\alpha} \right) b_2 \partial_{\vec{x}}^k f + \frac{2+\alpha}{\alpha_f} \nabla \cdot  \left( \tilde{\rho}^a \vec{u}^a \right) \partial_{\vec{x}}^k f  = b_2  \partial_{\vec x}^k R_{f}^ {\RN{1}}+ b_2  \partial_{\vec x}^k R_{f}^ {\RN{2}}.
\end{equation}
Integrating the inner product of $\bigl(\eqref{kg0}, \eqref{kh}, \eqref{kg0}-\eqref{kfb3}\bigr)$ and $\bigl(\partial_{\vec{x}}^k g^0, b_1 \partial_{\vec{x}}^k \vec{h}, b (\partial_{\vec{x}}^k g^0 - b_2 \partial_{\vec{x}}^k f)\bigr)$ over $\mathbb{R}^3$ with respect to $\vec{x}$, one has that
\begin{equation} \label{leen1}
	\begin{aligned}
		& \frac{1}{2}\frac{d}{dt} \int_{\mathbb{R}^3} \left( \partial_{\vec{x}}^k g^0 \right)^2 d\vec{x} + \frac{1}{2}\frac{d}{dt} \int_{\mathbb{R}^3} b_1 \left|\partial_{\vec{x}}^k \vec{h} \right|^2  d\vec{x} + \frac{b}{2} \frac{d}{dt} \int_{\mathbb{R}^3} \left( \partial_{\vec{x}}^k g^0 - b_2 \partial_{\vec{x}}^k f \right)^2 d\vec{x}\\
		& \quad + b \int_{\mathbb{R}^3}  b_2 \left[ \left(1+\frac{\alpha_f}{\alpha}+\frac{1}{2} \nabla\cdot\vec{u}^a \right) b_2 - \frac{2}{\alpha v_a} \right] \! \left( \partial_{\vec{x}}^k f \right)^2 d\vec{x}  - \frac{b+1}{2} \int_{\mathbb{R}^3} \left(\nabla \cdot\vec{u}^b\right) \! \left(\partial_{\vec{x}}^k g^0 \right)^2 d\vec{x} \\
		& \quad + \frac{\gamma}{2} \int_{\mathbb{R}^3} \left(b_1 \nabla\cdot\vec{u}^a + \frac{2}{\alpha} \left( \tilde{\rho}^a \right)^2 \chi^a \right) \! \left| \partial_{\vec{x}}^k \vec{h} \right|^2  d\vec{x} -  \left( 1 + \frac{\alpha_f}{\alpha} \right) a_0 b \int_{\mathbb{R}^3} \rho^a \partial_{\vec{x}}^k g^0 \partial_{\vec{x}}^k f \, d\vec{x} \\
		& \quad - \gamma \int_{\mathbb{R}^3} \left( \nabla p^a \cdot \partial_{\vec{x}}^k \vec{h} \right) \partial_{\vec{x}}^k g^0 \,  d\vec{x} - b \int_{\mathbb{R}^3} b_2 \vec{h} \cdot \partial_{\vec{x}}^k \nabla g^0 \partial_{\vec{x}}^k f \, d\vec{x} \\
		& \quad + \gamma b\int_{\mathbb{R}^3} \tilde{\rho}^a \left( p^a \tilde{v}^a - p^b \tilde{v}^b \right) \! \left(  \partial_{\vec{x}}^k g^0 - b_2 \partial_{\vec{x}}^k f \right) \partial_{\vec{x}}^k (\nabla \cdot \vec{h}) \, d\vec{x} \\
		& = \int_{\mathbb{R}^3} b_1 \partial_{\vec{x}}^k \vec{h} \cdot \left( \partial_{\vec{x}}^k R_{\vec h}^ {\RN{1}} + \partial_{\vec{x}}^k R_{\vec h}^ {\RN{2}}  +  \vec{u}^a \cdot \partial_{\vec{x}}^k \nabla \vec{h} + v^a \partial_{\vec{x}}^k \nabla g^0 \right) d\vec{x} \\
		& \quad + b \int_{\mathbb{R}^3} \Bigl( (1+b^{-1}) \partial_{\vec{x}}^k g^0 - b_2  \partial_{\vec{x}}^k f \Bigr)  \! \left( \partial_{\vec{x}}^k R_{g^0}^ {\RN{1}} + \partial_{\vec{x}}^k R_{g^0}^ {\RN{2}} + \vec{u}^b \cdot \partial_{\vec{x}}^k \nabla g^0 + \gamma p^a \partial_{\vec{x}}^k (\nabla \cdot \vec{h}) \right) d\vec{x} \\
		& \quad -b \int_{\mathbb{R}^3} b_2 \Bigl( \partial_{\vec{x}}^k g^0 - b_2 \partial_{\vec{x}}^k f \Bigr) \! \left( \partial_{\vec{x}}^k R_{f}^ {\RN{1}} + \partial_{\vec{x}}^k R_{f}^ {\RN{2}} + \vec{u}^a \cdot \nabla\partial_{\vec{x}}^k f + \frac{\alpha_f}{\alpha} \left( p^b \tilde{v}^b + p_a v^a \right) \partial_{\vec{x}}^k (\nabla \cdot \vec{h}) \right) d\vec{x},
	\end{aligned} 
\end{equation}
where $a_0 := (2+\alpha)/\alpha_f$. Here, $|p^a \tilde{v}^a - p^b \tilde{v}^b|$ could be estimated as follows:
$$|p^a \tilde{v}^a - p^b \tilde{v}^b| \le C(|p^a|+|p^b|)$$ and
$$ \left| p^a \tilde{v}^a - p^b \tilde{v}^b\right|  = \left|-g^0 \tilde{v}^a + p^b (\tilde{v}^a - \tilde{v}^b)\right| \le C \left[ |g^0| + |p^b| \bigl( \bigl|\vec{\theta}_b-\vec{\theta}_a\bigr| + |f| + |z| \bigr) \right].$$
Moreover, after direct algebraic calculations and rearrangements, the smallness of $\epsilon_a$ and Theorem~\ref{main1} yield that
$$\begin{aligned}
	 \left(1+\frac{\alpha_f}{\alpha}+\frac{1}{2} \nabla\cdot\vec{u}^a \right) b_2 - \frac{2}{\alpha v_a} 
	&= \left[ a_0 \left(1+\frac{\alpha_f}{\alpha}+\frac{1}{2} \nabla\cdot\vec{u}^a \right) - \frac{2}{\alpha} \right] \rho_a + a_0 \left( 1+\frac{\alpha_f}{\alpha}+\frac{1}{2} \nabla\cdot\vec{u}^a \right) \rho^a \\
	&\ge \left[ a_0 \left(1+\frac{\alpha_f}{\alpha}\right)-\frac{2}{\alpha} - C \epsilon_a \right]\rho_a - C\epsilon_a\\
	&= (a_0+1) \rho_a - C \epsilon_a(\rho_a + 1) = \frac{2}{\alpha b} \rho_a - C \epsilon_a(\rho_a + 1),
\end{aligned}$$
which further implies that
$$b b_2 \left[ \left(1+\frac{\alpha_f}{\alpha}+\frac{1}{2} \nabla\cdot\vec{u}^a \right) b_2 - \frac{2}{\alpha v_a} \right] \ge \frac{2a_0}{\alpha} \rho_a \tilde{\rho}^a - C \epsilon_a \tilde{\rho}^a \ge \frac{a_0}{2\alpha} \rho_a^2.$$
In addition, the higher-order terms 
$$\vec{u}^a \cdot \nabla\partial_{\vec{x}}^k \vec{h} + v^a \partial_{\vec{x}}^k \nabla g^0, \quad \vec{u}^b \cdot \partial_{\vec{x}}^k \nabla g^0 + \gamma p^a \partial_{\vec{x}}^k (\nabla \cdot \vec{h}), \quad \vec{u}^a \cdot \nabla\partial_{\vec{x}}^k f + \frac{\alpha_f}{\alpha} \bigl( p^b \tilde{v}^b + p_a v^a \bigr) \partial_{\vec{x}}^k (\nabla \cdot \vec{h})$$
on the right-hand side of \eqref{leen1} can be combined with the corresponding terms in $\partial_{\vec{x}}^k \bigl(R_{\vec h}^{\RN{2}}\bigr)$, $\partial_{\vec{x}}^k \bigl(R_{g^0}^{\RN{2}}\bigr)$, and $\partial_{\vec{x}}^k \bigl(R_{f}^{\RN{2}}\bigr)$, respectively, to give rise to
$$-\bigl[\partial_{\vec{x}}^k, \vec{u}^a \cdot \nabla\bigr] \vec{h} - \bigl[\partial_{\vec{x}}^k, v^a \nabla\bigr] g^0,\quad 
-\bigl[\partial_{\vec{x}}^k, \vec{u}^b \cdot \nabla\bigr]g^0 - \gamma\bigl[\partial_{\vec{x}}^k, p^a\bigr]\nabla\cdot\vec{h},$$
$$-\bigl[\partial_{\vec{x}}^k, \vec{u}^a \cdot \nabla\bigr] f - \frac{\alpha_f}{\alpha} \bigl[\partial_{\vec{x}}^k,  \bigl( p^b \tilde{v}^b + p_a v^a \bigr)\bigr] \nabla\cdot\vec{h}.$$
The higher-order terms on the left-hand side of \eqref{leen1}, such as $ b\displaystyle \int_{\mathbb{R}^3} b_2 \vec{h} \cdot \partial_{\vec{x}}^k \nabla g^0 \partial_{\vec{x}}^k f \, d\vec{x}$, can be estimated by the inequality $|\partial_{\vec{x}}^k \nabla g^0| \le |\partial_{\vec{x}}^k \nabla p^a| + |\partial_{\vec{x}}^k \nabla p^b|$ and upper bounds for individual solutions which have been obtained in previous sections.

Combining the definitions of the source terms $\bigl(R_{g^0}^{\RN{1}}, R_{\vec h}^{\RN{1}}, R_{f}^{\RN{1}}\bigr)$ and coupling terms $\bigl(R_{g^0}^{\RN{2}}, R_{\vec h}^{\RN{2}}, R_{f}^{\RN{2}}\bigr)$ in \eqref{lin} and \eqref{nlin} with the Sobolev inequalities \eqref{sob1} and \eqref{sob2}, and the energy estimate \eqref{genenes} for individual solutions, one can derive that
\begin{equation}\label{g0sen}
	\begin{aligned}
		& \frac{1}{2} \frac{d}{dt}  \sum_{0 \le |k| \le m} \left(\left\|\partial_{\vec{x}}^k g^0(\cdot, t)\right\|_{L^2}^2 + \left\| \sqrt{\gamma \tilde{p}^a \tilde{\rho}^a} \partial_{\vec{x}}^k \vec{h} (\cdot, t)\right\|_{L^2}^2 + b \left\| \left( \partial_{\vec{x}}^k g^0 - a_0 \tilde{\rho}^a \partial_{\vec{x}}^k f\right) \! (\cdot, t)\right\|_{L^2}^2 \right) + \frac{a_0}{8\alpha} \rho_a^2 \left\| f(\cdot, t)\right\|_{H^m}^2 \\
		& \le C \bigl|\vec{W}_b-\vec{W}_a\bigr|^2 \left( \left\|\nabla (p^b,\vec{u}^a,\vec{u}^b)(\cdot, t)\right\|_{H^m} + \left\|\chi^a (\cdot, t)\right\|_{H^m} + \left\| \chi^b (\cdot, t)\right\|_{H^{m+1}} \right)\\
		& \quad +  C \Bigl( \left\|g^0(\cdot, t)\right\|_{L^2}^2 + \left\|\vec{h}(\cdot, t)\right\|_{L^2}^2 \Bigr) \Bigl( \left\|\nabla (p^a,p^b,\vec{u}^a,\vec{u}^b,\chi^b)(\cdot, t)\right\|_{L^2} + \left\|(\chi^a,\chi^b)(\cdot, t)\right\|_{L^2} \Bigr)\\
		& \quad + C \bigl\|\nabla \vec{Y}_1(\cdot, t)\bigr\|_{H^{m-1}} \cdot \bigl\|\nabla z(\cdot, t)\bigr\|_{H^{m-1}} \cdot \left\|\nabla(\nabla p^b, \nabla \cdot \vec{u}^a, \nabla \cdot \vec{u}^b, \chi^a, \chi^b)(\cdot, t)\right\|_{H^{m-1}}\\
		& \quad + C \bigl\|\vec{Y}_1(\cdot, t)\bigr\|_{L^2} \cdot \bigl\|z(\cdot, t)\bigr\|_{L^2} \cdot \left\|(\nabla p^b, \nabla \cdot \vec{u}^a, \nabla \cdot \vec{u}^b, \chi^a, \chi^b)(\cdot, t)\right\|_{L^{\infty}}\\
		& \quad + C(\epsilon_a + \epsilon_b)\Bigl( \left\|\nabla g^0(\cdot, t)\right\|_{H^{m-1}}^2 + \left\|\nabla \cdot \vec{h}(\cdot, t)\right\|_{H^{m-1}}^2 \Bigr)
	\end{aligned}
\end{equation}
with $\vec{Y}_1 = \bigl(g^0, \vec{h}, f\bigr)$ defined in \eqref{yy1}. Herein, one employs the embedding inequality \eqref{sobemb}, the smallness of $\epsilon_{\beta}$ for $\beta=a,b$, the irrotational property which ensures that 
\begin{equation} \label{irrotah}
	\left\|\nabla \cdot \vec{h}(\cdot, t)\right\|_{H^{m-1}} = \left\|\nabla \vec{h}(\cdot, t)\right\|_{H^{m-1}},
\end{equation}
as well as the estimates $|v^b-v^a| = \mathcal{O}(1) \bigl(\bigl|\vec{\theta}_b-\vec{\theta}_a \bigr|+|g^0|+|f|+|z|\bigr)$ and 
\begin{equation} \label{vbva}
\begin{aligned}
	\sum_{1 \le |k| \le m} \left\|\partial_{\vec{x}}^k (v^b-v^a)(\cdot, t)\right\|_{L^2} & = \mathcal{O}(1) \Bigl[ \left( \bigl|\vec{\theta}_b-\vec{\theta}_a\bigr|+\left\|(g^0,f,z)(\cdot,t)\right\|_{L^\infty} \right) \left\| \nabla (p^b, \chi^b,s^b)(\cdot, t)\right\|_{H^{m-1}}\\
	& \quad \quad \quad \quad + \left\| \nabla (g^0, f, z)(\cdot, t)\right\|_{H^{m-1}} \Bigr].
\end{aligned}
\end{equation}

The equation for $\partial_t \bigl(\partial_{\vec x}^k f \, \partial_{\vec x}^k (\nabla \cdot \vec{h}) \bigr)$ could be derived from \eqref{kh}--\eqref{kf} and then integrate this new equation with respect to the spatial variable $\vec{x}$ over $\mathbb{R}^3$ to obtain that
\begin{equation} \nonumber 
	\begin{aligned}
		& \frac{d}{dt} \int_{\mathbb{R}^3} \partial_{\vec x}^k f \, \partial_{\vec x}^k (\nabla \cdot \vec{h}) \, d\vec{x} - v_a \int_{\mathbb{R}^3} \partial_{\vec x}^k \nabla f \cdot \partial_{\vec x}^k \nabla g^0 \, d\vec{x} + \frac{\alpha_f}{\alpha} p_a v_a \left\| \partial_{\vec x}^k (\nabla \cdot \vec{h})(\cdot, t)\right\|_{L^2}^2 \\
		& \quad + \left(1+\frac{\alpha_f}{\alpha} \right) \int_{\mathbb{R}^3} \partial_{\vec x}^k (\nabla \cdot \vec{h}) \, \partial_{\vec x}^k f \, d\vec{x} \\
		& = \int_{\mathbb{R}^3} \left[ -\partial_{\vec x}^k \nabla f \cdot \left(\partial_{\vec x}^k \left(R_{\vec h}^{\RN{1}}\right) + \partial_{\vec x}^k \left(R_{\vec h}^{\RN{2}}\right) \right) + \partial_{\vec x}^k (\nabla \cdot \vec{h}) \, \left( \partial_{\vec x}^k \left(R_{f}^{\RN{1}}\right) + \partial_{\vec x}^k \left(R_{f}^{\RN{2}}\right) \right) \right] d\vec{x},
	\end{aligned}
\end{equation}
which together with \eqref{lin}--\eqref{nlin}, \eqref{sob1}--\eqref{sob2}, \eqref{irrotah}--\eqref{vbva}, \eqref{sobemb}, and Theorem~\ref{genes} yields that 
\begin{equation}\label{hs}
	\begin{aligned}
		& \frac{d}{dt} \sum_{0 \le |k| \le m-1} \int_{\mathbb{R}^3} \partial_{\vec x}^k f \, \partial_{\vec x}^k (\nabla \cdot \vec{h}) \, d\vec{x} + \frac{\alpha_f}{\alpha} p_a v_a \left\|\nabla \cdot \vec{h} (\cdot, t)\right\|_{H^{m-1}}^2 - v_a  \sum_{0 \le |k| \le m-1} \int_{\mathbb{R}^3} \partial_{\vec x}^k \nabla f \cdot \partial_{\vec x}^k \nabla g^0 \, d\vec{x} \\
		& \le  C \epsilon_1^{-1} \left\|\left(R^{\RN{1}}_{\vec h},R^{\RN{1}}_{f}\right) (\cdot, t)\right\|_{H^{m-1}}^2 +C \epsilon_1^{-1} \left\|  f(\cdot, t)\right\|_{H^{m}}^2+C\epsilon_1  \left\|  \nabla\cdot\vec h(\cdot, t)\right\|_{H^{m-1}}^2 \\
		& \quad +\int_{\mathbb{R}^3}\Big|\partial_{\vec x}^k \nabla f \left(\partial_{\vec x}^k R_{\vec h}^{\RN{2}}\right) + \partial_{\vec x}^k (\nabla \cdot \vec{h})  \left(\partial_{\vec x}^k R_{f}^{\RN{2}} \right)\Big|d\vec x\\
		& \le  C( 1 + \epsilon_1^{-1}) \bigl|\vec{W}_b-\vec{W}_a\bigr|^2 \cdot  \left\|\nabla (p^b, \vec{u}^a, \vec{u}^b, \chi^b) (\cdot, t)\right\|_{H^{m-1}}^2 + C\left( \epsilon_a + \epsilon_b +\epsilon_1^{-1} \right) \left\|  f(\cdot, t)\right\|_{H^{m}}^2\\
		& \quad + C (\epsilon_a + \epsilon_b+\epsilon_1 )\Bigl( \left\| \nabla g^0(\cdot, t)\right\|_{H^{m-1}}^2 + \left\| \nabla \vec{h}(\cdot, t)\right\|_{H^{m-1}}^2 \Bigr) \\
		& \quad +\frac{\alpha_f}{\alpha} \int_{\mathbb{R}^3} \left| \partial_{\vec x}^k (\nabla \cdot \vec{h})\right|  \cdot  \left|\partial_{\vec x}^k\left(p^a(v^b-v^a)\nabla\cdot\vec u^a+p_b(v^b-v^a)\nabla\cdot\vec u^b\right)\right| d\vec x \\
		& \quad +\int_{\mathbb{R}^3} \left|\partial_{\vec x}^k \nabla f\right| \cdot \left|\partial_{\vec x}^k\left((v^b-v^a)\nabla p^b\right)\right| d\vec x \\
		& \le  C( 1 + \epsilon_1^{-1}) \bigl|\vec{W}_b-\vec{W}_a\bigr|^2 \cdot  \left\|\nabla (p^b, \vec{u}^a, \vec{u}^b, \chi^b) (\cdot, t)\right\|_{H^{m-1}}^2 \\
		& \quad + C( 1 + \epsilon_1^{-1}) \bigl\| z(\cdot, t)\bigr\|_{H^{m-1}}^2 \cdot  \left\|\nabla (\nabla p^b, \nabla \cdot \vec{u}^a, \nabla \cdot \vec{u}^b)(\cdot, t)\right\|_{H^{m-2}}^2\\
		& \quad + C (\epsilon_a + \epsilon_b+\epsilon_1 )\Bigl( \left\| \nabla g^0(\cdot, t)\right\|_{H^{m-1}}^2 + \left\| \nabla \vec{h}(\cdot, t)\right\|_{H^{m-1}}^2 \Bigr)+ C\left( \epsilon_a + \epsilon_b +\epsilon_1^{-1} \right) \left\|  f(\cdot, t)\right\|_{H^{m}}^2,
	\end{aligned}
\end{equation}
with a sufficiently small constant $\epsilon_1 \in (0,1)$ for applying Cauchy's inequality to eliminate the term $C(\epsilon_a+\epsilon_b+\epsilon_1) \| \nabla  \vec{h} (\cdot, t)\|_{H^{m-1}}^2$ from the left-hand side. In \eqref{hs}, one chooses a special treatment of the terms containing $v^b-v^a$ in $R_{\vec h}^{\RN{2}}, R_{f}^{\RN{2}}$ since $v^b-v^a$ contains the non-decaying component $z$ and the detailed estimates are given in \eqref{vbva}. This non-decaying $z$ results in a bound
$$\left\|\left(\nabla \cdot \vec{h}, \nabla f\right) \! (\cdot, t)\right\|_{H^{m-1}} \bigl\| z(\cdot, t)\bigr\|_{H^{m-1}} \left\|\nabla (\nabla p^b, \nabla \cdot \vec{u}^a, \nabla \cdot \vec{u}^b)(\cdot, t)\right\|_{H^{m-2}}$$
during the energy estimate in \eqref{hs} and it is estimated by applying Cauchy's inequality as follows:
\begin{equation}\nonumber
			\begin{aligned}
				&\left\|\nabla \cdot \vec{h} (\cdot, t)\right\|_{H^{m-1}} \bigl\| z(\cdot, t)\bigr\|_{H^{m-1}} \left\|\nabla (\nabla p^b, \nabla \cdot \vec{u}^a, \nabla \cdot \vec{u}^b)(\cdot, t)\right\|_{H^{m-2}}\\
				&\le \frac{\epsilon_1}{2} \left\|\nabla \cdot \vec{h} (\cdot, t)\right\|_{H^{m-1}}^2 + \frac{1}{2\epsilon_1} \bigl\| z(\cdot, t)\bigr\|_{H^{m-1}}^2 \left\|\nabla (\nabla p^b, \nabla \cdot \vec{u}^a, \nabla \cdot \vec{u}^b)(\cdot, t)\right\|_{H^{m-2}}^2.
			\end{aligned}
		\end{equation}
		This manipulation for the term involving $\|z(\cdot,t)\|_{H^{m-1}}^2$ is necessary since the desired bound for $\|z(\cdot,t)\|_{H^{m-1}}$ admits a time-growing factor $t$, while it holds that
		\[
		\left\|\nabla (\nabla p^b, \nabla \cdot \vec{u}^a, \nabla \cdot \vec{u}^b)(\cdot, t)\right\|_{H^{m-2}} \le (\epsilon_a + \epsilon_b)(1+t)^{-\frac{7}{4}}
		\]
		by Lemma~\ref{decay} and thus $\bigl\| z(\cdot, t)\bigr\|_{H^{m-1}}^2 \left\|\nabla (\nabla p^b, \nabla \cdot \vec{u}^a, \nabla \cdot \vec{u}^b)(\cdot, t)\right\|_{H^{m-2}}^2$ decays fast enough for time integral. Similarly, the use of the $L^\infty$ norm instead of the $L^2$ norm in the term
		$$\bigl\|\vec{Y}_1(\cdot, t)\bigr\|_{L^2} \bigl\|z(\cdot, t)\bigr\|_{L^2} \left\|(\nabla p^b, \nabla \cdot \vec{u}^a, \nabla \cdot \vec{u}^b, \chi^a, \chi^b)(\cdot, t)\right\|_{L^{\infty}}$$
		appearing in \eqref{g0sen} is also intended to furnish a sufficiently rapid decay rate.

In a similar vein, one deduces the equation for $\partial_t \bigl( \partial_{\vec{x}}^k \nabla g^0 \cdot \partial_{\vec{x}}^k \vec{h} \bigr)$ from \eqref{kg0}--\eqref{kh} and obtains that
\begin{equation} \nonumber
	\begin{aligned}
		& \frac{d}{dt} \int_{\mathbb{R}^3} \partial_{\vec{x}}^k \nabla g^0 \cdot \partial_{\vec{x}}^k \vec{h} \, d\vec{x} + v_a \left\| \partial_{\vec{x}}^k \nabla g^0 (\cdot, t)\right\|_{L^2}^2 - \gamma p_a \left\| \partial_{\vec x}^k (\nabla \cdot \vec{h})(\cdot, t)\right\|_{L^2}^2 - \frac{2}{\alpha v_a} \int_{\mathbb{R}^3} \partial_{\vec x}^k (\nabla \cdot \vec{h}) \, \partial_{\vec{x}}^k f \, d\vec{x} \\
		& = \int_{\mathbb{R}^3} \left[ \partial_{\vec x}^k (\nabla \cdot \vec{h})\, \left(\partial_{\vec x}^k \left(R_{g^0}^{\RN{1}}\right) + \partial_{\vec x}^k \left(R_{g^0}^{\RN{2}}\right) \right) + \partial_{\vec{x}}^k \nabla g^0 \cdot \left(\partial_{\vec x}^k \left(R_{\vec h}^{\RN{1}}\right) + \partial_{\vec x}^k \left(R_{\vec h}^{\RN{2}}\right) \right) \right] d\vec{x}
	\end{aligned}
\end{equation}
and
\begin{equation}\label{g0s}
	\begin{aligned}
		& \frac{d}{dt} \sum_{0 \le |k| \le m-1} \int_{\mathbb{R}^3}  \partial_{\vec{x}}^k \nabla g^0 \cdot \partial_{\vec{x}}^k \vec{h} \, d\vec{x} + \frac{v_a}{2} \left\|\nabla g^0(\cdot, t)\right\|_{H^{m-1}}^2  \\
		& \le C( 1 + \epsilon_2^{-1}) \bigl|\vec{W}_b-\vec{W}_a \bigr|^2 \cdot  \left\| (\nabla p^b, \nabla \vec{u}^b, \chi^a, \chi^b) (\cdot, t)\right\|_{H^{m-1}}^2 \\
		& \quad + C( 1 + \epsilon_2^{-1}) \bigl\| z(\cdot, t)\bigr\|_{H^{m-1}}^2 \cdot  \left\|\nabla (\nabla p^b, \chi^a, \chi^b)(\cdot, t)\right\|_{H^{m-2}}^2 \\
		& \quad + C\Bigl(\left\| f(\cdot, t)\right\|_{H^{m-1}}^2 +  \left\|\nabla \cdot \vec{h} (\cdot, t)\right\|_{H^{m-1}}^2 \Bigr), 
	\end{aligned}
\end{equation}
for a chosen sufficiently small constant $\epsilon_2 \in (0,1)$.

From \eqref{hs} and \eqref{g0s}, there exists a positive constant $\sigma_1 \in (0,1)$ such that
\begin{equation} \label{g0h}
	\begin{aligned}
		& \frac{d}{dt} \sum_{0 \le |k| \le m-1} \left( \int_{\mathbb{R}^3} \partial_{\vec{x}}^k f \, \partial_{\vec x}^k (\nabla \cdot \vec{h}) \, d\vec{x} + \sigma_1 \int_{\mathbb{R}^3} \partial_{\vec{x}}^k \nabla g^0 \cdot \partial_{\vec{x}}^k \vec{h} \, d\vec{x} \right) + \frac{\alpha_f}{8\alpha}p_a v_a \left\| \nabla \cdot \vec{h}(\cdot, t) \right\|_{H^{m-1}}^2\\
		& \quad + \frac{\sigma_1 v_a}{8} \left\| \nabla g^0(\cdot, t)\right\|_{H^{m-1}}^2 \\
		& \le C \bigl|\vec{W}_b-\vec{W}_a\bigr|^2 \left( \left\| \nabla (p^b,\vec{u}^a,\vec{u}^b)(\cdot, t)\right\|_{H^{m-1}}^2 + \left\| \chi^a(\cdot, t)\right\|_{H^{m-1}}^2 + \left\| \chi^b(\cdot, t)\right\|_{H^m}^2 \right)\\
		& \quad + C\left( \left\| f(\cdot, t)\right\|_{H^m}^2 + \bigl\| z(\cdot, t)\bigr\|_{H^{m-1}}^2 \cdot \left\| \nabla (\nabla p^b, \nabla \cdot \vec{u}^a, \nabla \cdot \vec{u}^b, \chi^a, \chi^b)(\cdot, t)\right\|_{H^{m-2}}^2 \right).
	\end{aligned}
\end{equation}
Then one could again choose a suitable sufficiently small constant $\sigma_2$ and \eqref{g0sen}$+\sigma_2 \, \eqref{g0h}$ yields that
\begin{equation} \label{eny}
	\begin{aligned}
		&\frac{1}{2} \frac{d}{dt} \sum_{0 \le \left|k\right| \le m} \left( \left\|\partial_{\vec{x}}^k g^0(\cdot, t)\right\|_{L^2}^2 + \left\| \sqrt{\gamma \tilde{p}^a \tilde{\rho}^a} \partial_{\vec{x}}^k \vec{h}(\cdot, t)\right\|_{L^2}^2 + b \left\| \left( \partial_{\vec{x}}^k g^0 - a_0 \tilde{\rho}^a \partial_{\vec{x}}^k f\right) \! (\cdot, t)\right\|_{L^2}^2 \right)\\
		&\quad + \sigma_2 \frac{d}{dt} \sum_{0 \le \left|k\right| \le m-1} \left( \int_{\mathbb{R}^3} \partial_{\vec{x}}^k f \, \partial_{\vec x}^k (\nabla \cdot \vec{h}) \, d\vec{x} + \sigma_1 \int_{\mathbb{R}^3} \partial_{\vec{x}}^k \nabla g^0 \cdot \partial_{\vec{x}}^k \vec{h} \, d\vec{x} \right)\\
		&\quad +\frac{a_0}{16\alpha} \rho_a^2 \left\| f(\cdot, t)\right\|_{H^m}^2 + \frac{\sigma_2 \alpha_f}{16\alpha} p_a v_a \left\|\nabla \cdot \vec{h}(\cdot, t)\right\|_{H^{m-1}}^2 + \frac{\sigma_1 \sigma_2 v_a}{16} \left\|\nabla g^0(\cdot, t)\right\|_{H^{m-1}}^2  \\
		& \le C \bigl|\vec{W}_b-\vec{W}_a\bigr|^2 \left( \left\|\nabla (p^b,\vec{u}^a,\vec{u}^b)(\cdot, t)\right\|_{H^m} + \left\|\chi^a(\cdot, t)\right\|_{H^m} + \left\| \chi^b(\cdot, t)\right\|_{H^{m+1}} \right) \\
		& \quad + C\bigl\| z(\cdot, t)\bigr\|_{H^m}^2 \cdot \left\| \nabla (\nabla p^b, \nabla \cdot \vec{u}^a,\nabla \cdot \vec{u}^b, \chi^a,\chi^b)(\cdot, t)\right\|_{H^{m-1}}^2 \\
		& \quad + C \Bigl(\left\| g^0(\cdot, t)\right\|_{L^2} + \left\| \vec{h}(\cdot, t)\right\|_{L^2} \Bigr) \cdot \bigl\| z(\cdot, t)\bigr\|_{L^2} \cdot \left\|(\nabla p^b, \nabla \cdot \vec{u}^a, \nabla \cdot \vec{u}^b, \chi^a, \chi^b)(\cdot, t)\right\|_{L^{\infty}}\\
		& \quad + C \Bigl(\left\| g^0(\cdot, t)\right\|_{L^2}^2 + \left\| \vec{h}(\cdot, t)\right\|_{L^2}^2 \Bigr) \Bigl( \left\|\nabla (p^a,p^b, \vec{u}^a, \vec{u}^b,\chi^b)(\cdot, t)\right\|_{L^2} + \left\| (\chi^a,\chi^b)(\cdot, t)\right\|_{L^2} \Bigr),
	\end{aligned}
\end{equation}
where the estimate for $z(\vec{x}, t)$ goes as follows directly from \eqref{kz}:
\begin{equation} \label{enz}
	\begin{aligned}
		\frac{1}{2} \frac{d}{dt} \bigl\| z(\cdot, t)\bigr\|_{H^m}^2 & \le  C\bigl\| z(\cdot, t)\bigr\|_{H^m} \Bigl[ \bigl|\vec{u}_b\bigr| \cdot\left\| \nabla s^b(\cdot, t)\right\|_{H^m} + \bigl|\vec{\theta}_b - \vec{\theta}_b \bigr| \cdot \left\| \chi^a(\cdot, t)\right\|_{L^\infty} \cdot \left\|\chi^a(\cdot, t)\right\|_{H^m} \\
		& \quad + \bigl\| z(\cdot, t)\bigr\|_{H^m} \left( \left\|\nabla^2 \vec{u}^b(\cdot, t)\right\|_{H^{m-2}}+\left\|\nabla \chi^a(\cdot, t)\right\|_{H^{m-1}} \right)\\
		& \quad + \left( \left\| g^0(\cdot, t)\right\|_{H^m} + \left\| f(\cdot, t)\right\|_{H^m} \right) \! \left\| \nabla (\chi^a, \chi^b)(\cdot, t)\right\|_{H^{m-1}}\\
		 & \quad + \left\|\nabla \vec{h}(\cdot, t)\right\|_{H^{m-1}} \cdot \left\|\nabla s^a(\cdot, t)\right\|_{H^m} \Bigr].
	\end{aligned}
\end{equation}
The above estimates in \eqref{eny} and \eqref{enz}, together with the conclusions about the global existence and time decay estimates for the solutions $\partial_{\vec{x}}^k (p^{\beta}, \vec{u}^{\beta}, \chi^{\beta}, s^{\beta})(\vec{x},t)$ $(\beta =a,b)$ of the individual perturbation problems, presented in Theorem~\ref{main1}, Theorem~\ref{genes} and Lemma~\ref{decay}, result in that

\begin{equation} \label{eny1}
	\begin{aligned}
		&\frac{1}{2} \frac{d}{dt} \sum_{0 \le \left|k\right| \le m} \left( \left\|\partial_{\vec{x}}^k g^0(\cdot, t)\right\|_{L^2}^2 + \left\| \sqrt{\gamma \tilde{p}^a \tilde{\rho}^a} \partial_{\vec{x}}^k \vec{h}(\cdot, t)\right\|_{L^2}^2 + b \left\| \left( \partial_{\vec{x}}^k g^0 - a_0 \tilde{\rho}^a \partial_{\vec{x}}^k f\right)\!(\cdot, t)\right\|_{L^2}^2 \right)\\
		&\quad + \sigma_2 \frac{d}{dt} \sum_{0 \le \left|k\right| \le m-1} \left( \int_{\mathbb{R}^3} \partial_{\vec{x}}^k f \, \partial_{\vec x}^k (\nabla \cdot \vec{h}) \, d\vec{x} + \sigma_1 \int_{\mathbb{R}^3} \partial_{\vec{x}}^k \nabla g^0 \cdot \partial_{\vec{x}}^k \vec{h} \, d\vec{x} \right)\\
		& \quad + \frac{a_0}{16\alpha} \rho_a^2 \left\| f(\cdot, t)\right\|_{H^m}^2 + \frac{\sigma_2 \alpha_f}{16\alpha} p_a v_a \left\|\nabla \cdot \vec{h}(\cdot, t)\right\|_{H^{m-1}}^2 + \frac{\sigma_1 \sigma_2 v_a}{16} \left\|\nabla g^0(\cdot, t)\right\|_{H^{m-1}}^2  \\
		& \le C (\epsilon_a + \epsilon_b)  (1 + t)^{-\frac{5}{4}} \bigl|\vec{W}_b-\vec{W}_a\bigr|^2\\
		& \quad + C(\epsilon_a + \epsilon_b) \left[ (1+t)^{-\frac{7}{2}}  \bigl\| z(\cdot, t)\bigr\|_{H^m}^2 + (1+t)^{-\frac{5}{2}} \bigl\| z(\cdot, t)\bigr\|_{L^2}  \Bigl(\left\| g^0(\cdot, t)\right\|_{L^2} + \left\| \vec{h}(\cdot, t)\right\|_{L^2} \Bigr)\right] \\
		& \quad + C(\epsilon_a + \epsilon_b)(1+t)^{-\frac{5}{4}} \Bigl(\left\| g^0(\cdot, t)\right\|_{L^2}^2 + \left\| \vec{h}(\cdot, t)\right\|_{L^2}^2 \Bigr)\\
		& \le C (\epsilon_a + \epsilon_b) \left[(1 + t)^{-\frac{5}{4}} \bigl|\vec{W}_b-\vec{W}_a\bigr|^2 + (1+t)^{-\frac{7}{2}}  \bigl\| z(\cdot, t)\bigr\|_{H^m}^2 + (1+t)^{-\frac{15}{4}} \bigl\| z(\cdot, t)\bigr\|_{L^2}^2 \right] \\
		& \quad + C(\epsilon_a + \epsilon_b)(1+t)^{-\frac{5}{4}} \Bigl(\left\| g^0(\cdot, t)\right\|_{L^2}^2 + \left\| \vec{h}(\cdot, t)\right\|_{L^2}^2 \Bigr)
	\end{aligned}
\end{equation}
and
\begin{equation} \label{enz1}
	\begin{aligned}
		\frac{d}{dt} \bigl\| z(\cdot, t)\bigr\|_{H^m} & \le  C(\epsilon_a + \epsilon_b) \Bigl[ \bigl|\vec{u}_b\bigr| + (1+t)^{-\frac{5}{2}} \bigl|\vec{\theta}_b - \vec{\theta}_a \bigr| + (1+t)^{-\frac{7}{4}} \bigl\| z(\cdot, t)\bigr\|_{H^m}\\ 
		& \quad \quad \quad \quad \quad \quad \quad + (1+t)^{-\frac{7}{4}} \left( \left\| g^0(\cdot, t)\right\|_{H^m}+\left\| f(\cdot, t)\right\|_{H^m}\right) + \left\|\nabla \vec{h}(\cdot, t)\right\|_{H^{m-1}}  \Bigr].
	\end{aligned}
\end{equation}

Set $$ I(t) := \underset{0 \le \tau \le t}{\sup} \left\{(1+\tau)^{-1} \left\| z(\cdot, \tau) \right\|_{H^m}+\left\|\vec{Y}_{1}(\cdot,t)\right\|_{H^m}\right\} .$$ 
Gronwall's inequality and \eqref{eny1},  along with the initial condition $\|\vec{Y}_1(\cdot, 0)\|_{H^m} \le C_{in} |\vec{W}_b-\vec{W}_a|$ yield that
\begin{equation} \label{y1pri}
	\begin{aligned}
		&\left\|\vec{Y}_{1}(\cdot,t)\right\|_{H^m}^{2} + \int_0^t \left( \left\| \nabla g^0 (\cdot, \tau)\right\|_{H^{m-1}}^2 + \left\|\nabla \cdot \vec{h} (\cdot, \tau)\right\|_{H^{m-1}}^2 + \left\|f(\cdot, \tau)\right\|_{H^m}^2 \right) d\tau \\
		& \le C(\epsilon_a+\epsilon_b) \int_{0}^{t} e^{\int_{\tau}^{t} C(\epsilon_a+\epsilon_b)(1+t_1)^{-\frac{5}{4}} dt_1} \left( (1+\tau)^{-\frac{5}{4}} \bigl|\vec{W}_b-\vec{W}_a \bigr|^2 + (1+\tau)^{-\frac{7}{2}} \left\|z(\cdot,t)\right\|_{H^{m}}^{2} \right) d\tau \\
		& \quad + C \left\|\vec{Y}_{1}(\cdot,0)\right\|_{H^{m}}^{2}\\
		&\le C \bigl|\vec{W}_b-\vec{W}_a \bigr|^2 + C(\epsilon_a+\epsilon_b) I^2(t).
	\end{aligned}
\end{equation}
Besides, integrating \eqref{enz1} in time over $[0,t]$ and one has that
\begin{equation} \nonumber
\begin{aligned}
	\bigl\| z(\cdot, t)\bigr\|_{H^m} &\le  \bigl\| z(\cdot, 0)\bigr\|_{H^m} + C(\epsilon_a+\epsilon_b) \left[ (1+t) \bigl|\vec{u}_b\bigr|+ \bigl|\vec{\theta}_b - \vec{\theta}_a \bigr|+ (1+t)^{\frac{1}{4}}\underset{0 \le \tau \le t}{\sup} (1+\tau)^{-1} \left\| z(\cdot, \tau) \right\|_{H^m} +I(t) \right]\\
	&\quad\quad +C(\epsilon_a+\epsilon_b)(1+t)^{\frac 12}\left(\int_0^t  \left\|\nabla  \vec{h} (\cdot, \tau)\right\|_{H^{m-1}}^2 d\tau\right)^{1/2},
\end{aligned}
\end{equation}
which together with \eqref{y1pri} results in that
\begin{equation} \label{it}
	I^2(t) \le C \bigl\| z(\cdot, 0)\bigr\|^2_{H^m}  + C\bigl|\vec{W}_b-\vec{W}_a \bigr|^2 \le C  \bigl|\vec{W}_b-\vec{W}_a \bigr|^2,
\end{equation}
and thus directly implies \eqref{ensz}. Combining \eqref{y1pri} and \eqref{it}, one can verify \eqref{enghf}, and thereby completes the proof.   
\end{proof}

The enhanced energy estimate for the system of $\vec{Y}$ is analogous to \eqref{fk} and Lemma~\ref{lefk} and the aim is to overcome the difficulty caused by the insufficient decay rates of terms containing $s^\beta$, such as $v^a \nabla g^0$ in $R_{\vec h}^{\RN{2}}$ (where $v^a$ is a function of $s^a$). Introduce the following functional:
\begin{equation} \nonumber
	\begin{aligned}
		N_{m}(t, \epsilon):= & \int_{C_{\epsilon t}} (\epsilon^2 |\vec{x}|^2 + 1)^{-\frac{9}{2}} \sum_{|k|=m} \left[ \left(  \partial_{\vec{x}}^k g^0 \right)^2 + \left|\partial_{\vec{x}}^k \vec{h}\right|^2 + \left(\partial_{\vec{x}}^k f \right)^2 \right] \! (\vec{x}, t) \, d\vec{x}\\
		& + \int_{t/2}^t \int_{C_{\epsilon t}} e^{-\epsilon (t - \tau)} (\epsilon^2 |\vec{x}|^2 + 1)^{-\frac{9}{2}} \left\lbrace \sum_{|k|=1}^{m} \left[ \left(\partial_{\vec{x}}^k g^0 \right)^2 + \left| \partial_{\vec{x}}^k \vec{h} \right|^2 \right]  + \sum_{|k|=0}^{m} \left(\partial_{\vec{x}}^k f\right)^2 \right\rbrace \!  (\vec{x}, \tau) \,  d\vec{x} d\tau,
	\end{aligned}
\end{equation}
with a multi-index $k$ denoting the order of differentiation and $\left| k \right| := \sum_{i=1}^3 k_i$, and $C_{\epsilon t}:=\left[-\epsilon t, \epsilon t \right] \times \left[-\epsilon t, \epsilon t \right] \times \left[-\epsilon t, \epsilon t \right]$.
\begin{Lemma} \label{nkle}
	Under the assumptions of Theorem~\ref{main2}, for all $t \ge 0$, $2 \le m \le 5$, and sufficiently small $\epsilon > 0$, the following recursive relation holds:
	\begin{equation} \label{recursive}
		N_{m}(t, \epsilon) \le \tilde{C}_U \left( N_{m-1}(t, \epsilon) + \epsilon^{-17} (1+t)^{-3} \bigl|\vec{W}_b-\vec{W}_a\bigr|^2 \right),
	\end{equation}
	where $\tilde{C}_U$ denotes a positive constant independent of $t$, $m$ and $\epsilon$.
\end{Lemma}
\begin{proof}
	First, similar to \eqref{leen1}, compute the inner product of $\bigl(\eqref{kg0}, \eqref{kh}, \eqref{kg0} - \eqref{kfb3}\bigr)$ and $\bigl(\partial_{\vec{x}}^k g^0, b_1 \partial_{\vec{x}}^k \vec{h},$ $b (\partial_{\vec{x}}^k g^0 - b_2 \partial_{\vec{x}}^k f)\bigr)$, with $b_1, b_2$ defined in \eqref{b123}. Then multiply the resulting inner product by a weight function $e^{-\epsilon (t-\tau)} (\epsilon^2 |\vec{x}|^2 + 1)^{-9/2}$, integrate it with respect to $\vec{x}$ over the domain $C_{\epsilon t}$ and with respect to $\tau$ over the interval $[t/2, t]$, and obtain that
	\begin{equation} \label{lenken1}
		\begin{aligned}
			& \int_{C_{\epsilon t}} (\epsilon^2 |\vec{x}|^2 + 1)^{-\frac{9}{2}} \left[ \left(\partial_{\vec{x}}^k g^0\right)^2 + b_1 \left|\partial_{\vec{x}}^k \vec{h}\right|^2 + b \left( \partial_{\vec{x}}^k g^0 - b_2 \partial_{\vec{x}}^k f \right)^2 \right] \! (\vec{x},t) \, d\vec{x}\\
			& \quad + \int_{t/2}^{t} \int_{C_{\epsilon t}} e^{-\epsilon (t-\tau)} (\epsilon^2 |\vec{x}|^2 + 1)^{-\frac{9}{2}} \left| \tilde{\rho}^a \partial_{\vec{x}}^k f\right|^2 \!(\vec{x},\tau) \, d\vec{x} d\tau\\
			& \le \int_{t/2}^{t} \int_{C_{\epsilon t}} e^{-\epsilon (t-\tau)} (\epsilon^2 |\vec{x}|^2 + 1)^{-\frac{9}{2}} b_1 \partial_{\vec{x}}^k \vec{h} \cdot \left( \partial_{\vec{x}}^k R_{\vec h}^{\RN{1}} + \partial_{\vec{x}}^k R_{\vec h}^{\RN{2}}+ \vec{u}^a \cdot \partial_{\vec{x}}^k \nabla \vec{h} + v^a \partial_{\vec{x}}^k \nabla g^0 \right) \! (\vec{x},\tau) \, d\vec{x} d\tau \\
			& \quad + b \int_{t/2}^{t} \int_{C_{\epsilon t}} e^{-\epsilon (t-\tau)} (\epsilon^2 |\vec{x}|^2 + 1)^{-\frac{9}{2}} \Bigl( (1+b^{-1}) \partial_{\vec{x}}^k g^0 - b_2 \partial_{\vec{x}}^k f \Bigr) \Bigl( \partial_{\vec{x}}^kR_{g^0}^{\RN{1}}+ \partial_{\vec{x}}^k R_{g^0}^{\RN{2}}+ \vec{u}^b \cdot \partial_{\vec{x}}^k \nabla g^0 \\
			& \quad \quad \quad \quad \quad \quad \quad + \gamma p^a \partial_{\vec x}^k (\nabla \cdot \vec{h}) \Bigr)(\vec{x},\tau) \, d\vec{x} d\tau\\
			& \quad - b \int_{t/2}^{t} \int_{C_{\epsilon t}} e^{-\epsilon (t-\tau)} (\epsilon^2 |\vec{x}|^2 + 1)^{-\frac{9}{2}} b_2 \Bigl( \partial_{\vec{x}}^k g^0 - b_2 \partial_{\vec{x}}^k f \Bigr) \Bigl( \partial_{\vec{x}}^k R_{f}^{\RN{1}} + \partial_{\vec{x}}^k R_{f}^{\RN{2}} + \vec{u}^a \cdot \partial_{\vec{x}}^k \nabla f \\
			& \quad \quad \quad \quad \quad \quad \quad + \frac{\alpha_f}{\alpha} \bigl(p^b \tilde{v}^b + p_a v^a\bigr) \partial_{\vec x}^k (\nabla \cdot \vec{h}) \Bigr)(\vec{x},\tau) \,  d\vec{x} d\tau\\
			& \quad + C(\epsilon + \epsilon_a + \epsilon_b) \int_{t/2}^{t} \int_{C_{\epsilon t}} e^{-\epsilon (t-\tau)} (\epsilon^2 |\vec{x}|^2 + 1)^{-\frac{9}{2}} \left( \left|\partial_{\vec{x}}^k g^0\right|^2 + \left|\partial_{\vec{x}}^k \vec{h}\right|^2 \right) \! (\vec{x},\tau) \, d\vec{x} d\tau\\
			& \quad + C(\epsilon_a + \epsilon_b)\epsilon^{-1} (1+t)^{-5} \bigl|\vec{W}_b-\vec{W}_a\bigr|^2 + C \epsilon^{-3} (\epsilon^4 t^2 + 1)^{-\frac{7}{2}} \bigl|\vec{W}_b-\vec{W}_a\bigr|^2.
		\end{aligned}
	\end{equation}
In the above inequality \eqref{lenken1}, one invokes the bound 
$$\left\| \partial_{\vec{x}}^k \nabla g^0(\cdot,t) \right\|_{L^{\infty}}, \, \left\|\partial_{\vec x}^k (\nabla \cdot \vec{h})(\cdot,t) \right\|_{L^{\infty}} \le C (\epsilon_a + \epsilon_b) (1+t)^{-5/2} \quad \text{for all } |k| \ge 1;$$ 
this bound is a direct consequence of the ansatz \eqref{anas} and the decay estimate \eqref{decaygen}. Specifically, this bound is applied as follows: 
$$\begin{aligned}
		I_w(t) & \equiv \left|\int_{\epsilon t}(\epsilon^2 |\vec{x}|^2 + 1)^{-\frac{9}{2}}\left(\gamma b \tilde{\rho}^a \bigl( p^a \tilde{v}^a - p^b \tilde{v}^b \bigr) \! \left( \partial_{\vec{x}}^k g^0 - b_2 \partial_{\vec{x}}^k f \right)\partial_{\vec x}^k (\nabla \cdot \vec{h})-b b_2 \partial_{\vec{x}}^k f \bigl( \vec{h} \cdot \partial_{\vec{x}}^k \nabla g^0 \bigr)\right) \! (\vec{x},t)d\vec x\right|\\
&\le C (\epsilon_a + \epsilon_b)(1+t)^{-\frac{5}{2}} \int_{\epsilon t}(\epsilon^2 |\vec{x}|^2 + 1)^{-\frac{9}{2}}\left(|g^0|+|\vec{h}| + |p^b| \bigl( \bigl|\vec{\theta}_b-\vec{\theta}_a\bigr| + |f| + |z| \bigr)\right)\\
& \quad \quad \quad \quad \quad \quad \quad \quad \quad \quad \quad \times \Bigl(\left|\partial_{\vec{x}}^k g^0\right| + \left|\tilde{\rho}^a \partial_{\vec{x}}^k f \right| \Bigr)d\vec x,
\end{aligned}$$ 
which combined with Theorem~\ref{enle} results in that
$$\begin{aligned}
		\int_{t/2}^t e^{-\epsilon(t-\tau)}I_w(\tau)d\tau & \le C\epsilon^{-1} (\epsilon_a + \epsilon_b)(1+t)^{-5} \bigl|\vec{W}_b-\vec{W}_a\bigr|^2\\
		& \quad + C(\epsilon_a + \epsilon_b) \left( \int_{t/2}^{t} \int_{C_{\epsilon t}} e^{-\epsilon (t-\tau)} (\epsilon^2 |\vec{x}|^2 + 1)^{-\frac{9}{2}} \left( \left|\partial_{\vec{x}}^k g^0\right|^2 + \left|\tilde{\rho}^a \partial_{\vec{x}}^k f\right|^2 \right) \! (\vec{x},\tau) \, d\vec{x} d\tau\right)\! .\end{aligned}$$ 
Furthermore, the source terms and coupling terms could be estimated as follows:
\begin{equation}\nonumber
	\begin{aligned}
		& \left| \partial_{\vec x}^k R_{g^0}^{\RN{1}} \right|+ \left| \partial_{\vec x}^k R_{\vec h}^{\RN{1}} \right| + \left| \partial_{\vec x}^k R_f^{\RN{1}} \right| \le C \bigl|\vec{W}_b-\vec{W}_a\bigr| \!  \left( \left| \partial_{\vec x}^k \nabla(p^b,\vec{u}^b,\chi^b) \right| + \left| \partial_{\vec x}^k \chi^b \right| \right) \le C \epsilon_b (1+t)^{-\frac{5}{2}} \bigl|\vec{W}_b-\vec{W}_a\bigr|,
	\end{aligned}
\end{equation}
\begin{equation}\nonumber
	\begin{aligned}
	& \left|\partial_{\vec x}^kR_{g^0}^{\RN{2}}+ \vec{u}^b \cdot \partial_{\vec x}^k \nabla g^0 + \gamma p^a \partial_{\vec x}^k (\nabla \cdot \vec{h}) \right| + \left|\partial_{\vec x}^k R_{\vec h}^{\RN{2}} + \vec{u}^a \cdot \partial_{\vec x}^k \nabla \vec{h} + v^a \partial_{\vec x}^k \nabla g^0  \right|\\
	& \quad + \left|\partial_{\vec x}^k R_f^{\RN{2}}  + \vec{u}^a \cdot \partial_{\vec x}^k \nabla f + \frac{\alpha_f}{\alpha} \bigl(p^b \tilde{v}^b + p_a v^a\bigr) \partial_{\vec x}^k (\nabla \cdot \vec{h}) \right|\\
	& \le C \Biggl( \sum_{|j|=1}^{|k|} \left|\partial_{\vec x}^j g^0\right| \cdot \left\|\nabla v^a \right\|_{W^{|k|-1,\infty}} +  \sum_{|j|=0}^{|k|-1} \left|\partial_{\vec x}^j (\nabla \cdot \vec{h})\right| \cdot \left\|\nabla (p^b \tilde{v}^b + p_a v^a)\right\|_{W^{|k|,\infty}}\\
	& \quad + \sum_{|j|=0}^{|k|} \left| \partial_{\vec x}^j f \right| \cdot \left\| \frac{v^a}{\tilde{v}^b} \right\|_{W^{|k|,\infty}}\Biggr) + C\Biggl( \bigl|\vec{\theta}_b-\vec{\theta}_a\bigr| + \sum_{|j|=0}^{|k|} \left| \partial_{\vec x}^j \vec{Y} \right| \Biggr) \Biggl( \left\| \nabla (p^a, p^b, \vec{u}^a, \vec{u}^b, \chi^b) \right\|_{W^{|k|,\infty}}\\
	& \quad + \left\|\tilde{v}^b \nabla \cdot \vec{u}^a \right\|_{W^{|k|,\infty}} + \left\|p^a \nabla \cdot \vec{u}^a \right\|_{W^{|k|,\infty}} + \left\| \frac{\chi^b}{\tilde{v}^b} \right\|_{W^{|k|,\infty}} + \left\| \frac{v^a \chi^a}{\tilde{v}^a \tilde{v}^b} \right\|_{W^{|k|,\infty}} \Biggr)\\
	& \le C (\epsilon_a + \epsilon_b) \Biggl(|f| + \sum_{|j|=1}^{|k|}  \left| \partial_{\vec x}^j \vec{Y}_1 \right| \Biggr) + C(\epsilon_a + \epsilon_b)(1+t)^{-\frac{5}{2}} \Biggl( \bigl|\vec{\theta}_b-\vec{\theta}_a\bigr| + |g^0| + |\vec{h}| + \sum_{|j|=0}^{|k|} \left|\partial_{\vec x}^j z \right| \Biggr),
	\end{aligned}
\end{equation}
which combined with \eqref{lenken1} yield that
\begin{equation} \label{lenkg}
		\begin{aligned}
			& \int_{C_{\epsilon t}} (\epsilon^2 |\vec{x}|^2 + 1)^{-\frac{9}{2}} \left[ \left(\partial_{\vec x}^k g^0\right)^2 + b_1 \left|\partial_{\vec x}^k \vec{h}\right|^2 + b \left( \partial_{\vec x}^k g^0 -b_2 \partial_{\vec x}^k f \right)^2 \right] \! (\vec{x},t) \, d\vec{x}\\
			& \quad + \int_{t/2}^{t} \int_{C_{\epsilon t}} e^{-\epsilon (t-\tau)} (\epsilon^2 |\vec{x}|^2 + 1)^{-\frac{9}{2}} \left| \tilde{\rho}^a \partial_{\vec x}^k f\right|^2 \! (\vec{x},\tau) \,  d\vec{x} d\tau\\
			& \le C(\epsilon + \epsilon_a + \epsilon_b) \! \int_{t/2}^{t} \int_{C_{\epsilon t}} e^{-\epsilon (t-\tau)} (\epsilon^2 |\vec{x}|^2 + 1)^{-\frac{9}{2}} \left\lbrace \sum_{|j|=1}^{|k|} \left[ \left(\partial_{\vec{x}}^j g^0\right)^2 + \left| \partial_{\vec{x}}^j \vec{h} \right|^2 \right]  + \sum_{|j|=0}^{|k|} \left(\partial_{\vec{x}}^j f\right)^2 \right\rbrace \! (\vec{x}, \tau) \, d\vec{x} d\tau\\
			& \quad + C(\epsilon_a + \epsilon_b)\epsilon^{-4} (1+t)^{-3} \bigl|\vec{W}_b-\vec{W}_a\bigr|^2 + C \epsilon^{-3} (\epsilon^4 t^2 + 1)^{-\frac{7}{2}} \bigl|\vec{W}_b-\vec{W}_a\bigr|^2.
		\end{aligned} 
	\end{equation}
Here, one uses the estimate
$$\begin{aligned}
		& \int_{t/2}^{t} \int_{C_{\epsilon t}} e^{-\epsilon (t-\tau)} (\epsilon^2 |\vec{x}|^2 + 1)^{-\frac{9}{2}} (1+\tau)^{-5} \Biggl( \bigl|\vec{\theta}_b-\vec{\theta}_a\bigr|^2 + |g^0|^2 + |\vec{h}|^2 +\sum_{|j|=0}^{|k|} \left(\partial_{\vec{x}}^j z\right)^2 \Biggr) (\vec{x},\tau) \, d\vec{x} d\tau\\
		& \le C \epsilon^{-4} (1+t)^{-3} \bigl|\vec{W}_b-\vec{W}_a\bigr|^2
	\end{aligned}
$$
obtained from Theorem~\ref{enle}.
	
Next, begin with the equations for $\partial_{\vec{x}}^{k-1} f \, \partial_{\vec x}^{k-1} (\nabla \cdot \vec{h})$ and 
	$$\begin{aligned}
		\frac{\alpha_f}{\alpha} p_a v_a \left| \partial_{\vec x}^{k-1} (\nabla \cdot \vec{h})\right|^2 & = -\partial_t \left(\partial_{\vec{x}}^{k-1} f \, \partial_{\vec x}^{k-1} (\nabla \cdot \vec{h})\right) + \left( \partial_{\vec{x}}^{k-1} f \right) \nabla \cdot \left(- v_a \nabla \partial_{\vec{x}}^{k-1} g^0 + \partial_{\vec x}^{k-1} R_{\vec h}^{\RN{1}} + \partial_{\vec x}^{k-1} R_{\vec h}^{\RN{2}} \right)\\
		&\quad + \partial_{\vec x}^{k-1} (\nabla \cdot \vec{h}) \left(-\left(1+\frac{\alpha_f}{\alpha} \right)\partial_{\vec{x}}^{k-1} f + \partial_{\vec x}^{k-1} R_{f}^{\RN{1}}+ \partial_{\vec x}^{k-1} R_{f}^{\RN{2}} \right),
	\end{aligned}$$
	$$
	\begin{aligned}
		v_a \left| \partial_{\vec{x}}^{k-1} \nabla g^0 \right|^2 & = -\partial_t \left(\partial_{\vec{x}}^{k-1} \nabla g^0 \cdot \partial_{\vec{x}}^{k-1} \vec{h} \right)+ \partial_{\vec{x}}^{k-1} \nabla g^0 \cdot \left(\partial_{\vec x}^{k-1}  R_{\vec h}^{\RN{1}} + \partial_{\vec x}^{k-1}  R_{\vec h}^{\RN{2}} \right) \\
		& \quad + \partial_{\vec{x}}^{k-1} \vec{h} \cdot \nabla \left( - \gamma p_a \partial_{\vec x}^{k-1} (\nabla \cdot \vec{h})- \frac{2}{\alpha v_a} \partial_{\vec{x}}^{k-1} f + \partial_{\vec x}^{k-1}  R_{g^0}^{\RN{1}} + \partial_{\vec x}^{k-1}  R_{g^0}^{\RN{2}} \right). 
	\end{aligned}
$$
	Multiplying both sides of these identities by the weight function $e^{-\epsilon (t-\tau)} (\epsilon^2 |\vec{x}|^2 + 1)^{-9/2}$ and integrating the results over $\vec{x} \in C_{\epsilon t}$ and $\tau \in [t/2, t]$, one obtains that
\begin{equation} \label{lenkhg0}
	\begin{aligned}
		& \int_{t/2}^t \int_{C_{\epsilon t}} e^{-\epsilon (t-\tau)} (\epsilon^2 |\vec{x}|^2 + 1)^{-\frac{9}{2}} \left( \frac{\alpha_f}{4\alpha} p_a v_a  \left|\partial_{\vec{x}}^{k-1} \nabla \vec{h} \right|^2 + \frac{\sigma_3 v_a}{4} \left| \partial_{\vec{x}}^{k-1} \nabla g^0 \right|^2 \right) \! (\vec{x},\tau) \, d\vec{x} d\tau\\
		& \le C \int_{C_{\epsilon t}} (\epsilon^2 |\vec{x}|^2 + 1)^{-\frac{9}{2}} \left( \left|\partial_{\vec{x}}^{k-1} \vec{h}\right|^2 + \left|\partial_{\vec{x}}^{k-1} f\right|^2 + \left| \partial_{\vec{x}}^{k-1} \nabla g^0\right|^2 + \left|\partial_{\vec x}^{k-1} (\nabla \cdot \vec{h}) \right|^2 \right) \! (\vec{x},t) \, d\vec{x} \\
		& \quad + C(\epsilon + \epsilon_a + \epsilon_b) \! \int_{t/2}^{t} \int_{C_{\epsilon t}} e^{-\epsilon (t-\tau)} (\epsilon^2 |\vec{x}|^2 + 1)^{-\frac{9}{2}} \Bigg\lbrace \sum_{|j|=1}^{|k|-1} \left[ \left(\partial_{\vec{x}}^j g^0\right)^2 + \left| \partial_{\vec{x}}^j \vec{h} \right|^2 \right]  + \sum_{|j|=0}^{|k|-1} \left(\partial_{\vec{x}}^j f\right)^2 \! \Bigg\rbrace (\vec{x}, \tau) \, d\vec{x} d\tau\\
		& \quad + C\int_{t/2}^{t} \int_{C_{\epsilon t}} e^{-\epsilon (t-\tau)} (\epsilon^2 |\vec{x}|^2 + 1)^{-\frac{9}{2}} \left( \left|\partial_{\vec{x}}^{k-1} f\right|^2 + \left|\partial_{\vec{x}}^{k-1} \nabla f\right|^2 \right) \! (\vec{x},\tau) \, d\vec{x} d\tau\\
		& \quad + C(\epsilon_a + \epsilon_b)\epsilon^{-4} (1+t)^{-3} \bigl|\vec{W}_b-\vec{W}_a\bigr|^2 + C \epsilon^{-3} (\epsilon^4 t^2 + 1)^{-\frac{7}{2}} \bigl|\vec{W}_b-\vec{W}_a\bigr|^2,
	\end{aligned} 
\end{equation}
for a chosen sufficiently small constant $\sigma_3 \in (0,1)$. Here, the replacement of $\partial_{\vec x}^{k-1} (\nabla \cdot \vec{h})$ by $\partial_{\vec{x}}^{k-1} \nabla \vec{h}$ on the left-hand side of inequality \eqref{lenkhg0} is consistent with the process in Remark~\ref{transform}.

Choose a sufficiently small constant $\sigma_4 \in (0,1)$ and $\sum_{|k|=m}\eqref{lenkg}+\sigma_4 \eqref{lenkhg0}$ yields that
$$\label{leeng}
	\begin{aligned}
		& \sum_{|k|=m} \int_{C_{\epsilon t}} (\epsilon^2 |\vec{x}|^2 + 1)^{-\frac{9}{2}} \left[ \left(\partial_{\vec x}^k g^0\right)^2 +b_1 \left|\partial_{\vec x}^k \vec{h}\right|^2 + b \left( \partial_{\vec x}^k g^0-b_2 \partial_{\vec x}^k f \right)^2 \right] \! (\vec{x},t) \, d\vec{x}\\
		& \quad + \sum_{|k|=m} \int_{t/2}^{t} \int_{C_{\epsilon t}} e^{-\epsilon (t-\tau)} (\epsilon^2 |\vec{x}|^2 + 1)^{-\frac{9}{2}} \Bigl( \frac{\sigma_3 \sigma_4 v_a}{4} \left|\partial_{\vec{x}}^{k-1} \nabla g^0 \right|^2 + \frac{\sigma_4 \alpha_f}{4\alpha} p_a v_a  \left|\partial_{\vec{x}}^{k-1} \nabla \vec{h} \right|^2\\
		& \quad \quad \quad \quad \quad \quad \quad \quad  + \left|\tilde{\rho}^a \partial_{\vec x}^k f\right|^2 \Bigr)(\vec{x},\tau) \, d\vec{x} d\tau \\
		& \le C (\sigma_4 + \epsilon + \epsilon_a + \epsilon_b) N_{m-1}(t, \epsilon) + C(\epsilon_a + \epsilon_b)\epsilon^{-4} (1+t)^{-3} \bigl|\vec{W}_b-\vec{W}_a\bigr|^2 + C \epsilon^{-3} (\epsilon^{-4} t^2 + 1)^{-\frac{7}{2}} \bigl|\vec{W}_b-\vec{W}_a\bigr|^2,
	\end{aligned}
$$
and thus 
$$N_{m}(t, \epsilon) \le C N_{m-1}(t, \epsilon) + C(\epsilon_a + \epsilon_b)\epsilon^{-4} (1+t)^{-3} \bigl|\vec{W}_b-\vec{W}_a\bigr|^2 + C \epsilon^{-3} (\epsilon^4 t^2 + 1)^{-\frac{7}{2}} \bigl|\vec{W}_b-\vec{W}_a\bigr|^2,$$
which immediately verifies \eqref{recursive} and completes the proof of Lemma~\ref{nkle}.
\end{proof}
\begin{Remark}
Once it holds that for a positive constant $C$ and all $t\ge 0$,
\begin{equation}\nonumber
N_{1}(t, \epsilon) <C \epsilon^{-17} (1 + t)^{-3} |\vec{W}_b - \vec{W}_a|^2,
\end{equation}
one could combine it with Lemma~\ref{nkle} to obtain that for $t\ge 0$ and $2\le m\le 5$, $N_{m}(t, \epsilon) = \mathcal{O}(1) \epsilon^{-17} (1 + t)^{-3} |\vec{W}_b - \vec{W}_a|^2.$ This upper bound for $N_m(t,\epsilon)$ together with Theorem~\ref{enle} and the $L^2$-integrability of $\nabla^i s^a$ for $i=0,1,2$, after an analysis analogous to that in \eqref{phi}--\eqref{slp2}, yields that:
	\begin{equation} \label{saky1}
		\bigl|\nabla^i s^a\bigr| \left(\left|\partial_{\vec{x}}^k g^0\right| + \left|\partial_{\vec{x}}^k \vec{h}\right| + \left|\partial_{\vec{x}}^k f\right| \right) \! (\vec{x}, t) \le C_{\epsilon^{-1}} (1 + t)^{-\frac{3}{2}} \bigl| \vec{W}_b - \vec{W}_a \bigr|, \ \ \text{ for }i=0,1,2 \text{ and } 1\le |k|\le 5.
	\end{equation}
\end{Remark}

\subsection{Comparison in Nonlinear Level and $L^p$-Estimates}

\subsubsection{$L^p$-Estimates for $\vec{Y}_1$}

After basic energy estimates in Theorem~\ref{enle} and enhanced estimates established by Lemma~\ref{nkle}, we are prepared for the nonlinear propagation. Based on initial propagation \eqref{ini0}--\eqref{iniz}, one poses the following ansatz assumptions on the difference solutions $(g^0,\vec h,f,z)$:

\textbf{Ansatz Assumption.} 
There exists a constant $\mathcal{D}>C_{in}>0$ such that for all $\vec{x}\in\mathbb{R}^3$ and $t\ge0$, the irrotational solution $(g^0,\vec h,f,z)=(\vec{Y}_1,z)$ of \eqref{gbp} satisfies that
\begin{align}
	&  \left\|g^0(\cdot , t) \right\|_{L^{\infty}},  \left\|\vec{h}(\cdot , t) \right\|_{L^{\infty}} \le 4 \mathcal{D} (1 + t)^{-1} \bigl| \vec{W}_b - \vec{W}_a \bigr|, \label{anasg0} \\
	& \left\| f (\cdot , t) \right\|_{L^{\infty}}, 	\left\|\nabla \vec{Y}_1(\cdot , t) \right\|_{L^{\infty}} \le 4 \mathcal{D} (1 + t)^{-\frac{3}{2}} \bigl| \vec{W}_b - \vec{W}_a \bigr|. \label{anasf0} 
\end{align}

Under the ansatz assumptions \eqref{anasg0}--\eqref{anasf0}, combining the energy estimates \eqref{enghf}--\eqref{ensz} and Remark~\ref{rebetalq}, one could establish the following $L^2$ estimates for the source terms and coupling terms: 
\begin{equation} \label{stl2}
	\begin{aligned}
		 & \left\| R_{g_0}^{\RN{1}}(\cdot , t) \right\|_{L^2}, \left\| R_{\vec h}^{\RN{1}}(\cdot , t) \right\|_{L^2} \le C \epsilon_b (1+t)^{-\frac{5}{4}} \bigl|\vec{W}_b - \vec{W}_a\bigr|, \\
		  & \left\| R_{f}^{\RN{1}}(\cdot , t) \right\|_{L^2} \le C \epsilon_b \left( (1 + t)^{-\frac{7}{4}} \bigl|\vec{u}_b - \vec{u}_a\bigr|+ (1 + t)^{-\frac{5}{4}} \bigl|\vec{\theta}_b - \vec{\theta}_a\bigr| \right).
		 	\end{aligned}
\end{equation}
Furthermore, denote $\vec{R} := \bigl(R_{g^0}, R_{\vec{h}}, R_{f}\bigr) = \vec{R}_1 + \vec{R}_2 = \vec{R}_{1,N} + \vec{R}_{2,N}$, with $\vec{R}_1 := \bigl(R_{g_0}^{\RN{1}}, R_{\vec h}^{\RN{1}}, R_{f}^{\RN{1}}\bigr)$, 
$\vec{R}_2 := \bigl(R_{g_0}^{\RN{2}}, R_{\vec h}^{\RN{2}}, R_{f}^{\RN{2}}\bigr)$, 
and
$$\vec{R}_{1,N} := \left(R_{g^0}^{\RN{1}} - \frac{2}{\alpha} \left( \frac{1}{v_a} - \frac{1}{v_b} \right) \chi^b, R_{\vec h}^{\RN{1}}, R_{f}^{\RN{1}}\right), \quad \vec{R}_{2,N} := \left(R_{g^0}^{\RN{2}}  + \frac{2}{\alpha} \left( \frac{1}{v_a} - \frac{1}{v_b} \right) \chi^b, R_{\vec h}^{\RN{2}}, R_{f}^{\RN{2}}\right),$$ 
$${\vec{R}_{*} := \left( \frac{2(v^b - v^a)}{\alpha v_b \tilde{v}^b}\chi^b, -(v^b - v^a) \nabla p^b, -\frac{\alpha_f}{\alpha} p_b(v^b - v^a) \nabla \cdot \vec{u}^b \right)},$$ 
and thus it holds that
\begin{equation} \label{stl1}
	\left\|\vec{R}_{1,N}(\cdot , t) \right\|_{L^1} \le C \epsilon_b \bigl| \vec{W}_b - \vec{W}_a \bigr|, \quad \left\|\vec{R}_{2,N}(\cdot , t) \right\|_{L^1} \le C(\epsilon_a + \epsilon_b)\left( (1 + t)^{-\frac{1}{4}} \bigl|\vec{u}_b - \vec{u}_a\bigr|+ \bigl|\vec{\theta}_b - \vec{\theta}_a\bigr| \right),
\end{equation}
\begin{equation} \label{stl20}
	\begin{aligned}
				  & \left\| (\vec{R}_2-\vec{R}_{*} ) (\cdot , t) \right\|_{L^2} \le C (\epsilon_a + \epsilon_b) (1 + t)^{-\frac{3}{2}} \bigl|\vec{W}_b - \vec{W}_a\bigr|,\\
		 & \left\|\vec{R}_{*}(\cdot , t) \right\|_{L^2} \le C \bigl| \vec{\theta}_b - \vec{\theta}_a \bigr| \cdot \left\|(\nabla p^b,\nabla \cdot \vec{u}^b,\chi^b)(\cdot , t) \right\|_{L^2} + C \left\|( g^0, f, z)(\cdot , t) \right\|_{L^2} \cdot  \left\|(\nabla p^b,\nabla \cdot \vec{u}^b,\chi^b)(\cdot , t) \right\|_{L^\infty}\\
		 & \le C \epsilon_b \left( (1 + t)^{-\frac{3}{2}} \bigl|\vec{u}_b - \vec{u}_a\bigr|+ (1 + t)^{-\frac{5}{4}} \bigl|\vec{\theta}_b - \vec{\theta}_a\bigr| \right).
	\end{aligned}
\end{equation}
The above rearrangements about  $\vec{R}$ will be used later in different settings for nonlinear wave propagations.

First, one establishes the $L^2$ estimate for $\vec{Y}_1(\vec x,t)=(g^0,\vec h,f)(\vec x,t):=\bigl(\vec{Y}_{1,1}, \vec{Y}_{1,2}, \vec{Y}_{1,3}\bigr)(\vec x,t)$:
\begin{equation} \label{yl2}
	\begin{aligned}
		\left\|\vec{Y}_{1,i}(\cdot, t) \right\|_{L^2} & \le \left\|\mathbb{G}_i (\cdot, t) \underset{\vec{x}}{\ast} \vec{Y}_1(\cdot, 0)\right\|_{L^2} + \int_0^t  \left\|\mathbb{G}_{L,i}^{t-\tau} \vec{R}(\cdot, \tau) \right\|_{L^2} d\tau + \int_0^t  \left\|\mathbb{G}_{S,i}^{t-\tau} \vec{R}(\cdot, \tau) \right\|_{L^2} d\tau \\
		& \le C_{in} (1 + t)^{-\frac{3}{4}-d_i} \bigl|\vec{W}_b - \vec{W}_a\bigr| +  C(\epsilon_a + \epsilon_b)(1 + t)^{-d_i} \Bigl( (1 + t)^{-\frac{1}{4}} \left|\vec{u}_b - \vec{u}_a\right|\\
		& \quad + (1 + t)^{-\frac{3}{20}} \bigl|\vec{\theta}_b -\vec{\theta}_a\bigr| \Bigr).
	\end{aligned}
\end{equation}
Here, the initial propagation $\bigl\|\mathbb{G}_i (\cdot, t) \ast \vec{Y}_1(\cdot, 0)\bigr\|_{L^2} $ has been obtained in \eqref{ini2} and the nonlinear propagations $ \displaystyle\int_0^t  \bigl\|\mathbb{G}_{L,i}^{t-\tau} \vec{R}(\cdot, \tau) \bigr\|_{L^2} d\tau$, $ \displaystyle\int_0^t  \bigl\|\mathbb{G}_{S,i}^{t-\tau} \vec{R}(\cdot, \tau) \bigr\|_{L^2} d\tau$ are estimated as follows: From Remark~\ref{rebetalq}, Theorem~\ref{lpg}, and Lemma~\ref{impr}, one has that
\begin{equation} \label{y2lr1} 
	\begin{aligned}
		\int_0^t \left\|\mathbb{G}_{L,i}^{t-\tau} \vec{R}_{1,N}(\cdot , \tau) \right\|_{L^2} d\tau & = \int_0^{t/2} \Biggl\| -\vec{u}_b \cdot \nabla \mathbb{G}_{L,i}(\vec{x},t-\tau) \underset{\vec{x}}{\ast}
		\begin{pmatrix}
			p^b\\
			\vec{u}^b\\
			\chi^b
		\end{pmatrix}(\vec{x},\tau) + \nabla \cdot \left( \mathbb{G}_{L} \right)_{i2} (\vec{x},t-\tau) \underset{\vec{x}}{\ast} p^b(\vec{x},\tau)\\
		& \quad \quad \quad \quad \quad \cdot (v_a-v_b) + \vec{u}^{b,t} (\vec{x},\tau) \underset{\vec{x}}{\ast} \nabla \mathbb{G}_{L,i}(\vec{x},t-\tau)
		\begin{pmatrix}
			\gamma (p_a-p_b)\\
			0_{1 \times 3}\\
			\frac{\alpha_f}{\alpha} (p_a v_a-p_b v_b) 
		\end{pmatrix} \Bigg\|_{L^2} d\tau \\
		& \quad + \int_{t/2}^t \left\|\mathbb{G}_{L,i}^{t-\tau} \vec{R}_{1,N}(\cdot , \tau) \right\|_{L^2} d\tau\\
		& \le C \int_0^{t/2} (1 + t - \tau)^{-\frac{1}{2}-d_i} \left\|(p^b,\vec{u}^b,\chi^b)(\cdot, \tau) \right\|_{L^2} d\tau \cdot \bigl|\vec{W}_b - \vec{W}_a\bigr| \\
		&\quad + C \int_{t/2}^t (1+t-\tau)^{-d_i} \left\|\vec{R}_{1,N}(\cdot, \tau) \right\|_{L^2} d\tau\\
		& \le C \epsilon_b \int_0^{t/2} (1 + t - \tau)^{-\frac{1}{2}-d_i} (1+\tau)^{-\frac{3}{4}} \, d\tau \cdot \bigl|\vec{W}_b - \vec{W}_a\bigr|\\
		& \quad + C \epsilon_b \int_{t/2}^t (1 + t - \tau)^{-d_i} (1+\tau)^{-\frac{5}{4}} \, d\tau \cdot \bigl|\vec{W}_b - \vec{W}_a\bigr|\\
		& \le C \epsilon_b (1+t)^{-\frac{1}{4}-d_i} \bigl|\vec{W}_b - \vec{W}_a\bigr|,
	\end{aligned}	
\end{equation}
\begin{equation} \label{y2l0} 
	\begin{aligned}
		\int_0^{t^{3/5}}  \left\|\mathbb{G}_{L,i}^{t-\tau} \vec{R}_{2,N}(\cdot , \tau) \right\|_{L^2} d\tau & \le C \int_0^{t^{3/5}} (1 + t - \tau)^{-\frac{3}{4}-d_i} \left\|\vec{R}_{2,N}(\cdot, \tau) \right\|_{L^1} d\tau \\
		& \le C(\epsilon_a + \epsilon_b) \int_0^{t^{3/5}} (1 + t - \tau)^{-\frac{3}{4}-d_i} \left( (1 + \tau)^{-\frac{1}{4}} \bigl|\vec{u}_b - \vec{u}_a\bigr|+ \bigl|\vec{\theta}_b - \vec{\theta}_a\bigr| \right) d\tau\\
		& \le C(\epsilon_a + \epsilon_b)(1+t-t^{\frac{3}{5}})^{-\frac{3}{4}-d_i} \left((1+t^{\frac{3}{5}})^{\frac{3}{4}} \bigl|\vec{u}_b - \vec{u}_a\bigr|+  t^{\frac{3}{5}} \bigl|\vec{\theta}_b - \vec{\theta}_a\bigr| \right)\\
		& \le C(\epsilon_a + \epsilon_b) (1+t)^{-d_i} \left( (1+t)^{-\frac{3}{10}} \bigl|\vec{u}_b - \vec{u}_a\bigr|+ (1+t)^{-\frac{3}{20}} \bigl|\vec{\theta}_b - \vec{\theta}_a\bigr| \right),
	\end{aligned}
\end{equation}
\begin{equation} \label{y2l1} 
	\begin{aligned}
		\int_{t^{3/5}}^t  \left\|\mathbb{G}_{L,i}^{t-\tau} \vec{R}_{2,N}(\cdot , \tau) \right\|_{L^2} d\tau & \le C \int_{t^{3/5}}^t (1 + t - \tau)^{-d_i} \left\|\vec{R}_{2,N}(\cdot, \tau) \right\|_{L^2} d\tau \\
		& \le C(\epsilon_a + \epsilon_b)\left( \int_{t^{3/5}}^{t/2} + \int_{t/2}^t \right) (1 + t - \tau)^{-d_i}  \Bigl( (1 + \tau)^{-\frac{3}{2}} \bigl|\vec{u}_b - \vec{u}_a\bigr|\\
		& \quad \quad \quad \quad \quad \quad \quad \quad \quad \quad \quad \quad \quad + (1 + \tau)^{-\frac{5}{4}} \bigl|\vec{\theta}_b - \vec{\theta}_a\bigr| \Bigr) \, d\tau\\
		& \le C(\epsilon_a + \epsilon_b) \left( (1+t/2)^{-d_i} (1+t^{\frac{3}{5}})^{-\frac{1}{2}} + (1+t/2)^{-d_i+1-\frac{3}{2}} \right) \bigl|\vec{u}_b - \vec{u}_a\bigr|\\
		& \quad + C(\epsilon_a + \epsilon_b) \left( (1+t/2)^{-d_i}  (1+t^{\frac{3}{5}})^{-\frac{1}{4}} + (1+t/2)^{-d_i+1-\frac{5}{4}} \right) \bigl|\vec{\theta}_b - \vec{\theta}_a\bigr| \\
		& \le C(\epsilon_a + \epsilon_b) (1 + t)^{-d_i} \left(  (1 + t)^{-\frac{3}{10}} \bigl|\vec{u}_b - \vec{u}_a\bigr|+ (1 + t)^{-\frac{3}{20}} \bigl|\vec{\theta}_b - \vec{\theta}_a\bigr| \right), 
	\end{aligned}	
\end{equation}
and 
\begin{equation} \nonumber
	\begin{aligned}
		\int_0^t  \left\|\mathbb{G}_S^{t-\tau} \vec{R}(\cdot , \tau) \right\|_{L^2} d\tau & \le C_2 \int_0^t e^{-(t-\tau)/C_2} \left\|\vec{R}(\cdot, \tau) \right\|_{L^2} d\tau \\
		& \le C(\epsilon_a + \epsilon_b) \int_0^t e^{-(t-\tau)/C_2} (1+\tau)^{-\frac{5}{4}} \, d\tau \cdot \bigl|\vec{W}_b - \vec{W}_a\bigr|\\
		& \le C(\epsilon_a + \epsilon_b)(1+t)^{-\frac{5}{4}} \bigl|\vec{W}_b - \vec{W}_a\bigr|.
	\end{aligned}	
\end{equation}
In \eqref{y2lr1}--\eqref{y2l1}, $\mathbb{G}_{L,i}$ denotes the $i$-th row of the block matrix $\mathbb{G}_L$, and $(\mathbb{G}_L)_{ij}$ denotes its $(i,j)$-th entry; the constants $d_1 = d_2 = 0$ and $d_3 = 1/2$ are defined in \eqref{di}. Moreover, in \eqref{y2l0}, $t^{3/5} < t/2$ is required, which implies $t > 2^{5/2}$, and estimates for the case $t \le 2^{5/2}$ could be obtained directly after straightforward computations. 

To estimate $\|\vec{Y}_1(\cdot, t)\|_{L^\infty}$, one first establishes the following $L^{\infty}$-estimates of the source terms and coupling terms from Remark~\ref{rebetalq} and \eqref{saky1}:
\begin{equation} \label{stl0}
	\begin{aligned}
		& \left\| \vec{R}_1(\cdot, t) \right\|_{L^\infty}, \left\| \nabla R_{g^0}^{\RN{1}}(\cdot, t) \right\|_{L^\infty},  \left\| \nabla R_{\vec h}^{\RN{1}}(\cdot, t) \right\|_{L^\infty} \le C \epsilon_b (1+t)^{-\frac{5}{2}} \bigl|\vec{W}_b - \vec{W}_a\bigr|,\\
		& \left\|\vec{R}_2(\cdot, t) \right\|_{L^\infty}, \left\| \nabla R_{g^0}^{\RN{2}}(\cdot, t) \right\|_{L^\infty}, \left\| \nabla R_{\vec h}^{\RN{2}}(\cdot, t) \right\|_{L^\infty} \le C (\epsilon_a + \epsilon_b) (1 + t)^{-\frac{3}{2}} \bigl|\vec{W}_b - \vec{W}_a\bigr|.
	\end{aligned}
\end{equation}
Then one could combine \eqref{stl0} with Theorem~\ref{lpg} and \eqref{stl2} to obtain that: 
\begin{equation} \label{y0s}
	\begin{aligned}
		\int_0^t  \left\|\mathbb{G}_S^{t-\tau} \vec{R}(\cdot , \tau) \right\|_{L^\infty} d\tau & \le C_0 \int_0^t e^{-(t-\tau)/C_0} \left( \left\|\vec{R}(\cdot, \tau) \right\|_{L^\infty} + \left\|\nabla R_{g^0} (\cdot, \tau) \right\|_{L^\infty} + \left\|\nabla R_{\vec{h}} (\cdot, \tau) \right\|_{L^\infty} \right) d\tau \\
		& \le C(\epsilon_a + \epsilon_b) \int_0^t e^{-(t-\tau)/C_0} (1 + \tau)^{-\frac{3}{2}} \, d\tau \cdot \bigl|\vec{W}_b - \vec{W}_a\bigr|\\
		& \le C(\epsilon_a + \epsilon_b) (1 + t)^{-\frac{3}{2}} \bigl|\vec{W}_b - \vec{W}_a\bigr|,
	\end{aligned}	
\end{equation}
and for $i=1,2,3$,
\begin{equation} \label{y0l1}
	\begin{aligned}
		\int_0^{t/2}  \left\|\mathbb{G}_{L,i}^{t-\tau} \vec{R}(\cdot, \tau) \right\|_{L^\infty} d\tau & \le C \int_0^{t/2} (1 + t - \tau)^{-2-d_i} \left\|\vec{R}(\cdot, \tau) \right\|_{L^1} d\tau \\
		& \le C(\epsilon_a + \epsilon_b) \int_0^{t/2} (1 + t - \tau)^{-2-d_i} \, d\tau \cdot \bigl|\vec{W}_b - \vec{W}_a\bigr|\\
		& \le C(\epsilon_a + \epsilon_b)(1+t)^{-1-d_i} \bigl|\vec{W}_b - \vec{W}_a\bigr|,
	\end{aligned}	
\end{equation}
\begin{equation} \label{y0l2}
	\begin{aligned}
		\int_{t/2}^t \left\| \mathbb{G}_{L,i}^{t-\tau} (\vec{R}-\vec{R}_{*})(\cdot, \tau) \right\|_{L^\infty} d\tau & \le C \int_{t/2}^t (1+t-\tau)^{\frac{1}{2}-d_i} \left\|\vec{R}_1(\cdot, \tau) \right\|_{L^\infty} d\tau\\
		& \quad + C \int_{t/2}^t (1 + t - \tau)^{-\frac{3}{4}-d_i} \left\|(\vec{R}_2-\vec{R}_{*} )(\cdot, \tau) \right\|_{L^2} d\tau \\
		& \le  C (\epsilon_a + \epsilon_b) \int_{t/2}^t \Bigl[(1+t-\tau)^{\frac{1}{2}-d_i} (1+\tau)^{-\frac{5}{2}} \\
		& \quad + (1 + t - \tau)^{-\frac{3}{4}-d_i} (1 + \tau)^{-\frac{3}{2}} \Bigr] \, d\tau \cdot \bigl|\vec{W}_b - \vec{W}_a\bigr|\\
		& \le C(\epsilon_a + \epsilon_b) (1 + t)^{-1-d_i} \bigl|\vec{W}_b - \vec{W}_a\bigr|,
	\end{aligned}	
\end{equation}
\begin{equation} \label{gh0lstar}
	\begin{aligned}
		\int_{t/2}^t \left\| \mathbb{G}_{L,i}^{t-\tau} \vec{R}_{*}(\cdot , \tau) \right\|_{L^\infty} d\tau & \le C \int_{t/2}^t (1 + t - \tau)^{-\frac{3}{4}-d_i} \left\|\vec{R}_{*}(\cdot , \tau) \right\|_{L^2} d\tau \\
		& \le C \epsilon_b \int_{t/2}^t (1 + t - \tau)^{-\frac{3}{4}-d_i} \left( (1 + \tau)^{-\frac{3}{2}} \bigl|\vec{u}_b - \vec{u}_a\bigr|+ (1 + \tau)^{-\frac{5}{4}} \bigl|\vec{\theta}_b - \vec{\theta}_a\bigr| \right) d\tau\\
		& \le C\epsilon_b \left( (1 + t)^{-\frac{5}{4}-\frac{d_i}2} \bigl|\vec{u}_b - \vec{u}_a\bigr|+ (1 + t)^{-1-\frac{d_i}2} \bigl|\vec{\theta}_b - \vec{\theta}_a\bigr| \right).
	\end{aligned}	
\end{equation}
Here, $d_1=d_2=0, d_3=1/2$ are constants defined in \eqref{di}. Compared with \eqref{y0l1}--\eqref{y0l2}, the decay rate in \eqref{gh0lstar} is a bit slower for $i=3$ and it could be improved based on the structure of 
$\vec{R}_{*}$:
$$\bigl|\vec{R}_{*}\bigr| \le C \bigl(|\chi^b|+|\nabla p^b|+|\nabla\cdot\vec{u}^b|\bigr) \bigl( \bigl|\vec{\theta}_b - \vec{\theta}_a\bigr| + |g^0| + |f| + |z| \bigr),$$
which combined with $L^2$ bounds for $(g^0,f,z)$ given in Theorem~\ref{enle} results in that
\begin{equation} \label{f0lstar}
	\begin{aligned}
		\int_{t/2}^t \left\|\mathbb{G}_{L,3}^{t-\tau} \vec{R}_{*}(\cdot , \tau) \right\|_{L^\infty} d\tau & = \int_{t/2}^t \left\| \mathbb{G}_{L,3}(\vec{x},t-\tau) \underset{\vec{x}}{\ast} \vec{R}_{*}(\vec{x}, \tau) \right\|_{L^\infty} d\tau \\
		&\le C \int_{t/2}^t \left\|(\nabla p^b,\nabla \cdot \vec{u}^b,\chi^b)(\cdot, \tau) \right\|_{L^\infty} d\tau \cdot \bigl|\vec{\theta}_b - \vec{\theta}_a\bigr|\\
		& \quad + C \int_{t/2}^t (1+t-\tau)^{-\frac{5}{4}} \left\|(g^0, f, z)(\cdot, \tau) \right\|_{L^2} \cdot  \left\|(\nabla p^b,\nabla \cdot \vec{u}^b,\chi^b)(\cdot, \tau) \right\|_{L^\infty} d\tau\\
		& \le C \epsilon_b \int_{t/2}^t \left( (1+\tau)^{-\frac{5}{2}} \bigl|\vec{\theta}_b - \vec{\theta}_a\bigr| + (1 + t - \tau)^{-\frac{5}{4}} (1+\tau)^{-\frac{3}{2}} \bigl|\vec{W}_b - \vec{W}_a\bigr| \right) d\tau \\
		& \le C \epsilon_b (1 + t)^{-\frac{3}{2}} \bigl|\vec{W}_b - \vec{W}_a\bigr|.
	\end{aligned}	
\end{equation}
Then combining \eqref{diffrep}, \eqref{ini0} and \eqref{y0s}--\eqref{f0lstar}, one establishes the estimate for $\|\vec{Y}_1(\cdot,t)\|_{L^\infty}$ as follows:
\begin{equation} \label{ghfl0}
	\begin{aligned}
		\left\|\vec{Y}_{1,i}(\cdot , t) \right\|_{L^\infty} & \le \left\|\mathbb{G}_i (\cdot, t)  \underset{\vec{x}}{\ast} \vec{Y}_1(\cdot, 0)\right\|_{L^\infty} + \int_0^t  \left\|\mathbb{G}_{L,i}^{t-\tau} \vec{R}(\cdot, \tau) \right\|_{L^\infty} d\tau + \int_0^t  \left\|\mathbb{G}_{S,i}^{t-\tau} \vec{R}(\cdot, \tau) \right\|_{L^\infty} d\tau \\
		& \le C_{in} \left( 1 + t \right)^{-2-d_i} \left| \vec{W}_b-\vec{W}_a \right| + C(\epsilon_a + \epsilon_b) \left( 1 + t \right)^{-1-d_i} \left| \vec{W}_b-\vec{W}_a \right|, \quad \text{for } i=1,2,3.
	\end{aligned}
\end{equation}

To verify the ansatz assumption \eqref{anasf0}, one still needs to proceed to estimate $\bigl\|\nabla \vec{Y}_1(\cdot,t)\bigr\|_{L^\infty}$. In the ansatz \eqref{anasf0}, there is no difference in decaying rates of each component of  $\nabla \vec{Y}_1$ and thus it will be treated as a whole. Theorem~\ref{lpg}, \eqref{stl1}--\eqref{stl20}, \eqref{stl0} and \eqref{f0lstar} yield that:
\begin{equation} \label{f1l0}
	\begin{aligned}
		\int_0^t \left\| \nabla \mathbb{G}_{L}^{t-\tau} \vec{R}(\cdot , \tau) \right\|_{L^\infty} d\tau & \le C \left( \int_0^{t/2} (1 + t - \tau)^{-\frac 52} \left\|\vec{R}(\cdot , \tau) \right\|_{L^1} d\tau + \int_{t/2}^t  \left\|\vec{R}_1(\cdot, \tau) \right\|_{L^\infty} d\tau\right)\\
		&\quad + C \int_{t/2}^t \left(   (1 + t - \tau)^{-\frac{5}{4}} \left\|(\vec{R}_2-\vec{R}_{*})(\cdot, \tau) \right\|_{L^2}+\left\| \nabla \mathbb{G}_{L}^{t-\tau} \vec{R}_{*} (\cdot, \tau) \right\|_{L^\infty} \right) d\tau \\
		& \le C (\epsilon_a + \epsilon_b)(1 + t)^{-\frac{3}{2}} \bigl|\vec{W}_b - \vec{W}_a\bigr| + C\int_{t/2}^t   \left\| \nabla \mathbb{G}_{L}^{t-\tau} \vec{R}_{*} (\cdot, \tau) \right\|_{L^\infty} d\tau,
	\end{aligned}
\end{equation}
where
\begin{equation} \label{0f1l0}
	\begin{aligned}
		\int_{t/2}^t   \left\| \nabla \mathbb{G}_{L}^{t-\tau} \vec{R}_{*} (\cdot, \tau) \right\|_{L^\infty} d\tau & \le C \int_{t/2}^t  \left\|(\nabla p^b,\nabla \cdot \vec{u}^b,\chi^b)(\cdot , \tau) \right\|_{L^\infty}\\
		& \quad \quad \quad \quad \cdot  
		\Bigl( \bigl|\vec{\theta}_b - \vec{\theta}_a\bigr| + (1+t-\tau)^{-\frac{5}{4}}\left\|(g^0, f, z)(\cdot , \tau) \right\|_{L^2} \Bigr) \, d\tau \\
		& \le C (\epsilon_a + \epsilon_b)(1 + t)^{-\frac{3}{2}} \bigl|\vec{W}_b - \vec{W}_a\bigr|.
	\end{aligned}
\end{equation}
In the estimate of the short-wave component, high order derivatives are involved and one could obtain the following $L^\infty$-estimates for the first and second derivatives of source terms and coupling terms:
	\begin{align}
		\left\|\nabla R_f^{\RN{1}}(\cdot, t) \right\|_{L^\infty}, \left\|\nabla^2 R_{g^0}^{\RN{1}}(\cdot, t) \right\|_{L^\infty}, \left\|\nabla^2 R_{\vec h}^{\RN{1}}(\cdot, t) \right\|_{L^\infty}  & \le C \epsilon_b (1+t)^{-\frac{5}{2}} \bigl|\vec{W}_b - \vec{W}_a\bigr|, \label{st1l0}\\
		\left\|\nabla R_f^{\RN{2}}(\cdot, t) \right\|_{L^\infty}, \left\|\nabla^2 R_{g^0}^{\RN{2}}(\cdot, t) \right\|_{L^\infty}, \left\|\nabla^2 R_{\vec h}^{\RN{2}}(\cdot, t) \right\|_{L^\infty}  & \le C(\epsilon_a + \epsilon_b) (1 + t)^{-\frac{3}{2}} \bigl|\vec{W}_b - \vec{W}_a\bigr|.
		\label{st2l0}
	\end{align}
Here, \eqref{st1l0} is a direct consequence of Remark~\ref{rebetalq} and Lemma~\ref{decay}; while Theorem~\ref{enle} together with \eqref{saky1} and the ansatz \eqref{anasf0} yields the estimate \eqref{st2l0} as follows:
\begin{equation} \nonumber
	\begin{aligned}
		\left\| \nabla^2 R_{g^0}^{\RN{2}}(\cdot , t) \right\|_{L^\infty} & \le C \left( \left\|f\right\|_{L^\infty} \cdot \left\|\nabla^2 \frac{v^a}{v_b \tilde{v}^b}\right\|_{L^\infty} + \left\|\nabla f\right\|_{L^\infty} \cdot \left\|\nabla \frac{v^a}{v_b \tilde{v}^b}\right\|_{L^\infty} + \left\|v^a \nabla^2 f \right\|_{L^\infty}\right) \\
		& +C\left( \bigl|\vec{\theta}_b-\vec{\theta}_a\bigr|+\left\|(g^0,f,z)\right\|_{W^{2,\infty}}\right)\left( \left\|\chi^a\right\|_{W^{2,\infty}} + \left\|\chi^b\right\|_{W^{2,\infty}}\right)\\
		& +C\left( \left\|g^0\right\|_{W^{3,\infty}} \cdot \left\|\vec{u}^b\right\|_{W^{3,\infty}} +  \left\|\vec{h}\right\|_{W^{3,\infty}} \cdot \left\|p^a\right\|_{W^{3,\infty}}\right)\\
		& \le C\epsilon_a \left\|f\right\|_{W^{1,\infty}} + C \left( \left\|\nabla^2 f\right\|_{L^\infty} \cdot \left\|(p^a,\chi^a)\right\|_{L^\infty} +\left\|s^a \nabla^2 f\right\|_{L^\infty}\right)\\
		&\quad +C(\epsilon_a+\epsilon_b)\left[ (1+t)^{-\frac{5}{2}}\left( \bigl|\vec{\theta}_b-\vec{\theta}_a\bigr|+\left\|(g^0,f,z)\right\|_{W^{2,\infty}}\right)+(1+t)^{-2}\left\|(g^0,\vec{h})\right\|_{W^{3,\infty}}\right]\\
		& \le C(\epsilon_a + \epsilon_b) (1 + t)^{-\frac{3}{2}} \bigl|\vec{W}_b - \vec{W}_a\bigr|.
	\end{aligned}
\end{equation}
The remaining ones $\bigl\|\nabla R_f^{\RN{2}}(\cdot, t)\bigr\|_{L^\infty}$ and $\bigl\|\nabla^2 R_{\vec h}^{\RN{2}}(\cdot, t)\bigr\|_{L^\infty}$ in \eqref{st2l0} could be obtained similarly.
Thus, by combining Theorem~\ref{lpg}, \eqref{stl0} and \eqref{st1l0}--\eqref{st2l0}, one has that:
\begin{equation} \nonumber
	\begin{aligned}
		\int_0^t \left\| \nabla \mathbb{G}_S^{t-\tau} \vec{R}(\cdot , \tau) \right\|_{L^\infty} d\tau & \le C_0 \int_0^t e^{-(t-\tau)/C_0} \left( \left\| \nabla \vec{R}(\cdot , \tau) \right\|_{L^\infty} + \left\|\nabla^2 R_{g^0} (\cdot , \tau) \right\|_{L^\infty} + \left\|\nabla^2 R_{\vec{h}} (\cdot , \tau) \right\|_{L^\infty} \right) d\tau \\
		& \le C(\epsilon_a + \epsilon_b) \int_0^t e^{-(t-\tau)/C_0} (1 + \tau)^{-\frac{3}{2}} \, d\tau \cdot \bigl|\vec{W}_b - \vec{W}_a\bigr|\\
		& \le C(\epsilon_a + \epsilon_b) (1 + t)^{-\frac{3}{2}} \bigl|\vec{W}_b - \vec{W}_a\bigr|,
	\end{aligned}	
\end{equation}
which together with \eqref{ini0} and \eqref{f1l0}--\eqref{0f1l0} results in that 
\begin{equation} \label{ghf1l0}
	\begin{aligned}
		\left\|\nabla \vec{Y}_{1,i}(\cdot , t) \right\|_{L^\infty} & \le \left\|\nabla \mathbb{G}_i (\cdot, t) \underset{\vec{x}}{\ast} \vec{Y}_1(\cdot, 0)\right\|_{L^\infty} + \int_0^t  \left\|\nabla \mathbb{G}_{L,i}^{t-\tau} \vec{R}(\cdot , \tau) \right\|_{L^\infty} d\tau + \int_0^t  \left\|\nabla \mathbb{G}_{S,i}^{t-\tau} \vec{R}(\cdot , \tau) \right\|_{L^\infty} d\tau  \\
		& \le C_{in} \left( 1 + t \right)^{-\frac{5}{2}-d_i} \left|\vec{W}_b-\vec{W}_a\right| + C(\epsilon_a + \epsilon_b) (1 + t)^{-\frac{3}{2}} \bigl|\vec{W}_b - \vec{W}_a\bigr|, \quad \text{for } i=1,2,3.
	\end{aligned}
\end{equation}

Finally, one could combine \eqref{ghfl0} and \eqref{ghf1l0} to verify the ansatz assumptions \eqref{anasg0}--\eqref{anasf0} since $\mathcal{D} \ge C_{in}$ and $C(\epsilon_a + \epsilon_b) \ll \mathcal{D}$.

\subsubsection{$L^p$-Estimates for the Entropy Difference Function $z(\vec{x},t)$}
Multiply both sides of the equation 
$$\partial_t z=-\vec u_b\cdot\nabla s^b+R_z^{\RN{2}}$$
 by the sign function $\mathrm{sgn}(z)= \begin{cases}
	-1 & \text{for } z<0,\\
	0 & \text{for } z=0,\\
	1 & \text{for } z>0,
\end{cases}$
and then integrate it over $\mathbb{R}^3$ with respect to $\vec{x}$ to obtain that
$$\begin{aligned}
	\partial_t \left\| z(\cdot,t) \right\|_{L^1} & = \int_{\mathbb{R}^3} \mathrm{sgn}(z) \left( -\vec{u}_b \cdot \nabla s^b + R_z^{\RN{2}} \right) d\vec{x}\\
	& = \int_{\mathbb{R}^3} \mathrm{sgn}(z) \left(- \vec{u}^b \cdot \nabla z + \widetilde{R_z^{\RN{2}}} \right) d\vec{x}, 
\end{aligned}$$
where $$ \widetilde{R_z^{\RN{2}}} :=-\vec{u}_b \cdot \nabla s^b - \vec{h} \cdot \nabla s^a + \frac{2}{\alpha_f R} \left[\frac{(\chi^b + \chi^a) f}{\tilde{T}_1^b \tilde{T}_2^b} + \left(\frac{\tilde{T}_1^a - \tilde{T}_1^b}{\tilde{T}_2^b \tilde{T}_1^a \tilde{T}_1^b } + \frac{ \tilde{T}_2^a - \tilde{T}_2^b}{\tilde{T}_1^a \tilde{T}_2^a \tilde{T}_2^b} \right) \! (\chi^a)^2 \right].$$
It together with the $L^1$-integrability of $\nabla s^\beta (\beta=a,b)$, Remark~\ref{rebetalq} and \eqref{anasg0} yields that
$$\begin{aligned}
	\left\| z(\cdot ,t) \right\|_{L^1} & \le \left\| z(\cdot ,0) \right\|_{L^1} + \int_0^t \int_{\mathbb{R}^3}\left| \widetilde{R_z^{\RN{2}}} \right| \, d\vec{x} d\tau + \int_0^t \int_{\mathbb{R}^3} \left|\nabla \cdot \vec{u}^b\right| |z| \, d\vec{x} d\tau\\
	&\le C_{in} |\vec{W}_b-\vec{W}_a| +C (\epsilon_a + \epsilon_b) \int_0^t \left( \bigl|\vec{u}_b - \vec{u}_a\bigr|+ (1 + \tau)^{-1} \bigl|\vec{\theta}_b - \vec{\theta}_a\bigr| \right)d\tau\\
	&\quad\quad+ C \epsilon_b\int_0^t (1+\tau)^{-5/2} \left\| z(\cdot ,\tau) \right\|_{L^1} d\tau,
\end{aligned}$$
which combined with the Gronwall's integral inequality yields that
\begin{equation} \label{zl1}
	\left\| z(\cdot,t) \right\|_{L^1} \le C_{in} \bigl| \vec{W}_b - \vec{W}_a \bigr| + C(\epsilon_a + \epsilon_b) \left( (1+t) \bigl| \vec{u}_b - \vec{u}_a \bigr|+ \log(1+t) \bigl| \vec{\theta}_b - \vec{\theta}_a \bigr| \right).
\end{equation}

The estimate of $z(\vec{x},t)$ in $L^\infty$-norm could be obtained similarly. From Remark~\ref{rebetalq}, Theorem~\ref{enle}, and \eqref{anasg0}, it holds that:
\begin{equation} \label{zl0}
	\begin{aligned}
		\left\| z(\cdot,t) \right\|_{L^\infty} & \le \left\| z(\cdot,0) \right\|_{L^\infty} + \int_0^t \left\| \left( -\vec{u}_b \cdot \nabla s^b + R_z^{\RN{2}}\right) \!(\cdot, \tau) \right\|_{L^\infty} d\tau \\
		& \le C_{in} \bigl|\vec{W}_b - \vec{W}_a\bigr| + C(\epsilon_a + \epsilon_b) \int_0^t \left( \bigl|\vec{u}_b\bigr| + (1+\tau)^{-1} \bigl|\vec{\theta}_b - \vec{\theta}_a\bigr| \right) d\tau \\
		& \le C_{in} \bigl|\vec{W}_b - \vec{W}_a\bigr| + C(\epsilon_a + \epsilon_b) \left( (1+t) \bigl|\vec{u}_b - \vec{u}_a\bigr|+ \log(1+t) \bigl|\vec{\theta}_b - \vec{\theta}_a\bigr| \right),
	\end{aligned}
\end{equation}
and we finally finish the proof of Theorem~\ref{main2} by combining \eqref{yl2}, \eqref{ghfl0}, and \eqref{ghf1l0}--\eqref{zl0}.

%%%%%%%%%%%%%%%%%%%%%%%%%%%%%%%%%%%%%%%%%%%%%%%%%%%%%%%%%%%%%%%%%%%%%%%%%%%%%%%%%5
%
%
%                       Appendix: 
%
%
%%%%%%%%%%%%%%%%%%%%%%%%%%%%%%%%%%%%%%%%%%%%%%%%%%%%%%%%%%%%%%%%%%%%%%%%%%%%%%%%%

%%%%%%%%%%%%%%%%%%%%%%%%%%%%%%%%%%%%%%%%%%%%%%%%%%%%%%%%%%%%%%%%%%%%%%%%%%%%%%%%%5
%
%
%                      Appendix: 
%
%
%%%%%%%%%%%%%%%%%%%%%%%%%%%%%%%%%%%%%%%%%%%%%%%%%%%%%%%%%%%%%%%%%%%%%%%%%%%%%%%%%

%%%%%%%%%%%%%%%%%%%%%%%%%%%%%%%%%%%%%%%%%%%%%%%%%%%%%%%%%%%%%%%%%%%%%%%%%%%%%%%%%5
%
%
%                        Appendix: 
%
%
%%%%%%%%%%%%%%%%%%%%%%%%%%%%%%%%%%%%%%%%%%%%%%%%%%%%%%%%%%%%%%%%%%%%%%%%%%%%%%%%%

\section{Appendix $\RN{1}$: Modification of the Green's Function for Compliance with the Irrotationality Condition}

In this section, we present the detailed modifications of Theorem~1.3 in \cite{z1}, and the main difference is a retreatment of the operators $(-1)^j \Delta^{-j}$ for $j \ge 1$, where $\Delta^{-1}$ denotes the inverse Laplacian.

\subsection{Existing Theoretical Results about the Green's Function in \cite{z1}}
Pointwise estimates for the Green's function have been obtained in \cite{z1}. For convenience of readers, we quote them as follows: for $|\vec{x}| \le M(1+t)$,
\begin{equation} \label{green1}
	\begin{aligned}
		\partial_{\vec{x}}^{k} \mathbb{G}(\vec{x},t) &= H(t-1) G_{k}^{*} (\vec{x},t) + \mathcal{O}(1)(1 + t)^{-\frac{5+|k|}{2}} e^{-\frac{(|\vec{x}|-\bar{c} t)^2}{C_0(1 + t)}} \\
		& \quad- H(\tilde{c}_1 (1+t)-|\vec{x}|)  \operatorname{diag}\! \left( 0, \partial_{\vec{x}}^{k} (\Delta^{-1} \nabla \nabla^t \delta(\vec{x})), 0 \right) \\
		& \quad + G_{k,1}^s(\vec{x},t) + G_{k,2}^s(\vec{x},t),
	\end{aligned}
\end{equation}
and for $|\vec{x}| > M(1+t)$,
\begin{equation} \label{green2}
	\partial_{\vec{x}}^{k} \mathbb{G}(\vec{x},t)  = \mathcal{O}(1) e^{-(|\vec{x}|+t)/C_0} + \tilde{G}_{k,1}^s (\vec{x},t),
\end{equation}
where the Heaviside step function $H(r)$ is defined in \eqref{stepfun}, $M > \bar{c}_f\equiv \sqrt{\gamma \bar{p}\bar{v}}$, and $0 < \tilde{c}_1, \tilde{c}_2 < \bar{c} \equiv \sqrt{\gamma \bar{p}\bar{v}/(1+b)}$ are positive constants (with $\gamma, b$ defined in \eqref{equ}--\eqref{bb}); the functions $G_{k}^{*}$, $G_{k,1}^s$, $\tilde{G}_{k,1}^s$, and $G_{k,2}^s$ are defined as follows:
\begin{equation} \label{Gkstar}
	\begin{aligned}
		G_{k}^{*}(\vec{x},t) &= H(|\vec{x}|-\tilde{c}_2 (1+t))  \frac{1}{2(2 \pi \bar{\alpha}_1 t)^{3/2}} \partial_{\vec{x}}^k \left[ \frac{|\vec{x}| - \bar{c} t}{t} e^{-\frac{(|\vec{x}|-\bar{c}t)^2}{2 \bar{\alpha}_1 t}} \right]\\
		& \quad \times \begin{pmatrix}
			\frac{1}{\bar{c}} & \frac{\bar{c}\bar{\rho}t}{|\vec{x}|^2} \vec{x}^t & -\frac{\bar{p}_{e_1}}{\bar{c} {\zeta}}\\
			\frac{\bar{v} t}{\bar{c}|\vec{x}|^2} \vec{x} & \frac{t}{|\vec{x}|^3} \vec{x} \vec{x}^t & -\frac{\bar{v} \bar{p}_{e_1} t}{\bar{c} {\zeta} |\vec{x}|^2} \vec{x}\\
			0 & 0_{1 \times 3} & 0
		\end{pmatrix} \\
		& \quad + H(\bar{c} t-|\vec{x}|) \operatorname{diag}\! \left(0, \mathcal{O}(1)(1+t)^{-3-|k|},0 \right)\\
		&= \mathcal{O}(1) t^{-2-\frac{|k|}{2}} e^{-\frac{(|\vec{x}|-\bar{c}t)^2}{Kt}} + H(\bar{c} t-|\vec{x}|) \operatorname{diag}\! \left(0, \mathcal{O}(1)(1+t)^{-3-|k|},0 \right), 
	\end{aligned}
\end{equation}
\begin{align}
	G_{k,1}^s (\vec{x},t) &= \partial_{\vec{x}}^k \Bigg\lbrace \delta(\vec{x}) \operatorname{diag}\! \left(0, \mathbb{I}_{3 \times 3},0 \right) + e^{-{\mu}_2 t} \sum_{j=0}^{|k|+2} L_{1,j} (t,\nabla)\delta(\vec{x}) \nonumber\\
	& \quad + e^{-{\mu}_1 t} \Biggl[ \, \sum_{j=0}^{|k|+2} L_{2,j} (t,\nabla) \omega_t(\vec{x},t;\bar{c}_f) + \sum_{j=-1}^{|k|+2} L_{3,j}(t,\nabla) \omega(\vec{x},t; \bar{c}_f) \Biggr] \label{Gk1s}\\
	&\quad + te^{-{\mu}_1 t} \Biggl[ \, \sum_{j=1}^{|k|+2} L_{4,j}(t,\nabla) \omega_t (\vec{x},t; \bar{c}_f) + \sum_{j=0}^{|k|+2} L_{5,j}(t,\nabla) \omega(\vec{x},t;\bar{c}_f) \Biggr] \Bigg\rbrace,\nonumber\\
	\tilde{G}_{k,1}^s (\vec{x},t) &= \partial_{\vec{x}}^k \Big\lbrace \delta(\vec{x}) \operatorname{diag}\! \left(0, \mathbb{I}_{3 \times 3}, 0 \right) + e^{-{\mu}_2 t} L_{1,0} \delta(\vec{x}) + e^{-{\mu}_1 t} \bigl[ L_{2,0} \omega_t(\vec{x},t;\bar{c}_f) \label{Gk1slead}\\
	& \quad + (L_{3,-1}(t,\nabla) + L_{3,0}) \omega(\vec{x},t;\bar{c}_f) \bigr] + te^{-{\mu}_1 t} L_{5,0} \omega(\vec{x},t; \bar{c}_f) \Big\rbrace, \nonumber\\
	G_{k,2}^s(\vec{x},t) &= \partial_{\vec{x}}^k \left\lbrace  {{e^{-{\mu}_1t}}} \operatorname{diag}\! \left(0, \Delta^{-1} \nabla \nabla^t, 0 \right) \! \bigl[ \omega_t(\vec{x},t; \bar{c}_f) + \left( {\mu}_1 + {\mu}_1 ({\mu}_1/2 + {\mu}_2)t \right) \! \omega(\vec{x},t; \bar{c}_f) \bigr] \right\rbrace. \label{Gk2s}
\end{align}
Here, the overbar notation denotes the thermodynamic variables evaluated at the constant equilibrium state, $K$ is an arbitrarily chosen constant satisfying $K>2\bar{\alpha}_1=4\alpha_f\bar{p}\bar{v}/(\alpha^2\zeta^2)$, and $p_{e_1}$, $\zeta$, $\mu_1$, $\mu_2$ are defined in \eqref{equ} and \eqref{mumu} respectively; $L_{1,0}$, $L_{2,0}$, $L_{3,-1}(t,\nabla)$, $L_{3,0}$, $L_{5,0}$ are defined in \eqref{leading}, while the definitions of $L_{m,j}(t,\nabla)$ for $1 \le m \le 5$ and $j \ge 1$ are as follows:
\begin{equation}\label{138}
\begin{aligned}
			& L_{1,1}(t,\nabla):= -\Delta^{-1}\begin{pmatrix}
				0&0&0\\
				q_0^{121}\nabla&0&q_0^{123}\nabla\\
				0&q_0^{132}\nabla^t&0
			\end{pmatrix}, \\
			& L_{1,2j}(t,\nabla) :=(-1)^j\Delta^{-j} \begin{pmatrix}
				q_{j-1}^{111}(t)&0&q_{j-1}^{113}(t)\\
				0&q_{j-1}^{111}(t)\Delta^{-1}\nabla\nabla^t&0\\
				q_{j-1}^{131}(t)&0&q_{j-1}^{133}(t)
			\end{pmatrix},\quad j\ge 1,\\ 
			& L_{1,2j+1}(t,\nabla) :=(-1)^{j+1}\Delta^{-j-1} \begin{pmatrix}
			0&q_{j-1}^{112}(t)\nabla^t&0\\
				q_{j}^{121}(t)\nabla&0&q_{j}^{123}(t)\nabla\\
				0&q_{j}^{132}(t)\nabla^t&0
			\end{pmatrix},\quad j\ge 1,
			\end{aligned}
\end{equation}
\begin{equation}\label{1382}
\begin{aligned}
			& L_{2,2j-1}(t,\nabla) :=(-1)^{j}\Delta^{-j} \begin{pmatrix}
			0&q_{2j-2}^{212}(t)\nabla^t&0\\
				q_{2j-2}^{221}(t)\nabla&0&q_{2j-2}^{223}(t)\nabla\\
				0&q_{2j-2}^{232}(t)\nabla^t&0
			\end{pmatrix},\quad j\ge 1,\\
			& L_{2,2j}(t,\nabla) :=(-1)^j\Delta^{-j} \begin{pmatrix}
				q_{2j}^{211}(t)&0&q_{2j-2}^{213}(t)\\
				0&q_{2j}^{211}(t)\Delta^{-1}\nabla\nabla^t&0\\
				q_{2j}^{231}(t)&0&q_{2j-2}^{233}(t)
			\end{pmatrix},\quad j\ge 1,
			\end{aligned}
\end{equation}
\begin{equation}\label{1383}
\begin{aligned}
			& L_{3,2j-1}(t,\nabla) :=(-1)^{j}\Delta^{-j} \begin{pmatrix}
			0&q_{2j}^{312}(t)\nabla^t&0\\
				q_{2j}^{321}(t)\nabla&0&q_{2j-2}^{323}(t)\nabla\\
				0&q_{2j}^{332}(t)\nabla^t&0
			\end{pmatrix},\quad j\ge 1,\\
			& L_{3,2j}(t,\nabla) :=(-1)^j\Delta^{-j} \begin{pmatrix}
				q_{2j}^{311}(t)&0&q_{2j}^{313}(t)\\
				0&q_{2j}^{311}(t)\Delta^{-1}\nabla\nabla^t&0\\
				q_{2j}^{331}(t)&0&q_{2j}^{333}(t)
			\end{pmatrix},\quad j\ge 1,
			\end{aligned}
\end{equation}
\begin{equation}\label{1384}
\begin{aligned}
			& L_{4,1}(t,\nabla):= \frac{{\mu}_1({\mu}_1+2{\mu}_2)}{2\bar{c}_f^2}\left(-\Delta^{-1}\right)\begin{pmatrix}
				0&\bar{\rho}\bar{c}_f^2\nabla^t&0\\
				\bar{v}\nabla&0&0\\
				0&\bar{a}\nabla^t&0
			\end{pmatrix}, \\
			& L_{4,2j}(t,\nabla) :=(-1)^j\Delta^{-j} \begin{pmatrix}
				q_{2j-2}^{411}(t)&0&q_{2j-2}^{413}(t)\\
				0&q_{2j-2}^{411}(t)\Delta^{-1}\nabla\nabla^t&0\\
				q_{2j-2}^{431}(t)&0&q_{2j-2}^{433}(t)
			\end{pmatrix},\quad j\ge 1,\\ 
			& L_{4,2j+1}(t,\nabla) :=(-1)^{j+1}\Delta^{-j-1} \begin{pmatrix}
			0&q_{2j}^{412}(t)\nabla^t&0\\
				q_{2j}^{421}(t)\nabla&0&q_{2j-2}^{423}(t)\nabla\\
				0&q_{2j}^{432}(t)\nabla^t&0
			\end{pmatrix},\quad j\ge 1,
			\end{aligned}
\end{equation}
\begin{equation}\label{1385}
\begin{aligned}
			& L_{5,2j-1}(t,\nabla) :=(-1)^{j}\Delta^{-j} \begin{pmatrix}
			0&q_{2j-2}^{512}(t)\nabla^t&0\\
				q_{2j-2}^{521}(t)\nabla&0&q_{2j-2}^{523}(t)\nabla\\
				0&q_{2j-2}^{532}(t)\nabla^t&0
			\end{pmatrix},\quad j\ge 1,\\
			& L_{5,2j}(t,\nabla) :=(-1)^j\Delta^{-j} \begin{pmatrix}
				q_{2j}^{511}(t)&0&q_{2j-2}^{513}(t)\\
				0&q_{2j}^{511}(t)\Delta^{-1}\nabla\nabla^t&0\\
				q_{2j}^{531}(t)&0&q_{2j-2}^{533}(t)
			\end{pmatrix},\quad j\ge 1,
		\end{aligned}
\end{equation}
where $q_j^{kmn}(t)$ denotes a universal polynomial in $t$ of degree not larger than $j$. 

Notably, $L_{m,2j}(t,\nabla)$ and $L_{m,2j-1}(t,\nabla)$ (for $j \ge 1$) contain the operators $(-1)^j \Delta^{-j}$, arising from the Taylor expansion of short-wave components around $|\vec{\eta}|^{-2}$; this expansion is performed on $e^{\lambda_d(|\vec{\eta}|)t}P_d(\vec{\eta}) + e^{\lambda_{+}(|\vec{\eta}|)t}P_{+}(\vec{\eta}) + e^{\lambda_{-}(|\vec{\eta}|)t}P_{-}(\vec{\eta})$ in the setting $|\vec{\eta}| \to \infty$. Specifically, $\lambda_{d}(|\vec{\eta}|)$, $\lambda_{+}(|\vec{\eta}|)$, $\lambda_{-}(|\vec{\eta}|)$ denote the eigenvalues of the Fourier-transformed coefficient matrix $E(\vec{\eta})$:
$$E(\vec{\eta}) = -\begin{pmatrix}
	0 & \bar{\rho} \bar{c}_f^2 i\vec{\eta}^t & \frac{\bar{p}_{e_1}}{\bar{\tau}}\\
	\bar{v} i \vec{\eta} & 0_{3 \times 3} & 0_{3 \times 1}\\
	0 & \bar{a} i \vec{\eta}^t & \frac{{\zeta}}{\bar{\tau}}
\end{pmatrix},$$
which possesses a double zero eigenvalue and three simple eigenvalues $\lambda_{\pm}(|\vec{\eta}|)$ and $\lambda_d(|\vec{\eta}|)$.
As $|\vec{\eta}| \to \infty$,
$$\lambda_{\pm}(|\vec{\eta}|) = -{\mu}_1 + \tilde{\lambda}_R(|\vec{\eta}|^{-2}) \pm i|\vec{\eta}| \bigl[\bar{c}_f + \tilde{\lambda}_I(|\vec{\eta}|^{-2})\bigr], \quad \lambda_d(|\vec{\eta}|) = -{\mu}_2 - 2\tilde{\lambda}_R(|\vec{\eta}|^{-2}) $$
where $\tilde{\lambda}_R(|\vec{\eta}|^{-2})$ and $\tilde{\lambda}_I(|\vec{\eta}|^{-2})$ are analytic functions of the variable $|\vec{\eta}|^{-2}$; $P_{\pm}(\vec{\eta})$ and $P_{d}(\vec{\eta})$ are the eigenprojections associated with the three non-zero eigenvalues of $E(\vec{\eta})$ (see \cite{z1} for more details).

As established in Lemma~5.1 of \cite{z1}, $(-1)^j \Delta^{-j} \delta(\vec{x}) = (-1)^{j-1} (4\pi\,(2j-2)!)^{-1}|\vec{x}|^{2j-3}$ for $j \ge 1$, whence the operator $(-1)^j \Delta^{-j}$ exhibits a singularity at $|\vec{x}|=0$ for $j=1$ and tends to infinity as $|\vec{x}|\rightarrow\infty$ for $j \ge 2$. To circumvent this behavior, following the approach in \cite{dy}, one applies the Taylor expansion in $(1+|\vec{\eta}|^2)^{-1}$ instead of $|\vec{\eta}|^{-2}$; they are equivalent during $|\vec{x}|\rightarrow\infty$. Evidently, the inverse Fourier transform of $(1+|\vec{\eta}|^2)^{-j}$ for $j \ge 2$ is analytic and decays fast at $|\vec{x}|=\infty$. It leads to a replacement of $(-1)^j \Delta^{-j}\delta(\vec{x})$ with the convolution of $j$ copies of the three-dimensional Yukawa potential $\mathbb{Y}_3(\vec{x}) = (4\pi|\vec{x}|)^{-1} e^{-|\vec{x}|}.$

Moreover, $G_{k,1}^s$ and $G_{k,2}^s$ arise from the inverse Fourier transform of the Taylor-expanded spectral decomposition as $|\vec{\eta}| \to \infty$, and the spectral decomposition presented in \cite{z1} takes the form
\begin{equation}\label{eigende25}\begin{aligned}
			& (i \vec{\eta})^k \left[ e^{\lambda_d(|\vec{\eta}|)t}P_d(\vec{\eta}) + e^{\lambda_{+}(|\vec{\eta}|)t} P_{+}(\vec{\eta}) + e^{\lambda_{-}(|\vec{\eta}|)t} P_{-}(\vec{\eta}) + \operatorname{diag} \! \left(0, \mathbb{I}_{3 \times 3}, 0\right)\right] \\
			&= \mathbb{F}[G_{k,1}^s](\vec{\eta},t) + \mathbb{F}[G_{k,2}^s](\vec{\eta},t) + \mathbb{F}[G_{k,h}](\vec{\eta},t) + \mathcal{O}(1) e^{-\mu_m t} |\vec{\eta}|^{-4},
	\end{aligned}\end{equation}
with $\mu_m < \min \{ {\mu}_1, {\mu}_2 \}$, and 
\begin{equation} \label{Gkh0}
	\begin{aligned}
		&\mathbb{F}[G_{k,h}](\vec{\eta},t)\\ 
		&\quad =\begin{cases} e^{- {\mu}_2 t} (i \vec{\eta})^k |\vec{\eta}|^{-|k|-4} 
		\begin{pmatrix}
		0 & q_{\frac{|k|}{2}}^{12}(t) i \vec{\eta}^t & 0 \\
		q_{\frac{|k|}{2} +1}^{21}(t) i \vec{\eta} & 0_{3 \times 3} & q_{\frac{|k|}{2} +1}^{23}(t) i \vec{\eta}\\
		0 & q_{\frac{|k|}{2} +1}^{32}(t) i \vec{\eta}^t & 0
		\end{pmatrix}\\
		 \quad + e^{-{\mu}_1 t} (i\vec{\eta})^k \cos(\bar{c}_f |\vec{\eta}|t) |\vec{\eta}|^{-|k|-4} 
		\begin{pmatrix}
			0 & q_{|k|+3}^{12}(t) i \vec{\eta}^t & 0 \\
			q_{|k|+3}^{21}(t) i \vec{\eta} & 0_{3 \times 3} & q_{|k|+2}^{23}(t) i \vec{\eta}\\
			0 & q_{|k|+3}^{32}(t) i \vec{\eta}^t & 0
		\end{pmatrix},\\
		\quad\quad\quad\quad\text{ for }|k|\text{ even},\\
		\\
		e^{- {\mu}_2 t} (i \vec{\eta})^k |\vec{\eta}|^{-|k|-3} 
 		\begin{pmatrix}
 			q_{\frac{|k|+1}{2}}^{11}(t) & 0_{1 \times 3} & q_{\frac{|k|+1}{2}}^{13}(t)\\
 			0_{3 \times 1} & q_{\frac{|k|+1}{2}}^{11}(t) |\vec{\eta}|^{-2} \vec{\eta} \vec{\eta}^t & 0_{3 \times 1}\\
 			q_{\frac{|k|+1}{2}}^{31}(t) & 0_{1 \times 3} & q_{\frac{|k|+1}{2}}^{33}(t)
 		\end{pmatrix}\\
 		\quad + e^{-{\mu}_1 t} (i\vec{\eta})^k \cos(\bar{c}_f |\vec{\eta}|t) |\vec{\eta}|^{-|k|-3} \begin{pmatrix}
 			q_{|k|+3}^{11}(t) & 0_{1 \times 3} & q_{|k|+2}^{13}(t)\\
 			0_{3 \times 1} & q_{|k|+3}^{11}(t) |\vec{\eta}|^{-2} \vec{\eta} \vec{\eta}^t & 0_{3 \times 1}\\
 			q_{|k|+3}^{31}(t) & 0_{1 \times 3} & q_{|k|+2}^{33}(t)
 		\end{pmatrix},\\
		\quad\quad\quad\quad\text{ for }|k|\text{ odd}.\\
		\end{cases}
	\end{aligned}
\end{equation}
Here, $q_j^{ml}(t)$ denotes a universal polynomial in $t$ of degree no greater than $j$.

In addition to the pointwise estimates in \eqref{green1} and \eqref{green2}, Theorem~1.4 of \cite{z1} concludes the following uniform estimate for the Green's function:
\begin{equation} \label{greenpie}
	\begin{aligned}
		\partial_{\vec{x}}^{k} \mathbb{G}(\vec{x},t)&= H(t-1) \left[ G_{k}^{*} (\vec{x},t)+\mathcal{O}(1)(1+t)^{-\frac{5+|k|}{2}} e^{-\frac{(|\vec{x}|-\bar{c} t)^2}{Kt}} \right] \\
		&\quad - H(\tilde{c}_1 (1+t)-|\vec{x}|) \operatorname{diag}\! \left(0, \partial_{\vec{x}}^{k} (\Delta^{-1} \nabla \nabla^t \delta(\vec{x})), 0 \right)\\
		&\quad + G_{k,1}^s(\vec{x},t) + G_{k,2}^s(\vec{x},t) + \mathcal{O}(1) e^{-t/C_0}.
	\end{aligned}
\end{equation}
\subsection{Modifications to the Taylor Expansion in Spectral Decomposition}
In the Taylor expansion during the spectral decomposition, one changes $|\vec{\eta}|^{-2}$ to $(1+|\vec{\eta}|^2)^{-1}$. Consequently, the terms $|\vec{\eta}|^{-2j}$ for $j \ge 1$ in $\mathbb{F}[G_{k,1}^s]$, $\mathbb{F}[G_{k,2}^s]$ and $\mathbb{F}[G_{k,h}]$ are replaced with $(1+|\vec{\eta}|^2)^{-j}$ which allows the terms $(-1)^j \Delta^{-j}$ in $G_{k,1}^s$, $G_{k,2}^s$, and $G_{k,h}$ to be replaced by $\mathbb{F}^{-1}[(1+|\vec{\eta}|^2)^{-j}] = \mathbb{Y}_3 \underset{j}{\ast} \mathbb{Y}_3$, where one denotes
$$\mathbb{Y}_3 \underset{j}{\ast} \mathbb{Y}_3 := \underbrace{\mathbb{Y}_3 \underset{\vec{x}}{\ast} \mathbb{Y}_3 \underset{\vec{x}}{\ast} \cdots \underset{\vec{x}}{\ast} \mathbb{Y}_3}_{j \text{ times}},$$
i.e. the convolution of $j$ copies of the three-dimensional Yukawa potential $\mathbb{Y}_3$. As established in Lemma~2.1 of \cite{dy}, it holds that for all $j \ge 2$,
\begin{equation} \label{y3con} \mathbb{F}^{-1}\left[\frac{1}{(1+|\vec{\eta}|^2)^j}\right] \le C e^{-|\vec{x}|/C},
\end{equation} 
where $C > 0$ denotes a universal positive constant. 

The modified forms of $G_{k,1}^s$, $G_{k,2}^s$, and $G_{k,h}$ are rearranged and denoted by $\bar{G}_{k,1}^s$, $\bar{G}_{k,2}^s$, and $\bar{G}_{k,h}$ as follows:
 \begin{equation} \label{gk1s}
 	\begin{aligned}
 		\bar{G}_{k,1}^s (\vec{x},t) &= \partial_{\vec{x}}^k \Bigg\lbrace \delta(\vec{x}) \operatorname{diag}\! \left(0, \mathbb{I}_{3 \times 3},0 \right)  + e^{-{\mu}_2 t} \sum_{j=0}^{|k|+2} L_{1,j}^{*}(t,\nabla) \delta(\vec{x}) + e^{-{\mu}_1 t} \Biggl[ \, \sum_{j=0}^{|k|+2}  L_{2,j}^{*} (t,\nabla) \omega_t (\vec{x},t;\bar{c}_f) \\
 		& \quad + \sum_{j=-1}^{|k|+2} L_{3,j}^{*} (t,\nabla) \omega(\vec{x},t;\bar{c}_f) \Biggr] + t e^{-{\mu}_1 t} \Biggl[ \, \sum_{j=1}^{|k|+2} L_{4,j}^{*} (t, \nabla) \omega_t (\vec{x},t;\bar{c}_f) + \sum_{j=0}^{|k|+2} L_{5,j}^{*}(t,\nabla) \omega(\vec{x},t;\bar{c}_f) \Biggr] \Bigg\rbrace,
 	\end{aligned} 
 \end{equation}
\begin{equation} \label{gk2s}
	\begin{aligned}
		\bar{G}_{k,2}^s (\vec{x},t) = \partial_{\vec{x}}^k \Bigg\lbrace e^{-{\mu}_1 t} \operatorname{diag}\! \left(0, -\nabla \nabla^t \mathbb{Y}_3(\vec{x}),0 \right) \underset{\vec{x}}{\ast} \bigl[ \omega_t(\vec{x},t; \bar{c}_f) + \left({\mu}_1 + {\mu}_1 ({\mu}_1/2 + {\mu}_2) t\right) \! \omega(\vec{x},t; \bar{c}_f) \bigr] \Bigg\rbrace,
	\end{aligned}
\end{equation}
\begin{equation} \label{gkh0}
	\begin{aligned}
		&\bar{G}_{k,h}(\vec{x},t)\\
		 &\quad =\begin{cases} \partial_{\vec{x}}^k \left\lbrace  e^{-{\mu}_2 t} 	
		\begin{pmatrix}
			0 & q_{\frac{|k|}2}^{12}(t) \nabla^t & 0 \\
			q_{\frac{|k|}{2} +1}^{21}(t) \nabla & 0_{3 \times 3} & q_{\frac{|k|}{2} +1}^{23}(t) \nabla\\
			0 & q_{\frac{|k|}{2} +1}^{32}(t) \nabla^t & 0
		\end{pmatrix} \mathbb{Y}_3(\vec{x}) \underset{\frac{|k|+4}{2}}{\ast} \mathbb{Y}_3(\vec{x})\right.\\
		 \left.+ e^{-{\mu}_1 t} 
		\begin{pmatrix}
			0 & q_{|k|+3}^{12}(t) \nabla^t & 0 \\
			q_{|k|+3}^{21}(t) \nabla & 0_{3 \times 3} & q_{|k|+2}^{23}(t) \nabla\\
			0 & q_{|k|+3}^{32}(t) \nabla^t & 0
		\end{pmatrix}\mathbb{Y}_3(\vec{x}) \underset{\frac{|k|+4}{2}}{\ast} \mathbb{Y}_3(\vec{x}) \underset{\vec{x}}{\ast} \omega_t(\vec{x},t;\bar{c}_f)\right\rbrace,\\		\quad\quad\quad\quad\quad\quad\quad\quad\quad\quad\quad\quad\quad\quad\quad\quad\quad\quad\quad\quad\quad\quad\quad\quad\text{ for }|k|\text{ even},\\
		\\
		\partial_{\vec{x}}^k \left\lbrace  e^{-{\mu}_2 t} 	
		\begin{pmatrix}
			q_{\frac{|k|+1}{2}}^{11}(t) & 0_{1 \times 3} & q_{\frac{|k|+1}{2}}^{13}(t)\\
			0_{3 \times 1} & - q_{\frac{|k|+1}{2}}^{11}(t) \nabla \nabla^t \mathbb{Y}_3(\vec{x}) & 0_{3 \times 1}\\
			q_{\frac{|k|+1}{2}}^{31}(t) & 0_{1 \times 3} & q_{\frac{|k|+1}{2}}^{33}(t)
		\end{pmatrix}\underset{\vec{x}}{\ast}\mathbb{Y}_3(\vec{x}) \underset{\frac{|k|+3}{2}}{\ast} \mathbb{Y}_3(\vec{x})\right.\\
		 \left.+ e^{-{\mu}_1 t} 
		\begin{pmatrix}
			q_{|k|+3}^{11}(t) & 0_{1 \times 3} & q_{|k|+2}^{13}(t)\\
			0_{3 \times 1} & - q_{|k|+3}^{11}(t) \nabla \nabla^t \mathbb{Y}_3(\vec{x}) & 0_{3 \times 1}\\
			q_{|k|+3}^{31}(t) & 0_{1 \times 3} & q_{|k|+2}^{33}(t)
		\end{pmatrix}\underset{\vec{x}}{\ast}\mathbb{Y}_3(\vec{x}) \underset{\frac{|k|+3}{2}}{\ast} \mathbb{Y}_3(\vec{x}) \underset{\vec{x}}{\ast} \omega_t(\vec{x},t;\bar{c}_f) \right\rbrace,\\
\quad\quad\quad\quad\quad\quad\quad\quad\quad\quad\quad\quad\quad\quad\quad\quad\quad\quad\quad\quad\quad\quad\quad\quad\text{ for }|k|\text{ odd}.
		\end{cases}
	\end{aligned}
\end{equation}
Here, for $j=-1,0$, $L_{m,j}^*(t,\nabla) := L_{m,j}(t,\nabla)$ (see \eqref{leading} for the explicit expressions). For $j \ge 1$, $L_{m,j}(t,\nabla)$ is given in \eqref{138}--\eqref{1385}, while $L_{m,j}^*(t,\nabla)$ is the redefined vertion of $L_{m,j}(t,\nabla)$ obtained by replacing $(-1)^i \Delta^{-i}$ with $\mathbb{Y}_3 \underset{i}{\ast} \mathbb{Y}_3$ (for $j=2i$) or $(-1)^{i+1} \Delta^{-(i+1)}$ with $\mathbb{Y}_3 \underset{i+1}{\ast} \mathbb{Y}_3$ (for $j=2i+1$). For instance, the operators $L_{1,2j}^*(t,\nabla)$ and $L_{1,2j+1}^*(t,\nabla)$ are defined as follows:  for $j \ge 1$,
\begin{equation} \label{lmjstar}
	\begin{aligned}
		& L_{1,2j}^{*}(t,\nabla) = 
		\begin{pmatrix}
			q_{j-1}^{111}(t) &0_{1 \times 3}& q_{j-1}^{113}(t)\\
			0_{3 \times 1} & -q_{j-1}^{111}(t) \nabla \nabla^t \mathbb{Y}_3 (\vec{x}) & 0_{3 \times 1}\\
			q_{j-1}^{131}(t) &0_{1 \times 3}& q_{j-1}^{133}(t) 
		\end{pmatrix}\underset{\vec{x}}{\ast}\mathbb{Y}_3(\vec{x}) \underset{j}{\ast} \mathbb{Y}_3(\vec{x}),\\
		& L_{1,2j+1}^{*}(t,\nabla) = 
		\begin{pmatrix}
			0 & q_{j-1}^{112}(t) \nabla^t & 0\\
			q_{j}^{121}(t) \nabla & 0_{3 \times 3} & q_{j}^{123}(t) \nabla\\
			0 & q_{j}^{132}(t) \nabla^t & 0
		\end{pmatrix}\mathbb{Y}_3(\vec{x}) \underset{j+1}{\ast} \mathbb{Y}_3(\vec{x}).
	\end{aligned}
\end{equation}

Note that in \cite{z1}, $G_{k,h}$ needs to be estimated and then one could obtain \eqref{greenpie} from \eqref{eigende25} and the estimate of $G_{k,h}$ follows from Lemma~4.6 of \cite{z1}. This lemma states that for all $\vec{x} \in \mathbb{R}^3$ and $t \in \mathbb{R}_{+}$ and a multi-index $k$ with odd $|k|$,
$$\left|\mathrm{PV} \int_{|\vec{\eta}| \ge R} e^{i \vec{x} \cdot \vec{\eta}} (i \vec{\eta})^{k} |\vec{\eta}|^{-|k|-3} \, d\vec{\eta} \, \right|, \left|\mathrm{PV} \int_{|\vec{\eta}| \ge R} e^{i \vec{x} \cdot \vec{\eta}} (i \vec{\eta})^{k} |\vec{\eta}|^{-|k|-3} \cos(\bar{c}_f |\vec{\eta}| t) \, d\vec{\eta} \, \right| \le C,$$
where $R>0$ is a constant, and $\mathrm{PV}$ denotes the principal value. No analogous result holds when $|\vec{\eta}|^{-|k|-3}$ is replaced with $(1+|\vec{\eta}|^2)^{-(|k|+3)/2}$, and thus $\bar{G}_{k,h}$ remains in the revised estimates after modifying \eqref{Gk1s}, \eqref{Gk2s} and \eqref{Gkh0} to \eqref{gk1s}--\eqref{gkh0}. Specifically, \eqref{greenpie}  can be recast as
\begin{equation} \label{greenmodi}
	\begin{aligned}
		\partial_{\vec{x}}^{k} \mathbb{G}(\vec{x},t) &= H(t-1)\left[ G_{k}^{*} (\vec{x},t) + \mathcal{O}(1)(1+t)^{-\frac{5 + |k|}{2}} e^{-\frac{(|\vec{x}|-\bar{c} t)^2}{Kt}} \right] + \bar{G}_{k,1}^s(\vec{x},t) + \bar{G}_{k,2}^s(\vec{x},t) + \bar{G}_{k,h} (\vec{x},t) \\
		& \quad - H(\tilde{c}_1 (1+t)-|\vec{x}|) \operatorname{diag}\! \left( 0, \partial_{\vec{x}}^{k} (\Delta^{-1} \nabla \nabla^t \delta(\vec{x})), 0 \right)  + \mathcal{O}(1) e^{-t/C_0}.
	\end{aligned}
\end{equation}
Denote 
\begin{equation} \label{remain}
	\begin{aligned}R_{k}^{s}(\vec{x},t)&:=\partial_{\vec{x}}^{k} \mathbb{G}(\vec{x},t)-\Big(H(t-1) G_{k}^{*} (\vec{x},t) + \bar{G}_{k,1}^s(\vec{x},t) + \bar{G}_{k,2}^s(\vec{x},t) + \bar{G}_{k,h} (\vec{x},t)\\
&\quad- H(\tilde{c}_1 (1+t)-|\vec{x}|) \operatorname{diag}\! \left( 0, \partial_{\vec{x}}^{k} (\Delta^{-1} \nabla \nabla^t \delta(\vec{x})), 0 \right) \Big)
\end{aligned}
\end{equation} 
and it is bounded above by a moving heat kernel $C(1+t)^{-(5+|k|)/2} e^{-(|\vec{x}|-\bar{c} t)^2/(Kt)}$ and an exponentially decaying function $Ce^{-t/C_0}$.

Combining \eqref{green2}, \eqref{Gk1slead} and \eqref{greenmodi}--\eqref{remain}, one has that for $|\vec{x}|>M(1+t)$,
$$\begin{aligned}
	& \Bigl| H(t-1) G_{k}^{*} (\vec{x},t) + \bar{G}_{k,1}^s(\vec{x},t) - \tilde{G}_{k,1}^s(\vec{x},t) + \bar{G}_{k,2}^s(\vec{x},t) + \bar{G}_{k,h} (\vec{x},t) \\
	& -H(\tilde{c}_1 (1+t)-|\vec{x}|)\operatorname{diag}\!\left(0, \partial_{\vec{x}}^{k} (\Delta^{-1} \nabla \nabla^t \delta(\vec{x})), 0 \right) + R_{k}^{s}(\vec{x},t) \Bigr|\\
	&= \mathcal{O}(1) e^{-(|\vec{x}|+t)/C_0}.
\end{aligned}$$
Since $M > \bar{c}_f > \bar{c} > \tilde{c}_1$, it follows that $H(\tilde{c}_1 (1+t)-|\vec{x}|) = 0$ for $|\vec{x}| > M(1+t)$, which combined with \eqref{Gkstar}, \eqref{Gk1slead}, \eqref{138}--\eqref{1385}, \eqref{y3con}--\eqref{lmjstar}, and Corollary~\ref{kirch}, results in that there exists a positive constant $C > 0$ such that for $|\vec{x}| \in \mathbb{R}^3$, $t \in \mathbb{R}_+$,
$$\begin{aligned}
	\left| R_{k}^{s}(\vec{x},t) \right| & \le \left|H(t-1) G_{k}^{*} (\vec{x},t) \right| + \left|\bar{G}_{k,1}^s(\vec{x},t) - \tilde{G}_{k,1}^s(\vec{x},t) \right| + \left|  \bar{G}_{k,2}^s(\vec{x},t) \right| + \left|\bar{G}_{k,h} (\vec{x},t) \right| + \mathcal{O}(1) e^{-(|\vec{x}|+t)/C_0}\\
	& \le C e^{-(|\vec{x}|+t)/C}.
\end{aligned}$$
Finally it combined with \eqref{greenmodi} yields that 
$$\begin{aligned}
	\partial_{\vec{x}}^{k} \mathbb{G}(\vec{x},t) &= H(t-1) \left[ G_{k}^{*} (\vec{x},t) + \mathcal{O}(1) (1+t)^{-\frac{5 + |k|}{2}} e^{-\frac{(|\vec{x}| - \bar{c} t)^2}{Kt}} \right] + \bar{G}_{k,1}^s(\vec{x},t) + \bar{G}_{k,2}^s(\vec{x},t) + \bar{G}_{k,h} (\vec{x},t) \\
	& \quad - H(\tilde{c}_1 (1+t) - |\vec{x}|)\operatorname{diag}\! \left(0, \partial_{\vec{x}}^{k} (\Delta^{-1} \nabla \nabla^t \delta(\vec{x})), 0\right) + \mathcal{O}(1) e^{-(|\vec{x}|+t)/C}, \quad \text{for } \vec{x} \in \mathbb{R}^3, t \in \mathbb{R}_+.
\end{aligned}$$
A more precise description about $G_{k}^{*}$, $\bar{G}_{k,1}^s, \bar{G}_{k,2}^s$ and $\bar{G}_{k,h}$ is left to the following subsections and after the long-wave and short-wave decomposition.

One could decompose $\partial_{\vec{x}}^{k} \mathbb{G}(\vec{x},t)$ into long-wave and short-wave components after the spectral decomposition as follows:
\begin{equation} \label{greenmodi1}
	\partial_{\vec{x}}^{k} \left(\mathbb{G}-\mathbb{G}_0\right) \! (\vec{x},t)= \partial_{\vec{x}}^{k} \mathbb{G}_L (\vec{x},t) + \partial_{\vec{x}}^{k} \mathbb{G}_S (\vec{x},t),
\end{equation}
where $\mathbb{G}_0$ arises from the eigen-projection corresponding to the zero eigenvalues:
\begin{equation}
	 \mathbb{G}_0(\vec{x},t) = \partial_{\vec{x}}^{k} \operatorname{diag}\! \left(0, (\mathbb{I}_{3 \times 3} - \Delta^{-1} \nabla \nabla^t ) \delta(\vec{x}), 0 \right), \label{g0ini}
\end{equation}
and $\mathbb{G}_L$ and $\mathbb{G}_S$ correspond to the long-wave and the short-wave components of $\mathbb G-\mathbb{G}_0$ respectively:
\begin{align}
	& \partial_{\vec{x}}^{k} \mathbb{G}_L(\vec{x},t) = H(t-1) \left[ G_{k}^{*} (\vec{x},t) + \mathcal{O}(1) (1 + t)^{-\frac{5 + |k|}{2}} e^{-\frac{(|\vec{x}| - \bar{c} t)^2}{Kt}} \right] \label{glini}\\
	& \quad \quad \quad \quad \quad \quad + H(|\vec{x}|-\tilde{c}_1 (1+t)) \operatorname{diag}\! \left(0, \partial_{\vec{x}}^{k} (\Delta^{-1} \nabla \nabla^t \delta(\vec{x})), 0 \right) + \mathcal{O}(1) e^{-(|\vec{x}|+t)/C}, \nonumber \\
	& \partial_{\vec{x}}^{k} \mathbb{G}_S(\vec{x},t) = \bar{G}_{k,1}^s(\vec{x},t) + \bar{G}_{k,2}^s(\vec{x},t) + \bar{G}_{k,h} (\vec{x},t) - \partial_{\vec{x}}^{k} \delta(\vec{x}) \operatorname{diag}\! \left(0, \mathbb{I}_{3 \times 3}, 0 \right)  + \mathcal{O}(1) e^{-(|\vec{x}|+t)/C}.\label{gsini}
\end{align}

\subsection{Estimates for the Long-Wave Component $\partial_{\vec{x}}^{k} \mathbb{G}_L(\vec{x},t)$} 
The long-wave component $\partial_{\vec{x}}^{k} \mathbb{G}_L(\vec{x},t)$ defined in \eqref{glini} can be more precisely formulated as
$$\begin{aligned}
	\partial_{\vec{x}}^{k} \mathbb{G}_L(\vec{x},t) &= H(t-1) \Bigl[ G_{k}^{*} (\vec{x},t) + H(|\vec{x}| - \tilde{c}_1 (1+t)) \operatorname{diag}\! \left(0, \partial_{\vec{x}}^{k} (\Delta^{-1} \nabla \nabla^t \delta(\vec{x})), 0 \right) \\
	& \quad + \mathcal{O}(1)(1 + t)^{-\frac{5 + |k|}{2}} e^{-\frac{(|\vec{x}| - \bar{c} t)^2}{Kt}} \Bigr] + \mathcal{O}(1) e^{-(|\vec{x}|+t)/C}. 
\end{aligned}$$
Here, the first summand {{$G_{k}^{*} (\vec{x},t)$}} arises primarily from the inverse Fourier transform of the leading term in the long-wave expansion, namely 
$$\mathbb{F}[G_L](\vec{\eta},t) = e^{-\bar{\alpha}_1|\vec{\eta}|^2 t/2} \bigl[ \cos(\bar{c} |\vec{\eta}| t) P_5 + \sin(\bar{c} |\vec{\eta}| t) P_6 \bigr] $$
with
$$P_5=\begin{pmatrix}
	1 & 0_{1 \times 3} & -\frac{\bar{p}_{e_1}}{{\zeta}}\\
	0_{3 \times 1} & \frac{1}{|\vec{\eta}|^2} \vec{\eta} \vec{\eta}^t & 0_{3 \times 1}\\
	0 & 0_{1 \times 3} & 0
\end{pmatrix}, \quad
P_6 = \begin{pmatrix}
	0 & -\frac{\bar{c}\bar{\rho}}{|\vec{\eta}|} i \vec{\eta}^t & 0\\
	- \frac{\bar{v}}{\bar{c}|\vec{\eta}|} i \vec{\eta} & 0_{3 \times 3} & \frac{\bar{v} \bar{p}_{e_1}}{\bar{c} {\zeta} |\vec{\eta}|} i \vec{\eta} \\
	0 & 0_{1 \times 3} & 0
\end{pmatrix}.$$
Notably, as shown in \cite{z1}, both $H(\bar{c} t - |\vec{x}|) (1+t)^{-3-|k|}$ and $H(|\vec{x}|-\tilde{c}_1 (1+t)) \partial_{\vec{x}}^{k} (\Delta^{-1} \nabla \nabla^t \delta(\vec{x}))$ arise exclusively from the estimate for $\left(\partial_{\vec{x}}^k G_{L}\right)_{22}$:
$$\left(\partial_{\vec{x}}^k G_{L}\right)_{22}(\vec{x},t) := \partial_{\vec{x}}^k \mathbb{F}^{-1} \! \left[e^{-\bar{\alpha}_1 |\vec{\eta}|^2 t/2} \cos(\bar{c} |\vec{\eta}| t) |\vec{\eta}|^{-2} \vec{\eta} \vec{\eta}^t\right],$$
where $H(\bar{c} t - |\vec{x}|) (1+t)^{-3 - |k|}$ appears in the estimate \eqref{Gkstar} of $G_k^{*} (\vec{x},t)$. It could be modified by retaining the operator $\Delta^{-1} \nabla \nabla^t$, i.e., via the identity $$(\partial_{\vec{x}}^k G_{L})_{22}(\vec{x},t) = \Delta^{-1} \nabla \nabla^t (\partial_{\vec{x}}^k G_{L})_{11}(\vec{x},t),$$ and the corresponding bound for $(\partial_{\vec{x}}^k G_{L})_{11}$ is as follows:
$$\begin{aligned}
	\left( \partial_{\vec{x}}^k G_{L} \right)_{11} \! (\vec{x},t) &=
	\begin{cases}
		\frac{1}{2 \bar{c} (2 \pi \bar{\alpha}_1 t)^{3/2}} \partial_{\vec{x}}^k \left(\frac{|\vec{x}|-\bar{c}t}{t} e^{-\frac{(|\vec{x}|-\bar{c}t)^2}{2\bar{\alpha}_1 t}} \right) & \text{for } |k|=0 \text{ or } |\vec{x}| \ge \tilde{c}_2 (1+t), \\
		0 & \text{otherwise}
	\end{cases}\\
	& \quad + \mathcal{O}(1) t^{-\frac{5+|k|}{2}}e^{-\frac{(|\vec{x}|-\bar{c}t)^2}{Kt}} \sum_{j=0}^{|k|} t^{-\frac{j}{2}},
\end{aligned}$$ 
which has been proved in \cite{z1}. Thus, $\partial_{\vec{x}}^{k} \mathbb{G}_L$ can be further expressed as
\begin{equation}\nonumber
	\begin{aligned}
		\partial_{\vec{x}}^{k} \mathbb{G}_L(\vec{x},t) &= H(t-1) \Bigl[ \tilde{G}_{k}^{*} (\vec{x},t) + \operatorname{diag}\! \left(0, \partial_{\vec{x}}^{k} (\Delta^{-1} \nabla \nabla^t \theta_1(\vec{x},t)), 0 \right) \\
		&\quad + \mathcal{O}(1)(1+ t)^{-\frac{5 + |k|}{2}} e^{-\frac{(|\vec{x}|-\bar{c} t)^2}{Kt}} \Bigr] + \mathcal{O}(1) e^{-(|\vec{x}| + t)/C}, 
	\end{aligned}
\end{equation}
where $\theta_1(\vec{x},t):=(G_L)_{11}$ and
\begin{equation} \label{gkstar}
	\begin{aligned}
		\tilde{G}_{k}^{*}(\vec{x},t) &= H(|\vec{x}|-\tilde{c}_2 (1+t)) \frac{1}{2(2 \pi \bar{\alpha}_1 t)^{3/2}} \partial_{\vec{x}}^k \left[ \frac{|\vec{x}|-\bar{c} t}{t} e^{-\frac{(|\vec{x}|-\bar{c}t)^2}{2 \bar{\alpha}_1 t}} \right]\\
		& \quad \times \begin{pmatrix}
			\frac{1}{\bar{c}} & \frac{\bar{c}\bar{\rho}t}{|\vec{x}|^2} \vec{x}^t & -\frac{\bar{p}_{e_1}}{\bar{c} {\zeta}}\\
			\frac{\bar{v} t}{\bar{c}|\vec{x}|^2} \vec{x} & 0_{3 \times 3}& -\frac{\bar{v} \bar{p}_{e_1} t}{\bar{c} {\zeta} |\vec{x}|^2} \vec{x}\\
			0 & 0_{1 \times 3} & 0
		\end{pmatrix}\\
		&= \mathcal{O}(1) t^{-2-\frac{|k|}{2}} e^{-\frac{(|\vec{x}|-\bar{c}t)^2}{Kt}} .  
	\end{aligned}
\end{equation}
In addition, from the relation
$$\big(\partial_{\vec{x}}^k \mathbb{G}_{L}\big)_{22}(\vec{x},t) = \Delta^{-1} \nabla \nabla^t \big(\partial_{\vec{x}}^k \mathbb{G}_{L}\big)_{11}(\vec{x},t),$$
one obtains a sharper result:
\begin{equation}  \label{glrefin}
	\begin{aligned}
		\partial_{\vec{x}}^{k} \Big(\mathbb{G}_L(\vec{x},t) -\operatorname{diag}\! \left(0, \Delta^{-1} \nabla \nabla^t \left(\mathbb{G}_L\right)_{11}(\vec{x},t), 0 \right)\Big) &= H(t-1) \Bigl[ \tilde{G}_{k}^{*} (\vec{x},t) + \mathcal{O}(1) (1 + t)^{-\frac{5 + |k|}{2}} e^{-\frac{(|\vec{x}|-\bar{c} t)^2}{Kt}} \Bigr]  \\
		& \quad + \mathcal{O}(1) e^{-(|\vec{x}|+t)/C}.
	\end{aligned}
\end{equation}

\subsection{Estimates for the Short-Wave Component $\partial_{\vec{x}}^k\mathbb{G}_S(\vec{x},t)$} The short-wave component defined in \eqref{gsini} can be further simplified. Specifically, since $\Delta \mathbb{Y}_3 = -\delta + \mathbb{Y}_3$, it immediately follows that
\begin{equation} \label{yy3}
	-\nabla \nabla^t \mathbb{Y}_3(\vec{x}) = \Delta^{-1} \nabla \nabla^t (\delta - \mathbb{Y}_3)(\vec{x}).
\end{equation}
Moreover, the following convolution identities including $\mathbb{Y}_3$ can be derived via polar coordinate transformation and the residue theorem:
\begin{equation} \label{y3wwt}
	\begin{aligned}
		\mathbb{Y}_3 (\vec{x}) \underset{\vec{x}}{\ast} \omega (\vec{x},t;\bar{c}_f) &= \frac{1}{8 \pi \bar{c}_f |\vec{x}|} \left( e^{-\bigl| |\vec{x}| - \bar{c}_f t \bigr|} - e^{-(|\vec{x}| + \bar{c}_f t)} \right),\\
		\mathbb{Y}_3(\vec{x}) \underset{\vec{x}}{\ast} \omega_t(\vec{x},t;\bar{c}_f) &= \frac{1}{8 \pi |\vec{x}|} \left( e^{-(|\vec{x}| + \bar{c}_f t)} + e^{-\bigl| |\vec{x}| - \bar{c}_f t \bigr|} \right),\\
		\mathbb{Y}_3 (\vec{x}) \underset{\vec{x}}{\ast} \mathbb{Y}_3 (\vec{x}) \underset{\vec{x}}{\ast} \omega (\vec{x},t;\bar{c}_f) &= \frac{1}{16 \pi \bar{c}_f |\vec{x}|} \left[ \bigl( \bigl| |\vec{x}| - \bar{c}_f t \bigr| + 1 \bigr) e^{-\bigl| |\vec{x}| - \bar{c}_f t \bigr|} - (|\vec{x}| + \bar{c}_f t + 1)  e^{-(|\vec{x}| + \bar{c}_f t)} \right],\\
		\mathbb{Y}_3(\vec{x}) \underset{\vec{x}}{\ast} \mathbb{Y}_3(\vec{x}) \underset{\vec{x}}{\ast} \omega_t (\vec{x},t;\bar{c}_f) &= \frac{1}{16 \pi |\vec{x}|} \left[ (|\vec{x}| + \bar{c}_f t) e^{-(|\vec{x}| + \bar{c}_f t)} + (|\vec{x}| - \bar{c}_f t)  e^{-\bigl| |\vec{x}| - \bar{c}_f t \bigr|} \right].		
	\end{aligned}
\end{equation}
By combining the results in \eqref{wg}--\eqref{wtg}, \eqref{y3con}--\eqref{gkh0}, \eqref{gsini} and \eqref{yy3}--\eqref{y3wwt}, the estimate for the short-wave component 
\begin{equation} \label{gsmodi}
	\mathbb{G}_S(\vec{x},t) = \bar{G}_{0,1}^s(\vec{x},t) - \delta(\vec{x}) \operatorname{diag}\! \left(0, \mathbb{I}_{3 \times 3}, 0 \right) + \bar{G}_{0,2}^s(\vec{x},t) + \bar{G}_{0,h}(\vec{x},t) + \mathcal{O}(1) e^{-(|\vec{x}|+t)/C}
\end{equation}
can be further modified as follows.

For $\bar{G}_{0,1}^s(\vec{x},t) - \delta(\vec{x}) \operatorname{diag}(0, \mathbb{I}_{3 \times 3}, 0)$, it holds that
\begin{equation}\label{g01bar}
	\begin{aligned}
		& \bar{G}_{0,1}^s (\vec{x},t) - \delta(\vec{x}) \operatorname{diag}\! \left(0, \mathbb{I}_{3 \times 3}, 0 \right)-G_{0,1}^{s,*}(\vec{x},t) - \operatorname{diag}\! \left(0, \Delta^{-1} \nabla \nabla^t \theta_2 (\vec{x},t), 0 \right)\\
		& = e^{-{\mu}_2 t} \sum_{j=1}^{2} L_{1,j}^{*}(t,\nabla) \delta(\vec{x}) + e^{-{\mu}_1 t} \Biggl[ \, \sum_{j=1}^{2}  L_{2,j}^{*}(t,\nabla) \omega_t (\vec{x},t;\bar{c}_f) + \sum_{j=1}^{2} L_{3,j}^{*} (t,\nabla) \omega(\vec{x},t;\bar{c}_f) \Biggr]\\
		& \quad + t e^{-{\mu}_1 t} \Biggl[ \, \sum_{j=1}^{2} L_{4,j}^{*} (t, \nabla) \omega_t (\vec{x},t;\bar{c}_f) + \sum_{j=1}^{2} L_{5,j}^{*}(t,\nabla) \omega(\vec{x},t;\bar{c}_f) \Biggr]- \operatorname{diag}\! \left(0, \Delta^{-1} \nabla \nabla^t \theta_2 (\vec{x},t), 0 \right)\\
		& \le  C e^{-{\mu}_2 t} \Bigl( \left| \mathbb{Y}_3(\vec{x}) \right|  + \left|\nabla \mathbb{Y}_3 (\vec{x}) \right| \Bigr) + C q_3(t) e^{-{\mu}_1 t} \left( \Bigl|\mathbb{Y}_3(\vec{x}) \underset{\vec{x}}{*} \omega (\vec{x},t; \bar{c}_f) \Bigr| \right.\\
		& \quad \left.+ \Bigl|\mathbb{Y}_3(\vec{x}) \underset{\vec{x}}{\ast} \omega_t (\vec{x},t; \bar{c}_f) \Bigr| + \Bigl|\nabla \mathbb{Y}_3 (\vec{x}) \underset{\vec{x}}{\ast} \omega (\vec{x},t; \bar{c}_f) \Bigr| + \Bigl|\nabla \mathbb{Y}_3 (\vec{x}) \underset{\vec{x}}{\ast} \omega_t (\vec{x},t; \bar{c}_f) \Bigr| \right)\\
		& \le   C |\vec{x}|^{-2} e^{-|\vec{x}|/D_1} \bigl( e^{-{\mu}_1 t/4} + e^{-{\mu}_2 t} \bigr),
	\end{aligned}
\end{equation}
where the leading-order term $G_{0,1}^{s,*}(\vec{x},t)$ is defined by
\begin{equation}\label{g01star}
	\begin{aligned}
		G_{0,1}^{s,*}(\vec{x},t)&:= e^{-{\mu}_2 t}L_{1,0} \delta (\vec{x}) + e^{-{\mu}_1 t} \bigl[ L_{2,0} \omega_t(\vec{x}, t; \bar{c}_f) + (L_{3, -1}(t, \nabla) + L_{3,0}) \omega(\vec{x}, t; \bar{c}_f)\bigr]\\
		& \quad + t e^{-{\mu}_1 t} L_{5,0} \omega(\vec{x},t; \bar{c}_f),  
	\end{aligned}
\end{equation}
and the corresponding estimate for the auxiliary function $\theta_2(\vec{x},t)$ is as follows:
\begin{equation}\label{theta2} 
	\begin{aligned}
		\bigl| \theta_2 (\vec{x},t) \bigr|&=\Big| \Big[e^{-\mu_2 t} q_0^{111}(t) \delta(\vec{x}) + e^{-\mu_1 t} \bigl( q_2^{211}(t)+tq_0^{411}(t)\bigr)  \omega_t (\vec{x},t; \bar{c}_f )+ e^{-\mu_1 t} \bigl( q_2^{311}(t)+tq_2^{511}(t)\bigr) \omega (\vec{x},t; \bar{c}_f )\Big]\\ 
		&\quad \quad \quad \underset{\vec{x}}{\ast} \mathbb{Y}_3(\vec x) \underset{\vec{x}}{\ast}(\delta - \mathbb{Y}_3)(\vec{x})\Big| \\
		& \le C |\vec{x}|^{-1} e^{-|\vec{x}|/D_1} \bigl( e^{-{\mu}_1 t/4} + e^{-{\mu}_2 t} \bigr).
	\end{aligned}
\end{equation}

For the term $\bar{G}_{0,2}^s(\vec{x},t)$, similarly, from \eqref{Gk2s}, \eqref{gk2s} and \eqref{yy3}, one has that
\begin{equation}\label{g02bar}
	\begin{aligned}
		\bar{G}_{0,2}^s (\vec{x},t) & = e^{-{\mu}_1 t} \operatorname{diag}\! \left(0, \Delta^{-1} \nabla \nabla^t (\delta - \mathbb{Y}_3)(\vec{x}), 0 \right) \underset{\vec{x}}{\ast} \bigl[ \omega_t(\vec{x},t; \bar{c}_f) + \left( {\mu}_1 + {\mu}_1 ({\mu}_1/2 + {\mu}_2)t \right) \! \omega(\vec{x},t; \bar{c}_f) \bigr] \\
		& = G_{0,2}^s (\vec{x},t) + \operatorname{diag}\! \left(0, \Delta^{-1} \nabla \nabla^t \theta_3 (\vec{x},t), 0 \right), 
	\end{aligned}
\end{equation}
where the auxiliary function $\theta_3(\vec{x},t)$ satisfies
\begin{equation}\label{theta3}
	\begin{aligned}
		\bigl| \theta_3 (\vec{x},t) \bigr| & = \Bigl|- e^{-{\mu}_1 t} \mathbb{Y}_3(\vec{x}) \underset{\vec{x}}{\ast} \bigl[ \omega_t(\vec{x},t; \bar{c}_f) + \left( {\mu}_1 + {\mu}_1 ({\mu}_1/2 + {\mu}_2) t \right) \! \omega(\vec{x},t; \bar{c}_f) \bigr]  \Bigr|\\
		& \le C |\vec{x}|^{-1} e^{-{\mu}_1 t/4-2|\vec{x}|/D_1}.
	\end{aligned}
\end{equation}

Finally, for the remaining term $\bar{G}_{0,h}(\vec{x},t)$, it holds that
\begin{equation}\label{g0hbar}
	\begin{aligned}
		\bar{G}_{0,h} (\vec{x},t) & \le C q_1(t) e^{-{\mu}_2 t} \nabla \Bigl( \mathbb{Y}_3 (\vec{x}) \underset{\vec{x}}{\ast} \mathbb{Y}_3 (\vec{x}) \Bigr)+ C q_3(t) e^{-{\mu}_1 t} \nabla \Bigl( \mathbb{Y}_3(\vec{x}) \underset{\vec{x}}{\ast} \mathbb{Y}_3(\vec{x}) \underset{\vec{x}}{\ast} \omega_t (\vec{x},t; \bar{c}_f )\Bigr)\\
		& \le C |\vec{x}|^{-2} e^{-|\vec{x}|/D_1} \bigl( e^{-{\mu}_1 t/4} + e^{-{\mu}_2 t} \bigr).
	\end{aligned}
\end{equation}
In \eqref{g01bar}--\eqref{g0hbar}, one denotes $D_1 := 4\max \{ 4 \bar{c}_f / {\mu}_1, 1\}$; $q_i(t)$ denotes a polynomial in $t$ of degree no greater than $i$ and $G_{0,2}^s$ is given in \eqref{Gk2s}.

By synthesizing the results in \eqref{gsmodi}--\eqref{g0hbar}, it follows further that 
\begin{equation} \nonumber
	\begin{aligned}
		\mathbb{G}_S(\vec{x},t) &= G_{0,1}^{s,*}(\vec{x},t) +  G_{0,2}^{s}(\vec{x},t) + \operatorname{diag}\! \left( 0, \Delta^{-1} \nabla \nabla^t (\theta_2 + \theta_3)(\vec{x}, t), 0 \right) + \mathcal{O}(1) e^{-(|\vec{x}|+t)/C}\\
		&\quad + \mathcal{O}(1) |\vec{x}|^{-2} e^{-|\vec{x}|/D_1} \bigl( e^{-{\mu}_1 t/4} + e^{-{\mu}_2 t} \bigr),
	\end{aligned}
\end{equation}
which combined with \eqref{Gk2s}, \eqref{greenmodi1}--\eqref{glrefin}, \eqref{g01star}--\eqref{theta2} and \eqref{theta3} verifies the pointwise estimate of the Green's function stated in Theorem~\ref{green}.

\subsection{Proof of Theorem~\ref{lpg}}
\begin{proof}
	First, based on \eqref{g1} and the irrotationality assumption $\nabla \times \vec{h}_0 \equiv 0$, it can be derived that
	$$ \partial_{\vec{x}}^{k} \mathbb{G}^t \vec{g} = \partial_{\vec{x}}^{k} \mathbb{G}_L^t \vec{g} + \partial_{\vec{x}}^{k} \mathbb{G}_S^t \vec{g} + \partial_{\vec{x}}^{k} \operatorname{diag}\!\left( 0, (\mathbb{I}_{3 \times 3} - \Delta^{-1} \nabla \nabla^t ) \delta(\vec{x}), 0 \right) \underset{\vec{x}}{\ast} \vec{g}(\vec{x}) = \partial_{\vec{x}}^{k} \mathbb{G}_L^t \vec{g} + \partial_{\vec{x}}^{k} \mathbb{G}_S^t \vec{g}.$$
	
	Using \eqref{gL} and \eqref{gkstar0}, one establishes the estimate for the long-wave component:
	$$\begin{aligned}
		\partial_{\vec{x}}^{k} \mathbb{G}_L^t \vec{g} = \partial_{\vec{x}}^{k} \mathbb{G}_L(\vec{x},t) \underset{\vec{x}}{\ast} \vec{g}(\vec{x})&=  \Big\lbrace H(t-1) \Bigl[ \tilde{G}_{k}^{*} (\vec{x},t) + \mathcal{O}(1) (1 + t)^{-\frac{5 + |k|}2} e^{-\frac{(|\vec{x}| - \bar{c} t)^2}{Kt}} \Bigr]\\
		& \quad \quad +\operatorname{diag}\! \left(0, \partial_{\vec{x}}^{k} (\Delta^{-1} \nabla \nabla^t \mathbb{G}_{11}(\vec{x},t)), 0 \right) + \mathcal{O}(1) e^{-(|\vec{x}|+t)/D_0} \Big\rbrace \underset{\vec{x}}{\ast} \vec{g}(\vec{x}).
	\end{aligned}
	$$
	Specifically, the component-wise estimates are given as follows: for $i=1,2,3$, 
	$$\begin{aligned}
		\partial_{\vec{x}}^{k} \mathbb{G}_{L,i}^t \vec{g} = \partial_{\vec{x}}^{k} \mathbb{G}_{L,i}(\vec{x},t) \underset{\vec{x}}{\ast} \vec{g}(\vec{x}) & \le C \Bigl( (1 + t )^{-2-d_i-\frac{|k|}{2}} e^{-\frac{(|\vec{x}| - \bar{c} t )^2}{Kt}} + e^{-(|\vec{x}|+t)/D_0}\Bigr) \underset{\vec{x}}{\ast} \left| \vec{g}(\vec{x}) \right|.
	\end{aligned}$$
For $p \ge 2$, and given the index relation $1/p + 1 = 1/r + 1/q$ required by Young's inequality, the $L^p$ estimate of $\partial_{\vec{x}}^{k} \mathbb{G}_L^t \vec{g}$ can be deduced as follows:
	\begin{equation} \nonumber 
		\begin{aligned}
			\left\|\partial_{\vec{x}}^k \mathbb{G}_{L,i}^t \vec{g} \right\|_{L^p} & \le C \Bigl(\Bigl\|(1 + t )^{-2-d_i-\frac{|k|}{2}} e^{-\frac{(|\vec{x}| - \bar{c} t)^2}{Kt}} \Bigr\|_{L^r} + \Bigl\| e^{-(|\vec{x}|+t)/D_0} \Bigr\|_{L^r}\Bigr)\left\| \vec{g}\right\|_{L^q}\\
			& \le C \left( (1 + t )^{-2-d_i-\frac{|k|}{2} + \frac{5}{2r}} + e^{-t/D_0}\right) \left\| \vec{g}\right\|_{L^q} \le C (1 + t )^{\frac{5}{2} \left(\frac{1}{p}-\frac{1}{q} \right) + \frac{1-|k|}{2}-d_i} \left\| \vec{g}\right\|_{L^q},
		\end{aligned}
	\end{equation}
which verifies \eqref{glpl}.

Based on \eqref{gS} and \eqref{leading}, the short-wave component could be estimated as follows:
\begin{equation} \label{gskg}
	\begin{aligned}
		\partial_{\vec{x}}^{k} \mathbb{G}_S^t \vec{g} =  \mathbb{G}_S(\vec{x},t) \underset{\vec{x}}{\ast} \partial_{\vec{x}}^{k} \vec{g}(\vec{x})
		&= \Bigl[ G_{0,1}^{s,*}(\vec{x},t) + G_{0,2}^{s}(\vec{x},t) + \operatorname{diag}\! \left(0, \Delta^{-1} \nabla \nabla^t \theta (\vec{x}, t), 0 \right) + \mathcal{O}(1) e^{-(|\vec{x}| + t)/D_0} \\
		&\quad + \mathcal{O}(1) |\vec{x}|^{-2} e^{-|\vec{x}|/D_1} \bigl( e^{-\mu_1 t/4} + e^{-\mu_2 t} \bigr) \Bigr] \underset{\vec{x}}{\ast} \partial_{\vec{x}}^{k} \vec{g}(\vec{x})\\
		&\le \left( G_{0,1}^{s,*}(\vec{x},t) + G_{0,2}^{s}(\vec{x},t) \right) \underset{\vec{x}}{\ast} \partial_{\vec{x}}^{k} \vec{g}(\vec{x}) + C |\vec{x}|^{-2} e^{-(|\vec{x}| + t)/C} \underset{\vec{x}}{\ast} \left| \partial_{\vec{x}}^{k} \vec{g}(\vec{x}) \right|,
	\end{aligned}
\end{equation}
where $C$ is a positive constant satisfying $C>\max\big\lbrace D_0, D_1, 4/\mu_1, 1/\mu_2 \big\rbrace$. It is obvious that by Young's inequality, for $d \in [0,2], p\ge 2$ and positive constants $D_2, D_3$,
\begin{equation} \label{exponkglp}
	\left\| |\vec{x}|^{-d} e^{-|\vec{x}|/D_2-t/D_3} \underset{\vec{x}}{\ast} \left| \partial_{\vec{x}}^{k} \vec{g}(\vec{x}) \right| \right\|_{L^p} \le C \left\| |\vec{x}|^{-d} e^{-|\vec{x}|/D_2-t/D_3} \right\|_{L^1} \cdot \left\| \partial_{\vec{x}}^{k} \vec{g}(\vec{x}) \right\|_{L^p} \le C e^{-t/D_3} \left\| \partial_{\vec{x}}^{k} \vec{g}(\vec{x}) \right\|_{L^p}.
\end{equation}

One proceeds to compute the convolution of $G_{0,1}^{s,*}(\vec{x},t)$ and $\partial_{\vec{x}}^{k} \vec{g}(\vec{x})$; from \eqref{gS} and \eqref{leading}, one has that
\begin{equation} \label{g01kg}
	\begin{aligned}
		G_{0,1}^{s,*}(\vec{x},t) \underset{\vec{x}}{\ast} \partial_{\vec{x}}^{k} \vec{g}(\vec{x})
		&= \Big\lbrace e^{- \mu_2 t} L_{1,0} \delta (\vec{x}) + e^{- \mu_1 t} \bigl[ L_{2,0} \omega_t(\vec{x}, t; \bar{c}_f) + (L_{3, -1}(t, \nabla) + L_{3,0}) \omega(\vec{x}, t; \bar{c}_f)\bigr] \\
		&\quad + t e^{-\mu_1 t} L_{5,0} \omega(\vec{x},t; \bar{c}_f) \Big\rbrace \underset{\vec{x}}{\ast} \partial_{\vec{x}}^{k} \vec{g}(\vec{x})\\
		&\le e^{- \mu_2 t}
		\begin{pmatrix}
			0\\
			0\\
			-\frac{\bar{a}\bar{v}}{\bar{c}_f^2} \partial_{\vec{x}}^{k} g_0(\vec{x}) + \partial_{\vec{x}}^{k} f_0(\vec{x})
		\end{pmatrix}
		+ C e^{- \mu_1 t} \Bigl( \, \Bigl| \omega_t(\vec{x},t; \bar{c}_f) \underset{\vec{x}}{\ast} \partial_{\vec{x}}^{k} g_0 (\vec{x}) \Bigr|\\
		&\quad + \Bigl| \omega(\vec{x},t; \bar{c}_f) \underset{\vec{x}}{\ast} \partial_{\vec{x}}^{k} \nabla g_0 (\vec{x}) \Bigr| + \Bigl| \omega(\vec{x},t; \bar{c}_f) \underset{\vec{x}}{\ast} \partial_{\vec{x}}^{k} \nabla \cdot \vec{h}_0 (\vec{x}) \Bigr| + (1+t) \Bigl| \omega(\vec{x},t; \bar{c}_f) \underset{\vec{x}}{\ast} \partial_{\vec{x}}^{k} g_0 (\vec{x}) \Bigr|\\
		&\quad + \Bigl| \omega(\vec{x},t; \bar{c}_f) \underset{\vec{x}}{\ast} \partial_{\vec{x}}^{k} f_0 (\vec{x}) \Bigr|\, \Bigr).
	\end{aligned}
\end{equation}
The convolutions involving $\omega, \omega_t$ in the right-hand side of  \eqref{g01kg} could be estimated by Parseval's identity:
\begin{equation} \label{wgl2}
	\begin{aligned}
		& \left\| \omega_t(\vec{x},t;\bar{c}_f) \underset{\vec{x}}{\ast} \partial_{\vec{x}}^{k} \vec{g}(\vec{x}) \right\|_{L^2} = \left\| \mathbb{F}[\omega_t](\vec{\eta},t;\bar{c}_f) \, \vec{\eta}^k \, \mathbb{F}[\vec{g}](\vec{\eta}) \right\|_{L_{\vec{\eta}}^2} \le \left\| \vec{\eta}^k \, \mathbb{F}[\vec{g}](\vec{\eta}) \right\|_{L_{\vec{\eta}}^2} = \left\|\partial_{\vec{x}}^k \vec{g} \right\|_{L^2},\\
		& \left\| \omega(\vec{x},t;\bar{c}_f) \underset{\vec{x}}{\ast} \partial_{\vec{x}}^{k} \vec{g}(\vec{x}) \right\|_{L^2} = \left\| \mathbb{F}[\omega](\vec{\eta},t;\bar{c}_f) \, \vec{\eta}^k \, \mathbb{F}[\vec{g}](\vec{\eta}) \right\|_{L_{\vec{\eta}}^2} = \left\|\frac{\sin(\bar{c}_f |\vec{\eta}| t)}{\bar{c}_f |\vec{\eta}|} \, \vec{\eta}^k \, \mathbb{F}[\vec{g}](\vec{\eta}) \right\|_{L_{\vec{\eta}}^2} \\
		& \quad \quad \quad \quad \quad \quad \quad \quad \quad \quad \, \, \, \le t \left\| \vec{\eta}^k \, \mathbb{F}[\vec{g}](\vec{\eta}) \right\|_{L_{\vec{\eta}}^2} = t \left\|\partial_{\vec{x}}^k \vec{g} \right\|_{L^2},\\
		& \left\| \omega(\vec{x},t;\bar{c}_f) \underset{\vec{x}}{\ast} \partial_{\vec{x}}^{k} \nabla g_0(\vec{x}) \right\|_{L^2} = \left\| \frac{\sin(\bar{c}_f |\vec{\eta}| t)}{\bar{c}_f} \, \vec{\eta}^k \, \mathbb{F}[g_0](\vec{\eta}) \right\|_{L_{\vec{\eta}}^2}  \le \frac{1}{\bar{c}_f} \left\| \vec{\eta}^k \, \mathbb{F}[g_0](\vec{\eta}) \right\|_{L_{\vec{\eta}}^2} = \frac{1}{\bar{c}_f}  \left\|\partial_{\vec{x}}^k g_0 \right\|_{L^2},\\
		& \left\| \omega(\vec{x},t;\bar{c}_f) \underset{\vec{x}}{\ast} \partial_{\vec{x}}^{k} \nabla \cdot \vec{h}_0(\vec{x}) \right\|_{L^2} = \left\| \frac{\sin(\bar{c}_f |\vec{\eta}| t)}{\bar{c}_f |\vec{\eta}|} \, \vec{\eta}^k \vec{\eta}^t \, \mathbb{F}[\vec{h}_0](\vec{\eta}) \right\|_{L_{\vec{\eta}}^2} \le \frac{1}{\bar{c}_f} \left\| \vec{\eta}^k \, \mathbb{F}[\vec{h}_0](\vec{\eta}) \right\|_{L_{\vec{\eta}}^2} = \frac{1}{\bar{c}_f} \left\|\partial_{\vec{x}}^k \vec{h}_0 \right\|_{L^2}.
	\end{aligned}
\end{equation}
Meanwhile, applying the convolution formulas \eqref{wg} and \eqref{wtg}, one derives an $L^\infty$ estimate. Specifically, for a function $F(\vec{x}) \in W^{1,\infty}(\mathbb{R}^3)$, the following estimate holds:
\begin{equation} \label{wgl00}
	\begin{aligned}
		& \omega(\vec{x},t;\bar{c}_f) \underset{\vec{x}}{\ast} F(\vec{x}) = \frac{t}{4 \pi} \int_{|\vec{y}|=1} F(\vec{x} + \bar{c}_f t\vec{y}) \, dS(\vec{y}) \le t \| F \|_{L^\infty},\\
		& \omega_t(\vec{x},t;\bar{c}_f) \underset{\vec{x}}{\ast} F(\vec{x}) = \frac{1}{4 \pi} \int_{|\vec{y}|=1} F(\vec{x} + \bar{c}_f t\vec{y}) \, dS(\vec{y}) + \frac{\bar{c}_f t}{4 \pi} \int_{|\vec{y}|=1} \nabla F(\vec{x} + \bar{c}_f t\vec{y}) \cdot \vec{y} \, dS(\vec{y})\\
		& \quad \quad \quad \quad \quad \quad \quad \ \ \le \| F \|_{L^\infty} + \bar{c}_f t \| \nabla F \|_{L^\infty}. 
	\end{aligned}
\end{equation}
Thus, combining \eqref{g01kg}--\eqref{wgl00}, one deduces that
\begin{equation} \label{g01kgl2}
	\begin{aligned}
		\left\| G_{0,1}^{s,*}(\vec{x},t) \underset{\vec{x}}{\ast} \partial_{\vec{x}}^{k} \vec{g}(\vec{x}) \right\|_{L^2}
		& \le C e^{-\mu_2 t} \left( \left\| \partial_{\vec{x}}^k g_0 \right\|_{L^2} + \left\| \partial_{\vec{x}}^k f_0 \right\|_{L^2} \right) + C e^{-\mu_1 t} \Bigl[ \bigl(1 + t(1+t)\bigr) \left\|\partial_{\vec{x}}^k g_0 \right\|_{L^2} + \left\|\partial_{\vec{x}}^k \vec{h}_0 \right\|_{L^2} \\
		& \quad + t\left\|\partial_{\vec{x}}^k f_0 \right\|_{L^2} \Bigr]\\
		& \le C \left(e^{-\mu_2 t} + e^{-\mu_1 t/2}\right) \left\| \partial_{\vec{x}}^k \vec{g} \right\|_{L^2}
	\end{aligned}
\end{equation}
and
\begin{equation} \label{g01kgl00}
	\begin{aligned}
		\left\| G_{0,1}^{s,*}(\vec{x},t) \underset{\vec{x}}{\ast} \partial_{\vec{x}}^{k} \vec{g}(\vec{x}) \right\|_{L^\infty}
		& \le C e^{-\mu_2 t} \left( \left\| \partial_{\vec{x}}^k g_0 \right\|_{L^\infty} + \left\| \partial_{\vec{x}}^k f_0 \right\|_{L^\infty} \right) + C e^{-\mu_1 t} \Bigl[ \bigl(1 + t(1+t)\bigr) \left\|\partial_{\vec{x}}^k g_0 \right\|_{L^\infty} \\
		& \quad + t \left\|\partial_{\vec{x}}^k \nabla g_0 \right\|_{L^\infty} + t \left\|\partial_{\vec{x}}^k \nabla \cdot \vec{h}_0 \right\|_{L^\infty} + t\left\|\partial_{\vec{x}}^k f_0 \right\|_{L^\infty} \Bigr]\\
		& \le C \left(e^{-\mu_2 t} + e^{-\mu_1 t/2}\right) \left( \left\| \partial_{\vec{x}}^k g_0 \right\|_{L^\infty} + \left\| \partial_{\vec{x}}^k f_0 \right\|_{L^\infty} \right)\\
		& \quad + C e^{-\mu_1 t/2} \left( \left\|\partial_{\vec{x}}^k \nabla g_0 \right\|_{L^\infty} + \left\|\partial_{\vec{x}}^k \nabla \cdot \vec{h}_0 \right\|_{L^\infty} \right).  
	\end{aligned}
\end{equation}

Finally, noting that $\Delta^{-1} \nabla \nabla^t \vec{h}_0 = \vec{h}_0$ under the irrotationality condition $\nabla \times \vec{h}_0 = 0$, the following $L^2$ and $L^\infty$ estimates for the spatial convolution $G_{0,2}^s (\vec{x},t) \underset{\vec{x}}{\ast} \partial_{\vec{x}}^{k} \vec{g}(\vec{x})$ are obtained:
\begin{equation} \label{g02kgl2}
	\begin{aligned}
		\left\| G_{0,2}^s (\vec{x},t) \underset{\vec{x}}{\ast} \partial_{\vec{x}}^{k} \vec{g}(\vec{x}) \right\|_{L^2} & \le C e^{-\mu_1 t} \left( \left\| \omega_t(\vec{x},t; \bar{c}_f) \underset{\vec{x}}{\ast} \partial_{\vec{x}}^{k} \vec{h}_0 (\vec{x}) \right\|_{L^2} + (1+t) \left\| \omega (\vec{x},t; \bar{c}_f) \underset{\vec{x}}{\ast} \partial_{\vec{x}}^{k} \vec{h}_0 (\vec{x}) \right\|_{L^2} \right)\\
		& \le C \bigl( 1 + t(1+t) \bigr) e^{-\mu_1 t} \left\| \partial_{\vec{x}}^{k} \vec{h}_0 \right\|_{L^2} \le C e^{-\mu_1 t/2} \left\| \partial_{\vec{x}}^{k} \vec{h}_0 \right\|_{L^2}
	\end{aligned}
\end{equation}
and
\begin{equation} \label{g02kgl00}
	\begin{aligned}
		\left\| G_{0,2}^s (\vec{x},t) \underset{\vec{x}}{\ast} \partial_{\vec{x}}^{k} \vec{g}(\vec{x}) \right\|_{L^\infty} & \le C e^{-\mu_1 t} \left( \left\| \omega_t(\vec{x},t; \bar{c}_f) \underset{\vec{x}}{\ast} \partial_{\vec{x}}^{k} \vec{h}_0 (\vec{x}) \right\|_{L^\infty} + (1+t) \left\| \omega (\vec{x},t; \bar{c}_f) \underset{\vec{x}}{\ast} \partial_{\vec{x}}^{k} \vec{h}_0 (\vec{x}) \right\|_{L^\infty} \right)\\
		& \le C e^{-\mu_1 t} \left[ \bigl( 1 + t(1+t) \bigr) \left\| \partial_{\vec{x}}^{k} \vec{h}_0 \right\|_{L^\infty} + t \left\| \partial_{\vec{x}}^{k} \nabla \vec{h}_0 \right\|_{L^\infty} \right]\\
		& \le C e^{-\mu_1 t/2} \left(\left\| \partial_{\vec{x}}^{k} \vec{h}_0 \right\|_{L^\infty} + \left\| \partial_{\vec{x}}^{k} \nabla \vec{h}_0 \right\|_{L^\infty} \right). 
	\end{aligned}
\end{equation}

By synthesizing the results from \eqref{gskg}--\eqref{exponkglp} and \eqref{g01kgl2}--\eqref{g02kgl00}, one concludes that there exist positive constants $C_2$ and $C_0$ such that
\begin{equation} \nonumber
	\begin{aligned}
		\left\| \partial_{\vec{x}}^{k} \mathbb{G}_S^t \vec{g} \right\|_{L^2} \le C \left(e^{-\mu_2 t} + e^{-\mu_1 t/2} + e^{-t/C} \right) \left\| \partial_{\vec{x}}^k \vec{g} \right\|_{L^2} \le C_2 e^{-t/C_2} \left\| \partial_{\vec{x}}^k \vec{g} \right\|_{L^2} 
	\end{aligned}
\end{equation}
and
\begin{equation} \nonumber
	\begin{aligned}
		\left\| \partial_{\vec{x}}^{k} \mathbb{G}_S^t \vec{g} \right\|_{L^\infty} & \le C \left(e^{-\mu_2 t} + e^{-\mu_1 t/2} + e^{-t/C} \right) \left\| \partial_{\vec{x}}^k \vec{g} \right\|_{L^\infty} + C e^{-\mu_1 t/2} \left( \left\|\partial_{\vec{x}}^k \nabla g_0 \right\|_{L^\infty} + \left\|\partial_{\vec{x}}^k \nabla \vec{h}_0 \right\|_{L^\infty} \right)\\
		& \le C_0 e^{-t/C_0} \left(\left\| \partial_{\vec{x}}^k \vec{g} \right\|_{L^\infty} + \left\|\partial_{\vec{x}}^k \nabla g_0 \right\|_{L^\infty} + \left\|\partial_{\vec{x}}^k \nabla \vec{h}_0 \right\|_{L^\infty} \right),
	\end{aligned}
\end{equation}
and we finish the proof of \eqref{glps2} and \eqref{glps0}.
\end{proof}

\subsection{Proof of Lemma~\ref{impr}}
\begin{proof}
	From Lemma~3.5 of \cite{z1}, it holds that for $|\vec{\eta}| \le r\ll 1$,
	\begin{equation} \label{Fgl00}
		\begin{aligned}
			&  \left|e^{\lambda_{+}(|\vec{\eta}|)t}P_{+}(\vec{\eta}) + e^{\lambda_{-}(|\vec{\eta}|)t}P_{-}(\vec{\eta}) + e^{\lambda_{d}(|\vec{\eta}|)t}P_{d}(\vec{\eta})\right| \\
			& =\left| e^{\left[-\frac{\bar{\alpha}_1}{2} |\vec{\eta}|^2 + \lambda_R(|\vec{\eta}|^2)\right]t} \Bigl[ \cos\bigl( |\vec{\eta}|\bigl(\bar{c} + \lambda_I(|\vec{\eta}|^2)\bigr)t \bigr) \left( P_{+}(\vec{\eta}) + P_{-}(\vec{\eta}) \right) + i \sin\bigl( |\vec{\eta}|\bigl(\bar{c} + \lambda_I(|\vec{\eta}|^2)\bigr)t \bigr) \right.\\
			& \quad \left.\cdot \left( P_{+}(\vec{\eta}) -  P_{-}(\vec{\eta}) \right) \Bigr] + e^{\left[-\zeta + \bar{\alpha}_1 |\vec{\eta}|^2 - 2\lambda_R(|\vec{\eta}|^2)\right]t} P_d(\vec{\eta}) \right|\\
			& \le C e^{\left[-\frac{\bar{\alpha}_1}{2} |\vec{\eta}|^2 + \lambda_R(|\vec{\eta}|^2)\right]t} \left[ 
			\begin{pmatrix}
				1 & |\vec{\eta}| e_{*}^t & 1\\
				|\vec{\eta}| e_{*} & e_{*} e_{*}^t & |\vec{\eta}| e_{*}\\
				|\vec{\eta}|^2 & |\vec{\eta}| e_{*}^t & |\vec{\eta}|^2
			\end{pmatrix} + 	
			\begin{pmatrix}
				|\vec{\eta}| & e_{*}^t & |\vec{\eta}| \\
				e_{*} & |\vec{\eta}| e_{*} e_{*}^t & e_{*}\\
				|\vec{\eta}| & |\vec{\eta}|^2 e_{*}^t & |\vec{\eta}|
			\end{pmatrix}\right]\\
			 &\quad + C e^{\left[-\zeta + \bar{\alpha}_1 |\vec{\eta}|^2 - 2\lambda_R(|\vec{\eta}|^2)\right]t}
			\begin{pmatrix}
				|\vec{\eta}|^2 & |\vec{\eta}| e_{*}^t & 1\\
				|\vec{\eta}|^3 e_{*} & |\vec{\eta}|^2 e_{*} e_{*}^t & |\vec{\eta}| e_{*}\\
				|\vec{\eta}|^2 & |\vec{\eta}| e_{*}^t & 1
			\end{pmatrix},
		\end{aligned}
	\end{equation}
	where $e_{*} := (1,1,1)^t$, and $\lambda_{\pm}$, 
	$\lambda_{d}$, $0$ and $P_{\pm}$, $P_d$, $P_0$ denote the eigenvalues and eigenprojections corresponding of $E(\vec{\eta}):= -
	\begin{pmatrix}
		0 & \gamma \bar{p} i \vec{\eta}^t & \frac{2}{\alpha \bar{v}}\\
		\bar{v} i \vec{\eta} & 0 & 0\\
		0 & \frac{\alpha_f}{\alpha} \bar{p} \bar{v} i \vec{\eta} & \left(1 + \frac{\alpha_f}{\alpha}\right)
	\end{pmatrix}$, the Fourier transform of the coefficient matrix given in \eqref{greeng}, respectively. Moreover, in \eqref{Fgl00}, $\lambda_R(|\vec{\eta}|^2)=\mathcal{O}(|\vec{\eta}|^4)$ and $\lambda_I(|\vec{\eta}|^2)=\mathcal{O}(|\vec{\eta}|^2)$ are analytic functions of $|\vec{\eta}|^2$. Thus, one has that
	\begin{equation} \label{Fgles}
		\begin{aligned}
			\left|\mathbb{F}[G_L](\vec{\eta},t)\right| &= \left|H(r- |\vec{\eta}|) \left(e^{\lambda_{+}(|\vec{\eta}|)t}P_{+}(\vec{\eta}) + e^{\lambda_{-}(|\vec{\eta}|)t}P_{-}(\vec{\eta}) + e^{\lambda_{d}(|\vec{\eta}|)t}P_{d}(\vec{\eta})\right)\right| \\
			&\le CH(r- |\vec{\eta}|) \left[ e^{-\bar{\alpha}_1 |\vec{\eta}|^2 t/4} 
			\begin{pmatrix}
				1 & e_{*}^t & 1\\
				e_{*} & e_{*} e_{*}^t & e_{*}\\
				|\vec{\eta}| & |\vec{\eta}| e_{*}^t & |\vec{\eta}|
			\end{pmatrix}
		+e^{-\zeta t/2}
			\begin{pmatrix}
				|\vec{\eta}|^2 & |\vec{\eta}| e_{*}^t & 1\\
				|\vec{\eta}|^3 e_{*} & |\vec{\eta}|^2 e_{*} e_{*}^t & |\vec{\eta}| e_{*}\\
				|\vec{\eta}|^2 & |\vec{\eta}| e_{*}^t & 1
			\end{pmatrix}.\right] 
		\end{aligned}
	\end{equation}
	Here, $e^{-\bar{\alpha}_1 |\vec{\eta}|^2 t/4} = e^{\bar{\alpha}_1 |\vec{\eta}|^2/4} \cdot  e^{-\bar{\alpha}_1 |\vec{\eta}|^2 (1+t)/4} \le e^{\bar{\alpha}_1 r^2/4} \cdot e^{-\bar{\alpha}_1 |\vec{\eta}|^2 (1+t)/4}$.
	
	Combining \eqref{Fgles} with Parseval's identity, one obtains that:
	$$\begin{aligned}
		\left\|\partial_{\vec{x}}^k \mathbb{G}_{L,i}^t \vec{g} \right\|_{L^2} = \left\|\mathbb{F}[G_{L,i}](\vec{\eta},t) \, \vec{\eta}^k \, \mathbb{F}[\vec{g}](\vec{\eta}) \right\|_{L_{\vec{\eta}}^2} & \le C(1+t)^{-\frac{|k|}{2}-d_i} \left\|\mathbb{F}[\vec{g}] \right\|_{L_{\vec{\eta}}^2}\\
		& = C(1+t)^{-\frac{|k|}{2}-d_i} \left\|\vec{g}\right\|_{L^2},
	\end{aligned}$$
	which verifies \eqref{impr0} and completes the proof.
\end{proof}

\section{Appendix $\RN{2}$: Energy Estimates for High Order Derivatives}

\begin{proof} [\textbf{Proof of Theorem~\ref{genes}.}] First, assume that the a priori assumption
\begin{equation} \label{prioglo}
		\sup_{0 \le t \le T} \left\|(\rho-\bar{\rho}, \vec{u}, e_1-\bar{e}_1, q-\bar{q})(\cdot, t)\right\|_{H^l} \le \delta_l,
	\end{equation}
	holds for a sufficiently small positive constant $\delta_l$.
	
	For some multi-index $k$ with $1\le |k|\le l$, using the positive definite matrix {{$\begin{pmatrix}1+b&0&-\frac{c_f^2}{a}b\rho\\0&\gamma\rho p&0\\-\frac{c_f^2}{a}b\rho&0&\left(\frac{c_f^2}{a}\rho\right)^2b\end{pmatrix}$}}, applying $\partial^k_{\vec{x}}$ to \eqref{vnep} yields that
	{{\begin{equation} \label{puxkt}
	\begin{cases}
	\left(\partial^k_{\vec x}p\right)_t+\vec u\cdot\nabla\partial^k_{\vec x}p=-\frac{2}{\alpha}\partial^k_{\vec x}\left(\rho\chi\right)-\gamma \partial^k_{\vec x}\left(p\nabla\cdot \vec u\right)-\Big[\partial^k_{\vec x},\vec u\cdot\nabla\Big]p,\\
	\left(\partial^k_{\vec x}\vec u\right)_t+\vec u\cdot\nabla\partial^k_{\vec x}\vec u=-\partial^k_{\vec x}\left(v\nabla p\right)-\Big[\partial^k_{\vec x},\vec u\cdot\nabla\Big]\vec u,\\
	\left(\partial^k_{\vec x}p-\frac{c_f^2}{a}\rho\partial^k_{\vec x}\chi\right)_t+\vec u\cdot\nabla\left(\partial^k_{\vec x}p-\frac{c_f^2}{a}\rho\partial^k_{\vec x}\chi\right)-\left(\frac{c_f^2}a+1\right)\rho \partial^k_{\vec x}\chi\\
	\ \ \ \ \ \ \ \ \ \ \ \ \ \ \ \ \ \ \ \ \ \ \ =\frac{c_f^2}a\rho\nabla\cdot \vec u\partial^k_{\vec x}\chi-\Big[\partial^k_{\vec x},\vec u\cdot\nabla\Big]p+\frac{c_f^2}a\rho\Big[\partial^k_{\vec x},\vec u\cdot\nabla\Big]\chi+\gamma \rho \Big[\partial^k_{\vec x},v\Big] \! \left(p\nabla\cdot \vec u\right)-\frac 2\alpha\Big[\partial^k_{\vec x},\rho\Big]\chi,
	\end{cases}
	\end{equation}
	and
	\begin{equation}\label{eem2}
	\left(\partial^k_{\vec x}s\right)_t+\vec u\cdot\nabla\partial^k_{\vec x}s=\frac 2{\alpha_fR}\partial^k_{\vec x}\left(\frac{\chi^2}{T_1T_2}\right)-\Big[\partial^k_{\vec x},\vec u\cdot\nabla\Big]s.
	\end{equation}}}
	Multiply \eqref{puxkt} by $\left(\partial^k_{\vec x}p,c_f^2\rho^2\partial^k_{\vec x}\vec u,b\left(\partial^k_{\vec x}p-\frac{c_f^2}{a}\rho\partial^k_{\vec x}\chi\right)\right)^t$ to obtain that
	\begin{equation} \label{hl221}
		\begin{aligned}
			& \frac{1}{2} \frac{d}{dt} \left( \left| \partial_{\vec{x}}^k p\right|^2 + \gamma \rho p \left| \partial_{\vec{x}}^k \vec{u} \right|^2 + b\left|\partial_{\vec{x}}^k p - \frac{c_f^2}{a} \rho \partial_{\vec{x}}^k \chi \right|^2 \right) + 2\alpha \frac{c_f^2}{a} \left|\rho \partial_{\vec{x}}^k \chi \right|^2 \\
			& \quad + \frac{1}{2} \nabla \cdot \left(\vec{u} \Biggl( \left| \partial_{\vec{x}}^k p \right|^2 + \gamma \rho p \left|\partial_{\vec{x}}^k \vec{u} \right|^2 + b\left|\partial_{\vec{x}}^k p - \frac{c_f^2}{a} \rho \partial_{\vec{x}}^k \chi \right|^2 \Biggr) + 2 \gamma p \partial_{\vec{x}}^k \vec{u} \, \partial_{\vec{x}}^k p \right) \\
			& = \frac{1}{2} \nabla \cdot \vec{u} \left( (1+b) \left|\partial_{\vec{x}}^k p \right|^2-\rho p \left|\gamma \partial_{\vec{x}}^k \vec{u} \right|^2-b\left|\frac{c_f^2}{a} \rho \partial_{\vec{x}}^k \chi \right|^2 \right) + \gamma \left(\nabla p \cdot \partial_{\vec{x}}^k \vec{u} \, \partial_{\vec{x}}^k p -\frac{\chi}{\alpha} \left| \rho \partial_{\vec{x}}^k \vec{u} \right|^2 \right) \\
			& \quad -\left( (1+b) \partial_{\vec{x}}^k p-b\frac{c_f^2}{a} \rho \partial_{\vec{x}}^k \chi \right) \! \left(\bigl[\partial_{\vec{x}}^k, \vec{u} \cdot \nabla \bigr]p + \frac{2}{\alpha} \bigl[\partial_{\vec{x}}^k, \rho\bigr] \chi \right) \\
			& \quad +b \rho \left(\partial_{\vec{x}}^k p - \frac{c_f^2}{a} \rho \partial_{\vec{x}}^k \chi \right) \! \left( \frac{c_f^2}{a} \bigl[\partial_{\vec{x}}^k, \vec{u} \cdot \nabla\bigr]\chi + \gamma \bigl[\partial_{\vec{x}}^k,v\bigr](p \nabla \cdot \vec{u}) \right) \\
			& \quad -\gamma \partial_{\vec{x}}^k p \, \bigl[\partial_{\vec{x}}^k,p\bigr] (\nabla \cdot \vec{u}) - \gamma \rho p \partial_{\vec{x}}^k \vec{u} \cdot \left( \bigl[\partial_{\vec{x}}^k, \vec{u} \cdot \nabla\bigr] \vec{u} + \bigl[\partial_{\vec{x}}^k,v\bigr]\nabla p\right).
		\end{aligned}
	\end{equation}
Summing \eqref{hl221} over the multi-index $k$ satisfying $|k| \in [m,l]$ and integrating the sum over $\mathbb{R}^3$ with respect to the spatial variable $\vec{x}$, one obtains the following inequality:
	\begin{equation} \label{puxm}
		\frac{d}{dt} \left\|\nabla^m (p, \vec{u}, \chi)(\cdot, t)\right\|_{H^{l-m}}^2 + \left\| \nabla^m \chi (\cdot, t)\right\|_{H^{l-m}}^2 \le C \delta_l \left(\left\|\nabla^m (p, \vec{u}, {{\chi}})(\cdot, t)\right\|_{H^{l-m}}^2 + {{\left\| (\nabla p, \nabla \cdot \vec{u}, \chi)(\cdot, t)\right\|_{L^{\infty}}^2}} \right),
	\end{equation}	
under the assumption \eqref{prioglo} and the Sobolev inequalities \eqref{sob1}--\eqref{sob2}.	
	
	Next, one treats the equations in \eqref{vnep} separately. Applying $\partial^k_{\vec{x}} (|k|\le l-1)$ to the equation for $\chi$ in \eqref{vnep} and then multiplying it by $\partial^k_{\vec{x}} \left(\nabla\cdot\vec u\right)$ yields
	\begin{equation} \label{hl223}
		\begin{aligned}
			\frac{\alpha_f}{\alpha} pv \left| \partial_{\vec{x}}^k (\nabla \cdot \vec{u} )\right|^2 = & -\frac{d}{dt} \left(\partial_{\vec{x}}^k \chi \, \partial_{\vec{x}}^k (\nabla \cdot \vec{u}) \right) + \nabla \cdot \left(\partial_{\vec{x}}^k \chi \, \partial_{\vec{x}}^k \partial_t \vec{u} \right) -\partial_{\vec{x}}^k(\vec{u} \cdot \nabla \vec{u} + v \nabla p) \cdot \nabla \partial_{\vec{x}}^k \chi \\
			& -\left(\zeta \partial_{\vec{x}}^k \chi + \partial_{\vec{x}}^k(\vec{u} \cdot \nabla \chi) + \frac{\alpha_f}{\alpha} [\partial_{\vec{x}}^k, pv](\nabla \cdot \vec{u}) \right) \! \partial_{\vec{x}}^k(\nabla \cdot \vec{u}),
		\end{aligned}
		\end{equation}
where one also makes use of the equations for $\vec u$ in \eqref{vnep}. Similarly, one could also use the equations for $p$ and $\vec u$ to obtain that
	\begin{equation} \label{hl224}
		\begin{aligned}
			v\left|\nabla \partial_{\vec{x}}^k p \right|^2 &= -\frac{d}{dt} \left( \partial_{\vec{x}}^k \vec{u} \cdot \partial_{\vec{x}}^k \nabla p \right) + \nabla \cdot \left(\partial_{\vec{x}}^k \vec{u} \, \partial_{\vec{x}}^k \partial_t p \right) + \partial_{\vec{x}}^k \! \left(\vec{u} \cdot \nabla p + \gamma p \nabla \cdot \vec{u} + \frac{2}{\alpha} \rho \chi \right) \! \partial_{\vec{x}}^k (\nabla \cdot \vec{u}) \\
			&\quad - \left(\partial_{\vec{x}}^k (\vec{u} \cdot \nabla \vec{u}) + \bigl[\partial_{\vec{x}}^k,v\bigr] \nabla p\right) \cdot  \partial_{\vec{x}}^k \nabla p.
		\end{aligned}
	\end{equation}
	Summing \eqref{hl223} and \eqref{hl224} respectively over the multi-index $k$ with $|k| \in [m, l-1]$ and integrating the sums over $\mathbb{R}^3$, one could find a small constant $\tilde{C}_0 \in (0,1)$ by Cauchy-Schwarz inequality such that for any $\tilde{\epsilon}_0 \in (0,\tilde{C}_0)$ the following estimates hold:
	\begin{equation} \label{uimpro}
		\begin{aligned}
			& \frac{d}{dt} \sum_{m \le |k| \le l-1} \int_{\mathbb{R}^3} \partial_{\vec{x}}^k \chi \, \partial_{\vec{x}}^k (\nabla \cdot \vec{u}) \, d\vec{x} + \left\| \nabla^m (\nabla \cdot \vec{u})(\cdot,t) \right\|_{H^{l-m-1}}^2 \\
			& \le C\left( \tilde{\epsilon}_0^{-1} + \delta_l \right) \left\| \nabla^{{{m+1}}} \chi (\cdot,t) \right\|_{H^{l-m-1}}^2 + C \tilde{\epsilon}_0 \left(\left\| \nabla^{m+1} p(\cdot,t) \right\|_{H^{l-m-1}}^2 {{+ \left\| \nabla p(\cdot,t) \right\|_{L^\infty}^2}} \right)\\
			& \quad {{+ C \delta_l \left(\left\| \nabla^{l-2} (\nabla \cdot \vec{u})(\cdot,t) \right\|_{L^2}^2 + \left\| \nabla \vec{u}(\cdot,t) \right\|_{L^\infty}^2 + \left\| \nabla \chi (\cdot,t) \right\|_{L^\infty}^2  \right),}}
		\end{aligned}
	\end{equation}
	\begin{equation} \label{pimpro}
		\begin{aligned}
			& \frac{d}{dt} \sum_{m \le |k| \le l-1} \int_{\mathbb{R}^3} \partial_{\vec{x}}^k \vec{u} \cdot \partial_{\vec{x}}^k \nabla p  \, d\vec{x} + \left\| \nabla^{m+1} p (\cdot,t) \right\|_{H^{l-m-1}}^2 \\
			& \le C\left( 1 + \delta_l \right) \left\| \nabla^{m+1} \vec{u} (\cdot,t) \right\|_{H^{l-m-1}}^2 + C\left(  \left\| \nabla^{m+1} \chi (\cdot,t) \right\|_{H^{l-m-2}}^2 + \left\| \chi (\cdot,t) \right\|_{L^\infty}^2 \right) + C \delta_l  \left\| \nabla p(\cdot,t) \right\|_{L^\infty}^2.
		\end{aligned}
	\end{equation}
	Here one uses the fact that
	\begin{equation} \label{irrotation}
		\left\|\nabla \vec{u}(\cdot,t) \right\|_{H^{k}}=\left\|\nabla \cdot \vec{u}(\cdot,t) \right\|_{H^{k}}\ \ \forall k\in\mathbb{N},
	\end{equation}
	which could be verified by combining the irrotationality of $\vec{u}$ and the Hodge's decomposition
	\begin{equation}\label{hodge}
	\vec u=\Delta^{-1}\nabla\nabla\cdot \vec u-\Delta^{-1}\nabla\times\nabla\times\vec u.
	\end{equation}
Consider \eqref{uimpro}$+\tilde{\epsilon}_1 \, \eqref{pimpro}$ by choosing a sufficiently small constant $\tilde{\epsilon}_1 \in (0,1)$ and if $\tilde{\epsilon}_0 \ll \tilde{\epsilon}_1$ it holds that
	\begin{equation} \label{puimpro}
		\begin{aligned}
			& \frac{d}{dt} \sum_{m \le |k| \le l-1} \int_{\mathbb{R}^3} \left( \partial_{\vec{x}}^k \chi \, \partial_{\vec{x}}^k (\nabla \cdot \vec{u}) + \tilde{\epsilon}_1 \partial_{\vec{x}}^k \vec{u} \cdot  \partial_{\vec{x}}^k \nabla p \right) d\vec{x}  \\
			& \quad + \frac{1}{2} \left\| \nabla^m (\nabla \cdot \vec{u})(\cdot,t) \right\|_{H^{l-m-1}}^2+ \frac{\tilde{\epsilon}_1}{2} \left\| \nabla^{m+1} p (\cdot,t) \right\|_{H^{l-m-1}}^2 \\
			& \le C\left( \tilde{\epsilon}_0^{-1} + \delta_l + \tilde{\epsilon}_1 \right) \left\| \nabla^{m+1} \chi (\cdot,t) \right\|_{H^{l-m-1}}^2 + C \tilde{\epsilon}_1 \left\| \chi (\cdot,t) \right\|_{L^\infty}^2 + C \left( \tilde{\epsilon}_0 + \tilde{\epsilon}_1 \delta_l \right) \left\| \nabla p(\cdot,t) \right\|_{L^\infty}^2 \\
			& \quad + C \delta_l \left( \left\| \nabla \vec{u}(\cdot,t) \right\|_{L^\infty}^2 + \left\| \nabla \chi (\cdot,t) \right\|_{L^\infty}^2  \right),
		\end{aligned}
	\end{equation}
where $l \ge m+2$ is required to ensure that the term $\delta_l\| \nabla^{l-2} (\nabla \cdot \vec{u})(\cdot,t)\|_{L^2}^2$ in \eqref{uimpro} could be controlled by the left-hand side term $\| \nabla^m (\nabla \cdot \vec{u})(\cdot,t)\|_{H^{l-m-1}}^2$ for $0<\delta_l\ll1$. 

Similarly, there exists a sufficiently small constant $\tilde{\epsilon}_2 \in (0,1)$ such that \eqref{puxm} $+ \tilde{\epsilon}_2 \, \eqref{puimpro}$ yields that
	\begin{equation} \label{puxenes}
		\begin{aligned}
			& \frac{d}{dt} \left\|\nabla^m (p, \vec{u}, \chi) (\cdot, t)\right\|_{H^{l-m}}^2 + \tilde{\epsilon}_2 \frac{d}{dt} \sum_{m \le |k| \le l-1} \int_{\mathbb{R}^3} \left( \partial_{\vec{x}}^k \chi \, \partial_{\vec{x}}^k (\nabla \cdot \vec{u}) + \tilde{\epsilon}_1  \partial_{\vec{x}}^k \vec{u} \cdot \partial_{\vec{x}}^k \nabla p \right) d\vec{x} \\
			& \quad + \frac{\tilde{\epsilon}_1 \tilde{\epsilon}_2}{4} \left\| \nabla^{m+1} p(\cdot,t) \right\|_{H^{l-m-1}}^2 + \frac{\tilde{\epsilon}_2}{4} \left\| \nabla^m (\nabla \cdot \vec{u})(\cdot,t) \right\|_{H^{l-m-1}}^2  + \frac{1}{2} \left\| \nabla^m \chi (\cdot,t) \right\|_{H^{l-m}}^2 \\
			& \le C \delta_l \left\|\nabla^m (p, \vec{u})(\cdot, t)\right\|_{L^2}^2 + C \bigl( \delta_l + \tilde{\epsilon}_2 \left( \tilde{\epsilon}_0 + \tilde{\epsilon}_1 \delta_l \right) \bigr) \left\| \nabla p(\cdot,t) \right\|_{L^\infty}^2 + C \delta_l \left( 1 + \tilde{\epsilon}_2 \right) \left\| \nabla \vec{u}(\cdot,t) \right\|_{L^\infty}^2 \\
			& \quad + C \left( \delta_l + \tilde{\epsilon}_1 \tilde{\epsilon}_2 \right)\left\| \chi (\cdot,t) \right\|_{L^\infty}^2 + C \tilde{\epsilon}_2 \delta_l \left\| \nabla \chi (\cdot,t) \right\|_{L^\infty}^2.  
		\end{aligned}
	\end{equation}
	
		Note that $\chi(\vec{x},t)$ decays faster than $p(\vec{x},t)$ and $\vec{u}(\vec{x},t)$ by a factor of $(1+t)^{-1/2}$ and the following procedure is adopted to preserve this difference: Differentiate the equation for $\chi$ in \eqref{vnep} to obtain the equation for $\partial_{\vec{x}}^k\chi$, then multiply it by $\partial_{\vec{x}}^k \chi$ and integrate the product over $\mathbb{R}^3$ to derive that
	\begin{equation} \label{xk}
		\frac{1}{2} \frac{d}{dt} \left\|\partial_{\vec{x}}^k \chi (\cdot, t)\right\|_{L^2}^2 + \zeta \left\|\partial_{\vec{x}}^k \chi (\cdot, t)\right\|_{L^2}^2 = \frac{1}{2} \int_{\mathbb{R}^3} (\nabla \cdot \vec{u}) ( \partial_{\vec{x}}^k \chi )^2 d\vec{x} - \int_{\mathbb{R}^3} \partial_{\vec{x}}^k \chi \left([\partial_{\vec{x}}^k, \vec{u} \cdot \nabla]\chi + \partial_{\vec{x}}^k (a \nabla \cdot \vec{u}) \right) d\vec{x},
	\end{equation}
	where $\zeta$ and $a$ are defined in \eqref{equ}. For $|k|=0$, \eqref{xk} goes back to  
	$$ \frac{d}{dt} \left\| \chi (\cdot, t)\right\|_{L^2}^2 + \left\| \chi (\cdot, t)\right\|_{L^2}^2 \le C\left\| \nabla \cdot \vec{u}(\cdot, t)\right\|_{L^2}^2,$$
	which is obtained in \cite{hl}.
	For $|k| \ge 1$, upon summing \eqref{xk} over $k$ with $|k|=m-1$, one has that 
	\begin{equation} \label{xm1}
		\begin{aligned}
			\frac{d}{dt} \left\|\nabla^{m-1} \chi (\cdot, t)\right\|_{L^2}^2 + \zeta \left\|\nabla^{m-1} \chi (\cdot, t)\right\|_{L^2}^2 & \le C\delta_{m_3}{{ \left( \left\|\nabla \cdot \vec{u} (\cdot, t)\right\|_{L^\infty}^2 +  \left\|\nabla \chi (\cdot, t)\right\|_{L^\infty}^2 \right)}}\\
			& \quad + C \tilde{\epsilon}_3^{-1} \left\|\nabla^{m-1} (\nabla \cdot \vec{u})(\cdot, t)\right\|_{L^2}^2
		\end{aligned}
	\end{equation}
	for $m_3:=\max\{3,m-1\}$ and a sufficiently small constant $\tilde{\epsilon}_3 \in (0,1)$. Here, the requirement $m_3 \ge 3$ comes from the estimate ${{\|\nabla (\vec{u},\chi)(\cdot, t)\|_{L^\infty} \le C\|\nabla^2 (\vec{u}, \chi)(\cdot, t)\|_{H^1}}} \le C \delta_3 \le C \delta_{m_3}$ used in the aforementioned derivation steps.
	
	Finally, one could combine \eqref{puxenes} and \eqref{xm1} to obtain that
	\begin{equation} \label{puximpro0}
		\begin{aligned}
			& \frac{d}{dt} \left( \left\|\nabla^{m-1} \chi(\cdot, t)\right\|_{H^{l-m+1}}^2  + \left\|\nabla^m (p, \vec{u})(\cdot, t)\right\|_{H^{l-m}}^2 \right) + \tilde{\epsilon}_2 \frac{d}{dt}  \sum_{m \le |k| \le l-1} \int_{\mathbb{R}^3} \Bigl( \partial_{\vec{x}}^k \chi \, \partial_{\vec{x}}^k (\nabla \cdot \vec{u})\\
			& \quad + \tilde{\epsilon}_1 \partial_{\vec{x}}^k \vec{u} \cdot  \partial_{\vec{x}}^k \nabla p \Bigr) d\vec{x} + \frac{\tilde{\epsilon}_1 \tilde{\epsilon}_2}{4} \left\| \nabla^m p(\cdot,t) \right\|_{H^{l-m}}^2 + \frac{\tilde{\epsilon}_2}{4} \left\| \nabla^{m-1} (\nabla \cdot \vec{u})(\cdot,t) \right\|_{H^{l-m}}^2  + \frac{1}{2} \left\| \nabla^{m-1} \chi(\cdot,t) \right\|_{H^{l-m+1}}^2 \\
			& \le C \left\|\nabla^m (p, \vec{u})(\cdot, t)\right\|_{L^2}^2 + 
			C \left( \delta_l + \tilde{\epsilon}_0 \tilde{\epsilon}_2 \right) \left\| \nabla p(\cdot,t) \right\|_{L^\infty}^2 + 
			C \delta_l  \left\| \nabla \vec{u}(\cdot,t) \right\|_{L^\infty}^2 + C \left( \delta_l + \tilde{\epsilon}_1 \tilde{\epsilon}_2 \right) \left\| \chi(\cdot,t) \right\|_{W^{1,\infty}}^2,
		\end{aligned}\end{equation}
		where one also adds $\|\nabla^m (p, \vec{u})(\cdot, t)\|_{L^2}^2$ to both sides so that the terms $\| \nabla^{m+1} p(\cdot,t)\|_{H^{j}}^2$, $\| \nabla^{m} (\nabla \cdot \vec{u})(\cdot,t)\|_{H^{j}}^2$ turn into $\|\nabla^m p(\cdot,t)\|_{H^{j+1}}^2$, $\| \nabla^{m-1} (\nabla \cdot \vec{u})(\cdot,t) \|_{H^{j+1}}^2$, respectively, with $j=l-m-1$. 
	
	Take $m=1$ in \eqref{puxenes}, together with the smallness of $\delta_l, \tilde{\epsilon}_0, \tilde{\epsilon}_1$ and $\tilde{\epsilon}_2$, the parameter hierarchy $\tilde{\epsilon}_0 \ll \tilde{\epsilon}_1$, and the Sobolev embedding inequality
	\begin{equation} \label{sobemb}
		\left\| f(\cdot,t) \right\|_{L^\infty} \le C \left\| \nabla f(\cdot,t) \right\|_{H^1} \le \left\| \nabla f(\cdot,t) \right\|_{H^{l-2}} , \quad \text{for } l \ge 3
	\end{equation}
	to obtain that
	\begin{equation} \nonumber
		\begin{aligned}
			& \frac{d}{dt} \left\|\nabla (p, \vec{u}, \chi)(\cdot, t)\right\|_{H^{l-1}}^2 + \tilde{\epsilon}_2 \frac{d}{dt} \sum_{1 \le |k| \le l-1} \int_{\mathbb{R}^3} \left( \partial_{\vec{x}}^k \chi \, \partial_{\vec{x}}^k (\nabla \cdot \vec{u}) + \tilde{\epsilon}_1 \partial_{\vec{x}}^k \vec{u} \cdot  \partial_{\vec{x}}^k \nabla p \right) d\vec{x} \\
			& \quad + \frac{\tilde{\epsilon}_1 \tilde{\epsilon}_2}{8} \left\| \nabla^2 p(\cdot,t) \right\|_{H^{l-2}}^2 + \frac{\tilde{\epsilon}_2}{8} \left\| \nabla (\nabla \cdot \vec{u})(\cdot,t) \right\|_{H^{l-2}}^2  + \frac{1}{4} \left\| \nabla \chi (\cdot,t) \right\|_{H^{l-1}}^2 \le C\delta_l \left\|\nabla (p, \vec{u})(\cdot, t)\right\|_{L^2}^2.
		\end{aligned}
	\end{equation}
	It together with the following basic energy estimates 
	\begin{equation} \nonumber
		\frac{d}{dt} \left\|(\rho-\bar{\rho}, \vec{u}, e_1-\bar{e}_1, q-\bar{q})(\cdot, t)\right\|_{L^2}^2 + c_0 \left\|\chi (\cdot, t)\right\|_{L^2}^2 \le 0,
	\end{equation}
	\begin{equation} \nonumber
		\frac{d}{dt} \int_{\mathbb{R}^3} \left( \chi \, \nabla \cdot \vec{u} + c_1 \vec{u} \cdot \nabla p \right) d\vec{x} + c_2 \left\|(\nabla p, \nabla \cdot \vec{u})(\cdot, t)\right\|_{L^2}^2 \le C \left\|\chi (\cdot, t)\right\|_{H^1}^2,
	\end{equation}
	obtained in Proposition 2.2 of \cite{hl}
	where $c_0,c_1,c_2$ are positive constants, yields that 
	\begin{equation} \label{vueqenes}
		\begin{aligned}
			& \frac{d}{dt} \left( \left\|(\rho-\bar{\rho}, \vec{u}, e_1-\bar{e}_1, q-\bar{q})(\cdot, t)\right\|_{L^2}^2 + \left\|\nabla (p, \vec{u}, \chi)(\cdot, t)\right\|_{H^{l-1}}^2 \right) + \min\left\{\frac{c_0}{2},\frac{1}{8}\right\} \left\|\chi (\cdot, t)\right\|_{H^l}^2 \\
			& \quad + \frac{\tilde{\epsilon}_1 \tilde{\epsilon}_2}{8} \left\| \nabla p(\cdot,t) \right\|_{H^{l-1}}^2 + \tilde{\epsilon}_2 \min\left\{ \frac{c_2}{2}, \frac{1}{8} \right\} \left\| (\nabla \cdot \vec{u})(\cdot,t) \right\|_{H^{l-1}}^2  \le 0,
		\end{aligned}
	\end{equation}
	if $\tilde{\epsilon}_2 \ll \min\{c_0/2, 1/8\}$.
	
	Now we turn to the estimates for the entropy function given in \eqref{vnepre}: multiply \eqref{eem2} by $\partial_{\vec{x}}^k s(\vec{x},t)$ and integrate the product over $\mathbb{R}^3$ to result in that
	\begin{equation} \label{skl2}
		\frac{1}{2} \frac{d}{dt} \left\|\partial_{\vec{x}}^k s(\cdot, t)\right\|_{L^2}^2 = \frac{1}{2} \int_{\mathbb{R}^3} \nabla \cdot \vec{u}\,(\partial_{\vec{x}}^k s)^2 \, d\vec{x} + \int_{\mathbb{R}^3} \partial_{\vec{x}}^k s \left( \frac{2}{\alpha_f R} \partial_{\vec{x}}^k \left( \frac{\chi^2}{T_1T_2}\right)  - [\partial_{\vec{x}}^k, \vec{u} \cdot \nabla] s \right) d\vec{x}. 
	\end{equation}
Summing \eqref{skl2} over $k$ with $1\le|k|\le l$, together with the irrotationality condition \eqref{irrotation} and the Sobolev embedding inequality \eqref{sobemb}, one derives the following estimate:
	$$\begin{aligned}
		\frac{d}{dt} \left\|\nabla s(\cdot, t)\right\|_{H^{l-1}}^2 & \le C \left\|\nabla (\nabla \cdot \vec{u})(\cdot, t)\right\|_{H^{l-2}} \cdot  \left\|\nabla s(\cdot, t)\right\|_{H^{l-1}}^2 + C (1+\delta_l)  \left\|\nabla s(\cdot, t)\right\|_{H^{l-1}} \cdot  \left\|\nabla \chi(\cdot, t)\right\|_{H^{l-1}}^2 \\
		& \le C \left\|\nabla (\nabla \cdot \vec{u})(\cdot, t)\right\|_{H^{l-2}} \cdot  \left\|\nabla s(\cdot, t)\right\|_{H^{l-1}}^2 + C \delta_l \left\|\nabla \chi(\cdot, t)\right\|_{H^{l-1}}^2,
	\end{aligned}$$
	which together with Gronwall's inequality results in that
	\begin{equation} \label{skenest}
		\begin{aligned}
			\left\|\nabla s(\cdot, t)\right\|_{H^{l-1}}^2 &\le \exp\left\{ \int_0^t C \left\|\nabla (\nabla \cdot \vec{u})(\cdot, \tau)\right\|_{H^{l-2}} d\tau \right\} \cdot  \left\|\nabla s_0\right\|_{H^{l-1}}^2 \\
			& \quad + C \delta_l \int_0^t \exp\left\{  \int_{\tau}^t  C \left\|\nabla (\nabla \cdot \vec{u})(\cdot, t_1)\right\|_{H^{l-2}} dt_1 \right\} \cdot \left\|\nabla \chi(\cdot, \tau)\right\|_{H^{l-1}}^2 d\tau.
		\end{aligned}
	\end{equation}
	It requires a decay estimate of $\left\|(\nabla \cdot \vec{u})(\cdot, t)\right\|_{H^{l-1}}$. Define the time-dependent functional
	$$ \mathcal{M}_0(t):= \underset{0 \le \tau \le t}{\sup} (1+\tau)^{\frac{5}{2}} \left( \left\|\chi (\cdot, \tau)\right\|_{H^l}^2 + \left\|\nabla (p, \vec{u})(\cdot, \tau)\right\|_{H^{l-1}}^2 \right).$$ 
	Similar to the proof of Proposition~2.3 in \cite{hl}, applying Theorem~\ref{lpg}, Lemma~\ref{impr} and Duhamel’s principle, one has that
	\begin{equation} \label{puderi1}
		\begin{aligned}
			\left\|\nabla (p, \vec{u})(\cdot, t)\right\|_{L^2} &\le C(1+t)^{-\frac{5}{4}} \left\|(p_0-\bar{p},\vec{u}_0,\chi_0)\right\|_{L^1 \cap H^1} + C \int_0^{t} (1+t-\tau)^{-\frac{5}{4}} \left\|(R_p, R_{\vec{u}}, R_{\chi})(\cdot, \tau)\right\|_{L^1 \cap H^1} d\tau\\
			& \le C(1+t)^{-\frac{5}{4}} \left\|(p_0-\bar{p},\vec{u}_0,\chi_0)\right\|_{L^1 \cap H^1} + C \delta_3 \int_0^{t} (1+t-\tau)^{-\frac{5}{4}} (1+\tau)^{-\frac{5}{4}} \sqrt{\mathcal{M}_0(\tau)} d\tau\\
			& \le C(1+t)^{-\frac{5}{4}} \left( \left\|(p_0-\bar{p},\vec{u}_0,\chi_0)\right\|_{L^1 \cap H^1} + \delta_3 \sqrt{\mathcal{M}_0(t)}\right). 
		\end{aligned}
	\end{equation}
Taking $m=1$ in \eqref{puximpro0}, applying Gronwall's inequality and then using \eqref{puderi1}, one has that
	\begin{equation} \nonumber
		\begin{aligned} \left\|\chi (\cdot, t)\right\|_{H^l}^2 + \left\|\nabla (p, \vec{u})(\cdot, t)\right\|_{H^{l-1}}^2
	&\le Ce^{-t/C}\left(\left\|\chi (\cdot, 0)\right\|_{H^l}^2 + \left\|\nabla (p, \vec{u})(\cdot, 0)\right\|_{H^{l-1}}^2\right)\\
	&+
	C\int_0^t e^{-(t-\tau)/C}(1+\tau)^{-\frac 52}\left\|(p_0-\bar{p},\vec{u}_0,\chi_0)\right\|_{L^1 \cap H^1}^2 d\tau\\
	&+
	C\left(\delta_3\right)^2\mathcal{M}_0(t)\int_0^t e^{-(t-\tau)/C}(1+\tau)^{-\frac 52}d\tau\\
	&\le C(1+t)^{-\frac{5}{2}} \left(\left\|(p_0-\bar{p},\vec{u}_0,\chi_0)\right\|_{L^1 \cap H^l}^2 + \left(\delta_3\right)^2\mathcal{M}_0(t) \right) 
	\end{aligned}
	\end{equation}
	and thus
	$$\mathcal{M}_0(t) \le C \left\|(p_0-\bar{p},\vec{u}_0,\chi_0) \right\|_{L^1 \cap H^l}^2,$$ which yields that
	$$\left\|\chi (\cdot, t)\right\|_{H^l}+ \left\|\nabla (p, \vec{u})(\cdot, t)\right\|_{H^{l-1}} \le C(1+t)^{-\frac{5}{4}} \left\|(p_0-\bar{p},\vec{u}_0,\chi_0) \right\|_{L^1 \cap H^l}.$$
	It combined with \eqref{skenest} results in that
	\begin{equation} \label{skenes}
		\left\|\nabla s(\cdot, t)\right\|_{H^{l-1}}^2 \le  e^{C\left\| (p_0 - \bar{p}, \vec{u}_0, \chi_0 ) \right\|_{L^1 \cap H^l}} \left\|\nabla s_0\right\|_{H^{l-1}}^2 + C \delta_l \int_0^t \left\|\nabla \chi(\cdot, \tau)\right\|_{H^{l-1}}^2 d\tau.
	\end{equation}
	
	Finally one could combine \eqref{vueqenes}, \eqref{skenes} and the smallness of $\delta_l$ to obtain that
	$$\begin{aligned}
		& \left\|(\rho-\bar{\rho}, \vec{u}, e_1-\bar{e}_1, q-\bar{q})(\cdot, t)\right\|_{L^2}^2 + \left\|\nabla (p, \vec{u}, \chi, s)(\cdot, t)\right\|_{H^{l-1}}^2  \\
		& \quad + \int_0^t \Bigl(\left\|\chi(\cdot, \tau)\right\|_{H^l}^2+ \left\| (\nabla p, \nabla \cdot \vec{u})(\cdot,\tau) \right\|_{H^{l-1}}^2 \Bigr) \, d\tau \\
		& \le C \left(  \left\|(\rho_0-\bar{\rho}, \vec{u}_0, e_{1,0}-\bar{e}_1, q_0-\bar{q})\right\|_{L^2}^2 + \left\|\nabla (p_0, \vec{u}_0, \chi_0) \right\|_{H^{l-1}}^2 \right.\\
		& \left. \quad \quad \quad+ e^{C \left\| (p_0 - \bar{p}, \vec{u}_0, \chi_0 ) \right\|_{L^1 \cap H^l}} \left\|\nabla s_0\right\|_{H^{l-1}}^2 \right)
	\end{aligned}$$
	for a positive constant $C$. It verifies the a priori assumption \eqref{prioglo} 
and we also finish the proof of Theorem~\ref{genes}.
\end{proof}

\section{Appendix $\RN{3}$: Proofs of Computational Lemmas in Section 2}

In this section, we provide the proofs for all the computational lemmas used for the wave propagations.

\subsection{Proof of Lemma~\ref{lemmains}}

\begin{proof}
For $d \ge 0$ , $i \ge 0$, and $m > 1$, consider the region $|\vec x|\le2\bar{c}t$ and one has that
\begin{equation}\nonumber
	\begin{aligned}
		& \int_0^t (1+\tau)^{-d} (1 + |\vec{x}| + \sqrt{\tau})^{-i} \left( (|\vec{x}|-\bar{c}\tau)^2 + \tau + 1 \right)^{-m} d\tau \\
		&=\left( \int_0^{|\vec{x}|/(2\bar{c})} + \int_{|\vec{x}|/(2\bar{c})}^{\min\{2|\vec{x}|/\bar{c},t\}} + \int_{\min\{2|\vec{x}|/\bar{c},t\}}^t \right) (1+\tau)^{-d} (1 + |\vec{x}| + \sqrt{\tau})^{-i} \left( (|\vec{x}|-\bar{c}\tau)^2 + \tau + 1 \right)^{-m} d\tau \\
		&\le (|\vec{x}| + 1)^{-i} \left[ (|\vec{x}|^2/4 + 1 )^{-m}  \int_0^{|\vec{x}|/(2\bar{c})} (1+\tau)^{-d} d\tau \right.\\
		& \quad + \big(|\vec{x}|/(2\bar{c}) + 1 \big)^{-d-m} \int_{|\vec{x}|/(2\bar{c})}^{\min\{2|\vec{x}|/\bar{c},t\}} \left( \left( \frac{\bar{c} \tau}{\sqrt{2|\vec x|/\bar{c}+1}}-\frac{|\vec x|}{\sqrt{2|\vec x|/\bar{c}+1}}\right)^2 + 1 \right)^{-m} d\tau \\
		& \quad \left.+ ( 2|\vec{x}|/\bar{c} + 1)^{-d}  \int_{\min\{2|\vec{x}|/\bar{c},t\}}^t (|\vec{x}|^2 + \tau + 1)^{-m} d\tau \right] \\
		&\le C(|\vec{x}| + 1)^{-i} \left[(|\vec{x}|^2/4 + 1)^{-m} \int_0^{|\vec{x}|/(2\bar{c})} (1+\tau)^{-d} d\tau +  \sqrt{|\vec x|+1} \big( |\vec{x}|/(2\bar{c}) + 1 \big)^{-d-m} \right. \\
		& \quad \left.+ ( 2|\vec{x}|/\bar{c} + 1)^{-d} \left(|\vec{x}|^2+  1  \right)^{-m+1} \right]\\
		&\le C (|\vec{x}|^2 + 1)^{-\frac{i}{2}-m} 
		\begin{cases}
			(1+|\vec{x}|/(2\bar{c}))^{-d+1} & \text{for } 0 \le d < 1 \\
			\ln(1+|\vec{x}|/(2\bar{c})) & \text{for } d=1\\
			1 & \text{for } d>1
		\end{cases}\\
		& \quad + C(|\vec{x}|^2 + 1)^{-\frac{i+d+m}{2}+\frac{1}{4}} + C(|\vec{x}|^2 + 1)^{-\frac{i+d}{2}-m+1},
	\end{aligned}
\end{equation}
and \eqref{tconv1} is verified for $|\vec x|\le2\bar{c}t$. For $|\vec x|>2\bar{c}t$, it is trivial since it holds that $(|\vec{x}|-\bar{c}\tau)^2 + \tau + 1 \ge |\vec x|^2/4+\tau+1$ and thus
\begin{equation}\nonumber
	\begin{aligned}
		& \int_0^t (1+\tau)^{-d} (1 + |\vec{x}| + \sqrt{\tau})^{-i} \left( (|\vec{x}|-\bar{c}\tau)^2 + \tau + 1 \right)^{-m} d\tau \\
		&\le C (|\vec{x}|^2 + 1)^{-\frac{i}{2}-m} \int_0^t (1+\tau)^{-d}d\tau\\
		&\le C (|\vec{x}|^2 + 1)^{-\frac{i}{2}-m} 
		\begin{cases}
			(1+|\vec{x}|/(2\bar{c}))^{-d+1} & \text{for } 0 \le d < 1, \\
			\ln(1+|\vec{x}|/(2\bar{c})) & \text{for } d=1, \\
			1 & \text{for } d>1.
		\end{cases}
	\end{aligned}
\end{equation}

To verify \eqref{tconv2}, one only needs to consider the region $|\vec x|\le\bar{c}t$, since $H(\bar{c}\tau - |\vec{x}|)\equiv 0$ for $|\vec x|>\bar{c}t$ and $\tau\in[0,t]$: 
\begin{equation}\nonumber
	\begin{aligned}
		& \int_0^t (1+\tau)^{-d} \frac{H(\bar{c}\tau - |\vec{x}|)}{1 + (|\vec{x}| + \sqrt{\tau})^i(\bar{c}\tau - |\vec{x}| + \sqrt{\tau})^j} d\tau \\
		& \le C H(\bar{c}t - |\vec{x}|)  (|\vec{x}| + 1)^{-d}  \int_{|\vec{x}|/\bar{c}}^t  (1+|\vec{x}| + \sqrt{\tau})^{-i}(\bar{c}\tau - |\vec{x}| + \sqrt{\tau} + 1)^{-j} d\tau \\
		& \le C H(\bar{c}t - |\vec{x}|) (|\vec{x}| + 1)^{-d-i} \left( \int_{|\vec{x}|/\bar{c}}^{\min\{2|\vec{x}|/\bar{c},t\}} + \int_ {\min\{2|\vec{x}|/\bar{c},t\}}^t \right) (\bar{c}\tau - |\vec{x}| + \sqrt{\tau} + 1)^{-j} d\tau\\
		& \le C H(\bar{c}t - |\vec{x}|) (|\vec{x}| + 1)^{-d-i} \left[\int_{|\vec{x}|/\bar{c}}^{\min\{2|\vec{x}|/\bar{c},t\}} (\sqrt{\tau}+1)^{-j} d\tau + \int_{\min\{2|\vec{x}|/\bar{c},t\}}^t (\bar{c} \tau/2 +1)^{-j} d\tau \right] \\
		& \le C H(\bar{c}t - |\vec{x}|) (|\vec{x}| + 1)^{-d-i}	\begin{cases}
			\ln(1+2|\vec{x}|/\bar{c}) + (|\vec{x}|+1)^{-1} & \text{for } j=2, \\
			(|\vec{x}|+1)^{-\frac{j}{2}+1} & \text{for } j>2.
		\end{cases}
	\end{aligned}
\end{equation}
\end{proof}

\subsection{Proof of Lemma~\ref{lemmain}}
\begin{proof}
First, it is obvious that \eqref{iniconv2} holds since $\displaystyle \int_{\mathbb{R}^3}|\vec x|^{-j}e^{-\mu|\vec x|}d\vec x\le C$. 

Next, one could verify \eqref{iniconv1} as follows:
\begin{equation}\nonumber
\begin{aligned}
& \int_{\mathbb{R}^3}e^{-\frac{(|\vec{x}-\vec y|-\bar{c}t)^2}{K(1+t)}} (|\vec{y}|^2+1)^{-d}d\vec y \\
&=\left( \int_{\left|\vec y\right|\le \big||\vec x|-\bar{c}t\big|/2}+\int_{\left|\vec y\right|\ge \big||\vec x|-\bar{c}t\big|/2}\right)e^{-\frac{(|\vec{x}-\vec y|-\bar{c}t)^2}{K(1+t)}} (|\vec{y}|^2+1)^{-d}d\vec y \\
&\le\left(H\big( \sqrt{t}-\bigl||\vec x|-\bar{c}t\bigr|\big)+ e^{-\frac{(|\vec{x}|-\bar{c}t)^2}{4K(1+t)}}\right)\int_{\mathbb{R}^3} (|\vec{y}|^2+1)^{-d}d\vec y\\
&\quad+C H\big(\bigl||\vec x|-\bar{c}t\bigr|-\sqrt{t}\big) \big(\big||\vec x|-\bar{c}t\big|^2+1\big)^{-d} \int_{\mathbb{R}^3}e^{-\frac{(|\vec{x}-\vec y|-\bar{c}t)^2}{K(1+t)}}d\vec y\\
&\le Ce^{-\frac{(|\vec{x}|-\bar{c}t)^2}{4K(1+t)}}+C(1+t)^{\frac{5}{2}}\left((|\vec x|-\bar{c}t)^2+t+1\right)^{-d}.
\end{aligned}
\end{equation}
Here one uses $\displaystyle\int_{\mathbb{R}^3}e^{-\frac{(|\vec{x}-\vec y|-\bar{c}t)^2}{K(1+t)}}d\vec y\le C(1+t)^{5/2}$ which could be proved similar to Lemma~5.2 in \cite{liuwang}.

Finally, for $|\vec x|\ge 2\bar{c}_ft$ and $|\vec y|=1$, it holds that $|\vec{x}+\bar{c}_ft\vec y|\ge |\vec x|-\bar{c}_ft|\vec y|>|\vec x|/2$; for $|\vec x|\le 2\bar{c}_ft$, one has that $e^{-\mu t}\le e^{-\mu t/2-\mu |\vec x|/(4\bar{c}_f)}$. Thus one obtains that
\begin{equation}\nonumber
\begin{aligned}
&e^{-\mu t} \omega(\vec{x},t;\bar{c}_f) \underset{\vec{x}}{\ast} (|\vec{x}|^2+1)^{-d}\le Cte^{-\mu t}\int_{|\vec y|=1}
(|\vec{x}+\bar{c}_ft\vec y|^2+1)^{-d}dS(\vec y)\\
&\le Cte^{-\mu t/2}\Big((|\vec{x}|^2+1)^{-d}+e^{-\mu |\vec x|/(4\bar{c}_f)}\Big)\int_{|\vec y|=1}dS(\vec y)\\
&\le Cte^{-\mu t/2}\Big((|\vec{x}|^2+1)^{-d}+e^{-\mu |\vec x|/(4\bar{c}_f)}\Big)
\end{aligned}
\end{equation}
and verifies \eqref{iniconv3}.
\end{proof}

\subsection{Proof of Lemma~\ref{lemmainker}}

\begin{proof}
In \eqref{heke2}, one has that
$$e^2_1(\vec x,t)\le C\left(\left(\left(|\vec x|-\bar{c}t\right)^2+t+1\right)^{-4}+\frac{H(\bar{c}t-|\vec x|)}{1+\left(|\vec x|+\sqrt{t}\right)^2\left(\bar{c}t-|\vec x|+\sqrt{t}\right)^4}\right)$$
and thus one considers the following integral first:
\begin{equation} \label{123ker}
		\begin{aligned}
			&\int_0^t \int_{\mathbb{R}^3} (1+t-\tau)^{-2-\frac{m}{2}} e^{-\frac{(|\vec{x}-\vec{y}|-\bar{c}(t-\tau))^2}{K(t-\tau+1)}} (1+\tau)^{-\frac{1}{2}} \left((|\vec y|-\bar{c}\tau)^2 + \tau + 1 \right)^{-4} d\vec{y} d\tau\\
			& = H(|\vec{x}|-\bar{c}t) \int_0^t \left( \int_{|\vec{y}| -\bar{c}\tau \le (|\vec{x}|-\bar{c}t)/2} +  \int_{|\vec{y}| -\bar{c}\tau > (|\vec{x}|-\bar{c}t)/2}\right)  (1+t-\tau)^{-2-\frac{m}{2}} e^{-\frac{(|\vec{x}-\vec{y}|-\bar{c}(t-\tau))^2}{K(t-\tau+1)}} (1+\tau)^{-\frac{1}{2}} \\
			& \quad \quad \quad \quad \quad \quad \quad \times \left((|\vec{y}|-\bar c\tau)^2 + \tau + 1 \right)^{-4} d\vec{y} d\tau\\
			& \quad + H(\bar{c}t-|\vec{x}|) \left( \int_0^{(t-|\vec{x}|/\bar{c}+\sqrt{t})/16}+\int_{(t-|\vec{x}|/\bar{c}+\sqrt{t})/16}^{t/8}\right)  \int_{\mathbb{R}^3} (1+t-\tau)^{-2-\frac{m}{2}} e^{-\frac{(|\vec{x}-\vec{y}|-\bar{c}(t-\tau))^2}{K(t-\tau+1)}} (1+\tau)^{-\frac{1}{2}} \\
			& \quad \quad \quad \quad \quad \quad \quad \quad \times \left((|\vec{y}|-\bar{c}\tau)^2 + \tau + 1 \right)^{-4} d\vec{y} d\tau\\
			& \quad + H(\bar{c}t-|\vec{x}|) \int_{t/8}^t \int_{\mathbb{R}^3} (1+t-\tau)^{-2-\frac{m}{2}} e^{-\frac{(|\vec{x}-\vec{y}|-\bar{c}(t-\tau))^2}{K(t-\tau+1)}} (1+\tau)^{-\frac{1}{2}} \left((|\vec{y}|-\bar{c}\tau)^2 + \tau + 1 \right)^{-4} d\vec{y} d\tau\\
			&:=(\RN{1})+(\RN{2})+(\RN{3}).
		\end{aligned}
	\end{equation}
	
	As shown in \cite{liuwang}, it holds that
	\begin{equation} \nonumber
		\begin{aligned}
			&\int_{\mathbb{R}^3} (1+\tau)^{-\frac{1}{2}} \left((|\vec y|-\bar c\tau)^2 + \tau + 1 \right)^{-2} d\vec{y}\\
			& \le C \int_0^{\infty} r^2 (1+\tau)^{-\frac{5}{2}} \left(\frac{(r-\bar c\tau)^2}{\tau+1} +1 \right)^{-2} dr\\
			& \le C(1+\tau)^{-\frac{3}{2}} \int_0^{\infty} \left(\frac{(r-\bar c\tau)^2}{\tau+1} + \bar{c}^2 \tau \right) \left(\frac{(r-\bar c\tau)^2}{\tau+1} +1 \right)^{-2} dr\\
			& \le C(1+\tau)^{-1} \left[ \int_0^{\infty} (r^2+1)^{-1}  d\tau + \tau \int_0^{\infty} (r^2+1)^{-2} dr\right] \le C.
		\end{aligned}
	\end{equation}
	Moreover, for $|\vec{x}| \ge \bar{c} t$ and $|\vec{y}| - \bar{c}\tau \le (|\vec{x}| - \bar{c}t)/2$, one has that $|\vec{x} - \vec{y}| - \bar{c}(t - \tau) \ge (|\vec{x}| - \bar{c}t)/2$. Hence,
	\begin{equation} \label{1ker}
		\begin{aligned}
			(\RN{1}) &\le C\left( e^{-\frac{(|\vec{x}| - \bar{c}t)^2}{8K(t + 1)}} + \left(\frac{(|\vec x|-\bar{c}t)^2}{4(t+1)}+1 \right)^{-2}\right)\\
			& \quad \times \left[ \int_0^{t/2} (1+t-\tau)^{-2-\frac{m}{2}} (1+\tau)^{-2} \int_{\mathbb{R}^3} (1+\tau)^{-\frac{1}{2}} \left((|\vec y|-\bar c\tau)^2 + \tau + 1 \right)^{-2} d\vec{y} d\tau \right.\\
			&\quad \quad \quad \left.+ \int_{t/2}^t (1+t-\tau)^{\frac{1}{2}-\frac{m}{2}} (1+\tau)^{-\frac{9}{2}} \int_{\mathbb{R}^3} (1+t-\tau)^{-\frac{5}{2}} e^{-\frac{(|\vec{x}-\vec{y}|-\bar{c}(t-\tau))^2}{2K(t-\tau+1)}} d\vec{y} d\tau \right]\\
			& \le C\left(\frac{(|\vec{x}|-\bar{c}t)^2}{t+1}+1 \right)^{-2} \left[(1+t)^{-2-\frac{m}{2}} + (1+t)^{-\frac{9}{2}} \begin{cases}
				(1+t/2)^{\frac{3-m}{2}} & \text{for } 0 \le m < 3 \\
				\ln(1+t/2) & \text{for } m=3 \\
				1 & \text{for } m > 3
			\end{cases} \right]\\
			& \le C (1+t)^{-\frac{\min\{m,5\}}{2}}\left((|\vec{x}|-\bar{c}t)^2 +t+1\right)^{-2}.
		\end{aligned}
	\end{equation}

	For $\tau \le (t - |\vec{x}|/\bar{c} + \sqrt{t})/16 \le t/8$, when $|\vec{y}| \le (\bar{c}t - |\vec{x}| + \bar{c}\sqrt{t})/8$, one has that:
	$$\bar{c} (t-\tau)-|\vec{x}-\vec{y}| +\bar{c} \sqrt{t-\tau} \ge \bar{c}t -|\vec{x}|- \frac{3(\bar{c}t-|\vec{x}|+\bar{c} \sqrt{t})}{16} + \bar{c} \sqrt{\frac{7t}{8}}= \frac{13(\bar{c}t-|\vec x|)}{16}+\left(\frac{\sqrt{14}}{4} - \frac{3}{16}\right)\bar{c} \sqrt{t}$$
	and thus
	\begin{equation} \label{exieq}
		e^{-\frac{(|\vec{x}-\vec{y}|-\bar{c}(t-\tau))^2}{K(t-\tau+1)}} \le e^{-\frac{(|\vec{x}-\vec{y}|-\bar{c}(t-\tau))^2 + \bar{c}^2 (t-\tau) }{K(t-\tau+1)}+\frac{\bar{c}^2}{K}} \le C e^{-\frac{(|\vec{x}-\vec{y}|-\bar{c}(t-\tau)+\bar{c}\sqrt{t-\tau})^2  }{2K(t-\tau+1)}} \le Ce^{-\frac{(|\vec{x}|-\bar{c}t)^2}{8K(t+1)}}.
	\end{equation}
	When $|\vec y| \ge (\bar{c}t-|\vec x|+\bar{c}\sqrt{t})/8$, one has
	\begin{equation} \nonumber 
		|\vec{y}|-\bar{c} \tau \ge (\bar{c}t-|\vec{x}|+\bar{c}\sqrt{t})/16.
	\end{equation}
	Thus,
	\begin{equation} \label{2ker}
		\begin{aligned}
			(\RN{2}) &\le C H(\bar{c}t-|\vec{x}|)\left[ \left( e^{-\frac{(|\vec{x}|-\bar{c}t)^2}{8K(t+1)}} + \left(\frac{(|\vec{x}|-\bar{c}t)^2}{(16)^2(t+1)}+1 \right)^{-2} \right) \right.\\ 
			&\quad \times \int_0^{t/8} (1+t-\tau)^{-2-\frac{m}{2}}(1+\tau)^{-2} \int_{\mathbb{R}^3} (1+\tau)^{-\frac{1}{2}} \left((|\vec{y}|-\bar c\tau)^2 + \tau + 1 \right)^{-2} d\vec{y} d\tau\\
			&\quad \left.+ \int_{(t-|\vec{x}|/\bar{c}+\sqrt{t})/16}^{t/8} (1+t-\tau)^{-2-\frac{m}{2}}(1+\tau)^{-2} \int_{\mathbb{R}^3} (1+\tau)^{-\frac{1}{2}} \left((|\vec{y}|-\bar{c}\tau)^2 + \tau + 1 \right)^{-2} d\vec{y} d\tau \right]\\
			& \le C H(\bar{c}t-|\vec{x}|)(1+t)^{-2-\frac{m}{2}}\left[ \left(\frac{(|\vec{x}|-\bar{c}t)^2}{t+1}+1 \right)^{-2} + (\bar{c}t-|\vec{x}|+\sqrt{t+1})^{-1}\right].
		\end{aligned}
	\end{equation}

	The final part $(\RN{3})$ could be estimated by straightforward computations:
	\begin{equation} \label{3ker}
		\begin{aligned}
			(\RN{3}) &\le H(\bar{c}t-|\vec{x}|) \int_{t/8}^t (1+t-\tau)^{\frac{1}{2}-\frac{m}{2}} (1+\tau)^{-\frac{9}{2}} \int_{\mathbb{R}^3} (1+\tau)^{-\frac{5}{2}} e^{-\frac{(|\vec{x}-\vec{y}|-\bar{c}(t-\tau))^2}{K(t-\tau+1)}} d\vec{y} d\tau\\
			& \le C H(\bar{c}t-|\vec{x}|)(1+t)^{-\frac{9}{2}} \begin{cases}
				(1+t)^{\frac{3-m}{2}} & \text{for } 0 \le m < 3, \\
				\ln(1+t) & \text{for } m=3, \\
				1 & \text{for } m > 3.
			\end{cases}
		\end{aligned}
	\end{equation}

	Combining \eqref{123ker}, \eqref{1ker}, \eqref{2ker}, and \eqref{3ker}, one has that
	\begin{equation} \label{kere11}
		\begin{aligned}
			&\int_0^t \int_{\mathbb{R}^3} (1+t-\tau)^{-2-\frac{m}{2}} e^{-\frac{(|\vec{x}-\vec{y}|-\bar{c}(t-\tau))^2}{K(t-\tau+1)}} (1+\tau)^{-\frac{1}{2}} \left((|\vec{y}|-\bar{c}\tau)^2 + \tau + 1 \right)^{-4} d\vec{y} d\tau\\
			& \le C(1+t)^{-\frac{\min\{m,5\}}{2}}\left((|\vec{x}|-\bar{c}t)^2 +t+1\right)^{-2}\\
			&\quad + C H(\bar{c}t - |\vec{x}|)\left[(1+t)^{-2-\frac{m}{2}} \left( (|\vec{x}| - \bar{c} t)^2 + t+1\right)^{-\frac{1}{2}} + (1+t)^{-\frac{9}{2}} 
			\begin{cases}
				(1+t)^{\frac{3-m}{2}} & \text{for } 0 \le m < 3, \\
				\ln(1+t) & \text{for } m=3, \\
				1 & \text{for } m > 3.
			\end{cases} \right]
		\end{aligned}
	\end{equation}

	To verify \eqref{heke2}, it remains to consider the integral below:
	\begin{equation} \label{456ker}
		\begin{aligned}
			&\int_0^t \int_{\mathbb{R}^3} (1+t-\tau)^{-2-\frac{m}{2}} e^{-\frac{(|\vec{x}-\vec{y}|-\bar{c}(t-\tau))^2}{K(t-\tau+1)}} (1+\tau)^{-\frac{1}{2}} \frac{H(\bar{c}\tau - |\vec{y}|)}{1 + (|\vec{y}| + \sqrt{\tau})^2(\bar{c}\tau - |\vec{y}| + \sqrt{\tau})^4} d\vec{y} d\tau\\
			& = H(|\vec{x}|-\bar{c}t) \left(\int_0^{t/2} + \int_{t/2}^t \right)\int_{|\vec{y}|\le \bar{c}\tau} (1+t-\tau)^{-2-\frac{m}{2}} e^{-\frac{(|\vec{x}-\vec{y}|-\bar{c}(t-\tau))^2}{K(t-\tau+1)}} \frac{(1+\tau)^{-\frac{1}{2}}}{1 + (|\vec{y}| + \sqrt{\tau})^2(\bar{c}\tau - |\vec{y}| + \sqrt{\tau})^4} d\vec{y} d\tau\\
			& \quad + H(\bar{c}t-|\vec{x}|) \left( \int_0^{(t-|\vec{x}|/\bar{c}+\sqrt{t})/16}+\int_{(t-|\vec{x}|/\bar{c}+\sqrt{t})/16}^{t/8}\right)  \int_{|\vec{y}|\le \bar{c}\tau} (1+t-\tau)^{-2-\frac{m}{2}} e^{-\frac{(|\vec{x}-\vec{y}|-\bar{c}(t-\tau))^2}{K(t-\tau+1)}} (1+\tau)^{-\frac{1}{2}} \\
			& \quad \quad \quad \quad \quad \quad \quad \quad \times \frac{1}{1 + (|\vec{y}| + \sqrt{\tau})^2(\bar{c}\tau - |\vec{y}| + \sqrt{\tau})^4} d\vec{y} d\tau\\
			& \quad + H(\bar{c}t-|\vec{x}|) \int_{t/8}^t \int_{|\vec{y}|\le \bar{c}\tau} (1+t-\tau)^{-2-\frac{m}{2}} e^{-\frac{(|\vec{x}-\vec{y}|-\bar{c}(t-\tau))^2}{K(t-\tau+1)}} \frac{(1+\tau)^{-\frac{1}{2}}}{1 + (|\vec{y}| + \sqrt{\tau})^2(\bar{c}\tau - |\vec{y}| + \sqrt{\tau})^4} d\vec{y} d\tau\\
			&:=(\RN{4})+(\RN{5})+(\RN{6}).
		\end{aligned}
	\end{equation}

	For $|\vec{x}| \ge \bar{c} t$ and $|\vec{y}| \le \bar{c}\tau$, one has that 
	$|\vec{x} - \vec{y}| - \bar{c}(t - \tau) \ge (|\vec{x}| - \bar{c}t)+(\bar{c}\tau-|\vec{y}|) \ge |\vec{x}|-\bar{c}t$, which together with
	\begin{equation} \label{intebou}
		\begin{aligned}
			&\left( \int_{|\vec{y}| \le \bar{c}\tau/2}+\int_{ \bar{c}\tau/2 <|\vec{y}|\le \bar{c} \tau}\right) (1+|\vec{y}| + \sqrt{\tau})^{-2}(1+\bar{c}\tau - |\vec{y}| + \sqrt{\tau})^{-4} d\vec y \\
			& \le C\tau^3 (1+\sqrt{\tau})^{-2}(1+\bar{c}\tau/2)^{-4} + (1+\bar{c}\tau/2)^{-2} \int_{ \bar{c}\tau/2 <|\vec{y}|\le \bar{c} \tau} \left(1+(|\vec{y}|-\bar{c}\tau)^2 + \tau  \right)^{-2} d\vec y\\
			& \le C(1+\tau)^{-2} + C(1+\tau)^{-{{\frac{3}{2}}}} \le C(1+\tau)^{-\frac{3}{2}},
		\end{aligned}
	\end{equation}
	yields that
	\begin{equation} \label{4ker}
		\begin{aligned}
			(\RN{4}) &\le C e^{-\frac{(|\vec{x}| - \bar{c}t)^2}{2K(t + 1)}} \left[ \int_0^{t/2} (1+t-\tau)^{-2-\frac{m}{2}} (1+\tau)^{-2} d\tau \right.\\
			&\quad \quad \quad \quad \quad \quad \quad \left.+ \int_{t/2}^t (1+t-\tau)^{\frac{1}{2}-\frac{m}{2}} (1+\tau)^{-\frac{9}{2}} \int_{\mathbb{R}^3} (1+t-\tau)^{-\frac{5}{2}} e^{-\frac{(|\vec{x}-\vec{y}|-\bar{c}(t-\tau))^2}{2K(t-\tau+1)}} d\vec{y} d\tau \right]\\
			& \le Ce^{-\frac{(|\vec{x}| - \bar{c}t)^2}{2K(t + 1)}} \left[(1+t)^{-2-\frac{m}{2}} + (1+t)^{-\frac{9}{2}} \begin{cases}
				(1+t/2)^{\frac{3-m}{2}} & \text{for } 0 \le m < 3 \\
				\ln(1+t/2) & \text{for } m=3 \\
				1 & \text{for } m > 3
			\end{cases} \right]\\
			& \le C (1+t)^{-2-\frac{\min\{m,5\}}{2}}e^{-\frac{(|\vec{x}| - \bar{c}t)^2}{2K(t + 1)}}.
		\end{aligned}
	\end{equation}

	By virtue of \eqref{exieq} and \eqref{intebou}, it holds that
	\begin{equation} \label{5ker}
		\begin{aligned}
			(\RN{5}) &\le C H(\bar{c}t-|\vec{x}|)\left[ e^{-\frac{(|\vec{x}|-\bar{c}t)^2}{8K(t+1)}} \int_0^{(t-|\vec{x}|/\bar{c}+\sqrt{t})/16} (1+t-\tau)^{-2-\frac{m}{2}} (1+\tau)^{-2} d\tau \right.\\
			& \left. \quad \quad \quad \quad \quad \quad \quad \quad + \int_{(t-|\vec{x}|/\bar{c}+\sqrt{t})/16}^{t/8} (1+t-\tau)^{-2-\frac{m}{2}} (1+\tau)^{-2} d\tau \right]\\
			& \le CH(\bar{c}t-|\vec{x}|)(1+t)^{-2-\frac{m}{2}} \left(  e^{-\frac{(|\vec{x}| - \bar{c}t)^2}{8K(t + 1)}}  + (\bar{c}t-|\vec{x}|+\sqrt{t+1})^{-1}\right),
		\end{aligned}
	\end{equation}
while direct computations lead to
	\begin{equation} \label{6ker}
		\begin{aligned}
			(\RN{6}) &\le C H(\bar{c}t-|\vec{x}|) \int_{t/8}^t (1+t-\tau)^{\frac{1}{2}-\frac{m}{2}} (1+\tau)^{-\frac{9}{2}} \int_{\mathbb{R}^3} (1+t-\tau)^{-\frac{5}{2}} e^{-\frac{(|\vec{x}-\vec{y}|-\bar{c}(t-\tau))^2}{K(t-\tau+1)}} d\vec{y} d\tau \\
			&\le C H(\bar{c}t-|\vec{x}|) (1+t)^{-\frac{9}{2}} \begin{cases}
				(1+t)^{\frac{3-m}{2}} & \text{for } 0 \le m < 3, \\
				\ln(1+t) & \text{for } m=3, \\
				1 & \text{for } m > 3.
			\end{cases}
		\end{aligned}
	\end{equation}

	Combining \eqref{456ker} with \eqref{4ker}--\eqref{6ker} results in that
		\begin{equation} \label{kere12}
		\begin{aligned}
			&\int_0^t \int_{\mathbb{R}^3} (1+t-\tau)^{-2-\frac{m}{2}} e^{-\frac{(|\vec{x}-\vec{y}|-\bar{c}(t-\tau))^2}{K(t-\tau+1)}} (1+\tau)^{-\frac{1}{2}} \frac{H(\bar{c}\tau - |\vec{y}|)}{1 + (|\vec{y}| + \sqrt{\tau})^2(\bar{c}\tau - |\vec{y}| + \sqrt{\tau})^4} d\vec{y} d\tau\\
			& \le C (1+t)^{-2-\frac{\min\{m,5\}}{2}}e^{-\frac{(|\vec{x}| - \bar{c}t)^2}{8K(t + 1)}}\\
			&\quad + CH(\bar{c}t-|\vec{x}|)\left[ (1+t)^{-2-\frac{m}{2}} (\bar{c}t-|\vec{x}|+\sqrt{t+1})^{-1} + (1+t)^{-\frac{9}{2}} \begin{cases}
				(1+t)^{\frac{3-m}{2}} & \text{for } 0 \le m < 3 \\
				\ln(1+t) & \text{for } m=3 \\
				1 & \text{for } m > 3
			\end{cases} \right],
		\end{aligned}
	\end{equation}
	and one verifies \eqref{heke2} by \eqref{kere11} and \eqref{kere12}.
	
	To estimate the first convolution in \eqref{heke20}, one could follow the region decompositions similar to the ones in \eqref{123ker} and combine with the following inequalities:
	\begin{equation} \label{hboun2}
		\begin{aligned}
			(1+\tau)^{-\frac{1}{2}} (|\vec{y}|^2+1)^{-d} \left((|\vec y|-\bar{c}\tau)^2 + \tau + 1 \right)^{-2}&\le CH(\bar{c}\tau/2 - |\vec{y}|)(1+\tau)^{-\frac{9}{2}} (|\vec{y}|^2+1)^{-d}\\
			& \quad + H(|\vec{y}|-\bar{c}\tau/2) (1+\tau)^{-\frac{5}{2}-2d}\\
			& \le C(1+\tau)^{-\frac{5}{2}-\min\{2,2d\}}, \ \ \text{ for }d\ge 0,                                                                                                                                                                                                                                                                                                                                                                                                                                                                                                                                                                                                                                                                  
		\end{aligned}
	\end{equation}
	\begin{equation} \label{intebou2}
		\begin{aligned}
			&\left( \int_{|\vec{y}| \le \bar{c}\tau/2}+\int_{ |\vec{y}| \ge \bar{c}\tau/2}\right) (1+\tau)^{-\frac{1}{2}} (|\vec{y}|^2+1)^{-\frac{9}{4}} \left((|\vec y|-\bar{c}\tau)^2 + \tau + 1 \right)^{-2} d\vec{y} \\
			& \le (1+\tau)^{-\frac{1}{2}} \left( (\bar{c} \tau)^2/4+1\right)^{-2} \int_{|\vec{y}| \le \bar{c}\tau/2} (|\vec{y}|^2+1)^{-\frac{9}{4}} d\vec{y}\\
			&\quad + \left( (\bar{c} \tau)^2/4+1\right)^{-\frac{9}{4}} \int_{ |\vec{y}| \ge \bar{c}\tau/2} (1+\tau)^{-\frac{1}{2}} \left((|\vec y|-\bar{c}\tau)^2 + \tau + 1 \right)^{-2} d\vec{y}\\
			& \le C(1+\tau)^{-\frac{9}{2}}.
		\end{aligned}
	\end{equation}
Thus one has that
	\begin{equation} \label{789ker}
		\begin{aligned}
			&\int_0^t \int_{\mathbb{R}^3} (1+t-\tau)^{-2-\frac{m}{2}} e^{-\frac{(|\vec{x}-\vec{y}|-\bar{c}(t-\tau))^2}{K(t-\tau+1)}} (1+\tau)^{-\frac{1}{2}} (|\vec{y}|^2+1)^{-\frac{9}{4}} \left((|\vec y|-\bar{c}\tau)^2 + \tau + 1 \right)^{-2} d\vec{y} d\tau\\
			& = H(|\vec{x}|-\bar{c}t) \int_0^t \left( \int_{|\vec{y}| -\bar{c}\tau \le (|\vec{x}|-\bar{c}t)/2} +  \int_{|\vec{y}| -\bar{c}\tau > (|\vec{x}|-\bar{c}t)/2}\right)  (1+t-\tau)^{-2-\frac{m}{2}} e^{-\frac{(|\vec{x}-\vec{y}|-\bar{c}(t-\tau))^2}{K(t-\tau+1)}} (1+\tau)^{-\frac{1}{2}} \\
			& \quad \quad \quad \quad \quad \quad \quad \times (|\vec{y}|^2+1)^{-\frac{9}{4}} \left((|\vec{y}|-\bar c\tau)^2 + \tau + 1 \right)^{-2} d\vec{y} d\tau\\
			& \quad + H(\bar{c}t-|\vec{x}|) \left( \int_0^{(t-|\vec{x}|/\bar{c}+\sqrt{t})/16}+\int_{(t-|\vec{x}|/\bar{c}+\sqrt{t})/16}^{t/8}\right)  \int_{\mathbb{R}^3} (1+t-\tau)^{-2-\frac{m}{2}} e^{-\frac{(|\vec{x}-\vec{y}|-\bar{c}(t-\tau))^2}{K(t-\tau+1)}} (1+\tau)^{-\frac{1}{2}} \\
			& \quad \quad \quad \quad \quad \quad \quad \quad \times (|\vec{y}|^2+1)^{-\frac{9}{4}} \left((|\vec{y}|-\bar{c}\tau)^2 + \tau + 1 \right)^{-2} d\vec{y} d\tau\\
			& \quad + H(\bar{c}t-|\vec{x}|) \int_{t/8}^t \int_{\mathbb{R}^3} (1+t-\tau)^{-2-\frac{m}{2}} e^{-\frac{(|\vec{x}-\vec{y}|-\bar{c}(t-\tau))^2}{K(t-\tau+1)}} (1+\tau)^{-\frac{1}{2}} (|\vec{y}|^2+1)^{-\frac{9}{4}}\\
			& \quad \quad \quad \quad \quad \quad \quad \quad \times \left((|\vec{y}|-\bar{c}\tau)^2 + \tau + 1 \right)^{-2} d\vec{y} d\tau\\
			&:=(\RN{7})+(\RN{8})+(\RN{9}),
		\end{aligned}
	\end{equation}
with
	\begin{equation} \label{7ker}
		\begin{aligned}
			(\RN{7}) &\le C e^{-\frac{(|\vec{x}| - \bar{c}t)^2}{8K(t + 1)}}  \int_0^{t} (1+t-\tau)^{-2-\frac{m}{2}} (1+\tau)^{-\frac{9}{2}} d\tau \\
			& \quad + C\left(\frac{(|\vec{x}|-\bar{c}t)^2}{4(t+1)}+1 \right)^{-2} \left[\int_0^{t/2} (1+t-\tau)^{-2-\frac{m}{2}} (1+\tau)^{-\frac{5}{2}} \int_{\mathbb{R}^3} (|\vec{y}|^2+1)^{-\frac{9}{4}} d\vec{y} d\tau \right.\\
			& \quad \left.+ \int_{t/2}^t (1+t-\tau)^{\frac{1}{2}-\frac{m}{2}} (1+\tau)^{-\frac{5}{2}} \left( (\bar{c}\tau)^2+1\right)^{-\frac{9}{4}} \int_{\mathbb{R}^3} (1+t-\tau)^{-\frac{5}{2}} e^{-\frac{(|\vec{x}-\vec{y}|-\bar{c}(t-\tau))^2}{K(t-\tau+1)}} d\vec{y} d\tau \right] \\
			& \le Ce^{-\frac{(|\vec{x}| - \bar{c}t)^2}{8K(t + 1)}} \left[(1+t)^{-2-\frac{m}{2}} + (1+t)^{-\frac{9}{2}} \right]\\
			& \quad + C\left(\frac{(|\vec{x}|-\bar{c}t)^2}{t+1}+1 \right)^{-2} \left[(1+t)^{-2-\frac{m}{2}} + (1+t)^{-7} \begin{cases}
				(1+t/2)^{\frac{3-m}{2}} & \text{for } 0 \le m < 3 \\
				\ln(1+t/2) & \text{for } m=3 \\
				1 & \text{for } m > 3
			\end{cases} \right]\\
			& \le C(1+t)^{-\frac{\min\{m,5\}}{2}}\left((|\vec{x}|-\bar{c}t)^2 +t+1\right)^{-2},
		\end{aligned}
	\end{equation}
		\begin{equation} \label{8ker}
		\begin{aligned}
			(\RN{8}) &\le C H(\bar{c}t-|\vec{x}|)\left[ e^{-\frac{(|\vec{x}|-\bar{c}t)^2}{8K(t+1)}} \int_0^{(t-|\vec{x}|/\bar{c}+\sqrt{t})/16} (1+t-\tau)^{-2-\frac{m}{2}}(1+\tau)^{-\frac{9}{2}} d\tau \right.\\
			&\quad \quad \quad + \int_0^{(t-|\vec{x}|/\bar{c}+\sqrt{t})/16} \int_{|\vec{y}| \ge (t-|\vec{x}|/\bar{c}+\sqrt{t})/8}   (1+t-\tau)^{-2-\frac{m}{2}} (1+\tau)^{-\frac{1}{2}}  \left((|\vec{y}|-\bar{c}\tau)^2 + \tau + 1 \right)^{-2} d\vec{y} d\tau\\
			&\quad \quad \quad \left.  \times \left((t-|\vec{x}|/\bar{c}+\sqrt{t})^2/64 +1\right)^{-\frac{9}{4}} + \int_{(t-|\vec{x}|/\bar{c}+\sqrt{t})/16}^{t/8} (1+t-\tau)^{-2-\frac{m}{2}}(1+\tau)^{-\frac{9}{2}} d\tau \right]\\
			& \le C H(\bar{c}t-|\vec{x}|)(1+t)^{-2-\frac{m}{2}}\left[ e^{-\frac{(|\vec{x}|-\bar{c}t)^2}{8K(t+1)}} + (1+t) \left((|\vec{x}|-\bar{c}t)^2 +t+1\right)^{-\frac{9}{4}}\right],
		\end{aligned}
	\end{equation}
\begin{equation} \label{9ker}
	(\RN{9}) \le C H(\bar{c}t-|\vec{x}|) \int_{t/8}^t (1+t-\tau)^{-2-\frac{m}{2}}(1+\tau)^{-\frac{9}{2}} d\tau \le C H(\bar{c}t-|\vec{x}|)(1+t)^{-\frac{9}{2}}.
\end{equation}
Combining \eqref{789ker} with \eqref{7ker}--\eqref{9ker}, one obtains that
\begin{equation} \label{kerse11}
	\begin{aligned}
		&\int_0^t \int_{\mathbb{R}^3} (1+t-\tau)^{-2-\frac{m}{2}} e^{-\frac{(|\vec{x}-\vec{y}|-\bar{c}(t-\tau))^2}{K(t-\tau+1)}} (1+\tau)^{-\frac{1}{2}} (|\vec{y}|^2+1)^{-\frac{9}{4}} \left((|\vec y|-\bar{c}\tau)^2 + \tau + 1 \right)^{-2} d\vec{y} d\tau\\
		& \le C(1+t)^{-\frac{\min\{m,5\}}{2}}\left((|\vec{x}|-\bar{c}t)^2 +t+1\right)^{-2}\\
		& \quad + C H(\bar{c}t-|\vec{x}|)\left[ (1+t)^{-1-\frac{m}{2}} \left((|\vec{x}|-\bar{c}t)^2 +t+1\right)^{-\frac{9}{4}}+ (1+t)^{-\frac{9}{2}}\right].
	\end{aligned}
\end{equation}

The other convolution in \eqref{heke20} is similar to \eqref{456ker}:
\begin{equation} \label{1012ker}
	\begin{aligned}
		&\int_0^t \int_{\mathbb{R}^3} (1+t-\tau)^{-2-\frac{m}{2}} e^{-\frac{(|\vec{x}-\vec{y}|-\bar{c}(t-\tau))^2}{K(t-\tau+1)}} (1+\tau)^{-\frac{1}{2}} (|\vec{y}|^2+1)^{-\frac{9}{4}} \frac{H(\bar{c}\tau - |\vec{y}|)}{1 + (|\vec{y}| + \sqrt{\tau})(\bar{c}\tau - |\vec{y}| + \sqrt{\tau})^2} d\vec{y} d\tau\\
		& = H(|\vec{x}|-\bar{c}t) \int_0^t \int_{|\vec{y}|\le \bar{c}\tau} (1+t-\tau)^{-2-\frac{m}{2}} e^{-\frac{(|\vec{x}-\vec{y}|-\bar{c}(t-\tau))^2}{K(t-\tau+1)}} (1+\tau)^{-\frac{1}{2}}(|\vec{y}|^2+1)^{-\frac{9}{4}} \\
		& \quad \quad \quad \quad \quad \quad \quad \times \frac{1}{1 + (|\vec{y}| + \sqrt{\tau})(\bar{c}\tau - |\vec{y}| + \sqrt{\tau})^2} d\vec{y} d\tau\\
		& \quad + H(\bar{c}t-|\vec{x}|) \left( \int_0^{(t-|\vec{x}|/\bar{c}+\sqrt{t})/16}+\int_{(t-|\vec{x}|/\bar{c}+\sqrt{t})/16}^{t/8}\right)  \int_{|\vec{y}|\le \bar{c}\tau} (1+t-\tau)^{-2-\frac{m}{2}} e^{-\frac{(|\vec{x}-\vec{y}|-\bar{c}(t-\tau))^2}{K(t-\tau+1)}} (1+\tau)^{-\frac{1}{2}} \\
		& \quad \quad \quad \quad \quad \quad \quad \quad \times (|\vec{y}|^2+1)^{-\frac{9}{4}} \frac{1}{1 + (|\vec{y}| + \sqrt{\tau})(\bar{c}\tau - |\vec{y}| + \sqrt{\tau})^2} d\vec{y} d\tau\\
		& \quad + H(\bar{c}t-|\vec{x}|) \int_{t/8}^t \int_{|\vec{y}|\le \bar{c}\tau} (1+t-\tau)^{-2-\frac{m}{2}} e^{-\frac{(|\vec{x}-\vec{y}|-\bar{c}(t-\tau))^2}{K(t-\tau+1)}}  \frac{(1+\tau)^{-\frac{1}{2}}(|\vec{y}|^2+1)^{-\frac{9}{4}}}{1 + (|\vec{y}| + \sqrt{\tau})(\bar{c}\tau - |\vec{y}| + \sqrt{\tau})^2} d\vec{y} d\tau\\
		&:=(\RN{10})+(\RN{11})+(\RN{12}).
	\end{aligned}
\end{equation}

Similar to \eqref{hboun2} and \eqref{intebou2}, one has that
\begin{equation} \label{hboun3}
	\begin{aligned}
		&(1+\tau)^{-\frac{1}{2}}(|\vec{y}|^2+1)^{-d} \frac{H(\bar{c}\tau - |\vec{y}|)}{1 + (|\vec{y}| + \sqrt{\tau})(\bar{c}\tau - |\vec{y}| + \sqrt{\tau})^2}\\
		&\le C H(\bar{c}\tau/2 - |\vec{y}|)(1+\tau)^{-3} (|\vec{y}|^2+1)^{-d}\\
		& \quad + C\bigl(H(\bar{c}\tau - |\vec{y}|)- H(\bar{c}\tau/2 - |\vec{y}|)\bigr) (1+\tau)^{-\frac{5}{2}} \left( (\bar{c} \tau)^2/4+1\right)^{-d} \le C(1+\tau)^{-\frac{5}{2}-\min\{\frac{1}{2},2d\}},\ \ \text{ for }d\ge 0                                                                                                                                                                                                                                                                                                                                                                                                                                                                                                                                                                                                                                                                                                                                          
	\end{aligned}
\end{equation}
and
\begin{equation} \nonumber 
	\begin{aligned}
		&\left( \int_{|\vec{y}| \le \bar{c}\tau/2}+\int_{\bar{c}\tau/2 \le |\vec{y}| \le \bar{c}\tau }\right) (1+\tau)^{-\frac{1}{2}} (|\vec{y}|^2+1)^{-\frac{9}{4}} (1+|\vec{y}| + \sqrt{\tau})^{-1}(\bar{c}\tau - |\vec{y}| + \sqrt{\tau} +1)^{-2} d\vec{y}\\
		& \le C(1+\tau)^{-1}  \left(\bar{c} \tau/2+1\right)^{-2} \int_{|\vec{y}| \le \bar{c}\tau/2} (|\vec{y}|^2+1)^{-\frac{9}{4}} d\vec{y}+ C(1+\tau)^{-\frac{5}{2}}\left( (\bar{c} \tau)^2/4+1\right)^{-\frac{9}{4}} \int_{ |\vec{y}| \le \bar{c}\tau} d\vec{y}\le C(1+\tau)^{-3},
	\end{aligned}
\end{equation}
which together with \eqref{exieq} lead to
\begin{equation} \label{10ker}
		(\RN{10}) \le C e^{-\frac{(|\vec{x}| - \bar{c}t)^2}{K(t + 1)}} \int_0^t (1+t-\tau)^{-2-\frac{m}{2}} (1+\tau)^{-3} d\tau
		\le C (1+t)^{-2-\frac{\min\{m,2\}}2}e^{-\frac{(|\vec{x}| - \bar{c}t)^2}{K(t + 1)}}
\end{equation}
and
\begin{equation} \label{11ker}
	\begin{aligned}
		(\RN{11}) &\le C H(\bar{c}t-|\vec{x}|)\left[ e^{-\frac{(|\vec{x}|-\bar{c}t)^2}{8K(t+1)}} \int_0^{(t-|\vec{x}|/\bar{c}+\sqrt{t})/16} (1+t-\tau)^{-2-\frac{m}{2}}(1+\tau)^{-3} d\tau \right.\\
		&\left. \quad \quad \quad \quad \quad \quad \quad \quad + \int_{(t-|\vec{x}|/\bar{c}+\sqrt{t})/16}^{t/8} (1+t-\tau)^{-2-\frac{m}{2}}(1+\tau)^{-3} d\tau \right]\\
		& \le C H(\bar{c}t-|\vec{x}|)(1+t)^{-2-\frac{m}{2}}\left[ e^{-\frac{(|\vec{x}|-\bar{c}t)^2}{8K(t+1)}} + (\bar{c}t-|\vec{x}|+\sqrt{t+1})^{-2}\right].
	\end{aligned}
\end{equation}

The part $(\RN{12})$ should be estimated under a more tedious region decomposition since a direct use of \eqref{hboun3} does not result in fast enough decay rates:
	\begin{equation} \label{12ker}
		\begin{aligned}
			(\RN{12}) &\le C H(\bar{c}t-|\vec{x}|) \left[\int_{t/8}^t (1+t-\tau)^{\frac{1}{2}-\frac{m}{2}} (1+\tau)^{-7} \int_{|\vec{y}|\ge \bar{c}\tau/2} (1+t-\tau)^{-\frac{5}{2}} e^{-\frac{(|\vec{x}-\vec{y}|-\bar{c}(t-\tau))^2}{K(t-\tau+1)}} d\vec{y} d\tau \right.\\
			& \quad \quad + \left( \int_{t/8}^{t-\sqrt{t}}+H(4\bar{c}\sqrt{t}-|\vec{x}|)\int_{t-\sqrt{t}}^t \right) (1+t-\tau)^{-2-\frac{m}{2}} (1+\tau)^{-3} \int_{|\vec{y}| \le \bar{c}\tau/2} (|\vec{y}|^2+1)^{-\frac{9}{4}} d\vec{y}d\tau\\
			& \quad \quad  + H(|\vec{x}|-4\bar{c}\sqrt{t})\int_{t-\sqrt{t}}^t (1+t-\tau)^{-2-\frac{m}{2}}(1+\tau)^{-3} \left( \int_{2\bar{c}\sqrt{t} \le |\vec{y}|\le \bar{c}\tau/2}+ \int_{|\vec{y}|\le 2\bar{c}\sqrt{t}}\right)\\ 
			& \left. \quad \quad \quad (|\vec{y}|^2+1)^{-\frac{9}{4}} e^{-\frac{(|\vec{x}-\vec{y}|-\bar{c}(t-\tau))^2}{K(t-\tau+1)}} d\vec{y} d\tau \right]\\
			& \le C H(\bar{c}t-|\vec{x}|)\Bigg[ (1+t)^{-7} \begin{cases}
				(1+t)^{\frac{3-m}{2}} & \text{for } 0 \le m < 3 \\
				\ln(1+t) & \text{for } m=3 \\
				1 & \text{for } m > 3
			\end{cases} + (1+t)^{-3} (1+\sqrt{t})^{-1-\frac{m}{2}} \\
		  &\quad \quad + H(4\bar{c}\sqrt{t}-|\vec{x}|)(1+t)^{-3} +H(|\vec{x}|-4\bar{c}\sqrt{t})\left((4\bar{c}^2 t+1)^{-\frac{9}{4}} \int_{t-\sqrt{t}}^t (1+t-\tau)^{\frac{1}{2}-\frac{m}{2}} (1+\tau)^{-3} d\tau \right.\\
		  & \left.\quad \quad + e^{-\frac{\bar{c}^2t}{K(\sqrt{t}+1)}}  \int_{t-\sqrt{t}}^t (1+t-\tau)^{-2-\frac{m}{2}}(1+\tau)^{-3} d\tau \right) \Bigg]\\
		  & \le C H(\bar{c}t-|\vec{x}|)\left[ (1+t)^{-\frac{21}{4}} \begin{cases}
		  	(1+t)^{\frac{3-m}{2}} & \text{for } 0 \le m < 3 \\
		  	\ln(1+t) & \text{for } m=3 \\
		  	1 & \text{for } m > 3
		  \end{cases} + (1+t)^{-\frac{7}{2}-\frac{m}{4}} \right] \\
		  &\quad + CH(4\bar{c}\sqrt{t}-|\vec{x}|)(1+t)^{-3}\\
		  & \le C H(\bar{c}t-|\vec{x}|)(1+t)^{-\frac{7}{2}-\frac{\min\{m,7\}}4} + CH(4\bar{c}\sqrt{t}-|\vec{x}|)(1+t)^{-3}.
		\end{aligned}
	\end{equation}

One could combine \eqref{1012ker} and \eqref{10ker}--\eqref{12ker} to deduce that
\begin{equation} \nonumber 
	\begin{aligned}
		&\int_0^t \int_{\mathbb{R}^3} (1+t-\tau)^{-2-\frac{m}{2}} e^{-\frac{(|\vec{x}-\vec{y}|-\bar{c}(t-\tau))^2}{K(t-\tau+1)}} (1+\tau)^{-\frac{1}{2}} (|\vec{y}|^2+1)^{-\frac{9}{4}} \frac{H(\bar{c}\tau - |\vec{y}|)}{1 + (|\vec{y}| + \sqrt{\tau})(\bar{c}\tau - |\vec{y}| + \sqrt{\tau})^2} d\vec{y} d\tau\\
		& \le C (1+t)^{-2-\frac{\min\{m,2\}}2}e^{-\frac{(|\vec{x}| - \bar{c}t)^2}{8K(t + 1)}} + CH(4\bar{c}\sqrt{t}-|\vec{x}|)(1+t)^{-3}\\
		& \quad +C H(\bar{c}t-|\vec{x}|)\left[  (1+t)^{-2-\frac{m}{2}} (\bar{c}t-|\vec{x}|+\sqrt{t+1})^{-2} + (1+t)^{-\frac{7}{2}-\frac{\min\{m,7\}}4}  \right], 
	\end{aligned}
\end{equation}
which together with \eqref{kerse11} verifies \eqref{heke20}.

Finally, the convolution in \eqref{heke21} could be evaluated similar to \eqref{123ker} after establishing the following estimates:
\begin{equation} \nonumber 
	\begin{aligned}
		&(1+|\vec{y}|+\sqrt{\tau})^{-2} \left((|\vec y|-\bar{c}\tau)^2 + \tau + 1 \right)^{-3}\\
		&\le CH( \bar{c}\tau/2-|\vec{y}|)(1+\tau)^{-7} +CH(|\vec{y}|-\bar{c}\tau/2) ( 1+\tau)^{-5} \le C(1+\tau)^{-5} ,                                                                                                                                                                                                                                                                                                                                                                                                                                                                                                                                                                                                                                                                                                                                                  
	\end{aligned}
\end{equation}
\begin{equation} \nonumber 
	\begin{aligned}
		&\left( \int_{|\vec{y}| \le \bar{c}\tau/2}+\int_{ |\vec{y}| \ge \bar{c}\tau/2}\right) (1+|\vec{y}|+\sqrt{\tau})^{-2} \left((|\vec y|-\bar{c}\tau)^2 + \tau + 1 \right)^{-3} d\vec{y} \\
		& \le \left( (\bar{c} \tau)^2/4+1\right)^{-3} \int_{|\vec{y}| \le \bar{c}\tau/2} (1+|\vec{y}|+\sqrt{\tau})^{-2} d\vec{y}\\
		&\quad + (1+\bar{c} \tau/2)^{-2}(1+\tau)^{-1} \int_{ |\vec{y}| \ge \bar{c}\tau/2} \left((|\vec y|-\bar{c}\tau)^2 + \tau + 1 \right)^{-2} d\vec{y}\\
		& \le C(1+\tau)^{-\frac{5}{2}}.
	\end{aligned}
\end{equation}
Thus one has that
\begin{equation} \nonumber
	\begin{aligned}
		&\int_0^t \int_{\mathbb{R}^3} (1+t-\tau)^{-2-\frac{m}{2}} e^{-\frac{(|\vec{x}-\vec{y}|-\bar{c}(t-\tau))^2}{K(t-\tau+1)}} (1+|\vec{y}|+\sqrt{\tau})^{-2} \left((|\vec y|-\bar{c}\tau)^2 + \tau + 1 \right)^{-3} d\vec{y} d\tau\\
		& = H(|\vec{x}|-\bar{c}t) \int_0^t \left( \int_{|\vec{y}| -\bar{c}\tau \le (|\vec{x}|-\bar{c}t)/2} +  \int_{|\vec{y}| -\bar{c}\tau > (|\vec{x}|-\bar{c}t)/2}\right)  (1+t-\tau)^{-2-\frac{m}{2}} e^{-\frac{(|\vec{x}-\vec{y}|-\bar{c}(t-\tau))^2}{K(t-\tau+1)}}\\
		& \quad \quad \quad \quad \quad \quad \quad \times (1+|\vec{y}|+\sqrt{\tau})^{-2} \left((|\vec y|-\bar{c}\tau)^2 + \tau + 1 \right)^{-3} d\vec{y} d\tau\\
		& \quad + H(\bar{c}t-|\vec{x}|) \left( \int_0^{(t-|\vec{x}|/\bar{c}+\sqrt{t})/16}+\int_{(t-|\vec{x}|/\bar{c}+\sqrt{t})/16}^{t/8}\right)  \int_{\mathbb{R}^3} (1+t-\tau)^{-2-\frac{m}{2}} e^{-\frac{(|\vec{x}-\vec{y}|-\bar{c}(t-\tau))^2}{K(t-\tau+1)}}  \\
		& \quad \quad \quad \quad \quad \quad \quad \quad \times (1+|\vec{y}|+\sqrt{\tau})^{-2} \left((|\vec y|-\bar{c}\tau)^2 + \tau + 1 \right)^{-3} d\vec{y} d\tau\\
		& \quad + H(\bar{c}t-|\vec{x}|) \int_{t/8}^t \int_{\mathbb{R}^3} (1+t-\tau)^{-2-\frac{m}{2}} e^{-\frac{(|\vec{x}-\vec{y}|-\bar{c}(t-\tau))^2}{K(t-\tau+1)}} (1+|\vec{y}|+\sqrt{\tau})^{-2} \left((|\vec y|-\bar{c}\tau)^2 + \tau + 1 \right)^{-3} d\vec{y} d\tau\\
		&:=(\RN{13})+(\RN{14})+(\RN{15}),
	\end{aligned}
\end{equation}
with
	\begin{equation} \nonumber
	\begin{aligned}
		(\RN{13}) &\le Ce^{-\frac{(|\vec{x}| - \bar{c}t)^2}{8K(t + 1)}} \left[ \int_0^{t/2} (1+t-\tau)^{-2-\frac{m}{2}} (1+\tau)^{-\frac{5}{2}} d\tau \right.\\
		&\quad \quad \quad \quad \quad \quad \left.+ \int_{t/2}^t (1+t-\tau)^{\frac{1}{2}-\frac{m}{2}} (1+\tau)^{-5} \int_{\mathbb{R}^3} (1+t-\tau)^{-\frac{5}{2}} e^{-\frac{(|\vec{x}-\vec{y}|-\bar{c}(t-\tau))^2}{K(t-\tau+1)}} d\vec{y} d\tau \right]\\
		& \quad +C \int_0^{t/2} (1+t-\tau)^{-2-\frac{m}{2}} \left((|\vec{x}|-\bar{c}t)/2+\sqrt{\tau}+1\right)^{-2} \left((|\vec{x}|-\bar{c} t)^2/4+\tau+1\right)^{-1} \int_{|\vec{y}| -\bar{c}\tau > (|\vec{x}|-\bar{c}t)/2}\\
		& \quad \quad \quad \quad \quad \left((|\vec y|-\bar c\tau)^2 + \tau + 1 \right)^{-2} d\vec{y} d\tau \\
		&\quad + C \left(\frac{(|\vec{x}|-\bar{c}t)^2}{4(t+1)}+1 \right)^{-3} \int_{t/2}^t (1+t-\tau)^{\frac{1}{2}-\frac{m}{2}} (1+\tau)^{-5} \int_{\mathbb{R}^3} (1+t-\tau)^{-\frac{5}{2}} e^{-\frac{(|\vec{x}-\vec{y}|-\bar{c}(t-\tau))^2}{K(t-\tau+1)}} d\vec{y} d\tau\\
		& \le C(1+t)^{-\frac{\min\{m,6\}}2}\left((|\vec{x}|-\bar{c}t)^2+t+1\right)^{-2},
	\end{aligned}
\end{equation}
\begin{equation} \nonumber 
	\begin{aligned}
		(\RN{14}) &\le C H(\bar{c}t-|\vec{x}|)\left[ e^{-\frac{(|\vec{x}|-\bar{c}t)^2}{8K(t+1)}}  \int_0^{t/8} (1+t-\tau)^{-2-\frac{m}{2}}(1+\tau)^{-\frac{5}{2}} d\tau\right.\\ 
		& \quad +\left(\frac{(|\vec{x}|-\bar{c}t)^2}{(16)^2(t+1)}+1 \right)^{-1}\left(1+\frac{\bar{c}t-|\vec{x}|+\bar{c}\sqrt{t}}{8}\right)^{-2} \int_0^{(t-|\vec{x}|/\bar{c}+\sqrt{t})/16} (1+t-\tau)^{-2-\frac{m}{2}}\\
		& \quad \quad \quad \times (1+\tau)^{-1} \int_{|\vec{y}|  > (\bar{c}t-|\vec{x}|+\bar{c}\sqrt{t})/8} \left((|\vec y|-\bar{c}\tau)^2 + \tau + 1 \right)^{-2} d\vec{y} d\tau\\
		&\quad \left.+ \int_{(t-|\vec{x}|/\bar{c}+\sqrt{t})/16}^{t/8} (1+t-\tau)^{-2-\frac{m}{2}}(1+\tau)^{-\frac{5}{2}} d\tau \right]\\
		& \le C H(\bar{c}t-|\vec{x}|)\left[ (1+t)^{-2-\frac{m}{2}}\left(e^{-\frac{(|\vec{x}|-\bar{c}t)^2}{8K(t+1)}} + (\bar{c}t-|\vec{x}|+\sqrt{t+1})^{-\frac{3}{2}}\right) \right.\\
		& \left. \quad \quad \quad \quad \quad \quad \quad \quad + (1+t)^{-\frac{1}{2}-\frac{m}{2}} \left((|\vec{x}|-\bar{c}t)^2+t+1\right)^{-2}\right],
	\end{aligned}
\end{equation}
\begin{equation} \nonumber
	\begin{aligned}
		(\RN{15}) &\le C H(\bar{c}t-|\vec{x}|) \int_{t/8}^t (1+t-\tau)^{\frac{1}{2}-\frac{m}{2}}(1+\tau)^{-5} \int_{\mathbb{R}^3} (1+t-\tau)^{-\frac{5}{2}} e^{-\frac{(|\vec{x}-\vec{y}|-\bar{c}(t-\tau))^2}{K(t-\tau+1)}} d\vec{y} d\tau\\
		&\le C H(\bar{c}t-|\vec{x}|) (1+t)^{-5} 
		\begin{cases}
			(1+t)^{\frac{3-m}{2}} & \text{for } 0 \le m < 3, \\
			\ln(1+t) & \text{for } m=3, \\
			1 & \text{for } m > 3,
		\end{cases}
	\end{aligned}
\end{equation}
and finally obtains that
\begin{equation} \nonumber 
	\begin{aligned}
		&\int_0^t \int_{\mathbb{R}^3} (1+t-\tau)^{-2-\frac{m}{2}} e^{-\frac{(|\vec{x}-\vec{y}|-\bar{c}(t-\tau))^2}{K(t-\tau+1)}} (1+|\vec{y}|+\sqrt{\tau})^{-2} \left((|\vec y|-\bar{c}\tau)^2 + \tau + 1 \right)^{-3} d\vec{y} d\tau\\
		& \le C(1+t)^{-\frac{\min\{m,6\}}2}\left((|\vec{x}|-\bar{c}t)^2+t+1\right)^{-2} \\
		& \quad +C H(\bar{c}t-|\vec{x}|)\left[  (1+t)^{-2-\frac{m}{2}} (\bar{c}t-|\vec{x}|+\sqrt{t+1})^{-\frac{3}{2}}+(1+t)^{-5} 
		\begin{cases}
			(1+t)^{\frac{3-m}{2}} & \text{for } 0 \le m < 3, \\
			\ln(1+t) & \text{for } m=3, \\
			1 & \text{for } m > 3,
		\end{cases}\right]
	\end{aligned}
\end{equation}
which verifies \eqref{heke21} and we finish the proof.
\end{proof}

\subsection{Proof of Lemma~\ref{lemmainexp}}

\begin{proof}
	For $i \ge 0$, $m \ge 0$, and $j=0,1,2$, it holds that
	\begin{equation} \label{12exp}
		\begin{aligned}
			&\int_0^t \int_{\mathbb{R}^3} |\vec{x}-\vec{y}|^{-j}  e^{-(|\vec{x}-\vec{y}|+t-\tau)/D_0} (1+\tau)^{-i} \left((|\vec y|-\bar{c}\tau)^2 + \tau + 1 \right)^{-2m} d\vec{y} d\tau\\
			& = H(|\vec{x}|-\bar{c}t) \int_0^t \left( \int_{|\vec{x}-\vec{y}| \le (|\vec{x}|-\bar{c}t)/2} +  \int_{|\vec{x}-\vec{y}|> (|\vec{x}|-\bar{c}t)/2}\right) |\vec{x}-\vec{y}|^{-j}  e^{-(|\vec{x}-\vec{y}|+t-\tau)/D_0} (1+\tau)^{-i}\\
			& \quad \quad \quad \quad \quad \quad \quad \times \left((|\vec{y}|-\bar c\tau)^2 + \tau + 1 \right)^{-2m} d\vec{y} d\tau\\
			& \quad + H(\bar{c}t-|\vec{x}|) \int_0^t \int_{\mathbb{R}^3} |\vec{x}-\vec{y}|^{-j}  e^{-(|\vec{x}-\vec{y}|+t-\tau)/D_0} (1+\tau)^{-i} \left((|\vec{y}|-\bar{c}\tau)^2 + \tau + 1 \right)^{-2m} d\vec{y} d\tau\\
		 &\le CH(|\vec{x}|-\bar{c}t)\left( \left(\frac{(|\vec{x}|-\bar{c}t)^2}{4(t+1)}+1 \right)^{-2m} + e^{-(|\vec{x}|-\bar{c}t)/(4D_0)} \right) \int_0^t e^{-(t-\tau)/D_0}(1+\tau)^{-i-2m} d\tau\\
		&\quad \quad \quad \quad \quad \quad \times \int_{\mathbb{R}^3} |\vec{x}-\vec{y}|^{-j} e^{-|\vec{x}-\vec{y}|/(2D_0)}d\vec{y}\\
		&\quad +CH(\bar{c}t-|\vec{x}|) \int_0^t e^{-(t-\tau)/D_0}(1+\tau)^{-i-2m} d\tau \int_{\mathbb{R}^3} |\vec{x}-\vec{y}|^{-j} e^{-|\vec{x}-\vec{y}|/D_0}d\vec{y}\\
		& \le CH(|\vec{x}|-\bar{c}t)(1+t)^{-i-2m}\left( \left(\frac{(|\vec{x}|-\bar{c}t)^2}{t+1}+1 \right)^{-2m}+e^{-(|\vec{x}|-\bar{c}t)/(4D_0)}\right)+CH(\bar{c}t-|\vec{x}|)(1+t)^{-i-2m}\\
		&\le C (1+t)^{-i}\left( (|\vec{x}| - \bar{c} t)^2 + t + 1\right)^{-2m} +CH(\bar{c}t-|\vec{x}|)(1+t)^{-i-2m}.
	\end{aligned}
\end{equation}
Here we split the analysis into two cases, $|\vec{x}| - \bar{c}t \le \bar{c}t$ and $|\vec{x}| - \bar{c}t > \bar{c}t$, to treat the exponential term $\exp\left\lbrace -(|\vec{x}|-\bar{c}t)/(4D_0)\right\rbrace$.

Next one evaluates the following integral:
\begin{equation} \nonumber 
	\begin{aligned}
		&\int_0^t \int_{\mathbb{R}^3} |\vec{x}-\vec{y}|^{-j}  e^{-(|\vec{x}-\vec{y}|+t-\tau)/D_0} (1+\tau)^{-i}\frac{H(\bar{c}\tau - |\vec{y}|)}{1 + (|\vec{y}| + \sqrt{\tau})^m(\bar{c}\tau - |\vec{y}| + \sqrt{\tau})^{2m}} d\vec{y} d\tau\\
		& \le C \left( H(|\vec{x}|-\bar{c}t) e^{-(|\vec{x}|-\bar{c}t)/(2D_0)} + H(\bar{c}t-|\vec{x}|)\right) \int_0^t e^{-(t-\tau)/D_0} (1+\tau)^{-i-2m} d\tau\\
		& \quad \quad \times \int_{\mathbb{R}^3} |\vec{x}-\vec{y}|^{-j}  e^{-(|\vec{x}-\vec{y}|)/(2D_0)}d\vec{y}\\
		& \le C(1+t)^{-i-2m} \left( H(|\vec{x}|-\bar{c}t) e^{-(|\vec{x}|-\bar{c}t)/(2D_0)} + H(\bar{c}t-|\vec{x}|) \right)\\
		& \le C(1+t)^{-i-2m} \left( e^{-\frac{(|\vec{x}|-\bar{c}t)^2}{2D_0\bar{c}t}} + e^{-\frac{(|\vec{x}|+2\bar{c}t)}{8D_0}} + H(\bar{c}t-|\vec{x}|) \right),
	\end{aligned}
\end{equation}
which together with \eqref{12exp} verifies \eqref{exponxt}.
 
 The convolution in \eqref{exponxt0} could be estimated as follows:
 \begin{equation} \nonumber
 	\begin{aligned}
 		&\int_0^t \int_{\mathbb{R}^3} |\vec{x}-\vec{y}|^{-j}  e^{-(|\vec{x}-\vec{y}|+t-\tau)/D_0} (1+\tau)^{-\frac{1}{2}} (|\vec{y}|^2+1)^{-\frac{9}{4}} \left((|\vec y|-\bar{c}\tau)^2 + \tau + 1 \right)^{-2} d\vec{y} d\tau\\
 		& = H(|\vec{x}|-\bar{c}t) \int_0^t \left( \int_{|\vec{x}-\vec{y}| \le (|\vec{x}|-\bar{c}t)/2} +  \int_{|\vec{x}-\vec{y}|> (|\vec{x}|-\bar{c}t)/2}\right) |\vec{x}-\vec{y}|^{-j}  e^{-(|\vec{x}-\vec{y}|+t-\tau)/D_0} (1+\tau)^{-\frac{1}{2}}\\
 		& \quad \quad \quad \quad \quad \quad \quad \times (|\vec{y}|^2+1)^{-\frac{9}{4}} \left((|\vec{y}|-\bar c\tau)^2 + \tau + 1 \right)^{-2} d\vec{y} d\tau\\
 		& \quad + H(\bar{c}t-|\vec{x}|) \int_0^t \int_{\mathbb{R}^3} |\vec{x}-\vec{y}|^{-j}  e^{-(|\vec{x}-\vec{y}|+t-\tau)/D_0} (1+\tau)^{-\frac{1}{2}} (|\vec{y}|^2+1)^{-\frac{9}{4}} \left((|\vec{y}|-\bar{c}\tau)^2 + \tau + 1 \right)^{-2} d\vec{y} d\tau\\
		&\le C\Biggl[ H(|\vec{x}|-\bar{c}t)(|\vec{x}|^2+1)^{-\frac{9}{4}}\left( \left(\frac{(|\vec{x}|-\bar{c}t)^2}{4(t+1)}+1 \right)^{-2} + e^{-(|\vec{x}|-\bar{c}t)/(2D_0)} \right) \int_0^t e^{-(t-\tau)/D_0}(1+\tau)^{-\frac{5}{2}} d\tau\\
		&\quad \quad \quad +H(\bar{c}t-|\vec{x}|) \int_0^t e^{-(t-\tau)/D_0}(1+\tau)^{-\frac{9}{2}} d\tau\Biggr] \times \int_{\mathbb{R}^3} |\vec{x}-\vec{y}|^{-j} e^{-|\vec{x}-\vec{y}|/(4D_0)}d\vec{y}\\
		& \le CH(|\vec{x}|-\bar{c}t)(1+t)^{-\frac{1}{2}}(|\vec{x}|^2+1)^{-\frac{9}{4}} \left((|\vec{x}|-\bar{c}t)^2+t+1\right)^{-2} + CH(\bar{c}t-|\vec{x}|)(1+t)^{-\frac{9}{2}},
	\end{aligned}
\end{equation}

\begin{equation} \nonumber 
	\begin{aligned}
		&\int_0^t \int_{\mathbb{R}^3} |\vec{x}-\vec{y}|^{-j}  e^{-(|\vec{x}-\vec{y}|+t-\tau)/D_0} (1+\tau)^{-\frac{1}{2}} (|\vec{y}|^2+1)^{-\frac{9}{4}} \frac{H(\bar{c}\tau - |\vec{y}|)}{1 + (|\vec{y}| + \sqrt{\tau})(\bar{c}\tau - |\vec{y}| + \sqrt{\tau})^2} d\vec{y} d\tau\\
		&=\int_0^t \left( H(|\vec{x}|-\bar{c}t) \int_{\mathbb{R}^3} + H(\bar{c}t-|\vec{x}|)\left( \int_{|\vec{y}|\le \bar{c}\tau/2} + \int_{|\vec{y}|\ge \bar{c}\tau/2} \right) \right) |\vec{x}-\vec{y}|^{-j}  e^{-(|\vec{x}-\vec{y}|+t-\tau)/D_0} (1+\tau)^{-\frac{1}{2}}\\
		&\quad \quad \quad \times (|\vec{y}|^2+1)^{-\frac{9}{4}}\frac{H(\bar{c}\tau - |\vec{y}|)}{1 + (|\vec{y}| + \sqrt{\tau})(\bar{c}\tau - |\vec{y}| + \sqrt{\tau})^2} d\vec{y} d\tau\\
		& \le CH(|\vec{x}|-\bar{c}t) (1+t)^{-\frac{5}{2}} (|\vec{x}|^2+1)^{-\frac{9}{4}} e^{-(|\vec{x}|-\bar{c}t)/(2D_0)} + CH(\bar{c}t-|\vec{x}|)\left[ (|\vec{x}|^2+1)^{-\frac{9}{4}}(1+t)^{-3}+(1+t)^{-7}\right]. 
	\end{aligned}
\end{equation}

Finally, \eqref{exponxt1} could be verified similarly as follows: 
\begin{equation} \nonumber
	\begin{aligned}
		&\int_0^t \int_{\mathbb{R}^3} |\vec{x}-\vec{y}|^{-j}  e^{-(|\vec{x}-\vec{y}|+t-\tau)/D_0} (1+
		|\vec{y}|+\sqrt{\tau})^{-2} \left((|\vec y|-\bar{c}\tau)^2 + \tau + 1 \right)^{-3} d\vec{y} d\tau\\
		& \le CH(|\vec{x}|-\bar{c}t) (|\vec{x}|^2+1)^{-1} \left[ \left(\frac{(|\vec{x}|-\bar{c}t)^2}{4(t+1)}+1 \right)^{-3} + e^{-(|\vec{x}|-\bar{c}t)/(2D_0)} \right] \int_0^t e^{-(t-\tau)/D_0}(1+\tau)^{-3} d\tau\\
		& \quad \times \int_{\mathbb{R}^3} |\vec{x}-\vec{y}|^{-j} e^{-|\vec{x}-\vec{y}|/(4D_0)}d\vec{y}\\
		& \quad +CH(\bar{c}t-|\vec{x}) \int_0^t e^{-(t-\tau)/D_0}(1+\tau)^{-5} d\tau \int_{\mathbb{R}^3} |\vec{x}-\vec{y}|^{-j} e^{-|\vec{x}-\vec{y}|/D_0}d\vec{y}\\
		& \le CH(|\vec{x}|-\bar{c}t)(|\vec{x}|^2+1)^{-1}  \left((|\vec{x}|-\bar{c}t)^2+t+1\right)^{-3}+CH(\bar{c}t-|\vec{x}|)(1+t)^{-5}.
	\end{aligned}
\end{equation}

\end{proof}

\subsection{Proof of Lemma~\ref{lemmainexpw}}
\begin{proof}
By Kirchhoff's formula \eqref{wg}, it holds that
	\begin{equation} \label{12expwwt}
		\begin{aligned}
			& e^{-\mu t} \omega(\vec{x},t;\bar{c}_f) \underset{(\vec{x},t)}{\ast} (1+t)^{-l} (|\vec{x}|^2+1)^{-d} \left((|\vec{x}|-\bar{c}t)^2+t+1\right)^{-2} \\
			& =\frac{1}{4 \pi} \int_0^t (t-\tau) e^{-\mu(t-\tau)}\int_{|\vec{y}|=1} (1+\tau)^{-l} (|\vec{x}+\bar{c}_f(t-\tau)\vec{y}|^2 + 1)^{-d}\\
			& \quad \quad \quad \times \left((|\vec{x}+\bar{c}_f(t-\tau)\vec{y}|-\bar{c}\tau)^2 + \tau + 1 \right)^{-2}dS(\vec{y}) d\tau\\
			&\le C\bigl( H(|\vec{x}|-\bar{c}t)+ H(\bar{c}t-|\vec{x}|)\bigr)\int_0^t  e^{-\mu(t-\tau)/2}\int_{|\vec{y}|=1} (1+\tau)^{-l} (|\vec{x}+\bar{c}_f(t-\tau)\vec{y}|^2 + 1)^{-d}\\
			&\quad \quad \times \left((|\vec{x}+\bar{c}_f(t-\tau)\vec{y}|-\bar{c}\tau)^2 + \tau + 1 \right)^{-2}dS(\vec{y}) d\tau\\
			&:=(\RN{1}_w)+(\RN{2}_w).
		\end{aligned}
	\end{equation}
For $(\RN{1}_w)$, since $|\vec{x}| \ge \bar{c}t$ and $|\vec{y}|=1$, if furthermore $\bar{c}_f(t-\tau) \le (|\vec{x}|-\bar{c}t)/2$, one has that
\begin{equation} \label{xcfy}
	|\vec{x}+\bar{c}_f(t-\tau)\vec{y}|-\bar{c}\tau \ge |\vec{x}|-\bar{c}_f(t-\tau) -\bar{c}t \ge (|\vec{x}|-\bar{c}t)/2, \quad 	|\vec{x}+\bar{c}_f(t-\tau)\vec{y}| \ge |\vec{x}|-\bar{c}_f(t-\tau) \ge (|\vec{x}|+\bar{c}t)/2.
\end{equation}
Otherwise, if $\bar{c}_f(t-\tau) \ge (|\vec{x}|-\bar{c}t)/2$, it holds that
\begin{equation} \label{expxy}
	e^{-\mu(t-\tau)/4} \le e^{-\mu(|\vec{x}|-\bar{c}t)/(8\bar{c}_f)}.
\end{equation}
Moreover, for both $(\RN{1}_w)$ and $(\RN{2}_w)$, it is valid that
\begin{equation} \label{wxppoly}
	\begin{aligned}
		e^{-\mu(t-\tau)/8} (|\vec{x}+\bar{c}_f(t-\tau)\vec{y}|^2 + 1)^{-d} &\le C H(\bar{c}_f(t-\tau)-|\vec{x}|/2)  e^{-\mu|\vec{x}|/(16\bar{c}_f)}\\
		& \quad + CH(|\vec{x}|/2-\bar{c}_f(t-\tau))(|\vec{x}|^2/4 + 1)^{-d}\\
		&\le C(|\vec{x}|^2 + 1)^{-d}.
	\end{aligned}
\end{equation}
The above estimates \eqref{xcfy}--\eqref{wxppoly} together with \eqref{hboun2} yield that
\begin{equation} \label{1wwt}
	\begin{aligned}
		(\RN{1}_w) &\le C H(|\vec{x}|-\bar{c}t) (|\vec{x}|^2 + 1)^{-d} \left[ \left(\frac{(|\vec{x}|-\bar{c}t)^2}{4(t+1)}+1 \right)^{-2} + e^{-\frac{\mu(|\vec{x}|-\bar{c}t)}{8\bar{c}_f}}\right] \int_0^t e^{-\mu(t-\tau)/8} (1+\tau)^{-l-2}d\tau \\
		&\le C H(|\vec{x}|-\bar{c}t)(1+t)^{-l}(|\vec{x}|^2 + 1)^{-d}  \left((|\vec{x}|-\bar{c}t)^2+t+1\right)^{-2}
	\end{aligned}
\end{equation}
and 
\begin{equation} \label{2wwt}
	\begin{aligned}
		(\RN{2}_w) &\le CH(\bar{c}t-|\vec{x}|)\left[ (|\vec{x}|^2+1)^{-d} \int_0^t e^{-\mu(t-\tau)/4}(1+\tau)^{-l-4} d\tau + \int_0^t e^{-\mu(t-\tau)/4}(1+\tau)^{-l-2-2d} d\tau \right]\\
		&\le CH(\bar{c}t-|\vec{x}|)\left[(|\vec{x}|^2+1)^{-d}(1+t)^{-l-4}+(1+t)^{-l-2-2d}\right]. 
	\end{aligned}
\end{equation}
Thus one could combine \eqref{1wwt}--\eqref{2wwt} with \eqref{12expwwt} to obtain that
\begin{equation} \label{wwte11}
	\begin{aligned}
		&e^{-\mu t} \omega(\vec{x},t;\bar{c}_f) \underset{(\vec{x},t)}{\ast} (1+t)^{-l} (|\vec{x}|^2+1)^{-d} \left((|\vec{x}|-\bar{c}t)^2+t+1\right)^{-2}\\
		 &\le C H(|\vec{x}|-\bar{c}t)(1+t)^{-l}(|\vec{x}|^2 + 1)^{-d}  \left((|\vec{x}|-\bar{c}t)^2+t+1\right)^{-2}\\
		 &\quad +CH(\bar{c}t-|\vec{x}|)\left[(|\vec{x}|^2+1)^{-d}(1+t)^{-l-4}+(1+t)^{-l-2-2d}\right]. 
	\end{aligned}
\end{equation}

For $|\vec{x}| \ge \bar{c}t$ and $|\vec{y}|=1$, if $ |\vec{x}+\bar{c}_f(t-\tau)\vec{y}|\le \bar{c}\tau$, it holds that
$$0\le |\vec x|-\bar{c}t\le  |\vec{x}+\bar{c}_f(t-\tau)\vec{y}|+\bar{c}_f(t-\tau)|\vec y|-\bar{c}t\le \left(\bar{c}_f-\bar{c}\right)(t-\tau),$$
which combined with \eqref{hboun3} and \eqref{wxppoly} yields that
\begin{equation}\nonumber
	\begin{aligned}
		& e^{-\mu t} \omega(\vec{x},t;\bar{c}_f) \underset{(\vec{x},t)}{\ast} (1+t)^{-l} (|\vec{x}|^2+1)^{-d} \frac{H(\bar{c}t - |\vec{x}|)}{1 + (|\vec{x}| + \sqrt{t})(\bar{c}t - |\vec{x}| + \sqrt{t})^2} \\
		& =\frac{1}{4 \pi} \int_0^t (t-\tau) e^{-\mu(t-\tau)}\int_{|\vec{y}|=1} (1+\tau)^{-l} (|\vec{x}+\bar{c}_f(t-\tau)\vec{y}|^2 + 1)^{-d} \\
		& \quad \quad \quad \times \frac{ H(\bar{c}\tau - |\vec{x}+\bar{c}_f(t-\tau)\vec{y}|)}{1 + (|\vec{x}+\bar{c}_f(t-\tau)\vec{y}| + \sqrt{\tau})(\bar{c}\tau - |\vec{x}+\bar{c}_f(t-\tau)\vec{y}| + \sqrt{\tau})^2} dS(\vec{y}) d\tau\\
		&\le CH(|\vec{x}|-\bar{c}t) (|\vec{x}|^2+1)^{-d} e^{-\mu(|\vec{x}|-\bar{c}t)/(4\bar{c}_f)} \int_0^t  e^{-\mu(t-\tau)/8} (1+\tau)^{-l-2} d\tau\\
		&\quad +CH(\bar{c}t-|\vec{x}|)\left[ (|\vec{x}|^2+1)^{-d} \int_0^t  e^{-\mu(t-\tau)/4} (1+\tau)^{-l-\frac{5}{2}} d\tau + \int_0^t e^{-\mu(t-\tau)/2} (1+\tau)^{-l-2-2d} d\tau\right]\\
		& \le CH(|\vec{x}|-\bar{c}t) (1+t)^{-l-2} (|\vec{x}|^2+1)^{-d} e^{-\mu(|\vec{x}|-\bar{c}t)/(4\bar{c}_f)}\\
		& \quad +CH(\bar{c}t-|\vec{x}|)\left[ (|\vec{x}|^2+1)^{-d}(1+t)^{-l-\frac{5}{2}} + (1+t)^{-l-2-2d}\right],
	\end{aligned}
\end{equation}
which together with \eqref{wwte11} verifies \eqref{exponwwt}.

Next, for the convolution
$e^{-\mu t} \omega(\vec{x},t;\bar{c}_f) \underset{(\vec{x},t)}{\ast} H(|\vec{x}| - \epsilon \sqrt{t}) (1+t)^{-5/{2}} (|\vec{x}|^2 + 1)^{-9/{4}}$ and the region $|\vec{x}| \ge \bar{c}t$, one considers the cases $\bar{c}_f(t-\tau) \le |\vec{x}|/2$ and $\bar{c}_f(t-\tau) \ge |\vec{x}|/2$ separately:
\begin{equation} \nonumber
	\begin{aligned}
		&e^{-\mu t} \omega(\vec{x},t;\bar{c}_f) \underset{(\vec{x},t)}{\ast} H(|\vec{x}| - \epsilon \sqrt{t}) (1+t)^{-\frac{5}{2}} (|\vec{x}|^2 + 1)^{-\frac{9}{4}}\\
		& =\frac{1}{4 \pi} \int_0^t (t-\tau) e^{-\mu(t-\tau)}\int_{|\vec{y}|=1} H(|\vec{x}+\bar{c}_f(t-\tau)\vec{y}|- \epsilon \sqrt{\tau}) (1+\tau)^{-\frac{5}{2}} (|\vec{x}+\bar{c}_f(t-\tau)\vec{y}|^2 + 1)^{-\frac{9}{4}} dS(\vec{y}) d\tau\\
		&\le CH(|\vec{x}|-\bar{c}t) \left[ (|\vec{x}|^2/4+1)^{-\frac{9}{4}} + e^{-\mu|\vec{x}|/(8\bar{c}_f)}\right] \int_0^t  e^{-\mu(t-\tau)/4} (1+\tau)^{-\frac{5}{2}} d\tau \int_{|\vec{y}|=1} 1 \, dS(\vec{y})\\
		&\quad + CH(\bar{c}t-|\vec{x}|) \int_0^t e^{-\mu(t-\tau)}(1+\tau)^{-\frac{5}{2}}(\epsilon^2 \tau +1)^{-\frac{9}{4}} d\tau \int_{|\vec{y}|=1} 1 \, dS(\vec{y})\\
		& \le CH(|\vec{x}|-\bar{c}t) (1+t)^{-\frac{5}{2}}(|\vec{x}|^2+1)^{-\frac{9}{4}}+ C\epsilon^{-\frac{9}{2}} H(\bar{c}t-|\vec{x}|) (1+t)^{-\frac{19}{4}},
	\end{aligned}
\end{equation}
which verifies \eqref{exponwwt0}.

Finally, for $|\vec{x}| \ge 2\epsilon \sqrt{t}$ and $|\vec{y}|=1$, if $|\vec{x}+\bar{c}_f(t-\tau)\vec{y}|\le \epsilon \sqrt{\tau}$, one has that
\begin{equation} \nonumber
	\bar{c}_f(t-\tau)\ge |\vec{x}+\bar{c}_f(t-\tau)\vec{y}|+\bar{c}_f(t-\tau)|\vec{y}|-\epsilon \sqrt{t}\ge |\vec{x}|-\epsilon \sqrt{t} \ge \frac{\epsilon \sqrt{t}}{2} + \frac{|\vec{x}|}{4},
\end{equation}
which implies
\begin{equation} \nonumber
	\begin{aligned}
		&e^{-\mu t} \omega(\vec{x},t;\bar{c}_f) \underset{(\vec{x},t)}{\ast} H(\epsilon \sqrt{t} - |\vec{x}|) (1 + t)^{-3}\\
		& =\frac{1}{4 \pi} \int_0^t (t-\tau) e^{-\mu(t-\tau)}\int_{|\vec{y}|=1} H(\epsilon \sqrt{\tau}-|\vec{x}+\bar{c}_f(t-\tau)\vec{y}|) (1+\tau)^{-3} dS(\vec{y}) d\tau\\
		& \le CH(|\vec{x}|-2\epsilon\sqrt{t}) e^{-\mu(|\vec{x}|-\epsilon\sqrt{t})/(4\bar{c}_f)} \int_0^t e^{-\mu(t-\tau)/4}(1+\tau)^{-3}d\tau\\
		& \quad +CH(2\epsilon\sqrt{t}-|\vec{x}|)\int_0^t e^{-\mu(t-\tau)/2} (1+\tau)^{-3} d\tau\\
		&\le C (1+t)^{-3} \left( H(|\vec{x}|-2\epsilon\sqrt{t}) e^{-\frac{\mu(|\vec{x}|+2\epsilon \sqrt{t})}{16\bar{c}_f}} + H(2\epsilon\sqrt{t}-|\vec{x}|) \right)
	\end{aligned}
\end{equation}
and thus verifies \eqref{exponwwt1}.
\end{proof}

\subsection{Proof of Lemma~\ref{lemmain1}}

\begin{proof}
	After direct computations, one has that
	\begin{equation} \nonumber
		\begin{aligned}
			&\int_0^t \int_{\mathbb{R}^3} |\vec{x}-\vec{y}|^{-j}  e^{-(|\vec{x}-\vec{y}|+t-\tau)/D_0} (1+\tau)^{-\frac{5}{2}} H(|\vec{y}|-\epsilon \sqrt{\tau})(|\vec{y}|^2+1)^{-\frac{9}{4}} d\vec{y} d\tau\\
			& \le CH(|\vec{x}|-\bar{c}t) (|\vec{x}|^2+1)^{-\frac{9}{4}}   \int_0^t e^{-(t-\tau)/D_0}(1+\tau)^{-\frac{5}{2}} d\tau \int_{\mathbb{R}^3} |\vec{x}-\vec{y}|^{-j} e^{-|\vec{x}-\vec{y}|/(2D_0)}d\vec{y}\\
			& \quad +CH(\bar{c}t-|\vec{x}|) \int_0^t e^{-(t-\tau)/D_0}(1+\tau)^{-\frac{5}{2}} (\epsilon^2 \tau+1)^{-\frac{9}{4}} d\tau \int_{\mathbb{R}^3} |\vec{x}-\vec{y}|^{-j} e^{-|\vec{x}-\vec{y}|/D_0}d\vec{y}\\
			& \le CH(|\vec{x}|-\bar{c}t)(1+t)^{-\frac{5}{2}}(|\vec{x}|^2+1)^{-\frac{9}{4}}+ C\epsilon^{-\frac{9}{2}} H(\bar{c}t-|\vec{x}|) (1+t)^{-\frac{19}{4}},
		\end{aligned}
	\end{equation}
which together with
\begin{equation} \nonumber
	\begin{aligned}
		&\int_0^t \int_{\mathbb{R}^3} |\vec{x}-\vec{y}|^{-j}  e^{-(|\vec{x}-\vec{y}|+t-\tau)/D_0} H(\epsilon \sqrt{\tau}-|\vec{y}|) (1+\tau)^{-3} d\vec{y}\\
		& \le C\left( H(|\vec{x}|-2\epsilon \sqrt{t}) e^{-(|\vec{x}|-\epsilon \sqrt{t})/(2D_0)} + H(2\epsilon \sqrt{t}-|\vec{x}|)\right)\int_0^t e^{-(t-\tau)/D_0}(1+\tau)^{-3} d\tau\\
		& \quad \quad \times \int_{\mathbb{R}^3} |\vec{x}-\vec{y}|^{-j} e^{-|\vec{x}-\vec{y}|/(2D_0)}d\vec{y}\\
		& \le C(1+t)^{-3}\left( H(|\vec{x}|-2\epsilon \sqrt{t}) e^{-(|\vec{x}|+2\epsilon \sqrt{t})/(8D_0)} + H(2\epsilon \sqrt{t}-|\vec{x}|)\right)
	\end{aligned}
\end{equation}
verifies \eqref{exponxt1t3}.
\end{proof}

\section*{Acknowledgement} 
The authors are supported by National Key R$\&$D Program of China (No. 2022YFA1007300).

\end{document}